\documentclass[12pt,makeidx]{memo-l}

\usepackage{graphicx,amscd,pinlabel}

\newtheorem{theorem}{Theorem}[section]        % change section to chapter to
\newtheorem{corollary}[theorem]{Corollary}    % number consecutively throughout
\newtheorem{lemma}[theorem]{Lemma}            % chapters
\newtheorem{proposition}[theorem]{Proposition}
\newtheorem*{numberlesscorollary}{Corollary}

\theoremstyle{definition}
\newtheorem{definition}[theorem]{Definition}

\newtheorem{example}[theorem]{Example}

\numberwithin{section}{chapter}
\numberwithin{figure}{chapter}

\newcommand{\longpage}{\enlargethispage{\baselineskip}}
\newcommand{\shortpage}{\enlargethispage{-\baselineskip}}

\newcommand{\abs}[1]{\vert#1\vert}            % absolute value 

\newcommand{\C}{\operatorname{{\mathbb C}}}
\newcommand{\cdim}{\operatorname{cdim}}
\newcommand{\Cinf}{\operatorname{C^\infty}}

\newcommand{\Dih}{\operatorname{Dih}}
\newcommand{\Diff}{\operatorname{Diff}}
\newcommand{\diff}{\operatorname{diff}}
\renewcommand{\dim}{\operatorname{dim}}

\newcommand{\Ecal}{\mathcal{E}}
\newcommand{\Exp}{\operatorname{Exp}}

\newcommand{\Fr}{\operatorname{Fr}}

\renewcommand{\H}{\operatorname{{\mathbb H}}}

\newcommand{\I}{\operatorname{I}}

\newcommand{\Img}{\operatorname{Img}}
\newcommand{\Imb}{\operatorname{Emb}}
\newcommand{\imb}{\operatorname{emb}}
\newcommand{\interior}{\operatorname{int}}
\newcommand{\Is}{\operatorname{{\mathcal I}}}
\newcommand{\Isom}{\operatorname{Isom}}
\newcommand{\isom}{\operatorname{isom}}
\newcommand{\Maps}{\operatorname{C^\infty}}
\newcommand{\MapsH}{\operatorname{C^\infty_H}}
\newcommand{\maps}{\operatorname{Maps}}
\newcommand{\Norm}{\operatorname{Norm}}
\newcommand{\normalizer}{\operatorname{norm}}
\renewcommand{\O}{\operatorname{\mathcal{O}}}
\newcommand{\Or}{\operatorname{O}}
\newcommand{\Ostar}{\mbox{$\Or(2)^*$}}
\newcommand{\Out}{\operatorname{Out}}
\newcommand{\orb}{\operatorname{{\mathcal O}}}
\renewcommand{\P}{\operatorname{{\mathbb P}}}
\newcommand{\R}{\operatorname{{\mathbb R}}}
\newcommand{\RP}{\operatorname{{\mathbb{RP}}}}
\newcommand{\rel}{\mbox{$\;\operatorname{rel}\;$}}
\newcommand{\set}[1]{\{ #1 \}}                % set braces { } 
\newcommand{\sA}{\textsc{a}}
\newcommand{\sB}{\textsc{b}}

\newcommand{\SO}{\operatorname{SO}}
\newcommand{\sX}{\textsc{x}}
\newcommand{\sY}{\textsc{y}}
\newcommand{\TExp}{\operatorname{TExp}}
\newcommand{\ttimes}{\operatorname{\,\widetilde{\times}\,}}
\newcommand{\vertical}{\operatorname{vert}}
\newcommand{\Vol}{\operatorname{Vol}}
\newcommand{\vbar}{\mbox{$\;\vert\;$}}        % vertical bar with wider spacing
\newcommand{\Z}{\operatorname{{\mathbb Z}}}

\newcommand{\mapdown}[1]{\big\downarrow % makes a downarrow with a mathematics
            \rlap            % mode symbol to the right
            {\smash{$\vcenter       % (for use in the "diagram" template)
            {\hbox{$         %
            \scriptstyle#1   %
            $}}$}}}           %

\newcommand{\mapright}[1]{\smash{      % makes a long rightarrow with a mathematics
            \mathop          % mode symbol above it
            {\longrightarrow % (for use in the "diagram" template)
            }\limits^{#1}}}  %

\newcommand{\indexdef}[1]{\index{#1|textbf}}
\newcommand{\indexstate}[1]{\index{#1|textit}}
\newcommand{\indexsym}[2]{\index{#1@#2}\index{ #1@#2}}
\newcommand{\indexsymdef}[2]{\index{#1@#2|textbf}\index{ #1@#2|textbf}}
\makeindex

\begin{document}

\frontmatter
\title[Diffeomorphisms of Elliptic $\mathbf{3}$-Manifolds]{Diffeomorphisms of\\
Elliptic $3$-Manifolds}

\author{Sungbok Hong}    % first author
\address{\parbox{4in}{Department of Mathematics\\ 
Korea University\medskip}}
\email{shong@korea.ac.kr}
\urladdr{math.korea.ac.kr/~shong/}
%\thanks{Partially supported by the National Science Foundation\\
%and the Sloan Foundation.}

\author{John Kalliongis}    % second author
\address{\parbox{4in}{Department of Mathematics\\ 
Saint Louis University\medskip}}
\email{kalliongisje@slu.edu}
\urladdr{mathcs.slu.edu/people/kalliongisje}

\author{Darryl McCullough}    % second author
\address{\parbox{4in}{Department of Mathematics\\ 
University of Oklahoma\medskip}}  %\\ Norman, OK 73019\medskip}}
\email{dmccullough@math.ou.edu}
\urladdr{math.ou.edu/~dmccullough/}
%\thanks{Partially supported by the National Science Foundation.}

\author{J. H. Rubinstein}
\address{\parbox{4in}{Department of Mathematics\\
University of Melbourne\medskip}}
\email{rubin@ms.unimelb.edu.au}
\urladdr{ms.unimelb.edu.au/~rubin/}

\date{\today}

\maketitle

\tableofcontents

\chapter*{Preface}

This work is ultimately directed at understanding the diffeomorphism groups
of elliptic $3$-manifolds--- those closed $3$-manifolds that admit a
Riemannian metric of constant positive curvature. The main results concern
the Smale Conjecture. The original \index{Smale Conjecture}Smale
Conjecture, proven by \index{Hatcher}A. Hatcher \cite{HSmale}, asserts that if $M$ is the
$3$-sphere with the standard constant curvature metric, the inclusion
$\Isom(M)\rightarrow \Diff(M)$ from the isometry group to the
diffeomorphism group is a homotopy equivalence. The 
\index{Generalized Smale Conjecture}\textit{Generalized Smale
Conjecture} (henceforth just called the Smale Conjecture) asserts this
whenever $M$ is an elliptic $3$-manifold.

Here are our main results:
\begin{enumerate}
\item[1.] The Smale Conjecture holds for elliptic $3$-manifolds containing
geometrically incompressible Klein bottles (Theorem~\ref{thm:one-sided}).
These include all quaternionic
and prism manifolds.
\item[2.] The Smale Conjecture holds for all lens spaces $L(m,q)$ with $m\geq
3$ (Theorem~\ref{thm:SCforLensSpaces}).
\end{enumerate}
\noindent Many of the cases in Theorem~\ref{thm:one-sided} were proven a
number of years ago \index{Ivanov}by N.~Ivanov \cite{I1c,I1,I5,I2} (see
Section~\ref{sec:SCtoday}).

Some of our other results concern the groups of diffeomorphisms
$\Diff(\Sigma)$ and fiber-preserving diffeomorphisms 
$\Diff_f(\Sigma)$ of
a Seifert-fibered Haken $3$-manifold $\Sigma$, and the coset space
$\Diff(\Sigma)/\Diff_f(\Sigma)$, which is called the
\index{space of Seifert fiberings}space of Seifert
fiberings (equivalent to the given fibering) of~$\Sigma$.
\begin{enumerate}
\item[3.] Apart from a small list of known exceptions, $\Diff_f(\Sigma)\to
\Diff(\Sigma)$ is a homotopy equivalence (Theorem~\ref{thm:DiffSigmaHE}).
\item[4.] The space of Seifert fiberings of $\Sigma$ has contractible
components (Theorem~\ref{space of sf's}), and apart from a small list of
known exceptions, it is contractible (Theorem~\ref{thm:DiffSigmaHE}).
\end{enumerate}
\noindent These may be already accepted as part of the overall
$3$-dimensional landscape, but we are unable to find any serious treatment
of them. And we have found that the development of the necessary tools and
their application to the $3$-dimensional context goes well beyond a routine
exercise.

This manuscript includes work done more than twenty years ago, as well as
work recently completed. In the mid-1980's, two of the authors (DM and JHR)
sketched an argument proving the Smale Conjecture for the $3$-manifolds
that contain one-sided Klein bottles (other than the lens space
$L(4,1)$). That method, which ultimately became Chapter~\ref{ch:one-sided}
below, underwent a long evolution as various additions were made to fill in
technical details.

The case of one-sided Klein bottles includes some lens spaces--- those of
the form $L(4n,2n-1)$ for $n\geq 2$. But for the general lens space case, a
different approach using Heegaard tori was developed by SH and DM starting
around 2000. It is based on a powerful methodology developed by JHR and
M. Scharlemann~\cite{RS}\index{Scharlemann}. It turned out that JHR was
working on the Smale Conjecture for lens spaces along exactly the same
lines as SH and DM, so the efforts were combined in the work that became
Chapter~\ref{ch:lens} below.

One more case of the Smale Conjecture may be accessible to existing
techniques. It seems likely that \index{Hatcher}A. Hatcher's approach to
the $S^3$ case in~\cite{HSmale} would also serve for $\RP^3$, but this has
yet to be carried out.

In summary, this is where the \index{Smale Conjecture!current status}Smale
\index{table!Smale Conjecture status}Conjecture now stands:
\medskip

\begin{center}
\renewcommand{\arraystretch}{1.3}
\setlength{\fboxsep}{0pt}
\setlength{\tabcolsep}{8pt}
\fbox{%
\begin{tabular}{l|l}
case&SC proven\\
\hline\hline
$S^3$&Hatcher \cite{HSmale}\\
\hline
$\RP^3$\\
\hline
lens spaces&Chapter~\ref{ch:lens} below\\
\hline
\begin{minipage}[c]{33ex}{\raggedright prism and quaternionic manifolds}\end{minipage}&%
%\parbox[c]{1.6in}{\raggedright \mbox{prism and quaternionic} manifolds}&%
\begin{minipage}[c]{25ex}{\mbox{\vrule height 2.3ex width 0ex depth 1 ex Ivanov \cite{I1,I1c,I5,I2}},
\newline \vrule height 1.8ex width 0ex depth 1ex Chapter~\ref{ch:one-sided} below}\end{minipage}\\
%\parbox[c]{1.7in}{\raggedright \mbox{Ivanov \cite{I1,I1c,I5,I2}},\\
%Chapter~\ref{ch:one-sided} below}\\
\hline
tetrahedral manifolds\\
\hline
octahedral manifolds&\\
\hline
icosahedral manifold\\
\end{tabular}}
\end{center}
\medskip

%\begin{center}
%\renewcommand{\arraystretch}{1.3}
%\setlength{\fboxsep}{0pt}
%\setlength{\tabcolsep}{8pt}
%\fbox{%
%\begin{tabular}{l|c|l}
%case&SC proven?&where proven\\
%\hline\hline
%$S^3$&Yes&Hatcher \cite{HSmale}\\
%\hline
%$\RP^3$&No&\\
%\hline
%lens spaces&Yes&Chapter~\ref{ch:lens} below\\
%\hline
%\parbox[c]{1.6in}{\raggedright prism manifolds\\
%(including quaternionic\\
%manifolds)}&Yes&\parbox[c]{1.7in}{\raggedright
%Ivanov \cite{I1,I1c,I5,I2},\\
%Chapter~\ref{ch:one-sided} below}\\
%\hline
%tetrahedral manifolds&No\\
%\hline
%octahedral manifolds&No&\\
%\hline
%icosahedral manifolds&No
%\end{tabular}}
%\end{center}
%\medskip

Our work on the Smale Conjecture requires some basic theory about spaces of
mappings of smooth manifolds, such as the fact that diffeomorphism groups
of compact manifolds and spaces of embeddings of submanifolds have the
homotopy type of CW-complexes, a result originally proven by
R. Palais\index{Palais}. This theory is well known to global analysts and
others, but not to many low-dimensional topologists. Also, most sources do
not discuss the case of manifolds with boundary, and we know of no existing
treatment of the case of fiber-preserving diffeomorphisms and embeddings,
which is the context of much of our technical work. For this reason, we
have included a fair dose of foundational material on diffeomorphism groups
in Chapter~\ref{ch:foundations}, which includes the case of manifolds with
boundary, with the additional boundary control that we will need.

A more serious gap in the literature is the absence of versions of the
fundamental restriction fibration theorems of \index{Palais}Palais and
\index{Cerf}Cerf in the context of fibered (and Seifert-fibered)
manifolds. These extensions of the well-known theory require some new
ideas, which were developed by JK and DM and form most of
Chapter~\ref{ch:Palais}. We work in a class of singular fiberings large
enough to include all Seifert fiberings of $3$-manifolds, except some
fiberings of lens spaces. These results are heavily used in our work in
Chapters~\ref{ch:one-sided} and~\ref{ch:lens}. Our results on
fiber-preserving diffeomorphisms and the space of fibered structures of a
Seifert-fibered Haken $3$-manifold are applications of this work, and also
appear in Chapter~\ref{ch:Palais}.

As a final note, we mention that much of our work here is unusually
detailed and technical. In considerable part, this is inherent
complication, but it also reflects the fact that over the years we have
filled in many arguments in response to recommendations from various
readers. Unfortunately, one reader's ``too sketchy'' can be another's ``too
much elaboration of well-known facts'', and personally we find some of the
current exposition to be somewhat too long and too detailed. To provide an
alternative, we have included Sections~\ref{sec:results}
and~\ref{sec:outline}, which are overviews of the proofs of the main
results. In the actual proofs, we trust that each reader will simply accept
the ``obvious'' parts and focus on the ``nontrivial'' parts, whichever they
may~be.

The authors are grateful to many sources of support during the lengthy
preparation of this work. These include the Australian Research Council, the
Korea Research Foundation, the Basic Science Research Center of Korea
University, Saint Louis University, the U. S. National Science Foundation,
the Mathematical Sciences Research Institute, the University of Oklahoma
Vice President for Research, and the University of Oklahoma College of Arts
and Sciences. We also thank the referees of versions of this work for
occasional corrections and many helpful suggestions.

\mainmatter

\chapter{Elliptic $3$-manifolds and the Smale Conjecture}
\label{ch:intro}

As noted in the Preface, the \index{Smale Conjecture}Smale Conjecture is
the assertion that the inclusion $\Isom(M)\to \Diff(M)$ is a homotopy
equivalence whenever $M$ is an elliptic $3$-manifold, that is, a
$3$-manifold admitting a Riemannian metric of constant positive
curvature. The \index{Geometrization Conjecture}Geometrization Conjecture,
now proven by \index{Perelman}Perelman, shows that all closed $3$-manifolds
with finite fundamental group are elliptic.

In this chapter, we will first review elliptic $3$-manifolds and their
isometry groups. In the second section, we will state our main results on
the Smale Conjecture, and provide some historical context. In the final two
sections, we discuss isometries of nonelliptic $3$-manifolds, and address
the possibility of applying Perelman's methods to the Smale Conjecture.

\section{Elliptic $3$-manifolds and their isometries}
\label{sec:isometries}

The \indexdef{elliptic} elliptic $3$-manifolds were completely classified
long ago. They are exactly the $3$-manifolds whose universal cover can be
uniformized as the unit sphere $S^3$ in $\R^4$ so that $\pi_1(M)$ acts
freely as a subgroup of $\Isom_+(S^3)=\SO(4)$. The subgroups of $\SO(4)$
that act freely were first determined by Hopf and Seifert-Threlfall, and
reformulated using \index{quaternions}quaternions by 
\index{Hattori}Hattori. References include \cite{Wolf}
(pp.~226-227), \cite{Orlik} (pp.~103-113), \cite{Scott} (pp.\ 449-457),
\cite{Sakuma}, and~\cite{M}.

The isometry groups of elliptic $3$-manifolds have also been known for a
long time, and are topologically rather simple: they are compact Lie groups
of dimension at most $6$. A detailed calculation of the isometry groups of
elliptic $3$-manifolds was given in \cite{M}, and in this section we will
recall the resulting groups.

To set notation, recall that there is a well-known $2$-fold covering
$S^3\to \SO(3)$, which is a homomorphism when $S^3$ is regarded as the
group of unit quaternions (see Section~\ref{isometry} for a fuller
discussion). The elements of $\SO(3)$ that preserve a given axis, say the
$z$-axis, form the orthogonal subgroup
\indexsymdef{O(2)}{$\Or(2)$}$\Or(2)$. We will denote by
\indexsymdef{O(2)star}{$\Or(2)^*$}$\Or(2)^*$ the inverse image in $S^3$ of
$\Or(2)$. When $H_1$ and $H_2$ are groups, each containing $-1$ as a central
involution, the quotient $(H_1\times H_2)/\langle (-1,-1)\rangle$ is
denoted by \indexsymdef{ttimes}{$\protect\ttimes$} $H_1\ttimes H_2$. In particular,
$\SO(4)$ itself is $S^3\ttimes S^3$, and contains the subgroups
\indexsymdef{S1timesS3}{$S^1\protect\ttimes S^3$}$S^1\ttimes S^3$,
\indexsymdef{OstartimesOstar}{$\protect\Ostar\protect\ttimes\protect\Ostar$}$\Ostar\ttimes\Ostar$, and
\indexsymdef{S1timesS1}{$S^1\protect\ttimes S^1$}$S^1\ttimes S^1$.  The latter is
isomorphic to $S^1\times S^1$, but it is sometimes useful to distinguish
between them. Finally, \indexsymdef{Dih}{$\protect\Dih(S^1\times
S^1)$}$\Dih(S^1\times S^1)$ is the semidirect product $(S^1\times
S^1)\circ C_2$, where $C_2$ acts by complex conjugation in both factors.

There are $2$-fold covering homomorphisms
\begin{equation*}
\Ostar\times \Ostar\to
\Ostar\ttimes\Ostar\to
\Or(2)\times \Or(2)\to
\Or(2)\ttimes\Or(2)\ .
\end{equation*}
Each of these groups is diffeomorphic to four disjoint copies of the torus,
but they are pairwise nonisomorphic. Indeed, they are easily distinguished
by examining their subsets of order~$2$ elements.  Similarly, $S^1\times
S^3$ and $S^1\ttimes S^3$ are diffeomorphic, but nonisomorphic.

Table~\ref{tab:isometries} gives the isometry groups of the elliptic
$3$-manifolds with non-cyclic fundamental group.  The first column, $G$,
indicates the fundamental group of $M$, where \indexsymdef{Cm}{$C_m$}$C_m$
denotes a cyclic group of order $m$, and
\indexsymdef{Dstar4m}{$D^*_{4m}$}$D^*_{4m}$,
\indexsymdef{Tstar24}{$T^*_{24}$}$T^*_{24}$,
\indexsymdef{Ostar48}{$O^*_{48}$}$O^*_{48}$, and
\indexsymdef{Istar120}{$I^*_{120}$}$I^*_{120}$ are the binary
dihedral, tetrahedral, octahedral, and icosahedral groups of the indicated
orders. The groups called \indexdef{diagonal subgroup}index $2$ and 
index
$3$ diagonal are certain subgroups of $D^*_{4m}\times C_{4m}$ and
$T^*_{24}\times C_{6n}$ respectively. The last two columns give the full
isometry group $\Isom(M)$, and the group
\indexsymdef{IM}{$\protect\mathcal{I}(M)$}$\mathcal{I}(M)$ of path components of
$\Isom(M)$.

Section~\ref{isometry} contains the detailed calculation of 
\indexsymdef{isomM}{$\protect\isom(M)$}$\isom(M)$, the
connected component of $\mathrm{id}_M$ in $\Isom(M)$, for the elliptic
$3$-manifolds that contain one-sided incompressible Klein bottles--- the
quaternionic and prism manifolds, and the lens spaces of the form
$L(4n,2n-1)$--- since the notation and some of the mechanics of this are
needed for the arguments in Chapter~\ref{ch:one-sided}.

Table~\ref{tab:lens spaces} gives the isometry groups of the elliptic
$3$-manifolds with cyclic fundamental group. These are the $3$-sphere
$L(1,0)$, real projective space $L(2,1)$, and the lens spaces 
\indexsymdef{L(m,q)}{$L(m,q)$}$L(m,q)$ with
$m\geq 3$.
\index{table!isometry groups}%
\begin{table}
\begin{small}
\renewcommand{\arraystretch}{1.5}
\setlength{\tabcolsep}{2 ex}
\setlength{\fboxsep}{0pt}
\fbox{%
\begin{tabular}{l|l|l|l}
$G$&$M$&$\Isom(M)$&$\Is(M)$\\
\hline
\hline
$Q_8$&quaternionic&$\SO(3)\times S_3$&$S_3$\\ 
\hline 
$Q_8\times C_n$&quaternionic&$\Or(2)\times S_3$&$C_2\times S_3$\\ 
\hline 
$D_{4m}^*$&prism&$\SO(3)\times  C_2$&$C_2$\\ 
\hline 
$D_{4m}^*\times C_n$&prism&$\Or(2)\times C_2$&$C_2\times C_2$\\ 
\hline 
index $2$ diagonal&prism&$\Or(2)\times C_2$&$C_2\times C_2$\\ 
\hline
$T_{24}^*$&tetrahedral&$\SO(3)\times C_2$&$C_2$\\ 
\hline 
$T_{24}^*\times C_n$&tetrahedral&$\Or(2)\times C_2$&$C_2\times C_2$\\ 
\hline 
index $3$ diagonal&tetrahedral&$\Or(2)$&$C_2$\\ 
\hline 
$O_{48}^*$&octahedral&$\SO(3)$&$\set{1}$\\ 
\hline 
$O_{48}^*\times C_n$&octahedral&$\Or(2)$&$C_2$\\ 
\hline 
$I_{120}^*$&icosahedral&$\SO(3)$&$\set{1}$\\ 
\hline 
$I_{120}^*\times C_n$&icosahedral&$\Or(2)$&$C_2$\\
\end{tabular}}
\end{small}
\bigskip
\caption{Isometry groups of $M=S^3/G$\hspace{2 ex}($m>2$, $n>1$)}
\label{tab:isometries}
\end{table}

\index{table!isometry groups}%
\begin{table}
\begin{scriptsize}
\renewcommand{\arraystretch}{1.5}
\newlength{\minipagewidth}%
\setlength{\tabcolsep}{1.5 ex}
\setlength{\fboxsep}{0pt}
\fbox{%
\begin{tabular}{l|l|l}
$m$, $q$&$\Isom(L(m,q))$&$\Is(L(m,q))$\\
\hline
\hline
$m=1$ ($L(1,0)=S^3$)&$\Or(4)$&$C_2$\\
\hline
$m=2$ ($L(2,1)=\RP^3$)&$(\SO(3)\times \SO(3))\circ C_2$&$C_2$\\
\hline
\settowidth{\minipagewidth}{$m>2$, $m$ even}%
\begin{minipage}{\minipagewidth}%
\noindent $m>2$, $m$ odd,\par\end{minipage} $q=1$&$\Ostar\ttimes S^3$&$C_2$\\  
\hline
$m>2$, $m$ even, $q=1$&$\Or(2)\times \SO(3)$&$C_2$\\
\hline
$m>2$, $1<q<m/2$, $q^2\not\equiv\pm1\bmod{m}$&$\Dih(S^1\times S^1)$&$C_2$\\
\hline
$m>2$, $1<q<m/2$, $q^2\equiv-1\bmod{m}$&$(S^1\ttimes S^1)\circ C_4$&$C_4$\\
\hline
$m>2$,
\settowidth{\minipagewidth}{$\gcd(m,q+1)\gcd(m,q-1)=m$}%
\begin{minipage}[t]{\minipagewidth}%
\noindent $1<q<m/2$, $q^2\equiv1\bmod{m}$,\par
\noindent $\gcd(m,q+1)\gcd(m,q-1)=m$\rule[-1.2 ex]{0mm}{0mm}\par%
\end{minipage}%
&$\Or(2)\ttimes \Or(2)$&$C_2\times C_2$\\
\hline
$m>2$,
\settowidth{\minipagewidth}{$\gcd(m,q+1)\gcd(m,q-1)=2m$}%
\begin{minipage}[t]{\minipagewidth}%
\noindent $1<q<m/2$, $q^2\equiv1\bmod{m}$,\par
\noindent $\gcd(m,q+1)\gcd(m,q-1)=2m$\rule[-1.2 ex]{0mm}{0mm}\par%
\end{minipage}%
&$\Or(2)\times \Or(2)$&$C_2\times C_2$
\end{tabular}}
\end{scriptsize}
\bigskip
\caption{Isometry groups of $L(m,q)$}
\label{tab:lens spaces}
\end{table}

\newpage\longpage
\section{The Smale Conjecture}
\label{sec:SCtoday}\index{Smale Conjecture}

S. Smale \cite{Smale}\index{Smale} proved that for the standard round
$2$-sphere $S^2$, the inclusion of the isometry group $\Or(3)$ into the
diffeomorphism group $\Diff(S^2)$ is a homotopy equivalence. He conjectured
that the analogous result holds true for the $3$-sphere, that is, that
$\Or(4)\to \Diff(S^3)$ is a homotopy
equivalence. J.~Cerf~\cite{CerfS3}\index{Cerf} proved that the inclusion
induces a bijection on path components, and the full conjecture was proven
by \index{Hatcher}A.~Hatcher \cite{HSmale}.

A weak form of the (generalized) Smale Conjecture is known. In \cite{M},
the calculations of $\Isom(M)$ for elliptic $3$-manifolds are combined with
results on mapping class groups of many authors, including
\cite{A,Boileau-Otal,Bonahon,R2,R-B}, to obtain the following statement:
\begin{theorem}
Let $M$ be an elliptic $3$-manifold. Then the inclusion of $\Isom(M)$ into
$\Diff(M)$ is a bijection on path components.\par
\label{thm:pi0 Smale}
\end{theorem}\indexstate{Smale Conjecture!pi zero part@$\pi_0$-part}
\noindent 
This can be called the ``$\pi_0$-part'' 
of the Smale Conjecture. By virtue of
this result, to prove the Smale Conjecture for any elliptic $3$-manifold,
it is sufficient to prove that the inclusion $\isom(M)\to\diff(M)$ of the
connected components of the identity map in $\Isom(M)$ and $\Diff(M)$ is a
homotopy equivalence.
\longpage

The earliest work on the Smale Conjecture was \index{Ivanov}by N. Ivanov. Certain
elliptic $3$-manifolds contain one-sided geometrically incompressible Klein
bottles. Fixing such a Klein bottle $K_0$, called the base Klein bottle,
the remainder of the $3$-manifold is an open solid torus, and (up to
isotopy) there are two Seifert fiberings, one for which the Klein bottle is
fibered by nonsingular fibers (the 
\index{meridional!fibering of $K_0$}``meridional'' 
fibering), and one for
which it contains two exceptional fibers of type $(2,1)$ (the
\index{longitudinal!fibering of $K_0$}``longitudinal'' fibering).
As will be detailed in Section~\ref{sec:M(m,n)} below, the manifolds then
fall into four types:
\begin{enumerate}
\item[I.] Those for which neither the meridional nor the longitudinal
fibering is nonsingular on the complement of $K_0$.
\item[II.] Those for which only the longitudinal fibering is nonsingular on
the complement of $K_0$. These are the lens spaces $L(4n,2n-1),\;n\geq2$.
\item[III.] Those for which only the meridional fibering is nonsingular on
the complement of $K_0$.
\item[IV.] The lens space $L(4,1)$, for which both the meridional and
longitudinal fiberings are nonsingular on the complement of~$K_0$.
\end{enumerate}
\noindent Cases~I and~III are the quaternionic and prism manifolds.

\index{Ivanov}Ivanov announced the Smale Conjecture for Cases~I and~II in \cite{I1c,I1},
and gave a detailed proof for Case~I in \cite{I5,I2}. One of our main
theorems extends those results to all cases:
\begin{theorem}[Smale Conjecture for elliptic $3$-manifolds containing
incompressible Klein bottles] Let $M$ be an elliptic $3$-manifold
containing a geometrically incompressible Klein bottle. Then $\Isom(M)\to
\Diff(M)$ is a homotopy equivalence.
\label{thm:one-sided}\indexstate{Smale Conjecture!for incompressible Klein bottle case}
\end{theorem}
\noindent Theorem~\ref{thm:one-sided} is proven in
Chapter~\ref{ch:one-sided}, except for the case of $L(4,1)$, which is
proven in Chapter~\ref{ch:lens}.

Our second main result concerns lens spaces, which for us refers only to
the lens spaces $L(m,q)$ with $m\geq 3$:

\begin{theorem}[Smale Conjecture for lens spaces]\indexstate{Smale
Conjecture!for lens spaces}
For any lens space $L$, the inclusion $\Isom(L)\to \Diff(L)$ is a homotopy
equivalence.
\label{thm:SCforLensSpaces}
\end{theorem}

\index{DiffM@$\Diff(M)$!homeomorphism type|(}One consequence of the Smale
Conjecture is the determination of the homeomorphism type of $\Diff(M)$.
Recall that a \indexdef{Frechet space@Fr\'echet space}Fr\'echet space is a
locally convex complete metrizable linear space.  In
Section~\ref{sec:Cinfinity}, we will review the fact that if $M$ is a
closed smooth manifold, then with the $\Cinf$-topology, $\Diff(M)$ is a
separable infinite-dimensional manifold locally modeled on the Fr\'echet
space of smooth vector fields on $M$.  By the \index{Anderson-Kadec
Theorem}Anderson-Kadec Theorem \cite[Corollary VI.5.2]{BP}, every
infinite-dimensional separable Fr\'echet space is homeomorphic to
$\R^\infty$, the countable product of lines. A theorem of Henderson and
Schori (\cite[Theorem~IX.7.3]{BP}, originally announced in \cite{HS}) shows
that if $Y$ is any locally convex space with $Y$ homeomorphic to
$Y^\infty$, then manifolds locally modeled on $Y$ are homeomorphic whenever
they have the same homotopy type. Therefore our main theorems give
immediately the homeomorphism type of $\Diff(M)$:

\begin{numberlesscorollary} Let $M$ be an elliptic $3$-manifold which
either contains an incompressible Klein bottle or is a lens space $L(m,q)$
with $m\geq 3$. Then $\Diff(M)$ is homeomorphic to $\Isom(M)\times
\R^\infty$.\par
\end{numberlesscorollary}

\noindent Combining this with the calculations of $\Isom(M)$ in
Table~\ref{tab:isometries} gives the following homeomorphism classification
of $\Diff(M)$, in which $P_n$ denotes the discrete space with $n$ points:
\begin{numberlesscorollary} Let $M$ be an elliptic $3$-manifold, not a lens
space, containing an incompressible Klein bottle.
\begin{enumerate}
\item If $M$ is the quaternionic manifold with fundamental group
$Q_8=D^*_8$, then $\Diff(M) \approx P_6 \times \SO(3)\times  \R^\infty$.
\item If $M$ is a quaternionic manifold with fundamental group
$Q_8\times C_n$, $n> 2$, then $\Diff(M) \approx P_{12}\times S^1
\times \R^\infty$.
\item If $M$ is a prism manifold with fundamental group
\indexsym{Dstar4m}{$D^*_{4m}$}$D^*_{4m}$, $m\geq 3$, then $\Diff(M) \approx P_2\times \SO(3)\times \R^\infty$.
\item If $M$ is any other prism manifold,
then $\Diff(M) \approx P_4\times S^1\times \R^\infty$.
\end{enumerate}
\end{numberlesscorollary}

\noindent
Similarly, using Table~\ref{tab:lens spaces}, we obtain a complete
classification of $\Diff(L)$ for lens spaces into four homeomorphism types:
\begin{numberlesscorollary}For a lens space $L(m,q)$ with $m\geq 3$, 
the homeomorphism type of $\Diff(L)$ is as follows:
\begin{enumerate}
\item For $m$ odd, $\Diff(L(m,1))\approx P_2 \times S^1 \times S^3 \times
\R^\infty$.
\item For $m$ even, $\Diff(L(m,1))\approx P_2 \times S^1 \times \SO(3) \times
\R^\infty$. 
\item For $q>1$ and $q^2\not\equiv \pm 1\pmod{m}$,
$\Diff(L(m,q))\approx P_2 \times S^1 \times S^1 \times \R^\infty$.
\item For $q>1$ and $q^2\equiv \pm 1\pmod{m}$,
$\Diff(L(m,q))\approx P_4 \times S^1 \times S^1 \times \R^\infty$.
\end{enumerate}
\end{numberlesscorollary}\indexstate{homeomorphism type of
  DiffM@homeomorphism type of $\Diff(M)$}

We remark that the homeomorphism classification is quite different from the
isomorphism classification. In fact, for \textit{any} smooth manifold, the
\textit{isomorphism type}\index{DiffM@$\Diff(M)$!isomorphism type} of
$\Diff(M)$ determines $M$. That is, an abstract isomorphism between the
diffeomorphism groups of two differentiable manifolds must be induced by a
diffeomorphism between the manifolds~\cite{Banyaga, Filip,
Takens}.\index{DiffM@$\Diff(M)$!homeomorphism type|)}

The Smale Conjecture has some other applications, beyond the problem of
understanding $\Diff(M)$. \index{Ivanov}Ivanov's results were used in \cite{F-W} to
construct examples of homeomorphisms of reducible $3$-manifolds that are
homotopic but not isotopic. Our results show that the construction applies
to a larger class of $3$-manifolds. In \cite{P-R},
Theorem~\ref{thm:one-sided} was applied to the classification problem for
$3$-manifolds which have metrics of positive Ricci curvature and universal
cover $S^3$.

\index{Smale Conjecture!in physics}The Smale Conjecture has attracted the
interest of physicists studying the theory of quantum gravity. Certain
physical configuration spaces can be realized as the quotient space of a
principal $\Diff_1(M,x_0)$-bundle with contractible total space, where
$\Diff_1(M,x_0)$ denotes the subgroup of $\Diff(M,x_0)$ that induce the
identity on the tangent space to $M$ at $x_0$. (This group is homotopy
equivalent to $\Diff(M\# D^3\rel \partial D^3)$.) Consequently the loop
space of the configuration space is weakly homotopy equivalent to
$\Diff_1(M,x_0)$. Physical significance of $\pi_0(\Diff(M))$ for quantum
gravity was first pointed out in \cite{F-S}. See also \cite{ABBJRS, G-L, I,
  S3, W2}. The physical significance of some higher homotopy groups of
$\Diff(M)$ was examined by D. Giulini~\cite{Giulini}.

\section{Isometries of nonelliptic $3$-manifolds}
\label{sec:others}\index{DiffM@$\Diff(M)$!nonelliptic 3-manifolds|(}

For Haken $3$-manifolds, \index{Hatcher}Hatcher (\cite{H}, combined with~\cite{HSmale}),
extending earlier work of \index{Laudenbach}Laudenbach \cite{L}, proved
that the components of $\Diff(M \rel \partial M)$ have the expected
homotopy types (contractible, except when $\partial M$ is empty, in which
case they are homotopy equivalent to $(S^1)^k$, where $k$ is the rank of
the center of $\pi_1(M)\,$). The same was accomplished by
\index{Ivanov}Ivanov~\cite{I3}. The main part of the argument is to show that the space
of embeddings of a two-sided incompressible surface $F$ that are disjoint
from a parallel copy of $F$ is a deformation retract of the space of all
embeddings of~$F$ (isotopic to the inclusion relative to $\partial F$).
Using his ``insulator'' methodology, D. Gabai \cite{Gabai}\index{Gabai} proved that the
components of $\Diff(M)$ are contractible for all hyperbolic $3$-manifolds.

The analogue of the Smale Conjecture holds for (compact) $3$-manifolds
whose interiors have constant negative curvature and finite volume: in
fact, D.~Gabai~\cite{Gabai} showed that both $\Isom(M)\to \Diff(M)$ and
$\Diff(M)\to \Out(\pi_1(M))$ are homotopy equivalences for finite-volume
hyperbolic $3$-manifolds (for hyperbolic $3$- manifolds that are also
Haken, this was already known by Mostow Rigidity, \index{Waldhausen's
  Theorem}Waldhausen's Theorem, and the work of \index{Hatcher}Hatcher and
Ivanov already discussed). The same statement has also been
proven~\cite{McC-Soma} when $M$ has an $\mathbb{H}^2\times \R$ or
$\widetilde{\mathrm{SL}}_2(\R)$ geometry and its (unique, up to isotopy)
Seifert-fibered structure has base orbifold the $2$-sphere with three cone
points. This is expected to hold for the $\mathrm{Nil}$ geometry as well.

In contrast, when the manifold has interior of constant negative curvature
and infinite volume, or has constant zero curvature, a diffeomorphism will
not in general be isotopic to an isometry (said differently, the
diffeomorphism group may have more components than the isometry
group). Even in these cases, however, Waldhausen's Theorem and Hatcher's
work show that for a maximally symmetric Riemannian metric on $M$,
$\Isom(M)\to \Diff(M)$ is a homotopy equivalence when one restricts to the
connected components of~$\mathrm{id}_M$.\index{DiffM@$\Diff(M)$!nonelliptic 3-manifolds|)}

\section{Perelman's methods}
\label{sec:Perelman}\index{Perelman|(}

It is natural to ask whether the Smale Conjecture can be proven using the
methodology that G.~Perelman developed to prove the 
\index{Geometrization Conjecture}Geometrization
Conjecture. The Smale Conjecture would follow if there were a flow
retracting the space $\mathcal{R}$ of all Riemannian metrics on an elliptic
$3$-manifold $M$ to the subspace $\mathcal{R}_c$ of metrics of constant
positive curvature. Here is why this is so. First, note that by rescaling,
$\mathcal{R}_c$ deformation retracts to the subspace $\mathcal{R}_1$ of
metrics of constant curvature~$1$. Now, $\Diff(M)$ acts by pullback on
$\mathcal{R}_1$; this action is transitive (given two constant curvature
metrics on $M$, the developing map gives a diffeomorphism which is an
isometry between the lifted metrics on the universal cover, and since the
action of $\pi_1(M)$ is known to be unique up to conjugation by an
isometry, this diffeomorphism can be composed with some isometry to make it
equivariant) and the stabilizer of each point is a subgroup conjugate to
$\Isom(M)$, so $\mathcal{R}_1$ may be identified with the coset space
$\Isom(M)\backslash\Diff(M)$. On the other hand, $\mathcal{R}$ is
contractible ($M$ is parallelizable and one can use a Gram-Schmidt
orthonormalization process). So the existence of a flow retracting
$\mathcal{R}$ to $\mathcal{R}_c$ would imply that
$\Isom(M)\backslash\Diff(M)$ is contractible, which is equivalent to the
Smale Conjecture. Finding a flow that retracts $\mathcal{R}$ to
$\mathcal{R}_c$ is, of course, the rough idea of the Hamilton-Perelman
program. At the present time, however, we do not see any way to carry this
out, due to the formation of singularities and the requisite surgery of
necks, and we are unaware of any progress in this direction.\index{Perelman|)}

\chapter{Diffeomorphisms and embeddings of manifolds}
\label{ch:foundations}

This chapter contains foundational material on spaces of diffeomorphisms
and embeddings. Such spaces are known to be Fr\'echet manifolds, separable
when the manifolds involved are compact. We will need versions of these and
related facts for manifolds with boundary, and also in the context of
fiber-preserving diffeomorphisms and maps. For the latter, a new (to us, at
least) idea is required--- the aligned exponential introduced in
Section~\ref{exponent}. It will also be heavily used in
Chapter~\ref{ch:Palais}.

Two convenient references for Fr\'echet spaces and Fr\'echet manifolds are
R. Hamilton~\cite{Hamilton} and A. Kriegl and
P. Michor~\cite{Kriegl-Michor}.

This is a good time to introduce some of our notational conventions. Spaces
of mappings will usually have names beginning with capital letters, such as
the diffeomorphism group $\Diff(M)$ or the space of embeddings $\Imb(V,M)$
of a submanifold of $M$. The same name beginning with a small
\index{small-letter convention}\index{lower-case convention}letter, as in
\indexsymdef{diff(M)}{$\diff(M)$}$\diff(M)$ or
\indexsymdef{imb(M)}{$\imb(M)$}$\imb(V,M)$, will indicate the path
component of the identity or inclusion map. We also use
\indexsymdef{I}{$\protect\mathrm{I}$}$\I$ to denote the standard unit
interval~$[0,1]$.

\section[The $\Cinf$-topology]
{The $\Cinf$-topology}
\label{sec:Cinfinity}

For now, let $M$ be a manifold with empty boundary. Throughout our work, we
will use the $\Cinf$-topology on $\Diff(M)$. For this topology,
$\Diff(M)$ is a Fr\'echet manifold, locally diffeomorphic to the Fr\'echet
space $\mathcal{Y}(M,TM)$\indexsymdef{Y(M,TM)}{$\protect\mathcal{Y}(M,TM)$}
of smooth vector fields on $M$. In fact, the
space \indexsymdef{CinfinityMN}{$\Maps(M,N)$}$\Maps(M,N)$ 
of smooth maps from $M$ to $N$ is a Fr\'echet
manifold, metrizable when $M$ is compact (see for example Theorem~42.1 and
Proposition~42.3 of~\cite{Kriegl-Michor}), and $\Diff(M)$ is an open subset
of $\Maps(M,M)$ (Theorem~43.1 of~\cite{Kriegl-Michor}).

When $M$ is compact, or more generally when one is working with maps and
sections supported on a fixed compact subset of $M$, $\mathcal{Y}(M,TM)$ is
a separable Fr\'echet space. By Theorem~II.7.3 of \cite{BP}, originally
announced in \cite{HS}, manifolds modeled on a separable Fr\'echet space
$Y$ are homeomorphic whenever they have the same homotopy
type. Theorem~IX.7.1 of \cite{BP} (originally Theorem~4 of
\cite{Henderson}) shows that $\Diff(M)$ admits an open embedding into
$Y$. Theorems II.6.2 and II.6.3 of \cite{BP} then show that $\Diff(M)$ has
the homotopy type of a \index{DiffM@$\Diff(M)$!CW-complex}CW-complex. (As
far as we know, this fact is due originally to \index{Palais}Palais
\cite{Palais}; he showed that many infinite-dimensional manifolds are
dominated by CW-complexes, but a space dominated by a CW-complex is
homotopy equivalent to some CW-complex~\cite[Theorem IV.3.8]{LW}.)

\section{Metrics which are products near the boundary}
\label{metrics}

We are going to work extensively with manifolds with boundary, and will
need special Riemannian metrics on them, which we develop in this section.

Recall that a Riemannian metric is called
\textit{complete}\indexdef{Riemannian metric!complete} if every Cauchy
sequence converges. For a complete Riemannian metric on $M$, a geodesic can
be extended indefinitely unless it reaches a point in the boundary of $M$,
where it may continue or it may fail to be extendible because it ``runs out
of the manifold.''

\begin{definition}
A Riemannian metric on $M$ is said to be a 
\indexdef{product near the boundary}\textit{product near the
boundary}\indexdef{Riemannian metric!product near the boundary} if there is
a collar neighborhood $\partial M\times \I$ of the boundary on which the
metric is the product of a complete metric on $\partial M$ and the standard
metric on $\I$.\par
\label{def:product_metric}
\end{definition}
\noindent Note that when the metric is a product near the boundary, the
exponential of any vector tangent to $\partial M$ is a point in $\partial
M$.

Given any collar $\partial M\times [0,2]$, it is easy to obtain a metric
that is a product near the boundary of $M$. On $\partial M\times [0,2)$,
fix a Riemannian metric that is the product of a metric on $\partial M$ and
the usual metric on $[0,2)$. Obtain the metric on $M$ from this metric and
any metric defined on all of $M$ by using a partition of unity subordinate
to the open cover $\set{\partial M\times[0,2), M-\partial M\times \I}$.

By a submanifold $V$ of $M$, we mean a smooth submanifold.  When $M$ has
boundary and $\dim(V)<\dim(M)$, we always require that $V$ be properly
embedded in the sense that $V\cap \partial M=\partial V$, and that every
inward pointing tangent vector to $V$ at a point in $\partial V$ be also
inward pointing in~$M$.

We will often work with codimension-$0$ submanifolds of bounded
manifolds. In that case, the submanifold is a \textit{manifold with
  corners,}\index{corners}\indexdef{manifold with corners} that is, locally diffeomorphic
to a product of half-lines and lines. In fact, all of our work should
extend straightforwardly into the full context of manifolds with corners,
but for simplicity we restrict to the cases we will need. When $V$ has
codimension~0, we require that the frontier of $V$ be a codimension-1
submanifold of $M$ as above.

\begin{definition}\indexdef{submanifold!meeting collar in $\protect\I$-fibers}
Suppose that the Riemannian metric on $M$ is a product near the boundary,
with respect to the collar $\partial M\times \I$.  A submanifold $V$ of $M$
is said to \textit{meet the collar $\partial M\times \I$ in $\I$-fibers} when
$V\cap \partial M\times \I$ is a union of $\I$-fibers of $\partial M\times
\I$.\par
\label{def:meet_the_collar}
\end{definition}
\noindent Note that when $V$ meets the collar of $M$ in $\I$-fibers, the
normal space to $V$ at any point $(x,t)$ in $\partial M\times \I$ is
contained in the subspace in $T_xM$ tangent to $\partial M\times
\{t\}$. Consequently, if one exponentiates the $(<\epsilon)$-length vectors
in the normal bundle to obtain a tubular neighborhood of $V$, then the
fiber at $(x,t)$ is contained in $\partial M\times \{t\}$.

Given a submanifold $V$, one may obtain a complete metric on $M$ that is a
product near $\partial M$ and such that $V$ meets the collar $\partial
M\times \I$ in $\I$-fibers as follows. First, obtain a collar of $\partial M$
that $V$ meets in $\I$-fibers, by constructing an inward-pointing vector field
on a neighborhood of $\partial M$ which is tangent to $V$, using the
integral curves associated to the vector field to produce the collar, then
carrying out the previous construction to obtain a metric that is a product
near the boundary for this collar. It is complete on the collar. To make it
complete on all of $M$, define $f\colon M-\partial M\to (0,\infty)$ by
putting $f(x)$ equal to the supremum of the values of $r$ such that $\Exp$
is defined on all vectors in $T_x(M)$ of length less than~$r$. Let $g\colon
M-\partial M\to (0,\infty)$ be a smooth map that is an
$\epsilon$-approximation to $1/f$, and let $\phi\colon M\to[0,1]$ be a
smooth map which is equal to~$0$ on $\partial M\times \I$ and is~$1$ on
$M-\partial M\times[0,2)$. Give $M\times[0,\infty)$ the product metric, and
define a smooth embedding $i\colon M\to M\times[0,\infty)$ by $i(x)=
(x,\phi(x)g(x))$ if $x\notin\partial M$ and $i(x)= (x,0)$ if $x\in
\partial M$. The restricted metric on $i(M)$ agrees with the product metric
on $\partial M\times \I$ and is complete.

We will always assume that Riemannian metrics have been chosen to be
complete.

\section{Manifolds with boundary}
\label{sec:manifolds_with_boundary}

In this section, we will extend the results of Section~\ref{sec:Cinfinity}
to the bounded case. We always assume that $M$ has a Riemannian metric
which is a product near the boundary for some collar $\partial M\times \I$.
\begin{definition}\label{def:Y}
Let $V$ be a submanifold of $M$. By
\indexsymdef{YTMV}{$\mathcal{Y}(V,TM)$, $\mathcal{Y}^L(V,TM)$}$\mathcal{Y}(V,TM)$ 
we denote the
Fr\'echet space of all sections from $V$ to the restriction of the
tangent bundle of $M$ to $V$. The zero section of $\mathcal{Y}(V,TM)$ is
denoted by \indexsymdef{Z}{$Z$}$Z$. 
For $L\subseteq M$, we denote by $\mathcal{Y}^L(V,TM)$
the subspace of $\mathcal{Y}(V,TM)$ consisting of the sections which
equal $Z$ on $V-L$.
\end{definition}

The following extension lemma will be useful.
\begin{lemma}\indexstate{Seeley!Extension Lemma}\indexstate{Extension Lemma!Seeley}
Form a manifold $N$ from $M$ and $\partial M\times (-\infty,0]$ by
identifying $\partial M$ with $\partial M\times \{0\}$, and extending the
metric on $M$ using the product of the complete metric on $\partial M$ and
the standard metric on $(-\infty,0]$.
\begin{enumerate}
\item[(i)] There is a continuous linear extension $E\colon
\Maps(M,\R)\to \Maps(N,\R)$ for which the image is contained in the
subspace of functions that vanish on $\partial M\times (-\infty,-1]$.
\item[(ii)] There is a continuous linear extension $E\colon
\mathcal{Y}(M,TM)\to \mathcal{Y}(N,TN)$ for which the image is contained in
the subspace of sections that vanish on $\partial M\times (-\infty,-1]$.
\end{enumerate}
\label{lem:Seeley}
\end{lemma}
\begin{proof} Part~(i) is basically what is established in the proof of
Corollary~II.1.3.7 of~\cite{Hamilton}. It was also proven by essentially
the same method, using series in place of integration and working only on a
half-space in $\R^n$, by R. Seeley~\cite{Seeley}. The extensions are first
performed in local coordinates $\R^{n-1}\times \R$, where the value of the
extension $Ef(x,t)$ for $t<0$ is given by an integral on the ray
$\{x\}\times [0,\infty)$. Fixing a collection of charts and a partition of
unity, these local extensions are pieced together to give $Ef$.
Multiplying by a smooth function which is $1$ on a neighborhood of $M$
and vanishes on $\partial M\times (-\infty,-1]$, we may achieve the final
property in~(i). Part (ii) follows from (i) since locally a vector field is
just a collection of $n$ real-valued functions.
\end{proof}

We are grateful to Tatsuhiko Yagasaki\index{Yagasaki} for bringing the
reference \cite{Seeley} to our attention.

Our proof that $\Diff(M)$ is a Fr\'echet manifold will use the \textit{tame
exponential} 
\indexdef{tame exponential}\indexsymdef{TExp}{$\TExp$}$\TExp$. Let $X$ be 
a vector field on $M$ such that for every $x\in M$, $\Exp(X(x))$ is
defined. Then $\TExp(X)$ is defined to be the map from $M$ to $M$ that
takes each $x$ to $\Exp(X(x))$. For a complete manifold $M$ without
boundary, the tame exponential defines local charts on $\Maps(M,M)$ (and
more generally on $\Maps(M,N)$ if instead of vector fields on $M$ one uses
sections of a pullback of $TN$ to a bundle over $M$), see for example
Theorem~42.1 of~\cite{Kriegl-Michor}.

\begin{definition} Let $V$ be a submanifold of $M$, and as always assume
that the metric on $M$ is a product near the boundary and $V$ meets
$\partial M\times \I$ in $\I$-fibers. By 
\indexsymdef{XVTM}{$\mathcal{X}(V,TM)$, $\mathcal{X}^L(V,TM)$}$\mathcal{X}(V,TM)$ we denote the
Fr\'echet subspace of $\mathcal{Y}(V,TM)$ consisting of those sections
which are tangent to $\partial M$ at all points of $V\cap \partial M$. For
$L\subseteq M$, we denote by $\mathcal{X}^L(V,TM)$ the 
subspace of sections that equal $Z$ on $V-L$.\par
\label{def:X}
\end{definition}
\index{DiffM@$\Diff(M)$!Lie structure|(}We remark that 
$\mathcal{X}(M,TM)$ is the tangent space at $1_M$ of the
infinite-dimensional Lie
group $\Diff(M)$, and the exponential map in that
context takes a vector field on $M$ to the map at time $1$ of the flow on
$M$ associated to the vector field. The resulting exponential map from
$\mathcal{X}(M,TM)$ to $\Diff(M)$ is not locally surjective near $Z$ and
$1_M$, even for $M=S^1$ (see for example Section~5.5.2
of~\cite{Hamilton}). We will always use the 
\index{tame exponential}tame exponential, which as
noted above is a local homeomorphism (in fact, a local diffeomorphism, for
appropriate structures on these spaces as infinite-dimensional
manifolds).\index{DiffM@$\Diff(M)$!Lie structure|)}

We can now give the Fr\'echet structure on
$\Diff(M)$.\index{DiffM@$\Diff(M)$!Fr\'echet structure} We will denote by
$\Maps((M,\partial M),(M,\partial M))$ the space of smooth maps from $M$ to
$M$ that take $\partial M$ to $\partial M$, with the $\Cinf$-topology.
\begin{theorem} The space $\Maps((M,\partial M),(M,\partial M))$ is a
Fr\'echet manifold locally modeled on $\mathcal{X}(M,TM)$, and $\Diff(M)$ is
an open subset of $\Maps((M,\partial M),(M,\partial M))$.\par
\label{thm:diffs_with_boundary}
\end{theorem}

\begin{proof} It suffices to find a local chart for
$\Maps((M,\partial M),(M,\partial M))$ at the identity $1_M$ that has
image in $\Diff(M)$. Form a manifold $N$ from $M$ as in
Lemma~\ref{lem:Seeley}, and let $E\colon \mathcal{Y}(M,TM)\to
\mathcal{Y}(N,TN)$ be a continuous linear extension as in part~(ii) of
\index{Extension Lemma!Seeley}Lemma~\ref{lem:Seeley}. Since $N$ is
complete, the 
\index{tame exponential}tame exponential
$\TExp\colon \mathcal{Y}(N,TN)\to \Maps(N,N)$ is defined. From
Theorem~43.1 of~\cite{Kriegl-Michor}, $\Diff(N)$ is an open subset of
$\Maps(N,N)$. Let $U\subset \mathcal{Y}(N,TN)$ be an open neighborhood
of $Z$ which $\TExp$ carries homeomorphically to an open neighborhood of
$1_N$ in $\Diff(N)$. Since vector fields in $\mathcal{X}(M,TM)$ are tangent
to the boundary, $\TExp$ carries $U\cap E(\mathcal{X}(M,TM))$ to
diffeomorphisms of $N$ taking $\partial M$ to $\partial M$. Therefore
$\TExp$ carries the open neighborhood $E^{-1}(U\cap E(\mathcal{X}(M,TM)))$
of $Z$ into~$\Diff(M)$.
\end{proof}

As in the case of manifolds without boundary, we can now conclude that
$\Diff(M)$ has the homotopy type of a CW-complex.

When $M$ is compact, $\Diff(M)$ is separable, and moreover $\Diff(M)$ is
\index{DiffM@$\Diff(M)$!local convexity}locally convex. Explicitly, our
local charts defined using the 
\index{tame exponential}tame exponential show that for any $f\in
\Diff(M)$, there is a neighborhood $U$ of $f$ such that for every $g\in U$,
the homotopy that moves points along the shortest geodesic from each $g(x)$
to $f(x)$ is an isotopy from $g$ to~$f$.

For a closed subset $X\subset M$, we denote by 
\indexsymdef{DiffMrelX}{$\Diff(M\protect\rel X)$}$\Diff(M\rel X)$ the
subgroup of $\Diff(M)$ consisting of the elements which take $X$ to $X$ and
restrict to the identity map on $X$. Adapting the previous arguments shows
that $\Diff(M\rel X)$ is modeled on the closed Fr\'echet subspace of
$\mathcal{X}(M,TM)$ consisting of sections that vanish on~$X$.

\newpage
\section{Spaces of embeddings}
\label{sec:spaces_of_embeddings}

When we work with embeddings, we always start with a fixed submanifold $V$
of the ambient manifold $M$. The inclusion map then furnishes a natural
basepoint of the space of imbeddings. In addition, this will allow a simple
definition of the 
\index{space of images of a submanifold}space of images of $V$ in $M$, given in
Definition~\ref{def:images} below.

\begin{definition} 
Let $V$ be a submanifold of $M$. When $M$ has boundary and
$\dim(V)<\dim(M)$, we always require that $V\cap \partial M=\partial V$,
and select our Riemannian metric on $M$ to be a product near the boundary
for which $V$ meets the collar $\partial M\times \I$ in
$\I$-fibers. Similarly, when $V$ is codimension-$0$, the frontier of $V$ is
a codimension-$1$ submanifold of $M$ assumed to meet $\partial M\times \I$
in $\I$-fibers. Denote by \indexsymdef{EmbVM}{$\Imb(V,M)$}$\Imb(V,M)$ 
the space of all smooth embeddings $j$ of $V$ into $M$ such that
\begin{enumerate}
\item[(i)] $j^{-1}(\partial M)= V\cap \partial M$, and
\item[(ii)] $j$ extends to a diffeomorphism from $M$ to $M$.
\end{enumerate}\par
\label{def:Imb}
\end{definition}
\noindent
Note that condition~(ii) implies that $j$ carries every inward-pointing
tangent vector of $V\cap \partial M$ to an inward-pointing tangent vector
of~$M$. It also implies that the natural map $\Diff(M)\to\Imb(V,M)$ that
sends each diffeomorphism to its restriction to $V$ is surjective.

With the $\Cinf$-topology\index{EmbVM@$\Imb(V,M)$!Fr\'echet structure},
$\Imb(V,M)$ is a Fr\'echet manifold locally modeled
on $\mathcal{X}(V,TM)$. For the closed case, this is proven in Theorem~44.1
of~\cite{Kriegl-Michor}, and adaptations like those in
Section~\ref{sec:manifolds_with_boundary} allow its extension in the
bounded and codimension-$0$ contexts (note that 
\index{Extension Lemma!Seeley}Lemma~\ref{lem:Seeley}
provides a continuous linear extension from $\mathcal{X}(V,M)$ to
$\mathcal{X}(V\cup (-\infty,0], M\cup (-\infty,0])$).  As in the case of
$\Diff(M)$, this Fr\'echet manifold structure shows that $\Imb(V,M)$ has
the homotopy type of a CW-complex.

\section{Bundles and fiber-preserving diffeomorphisms}
\label{sec:fiberpreservingdiffeos}

Let $p\colon E\to B$ be a locally trivial smooth map of manifolds, with
compact fiber. When $B$ and the fiber have nonempty boundary, $E$ should be
regarded as a manifold with 
\index{corners}corners at the boundary points of the fibers in
$p^{-1}(\partial B)$. The 
\indexdef{horizontal!boundary}%
\indexdef{boundary!horizontal}horizontal boundary
\indexsymdef{delhE}{$\partial_hE$, $\partial_vE$}$\partial_hE$ is
defined to be $\cup_{x\in B}\partial(p^{-1}(x))$, and the 
\indexdef{vertical!boundary}%
\indexdef{boundary!vertical}vertical boundary $\partial_vE$ to 
be~$p^{-1}(\partial B)$.

\begin{definition}
The space of
\indexdef{fiber-preserving!diffeomorphism}%
\indexdef{diffeomorphism!fiber-preserving}%
\textit{fiber-preserving diffeomorphisms} is the subspace
\indexsymdef{DifffE}{$\Diff_f(E)$}$\Diff_f(E)$ 
of $\Diff(E)$ consisting of the diffeomorphisms that take each
fiber of $E$ to a fiber. The 
\indexdef{vertical!diffeomorphism}%
\indexdef{diffeomorphism!vertical}%
\textit{vertical diffeomorphisms} 
\indexsymdef{DiffvE}{$\Diff_v(E)$}$\Diff_v(E)$
are the elements of $\Diff_f(E)$ that take each fiber to itself.\par
\end{definition}

Fibered submanifolds also play an important role.
\begin{definition}\label{def:fibered_submanifolds}
A submanifold $W$ of $E$ is called 
\indexdef{fibered submanifold}\indexdef{submanifold!fibered}\textit{fibered} 
or
\indexdef{vertical!submanifold}\indexdef{submanifold!vertical}\textit{vertical} 
if it is a union of fibers. For a fibered submanifold $W$
of $E$, define 
\indexsymdef{delhW}{$\partial_hW$, $\partial_vW$}$\partial_hW$ to be 
$W\cap \partial_hE$ and $\partial_vW$ to be $W\cap \partial_vE$. The space of 
\indexdef{fiber-preserving!embedding}%
\indexdef{embedding!fiber-preserving}\textit{fiber-preserving embeddings}
\indexsymdef{EmbfWE}{$\Imb_f(W,E)$}$\Imb_f(W,E)$ is the subspace of 
$\Imb(W,E)$ consisting of embeddings that
take each fiber of $W$ to a fiber of $E$, and the space of 
\indexdef{vertical!embedding}\indexdef{embedding!vertical}\textit{vertical
  embeddings} 
\indexsymdef{EmbvWE}{$\Imb_v(W,E)$}$\Imb_v(W,E)$ is the subspace of 
$\Imb_f(W,E)$ consisting of
embeddings taking each fiber to itself.\par
\end{definition}

At each point $x\in E$, let 
\indexsymdef{VxE}{$V_x(E)$}$V_x(E)$ denote the 
\indexdef{vertical!subspace of $T_xE$}\indexdef{subspace!vertical}\textit{vertical subspace}
of $T_x(E)$ consisting of vectors tangent to the fiber of~$p$. When $E$ has
a Riemannian metric, the orthogonal complement 
\indexsymdef{HxE}{$H_x(E)$}$H_x(E)$ of $V_x(E)$ in
$T_x(E)$ is called the 
\indexdef{horizontal!subspace of $T_xE$}\indexdef{subspace!horizontal}\textit{horizontal subspace.}  
We call the elements
of $V_x(E)$ and $H_x(E)$ 
\indexdef{vertical!tangent vector}\textit{vertical} and 
\indexdef{horizontal!tangent vector}\textit{horizontal}
respectively. Clearly $V_x(E)$ is the kernel of $p_*\colon T_x(E)\to
T_{p(x)}(B)$, while $p_*\vert_{H_x(E)}\colon H_x(E)\to T_{p(x)}(B)$ is an
isomorphism. Each vector $\omega\in T_x(E)$ has an orthogonal decomposition
\indexsymdef{wh}{$\omega_v$, $\omega_h$}$\omega= \omega_v+\omega_h$ 
into its vertical and horizontal parts.

A path $\alpha$ in $E$ is called 
\indexdef{horizontal!path}\indexdef{path!horizontal}\textit{horizontal} 
if $\alpha'(t)\in H_{\alpha(t)}(E)$ for all $t$ in the domain of $\alpha$. 
Let $\gamma\colon [a,b]\to B$ be a path such that $\gamma'(t)$ never
vanishes, and let $x\in E$ with $p(x)= \gamma(a)$. A horizontal path
$\widetilde{\gamma}\colon [a,b]\to E$ such that
$\widetilde{\gamma}(a)= x$ and $p\widetilde{\gamma}= \gamma$ is
called a 
\indexdef{horizontal!lift}\indexdef{lift!horizontal}\textit{horizontal lift} 
of $\gamma$ starting at~$x$.

To ensure that horizontal lifts exist, we will need a special metric
on $E$. 
\begin{definition}
A Riemannian metric on $E$ is said to be a 
\indexdef{product near $\partial_hE$}\textit{product near $\partial_hE$} when
\begin{enumerate}
\item[(i)] There is a collar neighborhood $\partial_hE\times \I$ of the
horizontal boundary on which the metric is the product of a complete metric
on $\partial_hE$ and the standard metric on $\I$,
\item[(ii)] For this collar $\partial_hE\times \I$, 
each $\{x\}\times \I$ lies in some fiber of~$p$.
\end{enumerate}\par
\end{definition}
\noindent Such metrics can be constructed using a partition of unity as
follows.  Using the local product structure, at each point $x$ in
$\partial_hE$ select a vector field defined on a neighborhood of $x$ that
\begin{enumerate}
\item[{\rm(a)}] points into the fiber at points of $\partial_hE$, and
\item[{\rm(b)}] is tangent to the fibers wherever it is defined.
\end{enumerate}
\noindent By (b), the vector field must be tangent to $\partial_vE$ at
points in $\partial_vE$. Since scalar multiples and linear combinations of
vectors satisfying these two conditions also satisfy them, we may piece
these local fields together using a partition of unity to construct a
vector field, nonvanishing on a neighborhood of $\partial_hE$, that
satisfies (a) and~(b). Using the integral curves associated to this vector
field we obtain a smooth collar neighborhood $\partial_hE\times [0,2]$ of
$\partial_hE$ such that each $[0,2]$-fiber lies in a fiber of $p$. On
$\partial_hE\times [0,2)$, fix a Riemannian metric that is the product of a
metric on $\partial_hE$ and the usual metric on $[0,2)$. Form a metric on
$E$ from this metric and any metric on all of $E$ using a partition of
unity subordinate to the open cover $\set{\partial_hE\times[0,2),
E-\partial_hE\times \I}$.

When the metric is a product near $\partial_hE$ such that the $\I$-fibers of
$\partial_hE\times \I$ are vertical, the horizontal subspace $H_x$ is
tangent to~$\partial_hE\times\set{t}$ whenever
$x\in\partial_hE\times\set{t}$. For $H_x$ is orthogonal to the fiber
$p^{-1}(p(x))$, and since the $\I$-fiber of $\partial_hE\times \I$ that
contains $x$ lies in $p^{-1}(p(x))$, $H_x$ is orthogonal to that $\I$-fiber
as well. Since $\partial_hE\times\set{t}$ meets the $\I$-fiber orthogonally,
with codimension~1, $H_x$ is tangent to $\partial_hE\times\set{t}$.

Since the horizontal subspaces are tangent to the
$\partial_hE\times\set{t}$, a horizontal lift starting in some
$\partial_hE\times \set{t}$ will continue in $\partial_hE\times
\set{t}$. Provided that the fiber is compact, as we are assuming, the
existence of horizontal lifts is assured.

\section{Aligned vector fields and the aligned exponential}
\label{exponent}

\begin{definition}
A vector field $X\colon E\to TE$ is called 
\indexdef{aligned!vector field}\indexdef{vector field!aligned}\textit{aligned} 
if $p(x)= p(y)$ implies that $p_*(X(x))= p_*(X(y))$ (these are often
called \textit{projectable} in the literature). This happens precisely when
there exist a vector field $X_B$ on $B$ and a vertical vector field $X_V$
on $E$ so that for all $x\in E$,
\[X(x)= (p_*\vert_{H_x})^{-1}(X_B(p(x)))+X_V(x)\ .\]
\noindent In particular, any vertical vector field is aligned. When
$X$ is aligned, the projected vector field $p_*X$ is well-defined.\par
\label{def:aligned_vector_fields}
\end{definition}

The idea of the aligned exponential $\Exp_a$ is that it behaves as would
the regular exponential if the metric on $E$ were locally the product of a
metric on $F$ and a metric on $B$. The key property of $\Exp_a$ is that if
$X$ is an aligned vector field on $E$, and $\Exp_a(X(x))$ is defined for
all $x$, then the map of $E$ defined by sending $x$ to $\Exp_a(X(x))$ will
be fiber-preserving.

\begin{definition} Let $\pi\colon TE\to E$ denote the tangent bundle of $E$. 
Assume that the metric on $E$ is a product near $\partial_hE$ such that the
$\I$-fibers of $\partial_hE\times \I$ are vertical. Each fiber $F$ of $E$
inherits a Riemannian metric from that of $E$, and has an exponential map
$\Exp_F$ which (where defined) carries vectors tangent to $F$ to points of
$F$. The path $\Exp_F(t\omega)$ is not generally a geodesic in $E$. The
\indexdef{vertical!exponential}\indexdef{exponential!vertical}\textit{vertical exponential} 
\indexsymdef{Expv}{$\Exp_v$}$\Exp_v$ is defined by
$\Exp_v(\omega)=\Exp_F(\omega)$, where $\omega$ is a vertical vector and
$F$ is the fiber containing~$\pi(\omega)$. The 
\indexdef{aligned!exponential}\textit{aligned exponential map} 
\indexsymdef{Expa}{$\Exp_a$}$\Exp_a$ is defined as follows. 
Consider a tangent vector $\omega\in
T_x(E)$ such that for the vector $p_*(\omega)\in T_{p(x)}(B)$,
$\Exp(p_*(\omega))$ is defined. A geodesic segment $\gamma_{p_*(\omega)}$
starting at $p(\pi(\omega))$ is defined by $\gamma_{p_*(\omega)}(t)=
\Exp(tp_*(\omega))$, $0\leq t\leq 1$. Define $\Exp_a(\omega)$ to be the
endpoint of the unique horizontal lift of $\gamma_{p_*(\omega)}$ starting
at $\Exp_v(\omega_v)$.\par
\label{def:aligned}
\end{definition}

Note that $\Exp_a(\omega)$ exists if and only if both $\Exp_v(\omega_v)$
and $\Exp(p_*(\omega))$ exist. Clearly, when $\Exp_a(\omega)$ is defined,
it lies in the fiber containing the endpoint of a lift of
$\gamma_{p_*(\omega)}$, and therefore $p(\Exp_a(\omega))=
\Exp(p_*(\omega))$. This immediately implies that if $X$ is an aligned
vector field on $E$ such that $\Exp_a(X(x))$ is defined for all $x\in E$,
then the map defined by sending $x$ to $\Exp_a(X(x))$ takes fibers to
fibers, and in particular if $X$ is vertical, it takes each fiber to
itself.

\begin{definition}\label{def:A}
Let $W$ be a vertical submanifold of $E$.
By 
\indexsymdef{AWTE}{${\mathcal{A}}(W,TE)$}${\mathcal{A}}(W,TE)$ 
we denote the Fr\'echet space of
sections $X$ from $W$ to $TE\vert_W$ such that
\begin{enumerate}
\item[(1)] $X$ is aligned, that is, if $p(w_1)\!=\!p(w_2)$ then
$p_*(X(w_1))\!=\! p_*(X(w_2))$,
\item[(2)] if $x\in \partial_hW$, then $X(x)$ is tangent to
$\partial_hE$, and if $x\in \partial_vW$, then $X(x)$ is tangent to
$\partial_vE$, and
\end{enumerate}
The elements of ${\mathcal{A}}(W,TE)$ such that
$p_*X(x)= Z(p(x))$ for all $x\in W$ are denoted by
\indexsymdef{VWTE}{$\mathcal{V}(W,TE)$}$\mathcal{V}(W,TE)$.\par
\label{def:aligned_sections}
\end{definition}

\noindent By condition (3), $\TExp_a(X)$ is defined for every $X$ in
$\mathcal{A}(W,TE)$ or in $\mathcal{V}(W,TE)$.\par 

The vector space structure on ${\mathcal{A}}(W,TE)$ is defined using the
vector space structures of the fibers of $TE$ and $TB$. Given $v,w\in
\mathcal{A}(W,TE)$, we decompose them into their vertical and horizontal
parts. The vertical parts are added by the usual addition in $TE$. The
horizontal parts are added by pushing down to $TB$, adding there, and
taking horizontal lifts.

Since horizontal lifts of geodesics in $B$ exist, $\Exp_a(\omega)$ is
defined whenever $\Exp_v(\omega)$ and $\Exp(p_*(\omega))$ are defined. In
particular, the 
\index{tame exponential}tame aligned exponential $\TExp_a$ carries a neighborhood of
$Z$ in $\mathcal{A}(W,TE)$ into $\mathrm{C}^\infty_f(W,E)$. Choosing the
neighborhood small enough to ensure that $\TExp_a(X)\in \Imb_f(W,E)$
provides local charts on $\Imb_f(W,E)$, that carry the vertical fields into
$\Imb_v(W,E)$. Thus we have:
\begin{theorem} The spaces $\Diff_f(E)$, 
$\Diff_v(E)$, $\Imb_f(W,E)$, and
$\Imb_v(W,E)$ are infinite-dimensional manifolds modeled on
Fr\'echet spaces of aligned vector fields.
\end{theorem}

\chapter{The method of Cerf and Palais}
\label{ch:Palais}

In \cite{P}, R.~Palais proved a very useful result relating diffeomorphisms
and embeddings. For closed $M$, it says that if $W\subseteq V$ are
submanifolds of $M$, then the mappings $\Diff(M)\to \Imb(V,M)$ and
$\Imb(V,M)\to \Imb(W,M)$ obtained by restricting diffeomorphisms and
embeddings are locally trivial, and hence are Serre fibrations. The same
results, with variants for manifolds with boundary and more complicated
additional boundary structure, were proven by J.~Cerf\index{Cerf}
in~\cite{Cerf}. Among various applications of these results, the
\index{Isotopy Extension Theorem}Isotopy
Extension Theorem follows by lifting a path in $\Imb(V,M)$ starting at the
inclusion map of $V$ to a path in $\Diff(M)$ starting at $1_M$. Moreover,
parameterized versions of isotopy extension follow just as easily from the
homotopy lifting property for $\Diff(M)\to \Imb(V,M)$ (see
Corollary~\ref{isotopy lifting}).

In this chapter, we will extend the theorem of Palais in various ways. Many
of our results concern fiber-preserving maps. For example,
in Section~\ref{project} we will prove the
\smallskip

\indexstate{Projection Theorem!Palais-Cerf}%
\noindent{\bf Projection Theorem} (Theorem \ref{project diffs})
\textit{Let $E$ be a bundle over a compact manifold $B$. Then
$\Diff_f(E)\to\Diff(B)$ is locally trivial. }
\smallskip

\noindent This should be considered a folk theorem. Below we will discuss
some of its antecedents.

The homotopy extension property for the projection fibration $\Diff_f(E)\to
\Diff(B)$ translates directly into the following.
\smallskip

\indexstate{Parameterized Isotopy Lifting Theorem}%
\noindent{\bf Parameterized Isotopy Lifting Theorem} (Corollary
\ref{isotopy lifting}) \textbf{Suppose that $p\colon E\to B$ is a fibering of
compact manifolds, and suppose that for each $t$ in a path-connected
parameter space $P$, there is an isotopy $g_{t,s}$ such that $g_{t,0}$
lifts to a diffeomorphism $G_{t,0}$ of $E$. Assume that sending $(t,s)\to
g_{t,s}$ defines a continuous function from $P\times [0,1]$ to $\Diff(B)$
and sending $t$ to $G_{t,0}$ defines a continuous function from $P$ to
$\Diff_f(E)$. Then the family $G_{t,0}$ extends to a continuous family
on $P\times \I$ such that for each $(t,s)$, $G_{t,s}$ is a
fiber-preserving diffeomorphism inducing $g_{t,s}$ on~$B$.}
\smallskip

For fiber-preserving and vertical embeddings of vertical submanifolds, we
have a more direct analogue of Palais' results.
\smallskip

\indexstate{Restriction Theorem!fiber-preserving}%
\index{Palais-Cerf Restriction Theorem}%
\index{Restriction Theorem}%
\noindent\textbf{Restriction Theorem} (Corollaries~\ref{corollary2}
and~\ref{corollary3}) \textit{Let $V$ and $W$ be vertical submanifolds of $E$
with $W\subseteq V$, each of which is either properly embedded or
codimension-zero. Then the restrictions $\Diff_f(M)\to \Imb_f(V,M)$,
$\Diff_v(M)\to \Imb_v(V,M)$, $\Imb_f(V,E)\to \Imb_f(W,E)$ and
$\Imb_v(V,E)\to \Imb_v(W,E)$ are locally trivial.}
\medskip

\noindent As shown in Theorem \ref{square}, the Projection and
Restriction Theorems can be combined into a single commutative square,
called the 
\index{square!projection-restriction}\index{projection-restriction square}\textit{projection-restriction square,}
in which all four maps are locally trivial:
\[\vbox{\halign{\hfil#\hfil\quad&#&\quad\hfil#\hfil\cr
$\Diff_f(E)$&$\longrightarrow$&$\Imb_f(W,E)$\cr
\noalign{\smallskip}
$\mapdown{}$&&$\mapdown{}$\cr
\noalign{\smallskip}
$\Diff(B)$&$\longrightarrow$&$\Imb(p(W),B)\rlap{\ .}$\cr}}\]

In 3-dimensional topology, a key role is played by manifolds admitting a
more general kind of fibered structure, called a Seifert fibering. Some
general references for 
\index{Seifert-fibered $3$-manifold}Seifert-fibered 3-manifolds are
\cite{Hempel,Jaco,JS,Orlik,OVZ,Scott,Seifert,Waldhausen1,Waldhausen2}.  In
Section~\ref{sfiber}, we prove the analogues of the results discussed above
for most Seifert fiberings $p\colon\Sigma\to\O $. Actually, we work in a
somewhat more general context, called 
\index{singular fiberings}\textit{singular fiberings}, which
resemble Seifert fiberings but for which none of the usual structure of the
fiber as a homogeneous space is required.

In the late 1970's fibration results akin to our Projection Theorem for the
singular fibered case were proven by W.~Neumann\index{Neumann} and
F.~Raymond\index{Raymond}~\cite{N-R}. They were interested in the case when
$\Sigma$ admits an action of the $k$-torus~$T^k$ and $\Sigma\to\O $ is the
quotient map to the orbit space of the action. They proved that the space
of (weakly) $T^k$-equivariant homeomorphisms of $\Sigma$ fibers over the
space of homeomorphisms of $\O $ that respect the orbit types associated to
the points of $\O $. A detailed proof of this result when the dimension of
$\Sigma$ is $k+2$ appears in the dissertation of 
C.~Park~\cite{Park}. Park\index{Park}
also proved analogous results for space of weakly $G$-equivariant maps for
principal $G$-bundles and for Seifert fiberings of arbitrary
dimension~\cite{Park,Park1}. These results do not directly overlap ours
since we always consider the full group of fiber-preserving diffeomorphisms
without any restriction to $G$-equivariant maps (indeed, no assumption of a
$G$-action is even present).

The results of this chapter will be used heavily in the later chapters. In
this chapter, we give one main application.  For a Seifert-fibered manifold
$\Sigma$, $\Diff(\Sigma)$ acts on the set of Seifert fiberings, and the
stabilizer of the given fibering is $\Diff_f(\Sigma)$, thus the space of
cosets $\Diff(\Sigma)/\Diff_f(\Sigma)$ can be regarded as the 
\index{space of Seifert fiberings}\textit{space of Seifert fiberings} 
of $\Sigma$ equivalent to the given one.  We prove in Section~\ref{sfspace} 
that for a Seifert-fibered Haken
3-manifold, each component of the space of Seifert fiberings is
contractible (apart from a small list of well-known exceptions, the space
of Seifert fiberings is connected). This too should be considered a folk
result; it appears to be widely believed and regarded to be a direct
consequence of the work of \index{Hatcher}Hatcher and 
\index{Ivanov}Ivanov on the diffeomorphism groups
of Haken manifolds. We have found, however, that a real proof requires more
than a little effort.

Our results will be proven by adapting the Palais method of~\cite{P}, using
the aligned exponential defined in Section~\ref{exponent}. In
Section~\ref{palais}, we reprove the main result of \cite{P} for manifolds
which may have boundary. This duplicates~\cite{Cerf} (in fact, the boundary
control there is more refined than ours), but is included to furnish lemmas
as well as to exhibit a prototype for the approach we use to deal with the
bounded case in our later settings. In Section~\ref{orbifold}, we give the
analogues of the results of Palais and Cerf for smooth orbifolds, which for
us are quotients $\widetilde{\O}/H$ where $\widetilde{\O}$ is a manifold
and $H$ is a group acting smoothly and properly discontinuously on
$\widetilde{\O}$. Besides being of independent interest, these analogues
are needed for the case of singular fiberings.

Throughout this chapter, all Riemannian metrics are assumed to be products
near the boundary, or near the horizontal boundary for total spaces of
bundles, such that any submanifolds under consideration meet the collars in
$\I$-fibers. Let $V$ be a submanifold of $M$.  As in Definition~\ref{def:X},
the notation $\mathcal{X}(V,TM)$ means the Fr\'echet space of sections from
$V$ to the restriction of the tangent bundle of $M$ to $V$ that are tangent
to $\partial M$ at all points of $V\cap \partial M$. We also utilize
various kinds of control, as indicated in the following definitions.
\begin{definition}\label{def:diffcontrol}
The notations 
\indexsymdef{DiffMrelX}{$\Diff(M\protect\rel X)$}$\Diff(M\rel X)$ and 
\indexsymdef{DiffM-XM}{$\Diff^{M-X}(M)$}$\Diff^{M-X}(M)$ mean the space of
diffeomorphisms which restrict to the identity map on each point of the
subset $X$ of $M$.  These notations may be combined, for example
\indexsymdef{DiffLMrelX}{$\Diff^L(M\protect\rel X)$}$\Diff^L(M\rel X)$ is the 
space of diffeomorphisms that are the identity on $X\cup (M-L)$.
\end{definition}

\begin{definition}\label{def:imbcontrol}
For $X\subseteq M$ we say that $K\subseteq M$ is a 
\indexdef{neighborhood of a submanifold}neighborhood of $X$ when
$X$ is contained in the topological interior of $K$. If $K$ is a
neighborhood of a submanifold $V$ of $M$, then 
\indexsymdef{EmbVMK}{$\Imb^K(V,M)$}$\Imb^K(V,M)$ means the
elements $j$ in $\Imb(V,M)$ such that $K$ is a neighborhood of~$j(V)$.
Suppose that $S$ is a closed neighborhood in $\partial M$ of $V\cap S$.
Note that this implies that $S\cap\partial V$ is a union of components of
$V\cap\partial M$. We denote by 
\indexsymdef{EmbVMrelS}{$\Imb(V,M\protect\rel S)$}$\Imb(V,M\rel S)$ the elements 
$j$ that equal the inclusion on $V\cap S$ and carry $V\cap (\partial M-S)$ into
$\partial M-S$. For a neighborhood $K$ of $V$, the superscript notation of
Definition~\ref{def:diffcontrol} may be used, as in 
\indexsymdef{EmbVMrelSK}{$\Imb^K(V,M\protect\rel S)$}$\Imb^K(V,M\rel S)$.
\end{definition}

\begin{definition}\label{def:vfcontrol}
Recall from Definition~\ref{def:X} that for $L\subseteq M$, 
$\mathcal{X}^L(V,TM)$ means 
the elements of $\mathcal{X}(V,TM)$ that equal the zero section $Z$ on 
$V-L$. We extend this to the aligned and vertical sections (see
Definition~\ref{def:aligned_sections}), so that if $L\subset E$ then
\indexsymdef{AWTEL}{$\mathcal{A}^L(W,TE)$}$\mathcal{A}^L(W,TE)$ and 
\indexsymdef{VWTEL}{$\protect\mathcal{V}^L(W,TE)$}$\mathcal{V}^L(W,TE)$ and 
$\mathcal{V}^L(W,TE)$ have the corresponding meanings.
\end{definition}

\section{The Palais-Cerf Restriction Theorem}
\label{palais}

We begin with a review of the method of Palais~\cite{P}.

\begin{definition}\label{def:localcrosssections}
Let $X$ be a $G$-space and $x_0\in X$. A 
\indexdef{local cross-section}\textit{local cross-section} (or
\textit{$G$ local cross-section}) for $X$ at $x_0$ is a map $\chi$ from a
neighborhood $U$ of $x_0$ into $G$ such that $\chi(u)x_0= u$ for all
$u\in U$. By replacing $\chi(u)$ by $\chi(u)\chi(x_0)^{-1}$, one may always
assume that $\chi(x_0)= 1_G$. If $X$ admits a local cross-section at each
point, it is said to admit local cross-sections. 
\end{definition}
Note that a local cross-section $\chi_0\colon U_0\to G$ at a single point
$x_0$ determines a local cross section $\chi\colon gU_0\to G$ at any point
$gx_0$ in the orbit of $x_0$, by the formula
$\chi(u)=g\chi_0(g^{-1}u)g^{-1}$, since then
$\chi(u)(gx_0)=g\chi_0(g^{-1}u)g^{-1}gx_0=g\chi_0(g^{-1}u)x_0=gg^{-1}u=u$. In
particular, if $G$ acts transitively on $X$, then a local cross section at
any point provides local cross sections at all points.

From \cite{P} we have
\begin{proposition} Let $G$ be a topological group and $X$ a
$G$-space admitting local cross-sections. Then any equivariant map of
a $G$-space into $X$ is locally trivial.
\label{theoremA}
\end{proposition}

\noindent In fact, when $\pi\colon Y\to X$ is $G$-equivariant, the
local coordinates on $\pi^{-1}(U)$ are just given by sending the point
$(u,z)\in U\times \pi^{-1}(y_0)$ to $\chi(u)\cdot z$. Some additional
properties of the bundles obtained in Proposition~\ref{theoremA} are
given in~\cite{P}.

\begin{example}\label{ex:Lie_groups}\index{examples!example1@local triviality and
Lie groups}
For a closed subgroup $H$ of a Lie group $G$, the projection $G\to G/H$ to
the space of left cosets of $H$ always has local $G$ cross-sections, and
hence is locally trivial. To check this, recall first that since $G$ acts
transitively on $G/H$, it is sufficient to find a local cross-section
$\chi_0$ at the coset $e\,H$, where $e$ is the identity element of $G$. To
construct $\chi_0$, fix a Riemannian metric on $G$. The tangent space
$T_eH$ is a subspace of $T_eG$. Let $W$ be a complementary subspace. Let
$V$ be an open neighborhood of $0$ in $T_eG$ such that $\Exp\colon V\to U$
is a diffeomorphism onto an open neighborhood of $e$ in $G$, and so that
the submanifold $\Exp(W\cap V)$ is transverse to the cosets $uH$ for all
$u\in U$. Defining $\chi_0(uH)$ to be $\Exp(w)$ for the unique element
$w\in W\cap U$ such that $\Exp(w)U=uH$ gives the local cross-section
at~$e$.
\end{example}

The following technical lemma will simplify some of our applications
of Proposition~\ref{theoremA}.

\begin{proposition} Let $M$ be a $G$-space and let $V$ be a subspace
of $M$, possibly equal to $M$. Let $I(V,M)$ be a space of embeddings of $V$
into $M$, on which $G$ acts by composition on the left. Suppose that for
every $i\in I(V,M)$, the space of embeddings $I(i(V),M)$ has a local $G$
cross-section at the inclusion map of $i(V)$ into $M$. Then $I(V,M)$ has
local $G$ cross-sections.
\label{prop:inclusion}
\end{proposition}

\begin{proof} Fix $i\in I(V,M)$, and denote by $j_{i(V)}$ the inclusion 
map of $i(V)$ into $M$. Define $Y\colon I(V,M)\to I(i(V),M)$ by
$Y(j)=ji^{-1}$. For a local cross-section $\chi\colon U\to G$ at
$j_{i(V)}$, define $Y_1$ to be the restriction of $Y$ to $Y^{-1}(U)$, a
neighborhood of $i$ in $I(V,M)$. Then $\chi Y_1\colon Y^{-1}(U)\to G$ is
a local cross-section for $I(V,M)$ at~$i$. For if $j\in Y^{-1}(U)$, then
$\chi (Y_1(j)) \circ i= \chi (Y_1(j))\circ j_{i(V)} \circ i=Y_1(j)\circ i=j$.
\end{proof}

\index{Palais method|(}In our context, a typical procedure for finding a
local cross-section using the Palais method is as follows. Suppose, for
example, that one wants to find a local cross-section from a space of
embeddings of a submanifold to a space of diffeomorphisms of the ambient
manifold. First, take the ``logarithm'' of an embedding $j$, that is, find
a section from the submanifold to the tangent bundle of $M$ so that the
exponential of the vector at each $x$ is the image $j(x)$. Then obtain an
``extension'' of this section to a vector field on $M$. Finally,
``exponentiate'' the extended vector field to obtain the diffeomorphism of
$M$ that agrees with $j$ on the submanifold. The extension process must be
canonical enough so that sending the embedding to the resulting
diffeomorphism is a local cross-section.

This three-step procedure depends in large part on three lemmas, or
appropriate versions of them, called Lemmas~d, c and~b in~\cite{P}. As our
first instance of them, we give the following versions for manifolds with
boundary and submanifolds that may be of codimension~$0$.

\begin{lemma}[Logarithm Lemma]\indexstate{Palais method!Logarithm
Lemma}\indexstate{Logarithm Lemma}
Assume that the metric on $M$ is a product near $\partial M$, and let $V$
be a compact submanifold of $M$ that meets $\partial M\times \I$ in
$\I$-fibers. Then there are an open neighborhood $U$ of the inclusion $i_V$
in $\Imb(V,M)$ and a continuous map $X\colon U \to\mathcal{X}(V,TM)$ such
that for all $j\in U$, $\Exp(X(j)(x))$ is defined for all $x\in V$ and
$\Exp(X(j)(x))= j(x)$ for all $x\in V$. Moreover, $X(i_V) = Z$.\par
\label{logarithm}
\end{lemma}

\begin{proof} Choose $\epsilon$ small enough so that for all $x\in
V$, $\Exp$ carries the $\epsilon$-ball about $0$ in $T_x(M)$ (that is,
the portion of this $\epsilon$-ball on which it is defined, which may be as
small as a closed half-ball for $x\in \partial M$) diffeomorphically to a
neighborhood $W_x$ of $x$ in $M$. Choose a neighborhood $U$ of $i_V$ in
$\Imb(V,M)$ so that if $j\in U$ then $j(x)\in W_x$. For $j\in U$ define
$X(j)(x)$ to be the unique vector in $T_x(M)$ of length less than
$\epsilon$ for which $\Exp(X(j)(x))$ equals~$j(x)$. Since the metric is a
product near the boundary, $X(j)$ is in $\mathcal{X}(V,M)$, and the remark
about $i_V$ is clear.
\end{proof}

The Extension Lemma uses the notation from Definition~\ref{def:vfcontrol}.
\begin{lemma}[Extension Lemma]\index{Palais method!Extension
Lemma}\indexstate{Extension Lemma}
Assume that the metric on $M$ is a product 
near $\partial M$, and let $V$ be a compact submanifold of $M$ that meets
$\partial M\times \I$ in $\I$-fibers. Let $L$ be a neighborhood of $V$ in
$M$. Then there exists a continuous linear map $k\colon
\mathcal{X}(V,TM)\to \mathcal{X}^L(M,TM)$ such that $k(X)(x)= X(x)$ for
all $x$ in $V$ and all $X$ in $\mathcal{X}(V,TM)$, and moreover if $S$ is a
closed neighborhood in $\partial M$ of $S\cap \partial V$, and if $X(x)=
Z(x)$ for all $x\in S\cap \partial V$, then $k(X)(x)= Z(x)$ for all $x\in
S$.
\label{extension}
\end{lemma}

\begin{proof} Suppose first that $V$ has positive codimension. Let
$\nu_\epsilon(V)$ denote the subspace of the normal bundle of $V$
consisting of vectors of length $\epsilon$, and let $e\colon
\nu_\epsilon(V)\to M$ be the exponentiation map. For $\epsilon$
sufficiently small, $e$ is a diffeomorphism onto a tubular neighborhood of
$V$ in $M$. Since the metric on $M$ is a product near the boundary, and $V$
meets $\partial M\times \I$ in $\I$-fibers, the fibers of $\nu_\epsilon(V)$
are carried into the submanifolds $\partial{M}\times \{t\}$ near the
boundary.

Suppose that $v\in T_x(M)$ and that $\Exp(v)$ is defined. For all $u\in
T_x(M)$ define $P(u,v)$ to be the vector that results from parallel
translation of $u$ along the path that sends $t$ to~$\Exp(tv)$, $0\leq
t\leq 1$. In particular, $P(u,Z(x))= u$ for all $u$. Let $\alpha\colon
M\to [0,1]$ be a smooth function which is identically~1 on $V$ and
identically~0 on $M-e(\nu_{\epsilon/2}(V))$. Define $k\colon
\mathcal{X}(V,TM)\to\mathcal{X}^L(M,TM)$ by
\[k(X)(x)=
\begin{cases}
\alpha(x)P(X(\pi(e^{-1}(x))),e^{-1}(x))&\text{for $x\in e(\nu_{\epsilon}(V))$}\\
Z(x)&\text{for $x\in M-e(\nu_{\epsilon/2}(V))$}
\end{cases}\]

\noindent For $x\in V$, $e^{-1}(x)= Z(x)$ and $\alpha(x)= 1$, so
$k(X)(x)= X(x)$. Similarly, $k(X)(x)=Z(x)$ for $x\in M-L$. For $x\in
\partial M$, $\pi(e^{-1}(x))$ is also in $\partial M$, so
$X(\pi(e^{-1}(x)))$ is tangent to the boundary. Since the metric is a
product near the boundary, $P(X(\pi(e^{-1}(x))),e^{-1}(x))$ is also tangent
to the boundary. Therefore $k(X)\in\mathcal{X}^L(M,TM)$.

Assume now that $V$ has codimension zero, so that its frontier $W$ is a
properly embedded submanifold. Fix a tubular neighborhood $W\times
(-\infty,\infty)$, contained in $L$, with $V\cap (W\times
(-\infty,\infty))=W\times [0,\infty)$. As in 
\index{Extension Lemma!Seeley}Lemma~\ref{lem:Seeley}, there
is a continuous linear extension map $E\colon \mathcal{X}(V,TM)\to
\mathcal{Y}(V\cup (W\times (-\infty,\infty)), TM)$. Note that since $M$ may
have boundary, it is necessary to use the half-space version of reference
\cite{Seeley} at points of $V\cap \partial M$. The extended vector fields
are $Z$ on $W\times [1,\infty)$, so extend using $Z$ on $M- (V\cup
(W\times (-\infty,\infty)))$. At points of $\partial M$, the component of
each vector in the direction perpendicular to $\partial M$ is $0$, so 
since $E$ is linear, the extended component is also $0$ and therefore the
extended vector field is also tangent to the boundary. This defines
$k\colon \mathcal{X}(V,TM)\to \mathcal{X}^L(M,TM)$.

The final sentence of the proof holds provided that we choose the tubular
neighborhoods small enough to have fibers contained in $S$ at points in
$V\cap \partial M$, or in $\partial M-S$ at points in $V\cap (\partial
M-S)$.
\end{proof}

\begin{lemma}[Exponentiation Lemma]\index{Palais method!Exponentiation
Lemma}\indexstate{Exponentiation Lemma}
Assume that the metric on $M$ is a
product near the boundary, and let $K$ be a compact subset of $M$. Then
there exists a neighborhood $U$ of $Z$ in $\mathcal{X}^K(M,TM)$ such that
$\TExp(X)$ is defined for all $X\in U$, and $\TExp$ carries $U$ into
$\Diff^KM)$.\par
\label{lem:exponentiation}
\end{lemma}

\begin{proof}
Form a manifold $N$ from $M$ and $\partial M\times (-\infty,0]$ by
identifying $\partial M$ with $\partial M\times \{0\}$, and extending the
metric on $M$ using the product of the complete metric on $\partial M$ and
the standard metric on $(-\infty,0]$. By 
\index{Extension Lemma!Seeley}Lemma~\ref{lem:Seeley}, there is a
continuous linear extension $E\colon \mathcal{Y}(M,TM)\to
\mathcal{Y}(N,TN)$ for which the image is contained in the subspace of
sections that vanish on $\partial M\times (-\infty,-1]$.  Put $L=K\cup
(K\cap \partial M)\times [-1,0]$. As seen in the proof of
Lemma~\ref{lem:Seeley}, the extended vector fields may be chosen to lie in
$\mathcal{Y}^L(N,TN)$. Since $N$ is complete and open, there is a
neighborhood $W$ of $Z$ in $\mathcal{Y}^L(N,TN)$ for which $\Exp(E(Y(x)))$
is defined for all $Y\in W$ and $x\in N$. That is, $\TExp\colon W \to
\Maps(N,N)$ is defined.

Since $\Diff(N)$ is an open subset of $\Maps(N,N)$, $W$ may be chosen
smaller, if necessary, to ensure that it is carried into $\Diff(N)$ by
$\TExp$. Diffeomorphisms obtained from extended vector fields of
$\mathcal{X}(M,TM)$ carry $\partial M$ to $\partial M$, so $\TExp$ carries
the neighborhood $U=\mathcal{X}^K(M,TM)\cap E^{-1}(W)$ of $Z$ in
$\mathcal{X}^K(M,TM)$ into $\Diff^K(M)$.
\end{proof}

We are now set up for the main results of this section. At this point the
reader may wish to review Definitions~\ref{def:diffcontrol}
and~\ref{def:imbcontrol}.
\begin{theorem} Let $V$ be a compact submanifold of $M$, and let 
$S$ be a closed neighborhood in $\partial M$ of $S\cap
\partial V$. Let $L$ be a compact neighborhood of $V$ in $M$. Then
$\Imb^L(V,M\rel S\cap \partial V)$ admits local $\Diff^L(M\rel S)$ 
cross-sections.
\label{palaistheoremB}
\end{theorem}

\begin{proof} By Proposition~\ref{prop:inclusion} it suffices to find a local 
cross-section at the inclusion map $i_V$. Using Lemmas~\ref{logarithm}
and~\ref{extension}, we obtain an open neighborhood $W$ of $i_V$ in
$\Imb(V,M \rel S\cap \partial V)$ and continuous maps $X\colon W\to
\mathcal{X}(V,TM)$ and $k\colon \mathcal{X}(V,TM)\to \mathcal{X}^L(M,TM)$.
By Lemma~\ref{lem:exponentiation}, there is a neighborhood $U$ of $Z$ in
$\mathcal{X}^L(M,TM)$ for which the map $F\colon U\to \Diff^L(M)$ sending
$Y$ to $\TExp(Y)$ is defined and continuous. Choosing a neighborhood $U_1$
of $i_V$ contained in $W\cap (k\circ X)^{-1}(U)$, the function $F\circ
k\circ X\colon U_1\to \Diff^L(M\rel S)$ will be the desired cross-section.

To see that the image of this function lies in $\Diff^L(M\rel S)$, suppose
that $j(x)=x$ for all $x\in V\cap S$. Then $X(j)(x)=Z(x)$ for $x\in V\cap
S$. By the condition in Lemma~\ref{extension}, $k(X(j))(x)= Z(x)$ and
consequently $(FkX(j))(x)=x$ for all $x\in S$.
\end{proof}\index{Palais method|)}

Using Proposition~\ref{theoremA} we obtain immediate corollaries of
Theorem~\ref{palaistheoremB}:

\indexstate{Restriction Theorem!Palais-Cerf!relative version}%
\begin{corollary} Let $V$ be a compact submanifold of
$M$. Let $S\subseteq\partial M$ be a closed neighborhood in $\partial
M$ of $S\cap \partial V$, and $L$ a neighborhood of $V$ in $M$. Then
the restriction $\Diff^L(M\rel S)\to \Imb^L(V,M\rel S)$ is locally
trivial.
\label{palaiscoro2}
\end{corollary}

\begin{corollary} Let $V$ and $W$ be compact
submanifolds of $M$, with $W\subseteq V$. Let $S\subseteq\partial M$ a
closed neighborhood in $\partial M$ of $S\cap \partial V$, and $L$ a
neighborhood of $V$ in $M$. Then the restriction $\Imb^L(V,M\rel S)
\to\Imb^L(W,M\rel S)$ is locally trivial.
\label{palaiscoro3}
\end{corollary}

\section{The space of images}
\label{sec:images}

As an initial application of these methods, we examine the space of
images. This is well-known material (see for example Section~44
of~\cite{Kriegl-Michor}), although it seems to be rarely examined in the
bounded case. In the next definition, 
\indexsymdef{DiffMV}{$\Diff(M,V)$}$\Diff(M,V)$ denotes the
subgroup of $\Diff(M)$ consisting of the diffeomorphisms that take the
submanifold $V$ onto~$V$.

\begin{definition} Let $V$ be a submanifold of $M$ as in
Definition~\ref{def:Imb}. The space 
\indexdef{space of images of a submanifold}\indexsymdef{ImgVM}{$\Img(V,M)$}$\Img(V,M)$ of \textit{images} of $V$
in $M$ is the space of orbits $\Diff(M)/\Diff(M,V)$.
\label{def:images}
\end{definition}
For $j,k\in \Diff(M)$, $j=k$ in $\Img(V,M)$ if and only if
$j(V)=k(V)$. Consequently, we may write elements of $\Img(V,M)$ as $j(V)$
with $j\in \Diff(M,V)$.

The next result is basically Theorem~44.1 of~\cite{Kriegl-Michor}.
\begin{theorem} Let $V$ be a submanifold of a compact manifold $M$.
\begin{enumerate}
\item[(i)] If $V$ has positive codimension, then $\Img(V,M)$ is a Fr\'echet
manifold, locally modeled on the Fr\'echet space of sections from $V$ to its normal
bundle in $M$. Moreover, $\Diff(M)$ is the total space of a locally
trivial principal bundle with structure group $\Diff(M,V)$,
whose base space is $\Img(V,M)$.
\item[(ii)] If $V$ has codimension zero, and $W$ is the frontier of $V$ in
$M$, then the restriction map $\Img(V,M)\to\Img(W,M)$ is either a two-sheeted
covering map or a homeomorphism, according to whether or not there exists a
diffeomorphism of $M$ that preserves $W$ and interchanges $V$ and 
$\overline{M-V}$.
\end{enumerate}\par
\label{thm:space_of_images}
\end{theorem}\par
\begin{proof} Assume first that $V$ has positive codimension.
The map $\Imb(V,M)\to \Img(V,M)$ is $\Diff(M)$-equivariant, so to prove
that $\Imb(V,M)\to \Img(V,M)$ is locally trivial, it it suffices to find
local $\Diff(M)$ cross-sections at points in $\Img(V,M)$. Since $\Diff(M)$
acts transitively on $\Img(V,M)$, it suffices to find a local cross-section
at $1_M\Diff(M,V)$. 

Let $\nu(V)$ be the normal bundle of $V$. For some $\epsilon$, $\Exp\colon
\nu_\epsilon(V)\to M$ is a tubular neighborhood of $V$, where
$\nu_\epsilon(V)$ is the space of vectors of length less than
$\epsilon$. For each $g\Diff(M,V)$ in some neighborhood $U$ of
$1_M\Diff(M,V)$, $g(V)$ meets each fiber of the tubular neighborhood in a
single point. So at each $x\in V$, there is a unique normal vector
$X(g)(x)$ in the fiber $\nu_x(V)$ of the normal bundle of $V$ at $x$ such
that $\Exp(X(g)(x)) = g(V)\cap \Exp(\nu_x(V))$. This defines $X\colon U\to
\mathcal{X}(V,TM)$. Note that $X^{-1}$ defines a local chart for the
Fr\'echet structure of $\Img(V,M)$ at $i$, showing that $\Imb(V,M)$ is a
Fr\'echet manifold.

By the Extension Lemma~\ref{extension}, there is a continuous linear map
$k\colon \mathcal{X}(V,TM)\to \mathcal{X}(M,TM)$ such that $k(Y)(x)=
Y(x)$ for all $x$ in $V$ and all $Y$ in $\mathcal{X}(V,TM)$.  By the
Exponentiation Lemma~\ref{lem:exponentiation}, there is a 
neighborhood $K$ of $Z$ in $\mathcal{X}(M,TM)$ such that $\TExp(Y)$ is
defined for all $Y\in W$, and $\TExp$ carries $W$ into $\Diff(M)$.
Provided that our original $U$ was selected small enough that
$k(X(U))\subset K$, the composition $\TExp\circ k\circ X\colon U\to
\Diff(M)$ is the desired local cross-section. 

Suppose now that $V$ has codimension zero, with frontier $W$. If there does
not exist an element of $\Diff(M,W)$ that interchanges $V$ and
$\overline{M-V}$. 
then $\Diff(M,W)=\Diff(M,V)$, and the restriction map sending $j\Diff(M,V)$
to $j|_W\Diff(M,W)$ defines a homeomorphism between $\Img(V,M)$ and
$\Img(W,M)$. In the remaining case, we fix an element $H_0\in \Diff(M,W)$
that interchanges $V$ and $\overline{M-V}$.

Define $\rho\colon \Img(V,M)\to \Img(W,M)$ by sending $j\Diff(M,V)$ to
$j|_W\Diff(M,W)$, i.~e.~sending $j(V)$ to $j(W)$. This is well-defined,
since if $j_1(V)=j_2(V)$ then $j_1(W)=j_2(W)$.

A free involution $\tau$ on $\Img(V,M)$ is defined by sending $j(V)$ to
$j(H_0(V))$. To see that $\Img(W,M)$ is the quotient of $\Img(V,M)$ by this
involution, let $j_1(V),j_2(V)\in \Img(V,M)$ and suppose that
$j_1(W)=j_2(W)$. Then either $j_1(V)=j_2(V)$ or
$j_1(V)=j_2(\overline{M-V})$. The latter case says exactly that
$j_1(V)=j_2(H_0(V))$.
\end{proof}

\section{Projection of fiber-preserving diffeomorphisms}
\label{project}

Throughout this section and the next, it is understood that $p\colon E\to
B$ is a locally trivial smooth map as in Section~\ref{exponent}, such that
the metric on $B$ is a product near $\partial B$, and the metric on $E$ is
a product near $\partial_hE$ such that the $\I$-fibers of $\partial_hE\times
I$ are vertical.  When $W$ is a vertical submanifold of $E$, it is then
automatic that $W$ meets the collar $\partial_hE\times \I$ in $\I$-fibers.
By rechoosing the metric on $B$, we may assume that $p(W)$ meets the collar
$\partial B\times \I$ in $\I$-fibers. From
Definition~\ref{def:fibered_submanifolds}, we have the notations $\partial_hW=
W\cap\partial_hE$ and $\partial_vW= W\cap\partial_vE$.

We now examine the fundamental lemmas of \cite{P} in the fiber-preserving
case. Lemma~\ref{logarithm} adapts straightforwardly using the aligned
exponential from Definition~\ref{def:aligned} and the aligned vector fields
from Definition~\ref{def:aligned_vector_fields}.

\begin{lemma}[Logarithm Lemma for fiber-preserving maps]\indexstate{Logarithm
Lemma!for fiber-preserving maps}
Assume that $p\colon E\to B$ has compact fiber, and suppose that $W$ is a
compact vertical submanifold of $E$. Then there are an open neighborhood
$U$ of the inclusion $i_W$ in $\Imb_f(W,E)$ and a continuous map $X\colon U
\to\mathcal{A}(W,TE)$ such that for all $j\in U$, $\Exp_a((X(j))(x))$ is
defined for all $x\in W$ and $\Exp_a((X(j))(x))= j(x)$ for all $x\in
W$. Moreover, $X(i_W)=Z$.
\label{lem:aligned_logarithm}
\end{lemma}
\begin{proof}
We adapt the argument in Lemma~\ref{logarithm}, using the aligned
exponential. Choose $\epsilon$ small enough so that for all $x\in W$,
$\Exp$ carries the $\epsilon$-ball about $0$ in $T_x(E)$ (that is, the
portion of this $\epsilon$-ball on which it is defined, which may be as
small as a closed quarter-ball when $x\in \partial_vE\cap \partial_hE$)
diffeomorphically to a neighborhood $W_x$ of $x$ in $E$. Choose a
neighborhood $U$ of $i_W$ in $\Imb_f(W,E)$ so that if $j\in U$ then
$j(x)\in W_x$. For $j\in U$ define $X(j)(x)$ to be the unique vector in
$T_x(M)$ of length less than $\epsilon$ for which $\Exp_a(X(j)(x))$
equals~$j(x)$. Since $p\circ j(x)=p\circ j(y)$ whenever $p(x)=p(y)$, and
$j$ is close to the inclusion, $X\in \mathcal{A}(W,TE)$.
\end{proof}

\begin{lemma}[Extension Lemma for fiber-preserving maps]\index{Extension
Lemma! for fiber-preserving maps}\label{lem:fiber_preserving_aligned_extension}
Let $W$ be a compact vertical submanifold of $E$. Let $T$ be a closed
fibered neighborhood in $\partial_vE$ of $T\cap\partial_vW$, and let
$L\subseteq E$ be a neighborhood of $W$. Then there is a continuous linear
map $k\colon \mathcal{A}(W,TE) \to {\mathcal{A}}^L(E,TE)$ such that
$k(X)(x)= X(x)$ for all $x\in W$ and all $X\in \mathcal{A}(W,TE)$. If
$X(x)= Z(x)$ for all $x\in T\cap\partial_vW$, then $k(X)(x)= Z(x)$ for
all $x\in T$. Furthermore, $k(\mathcal{V}(W,TE))\subset
\mathcal{V}^L(E,TE)$.
\end{lemma}

Figure~\ref{fig:T} illustrates the neighborhood $T$ in
Lemma~\ref{lem:fiber_preserving_aligned_extension}.
\index{figures!figure1@$T$ a neighborhood of $T\cap S$}%
\begin{figure}
\labellist
\pinlabel $E$ [B] at 24 127
\pinlabel $B$ [B] at 24 26
\pinlabel $W$ [B] at 60 117
\pinlabel $p(W)$ [B] at 60 16
\pinlabel $T$ [B] at 119 60
\pinlabel $T$ [B] at 101 46
\pinlabel $p(T)$ [B] at 126 18
\pinlabel $p(T)$ [B] at 108 4
\endlabellist
\begin{center}
\includegraphics[width=5cm]{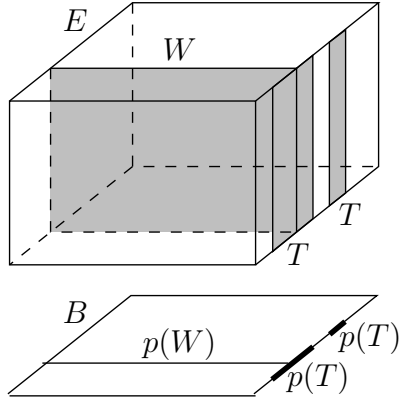}
\caption{The neighborhood $T$ in Lemma~\ref{def:aligned_vector_fields}.}
\label{fig:T}
\end{center}
\end{figure}

\begin{proof}[Proof of Lemma~\ref{lem:fiber_preserving_aligned_extension}]
For an aligned vector field $X$, we use Lemma~\ref{extension} followed by
projection to the vertical components to extend the vertical part. When
$X=Z$ on $T\cap \partial_vW$, the extension and hence its vertical
projection are $Z$ on $T$. For the horizontal part, project to $B$, extend
using Lemma~\ref{extension}, and take horizontal lifts. The extensions in
$B$ are selected to vanish outside a neighborhood $L'$ whose inverse image lies
in $L$. In addition, taking $S=p(T)$ in Lemma~\ref{extension}, the
extension in $B$ is $Z$ on $p(T)$ when $X=Z$ on $T\cap \partial_v(W)$,
ensuring that the lift is $Z$ on $T$ in this case.
\end{proof}

\begin{lemma}[Exponentiation Lemma for fiber-preserving
maps]\indexstate{Exponentiation Lemma!for fiber-preserving maps}
Let $p\colon E\to B$ be a fiber bundle with compact fiber, and assume that
the metric on $E$ is a product near $\partial_h(E)$. Let $K$ be a compact
subset of $E$. Then there exists a neighborhood $U$ of $Z$ in
$\mathcal{A}^K(E,TE)$ such that $\TExp_a(X(x))$ is defined for all $X\in U$
and $\TExp_a$ carries $U$ into $\Diff_f^K(E)$.\par
\label{lem:aligned_exponentiation}
\end{lemma}

\begin{proof} As in the proof of Lemma~\ref{lem:exponentiation}, enlarge
$E$ to a complete open manifold $N$ by attaching $\partial_hE\times
[0,\infty)$ along $\partial_hE$. For each fiber $F$ in $E$, put
$\partial_hF = F\cap \partial_hE$. Then $N$ is still a fiber bundle over
$B$, where each fiber $F$ has been enlarged to an open manifold by
attaching $\partial_hF \times [0,\infty)$ along $F\cap \partial_hE$. We
denote this fibering by $p_N\colon N\to B$.

Now, consider a vector field $X\in \mathcal{A}^K(E,TE)$. We extend the
vertical and horizontal parts $X_v$ and $X_h$ to $N$ separately. For
$X_v$, we extend using 
\index{Extension Lemma!Seeley}Lemma~\ref{lem:Seeley}, then project into the
vertical subspace at each point. For $X_n$, at each point $x\in N-E$, we
just take the horizontal lift of $p_*(X(y))$ for some $y\in E$ with
$p(y)=p_N(x)$, so that the extended vector field is aligned. Its
restriction to $E\cup \partial_hE\times [0,1]$ is tangent to $\partial_hE$.
The vertical part of the extension, constructed using
Lemma~\ref{lem:Seeley}, vanishes off of $K\cup (K\cap \partial_hE)\times
[0,1]$. The proof is now completed as in Lemma~\ref{lem:exponentiation},
using the aligned exponential on $E\cup \partial_hE\times [0,1]$, and
taking $L=K\cup (K\cap \partial_hE)\times [0,1]$.
\end{proof}

\begin{definition}\label{def:projection}
Let $p\colon E\to B$ be a fiber bundle. For an element $g$ of $\Diff_f(E)$,
the induced diffeomorphism of $B$ will be denoted by 
\indexsymdef{g}{$\overline{g}$}$\overline{g}$. More
generally, if $W$ is a vertical submanifold of $E$, each $j\in
\Imb_f(W,E)$ induces an embedding of $p(W)$ into $B$, denoted by
$\overline{j}$. Note that $\Diff_f(E)$ acts on $\Diff(B)$ and on
$\Imb(p(W),B)$ by sending $h$ to $\overline{g}h$. More generally, if
$S\subset \partial B$ is a neighborhood of $S\cap \partial\, p(W)$ and $K$ is
a neighborhood of $W$, then $\Diff_f^{p^{-1}(K)}(E\rel p^{-1}(S))$ acts in
this way on $\Diff^K(B\rel S)$ and $\Imb^K(W,B\rel S)$.
\end{definition}

\begin{theorem} Let $K$ be a compact subset of $B$,
let $S$ be a subset of $\partial B$, and put $T= p^{-1}(S)$.  Then
$\Diff^K(B\rel S)$ admits local $\Diff^{p^{-1}(K)}_f(E\rel T)$
cross-sections.
\label{theorem1}
\end{theorem}

\begin{proof} By Proposition~\ref{prop:inclusion},
it suffices to find a local cross-section at the identity $id_B$. Let $L$
be a compact codimension-zero submanifold of $B$ that contains $K$ in its
topological interior (and such that as usual, $L$ meets $\partial B\times
I$ in $\I$-fibers). By Lemma~\ref{logarithm}, there are a neighborhood $U$
of the inclusion $i_L$ in $\Imb(L,B)$ and a continuous map $X\colon
U\to\mathcal{X}(L,TB)$ such that for all $j\in U$ and all $x\in L$,
$\Exp(X(j)(x))$ is defined and $\Exp(X(j)(x))= j(x)$.

Suppose that $f\in \Diff^K(B\rel S)$. Then $X(f|_L)$ vanishes on a
neighborhood in $L$ of the frontier of $L$ in $B$, and on $L\cap S$, so
the vector field $X(f|_L)$ extends to a smooth vector field $X'(f|_L)$ on
$B$ using $Z$ on $B-K$, which vanishes on $S$. For each $x\in B$,
$\Exp(X'(j)(x))$ is defined and $\Exp(X'(j)(x))= f(x)$.  At each point $y$
of $E$, let $\widetilde{X'}(j)(y)$ be the horizontal lift of
$X'(j)(p(y))$. This produces an aligned vector field in
${\mathcal{A}}^{p^{-1}(L)}(E,TE)$, which vanishes on~$T$.

Choose a neighborhood $V$ of $id_B$ in $\Diff^K(B\rel S)$ such that
$f|_L\in U$ for each $f\in V$. On $V$, define
$\chi(f)=\TExp_a(\widetilde{X'}(f|_L))$. Since $\widetilde{X'}(f|_L))$
vanishes on $T$ and off of $p^{-1}(K)$, this defines $\chi\colon V\to
\Diff^{p^{-1}(K)}_f(E\rel T)$. This is a local cross-section, since given
$b\in K\subseteq L$ we may choose any $y$ with $p(y)= b$ and calculate that
for the induced diffeomorphism $\overline{\chi(f)}$ on $B$,
\begin{align*}
\overline{\chi(f)}(b)&= p(\chi(f)(x))\\
&=p\big(\Exp_a\big(\widetilde{X}(\rho(f))(x)\big)\big)\\
&=\Exp\big(X(\rho(f))(b)\big)\\
&=f(b)
\end{align*}
while at points in $B-K$, $\overline{\chi(f)}(b)=b$.
\end{proof}

From Proposition~\ref{theoremA}, we have immediately
\indexstate{Projection Theorem!bundles}%
\begin{theorem} Let $K$ be a compact subset of $B$.
Let $S\subseteq\partial B$ and let $T= p^{-1}(S)$. Then
$\Diff^{p^{-1}(K)}_f(E\rel T)\to \Diff^K(B\rel S)$ is locally trivial.
\label{project diffs}
\end{theorem}

Each of the fibration theorems we prove has a corresponding corollary
involving parameterized lifting or extension, but since the statements are
all analogous we give only the following one as a prototype.

\begin{corollary}{(Parameterized Isotopy Lifting
Theorem)}\index{Parameterized Isotopy Lifting Theorem}\index{lifting isotopies 
to bundles}\index{isotopy lifting} 
Let $K$ be a compact subset of $B$, let $S\subseteq\partial B$, and let $T=
p^{-1}(S)$. Suppose that for each $t$ in a path-connected parameter space
$P$ there is an isotopy $g_{t,s}$, such that each $g_{t,s}$ is the identity
on $S$ and outside of $K$, and such that $g_{t,0}$ lifts to a
diffeomorphism $G_{t,0}$ of $E$ which is the identity on $T$. Assume that
sending $(t,s)\to g_{t,s}$ defines a continuous function from $P\times
[0,1]$ to $\Diff(B\rel S)$ and sending $t$ to $G_{t,0}$ defines a
continuous function from $P$ to $\Diff(E\rel T)$. Then the family $G_{t,0}$
extends to a continuous family on $P\times \I$ such that for each $(t,s)$,
$G_{t,s}$ is a fiber-preserving diffeomorphism inducing $g_{t,s}$ on~$B$.
\label{isotopy lifting}
\end{corollary}

\section{Restriction of fiber-preserving diffeomorphisms}
\label{restrict}

In this section we present the analogues of the main results of
Palais~\cite{P} in the fibered case. As usual, we tacitly assume that
metrics are products near the boundary and that submanifolds meet the
boundary in $\I$-fibers. We remind the reader about
Figure~\ref{fig:T}, which indicates the setup in the next
result.
\begin{theorem}
Let $W$ be a compact vertical submanifold of $E$. Let $T$ be a closed
fibered neighborhood in $\partial_vE$ of $T\cap \partial_vW$, and let
$L$ be a neighborhood of $W$. Then
\begin{enumerate}
\item[{\rm (i)}] $\Imb_f^L(W,E\rel T)$ admits local $\Diff_f^L(E\rel
T)$ cross-sections, and
\item[{\rm (ii)}] $\Imb_v(W,E\rel T)$ admits local $\Diff_v^L(E\rel T)$ 
cross-sections.
\end{enumerate}
\label{theorem2}
\end{theorem}

\begin{proof} By Proposition~\ref{prop:inclusion}, it suffices to find
local cross-sections at the inclusion $i_W$. Choose a compact neighborhood
$K$ of $W$ with $K\subseteq L$ and $K=p^{-1}(p(K))$.

By Lemma~\ref{lem:aligned_logarithm}, there are a neighborhood $U_1$ of the
inclusion $i_W$ in $\Imb_f(W,E)$ and a continuous map $X\colon
U_1\to\mathcal{A}(W,TE)$ such that for all $j\in U_1$ and all $x\in W$,
$\Exp_a(X(j)(x))$ is defined and $\Exp_a(X(j)(x))= j(x)$.  By
Lemma~\ref{lem:fiber_preserving_aligned_extension}, there is a continuous linear map
$k\colon \mathcal{A}(W,TE)\to \mathcal{A}^K(E,TE)$, with
$k(\mathcal{V}(W,TE))\subset \mathcal{V}^K(E,TE))$, such that $k(X)(x) =
X(x)$ for all $x\in W$. 
\index{Exponentiation Lemma!for fiber-preserving maps}Lemma~\ref{lem:aligned_exponentiation} now gives a
neighborhood $U_2$ of $Z$ in $\mathcal{A}^K(E,TE)$ such that $\Exp_a(X)$ is
defined for all $X\in U_2$, and $\TExp_a$ has image in
$\Diff_f^K(E)$. Putting $U=X^{-1}(k^{-1}(U_2))$, the composition
$\TExp_a\circ k \circ X\colon U\to \Diff_f^K(E)$ is the desired
cross-section for~(i).

Since $X$ carries $\Imb_v(W,B)$ into $\mathcal{V}(W,TE)$, $k$ carries
$\mathcal{V}(W,TE)$ into $\mathcal{V}^K(E,TE)$, and $\TExp_a$ carries
$U_2\cap \mathcal{V}^K(E,TE)$ into $\Diff_v(E)$, this cross-section
restricts on $\Imb_v(W,E\rel T)$ to a $\Diff_v^L(E\rel T)$ cross-section,
giving~(ii).
\end{proof}

Proposition~\ref{theoremA} has the following immediate
corollaries.
\indexstate{Restriction Theorem!relative fiber-preserving}%
\begin{corollary} Let $W$ be a compact vertical submanifold of $E$. Let 
$T$ be a closed fibered neighborhood in $\partial_vE$ of $T\cap
\partial_vW$, and $L$ a neighborhood of $W$. Then the following
restrictions are locally trivial:
\begin{enumerate}
\item[{\rm(i)}] $\Diff_f^L(E\rel T)\to\Imb_f^L(W,E\rel T)$, and
\item[{\rm(ii)}] $\Diff_v^L(E\rel T)\to \Imb_v(W,E\rel T)$.
\end{enumerate}
\label{corollary2}
\end{corollary}

\begin{corollary} Let $V$ and $W$ be vertical
submanifolds of $E$, with $W\subseteq V$. Let $T$ be a closed fibered
neighborhood in $\partial_vE$ of $T\cap \partial_vV$, and let $L$ a
neighborhood of $V$. Then the following restrictions are locally trivial:
\begin{enumerate}
\item[{\rm(i)}] $\Imb_f^L(V,E\rel T)\to\Imb_f^L(W,E\rel T)$.
\item[{\rm(ii)}] $\Imb_v(V,E\rel T)\to \Imb_v(W,E\rel T)$.
\end{enumerate}
\label{corollary3}
\end{corollary}

\newpage
The final result of this section is the projection-restriction square for
bundles.\index{square!projection-restriction!fiber-preserving}%
\index{projection-restriction square!fiber-preserving}
\begin{theorem} Let $W$ be a compact vertical submanifold of
$E$. Let $K$ be a compact neighborhood of $p(W)$ in $B$. Let $T$ be a
closed fibered neighborhood in $\partial_vE$ of $T\cap
\partial_vW$, and put $S= p(T)$. Then all four maps in the following
commutative square are locally trivial:
$$\vbox{\halign{\hfil#\hfil\quad&#&\quad\hfil#\hfil\cr
$\Diff_f^{p^{-1}(K)}(E\rel T)$&$\longrightarrow$&%
$\Imb_f^{p^{-1}(K)}(W,E\rel T)$\cr
\noalign{\smallskip}
$\mapdown{}$&&$\mapdown{}$\cr
\noalign{\smallskip}
$\Diff^K(B\rel S)$&$\longrightarrow$&%
$\Imb^K(p(W),B\rel S)$\rlap{\ .}\cr}}$$\par
\label{square}
\end{theorem}

\begin{proof} The top arrow is 
\index{Restriction Theorem!relative fiber-preserving}%
Corollary~\ref{corollary2}(i), the
left vertical arrow is Theorem~\ref{theorem1}, and the bottom arrow is
\index{Restriction Theorem!Palais-Cerf!relative version}%
Corollary~\ref{palaiscoro2}.  For the right vertical arrow, we will
first show that $\Imb^K(p(W),B \rel\allowbreak S)$ admits local
$\Diff_f^{p^{-1}(K)}(E\allowbreak \rel\allowbreak T)$ cross-sections.
Let $i\in \Imb^K(p(W),B\rel S)$.  Using Theorems~\ref{theorem2}
and~\ref{theorem1}, choose local cross-sections $\chi_1\colon U\to
\Diff^K(B\rel S)$ at $i$ and $\chi_2\colon
V\to\Diff_f^{p^{-1}(K)}\allowbreak(E\rel T)$ at $\chi_1(i)$.  Let
$U_1=\chi_1^{-1}(V)$, then for $j\in U_1$ we have
\[\overline{\chi_2\chi_1(j)}i=\overline{\chi_2(\chi_1(j))}i=
\chi_1(j)i=j\ .\]
\noindent Since the right vertical arrow is
$\Diff_f^{p^{-1}(K)}(E\rel T)$-equivariant, Proposition \ref{theoremA}
implies that it is locally trivial.
\end{proof}

\section{Restriction theorems for orbifolds}
\label{orbifold}

Throughout this section, indeed in all of our work, an 
\indexdef{orbifold}orbifold means an
orbifold in the standard sense
whose universal covering $\pi\colon \widetilde{\O}\to \O$ is a
manifold. We assume further that $\O$ is a smooth orbifold, meaning that
$\widetilde{\O}$ is a smooth manifold and the group $H$ of covering
transformations consists of diffeomorphisms.

\begin{definition}\label{def:equivariant_diffeos}\indexdef{equivariant}
A map $f\colon \widetilde{\O}\to \widetilde{\O}$ is called \textit{(weakly)
$H$-equivariant} if for some automorphism $\alpha$ of $H$,
$f(h(x))=\alpha(h)(f(x))$ for all $x\in \widetilde{\O}$ and $h\in
H$. Define 
\indexsymdef{CinfinityOH}{$\protect\MapsH(\protect\mathcal{O})$}$\MapsH(\widetilde{\O})$ 
to be the space of $H$-equivariant boundary-preserving smooth maps from
$\widetilde{O}$ to $\widetilde{O}$, and $\Diff_H(\widetilde{\O})$ to be the
$H$-equivariant diffeomorphisms of $\widetilde{\O}$. Note that
$\Diff_H(\widetilde{\O})$ is the normalizer of $H$ in
$\Diff(\widetilde{\O})$.
\end{definition}

\begin{definition}\label{def:orbifold_diffeomorphism}\indexdef{orbifold!homeomorphism}\indexdef{orbifold!diffeomorphism}\indexdef{diffeomorphism!orbifold}
An \textit{orbifold homeomorphism} of $\O$ is a homeomorphism of the
underlying topological space of $\O$ that is induced by an $H$-equivariant
homeomorphism of $\widetilde{\O}$, called a \textit{lift} of the orbifold
homeomorphism. An orbifold diffeomorphism of $\O$ is an orbifold
homeomorphism for which some and hence all lifts to $\widetilde{\O}$ are
diffeomorphisms. Define $\Diff(\O )$ to be the group of orbifold
diffeomorphisms.  Note that $\Diff(\O )$ is the quotient of
$\Diff_H({\widetilde{\O}})$ by the normal subgroup~$H$. We give $\Diff(\O)$
the quotient topology of the $\Cinf$-topology on
$\Diff_H({\widetilde{\O}})$.
\end{definition}

\begin{definition}\label{def:suborbifolds}
An orbifold $\mathcal{W}$ contained in $\O$ is called a
\indexdef{suborbifold}\textit{suborbifold} of $\O$ if its inverse 
image $\widetilde{\mathcal{W}}$ in $\widetilde{\O} $ is a submanifold. 
An element of $\Imb(\widetilde{\mathcal{W}},\widetilde{\O})$ is called
\index{equivariant}\textit{$H$-equivariant} if it extends to an element of
$\Diff_H(\widetilde{\O})$, and the subspace of $H$-equivariant embeddings
is denoted by 
\indexsymdef{EmbHWO}{$\protect\mathrm{Emb}_H(\protect\widetilde{\protect\mathcal{W}},\protect\widetilde{\protect\mathcal{O}})$}$\Imb_H(\widetilde{\mathcal{W}},\widetilde{\O})$. An
\indexdef{embedding!orbifold}\textit{embedding} of $\mathcal{W}$ into $\mathcal{\O}$ is an embedding
induced by an element of $\Imb_H(\widetilde{\mathcal{W}},\widetilde{\O})$,
and the space of embeddings is denoted by 
\indexsymdef{EmbWO}{$\protect\Imb(\protect\mathcal{W},\protect\O )$}$\Imb(\mathcal{W},\O )$.
\end{definition}
\noindent Throughout this section, $\mathcal{W}$ will denote a compact
suborbifold of $\O $.

\begin{definition}\label{def:equivariant_vector_field}
A section from an $H$-equivariant subset $\widetilde{L}$ of
$\widetilde{\O}$ to $T\widetilde{\O}|_{\widetilde{L}}$ is called
\indexdef{equivariant!vector field}\textit{$H$-equivariant} if for each $x\in \widetilde{L}$ and each $h\in
H$, $h_*(X(x))=X(h(x))$.  In general, we use a subscript $H$ to indicate
the $H$-equivariant elements of any of the spaces of sections that we have
defined, thus for example
\indexsymdef{XHWTO}{$\protect\mathcal{X}_H(\protect\widetilde{\protect\mathcal{W}},T\protect\widetilde{\protect\O})$}$\mathcal{X}_H(\widetilde{\mathcal{W}},T\widetilde{\O})$ means the
$H$-equivariant elements of $\mathcal{X}(\widetilde{\mathcal{W}},
T\widetilde{\mathcal{O}})$.
\end{definition}

The next two lemmas provide equivariant functions and metrics.
 
\begin{lemma}\index{equivariant!smooth function} 
Let $H$ be a group acting
smoothly and properly discontinuously on a manifold $M$, possibly with
boundary, such that $M/H$ is compact. Let $A$ be an $H$-invariant
closed subset of $M$, and $U$ an $H$-invariant neighborhood of
$A$. Then there exists an $H$-equivariant smooth function
$\gamma\colon M\to [0,1]$ that is identically equal to~$1$ on $A$ and
whose support is contained in~$U$.
\label{equivariant function}
\end{lemma}

\begin{proof} Fix a compact subset $C$ of $M$ which maps surjectively
onto $M/H$ under the quotient map.  Let $\phi\colon M\to[0,\infty)$ be
a smooth function such that $\phi(x)\geq 1$ for all $x\in C\cap A$ and
whose support is compact and contained in $U$. Define $\psi$ by
$\psi(x)= \sum_{h\in H}\phi(h(x))$. Now choose $\eta\colon\R\to[0,1]$
such that $\eta(r)= 0$ for $r\leq 0$ and $\eta(r)= 1$ for $r\geq
1$, and put $\gamma= \eta\circ\psi$.
\end{proof}

When $\O $ is compact, the following lemma provides a Riemannian
metric on $\widetilde{\O}$ for which the covering transformations
are isometries.

\begin{lemma}\index{equivariant!Riemannian metric}
Let $H$ be a group acting
smoothly and properly discontinuously on a manifold $M$, possibly with
boundary, such that $M/H$ is compact. Let $N$ be a properly embedded
$H$-invariant submanifold, possibly empty. Then $M$ admits a complete
$H$-equivariant Riemannian metric, which is a product near $\partial M$,
and such that $N$ meets $\partial M\times \I$ in $\I$-fibers. Moreover, the
action preserves the collar, and if $(y,t)\in \partial M\times \I$ and $h\in
H$, then $h(y,t)= (h\vert_{\partial M}(y),t)$.\par
\label{covering isometries}
\end{lemma}

\begin{proof} We first prove that equivariant Riemannian metrics
exist. Choose a compact subset $C$ of $M$ that maps surjectively onto
$M/H$ under the quotient map. Let $\phi\colon M\to [0,\infty)$ be a
compactly supported smooth function which is positive on $C$. Choose
a Riemannian metric $R$ on $M$ and denote by $R_x$ the inner product
which $R$ assigns to $T_x(M)$. Define a new metric $R'$ by
$$R'_x(v,w)=\sum_{h\in H}\phi(h(x))\,R_{h(x)}(h_*(v),h_*(w))\ .$$
\noindent Since $\phi$ is compactly supported, the sum is finite, and
since every orbit meets the support of $\phi$, $R'$ is positive
definite. To check equivariance, let $g\in H$. Then
\begin{align*}R'_{g(x)}(g_*(v),g_*(w))&=
\sum_{h\in H}\phi(h(g(x)))\,
R_{h(g(x))}(h_*(g_*(v)),h_*(g_*(w)))\\
&=\sum_{h\in H}\phi(hg(x))\,
R_{hg(x)}((hg)_*(v),(hg)_*(w))\\
&=R'_x(v,w)\ .
\end{align*}
\longpage\longpage

We need to improve the metric near the boundary. First, note that
$C\cap\partial M$ maps surjectively onto the image of $\partial M$. Choose
an inward-pointing vector field $\tau'$ on a neighborhood $U$ of $C\cap
\partial M$, which is tangent to $N$. Choose a smooth function $\phi\colon
M\to [0,\infty)$ which is positive on $C\cap \partial M$ and has compact 
support contained in $U$. The field $\phi\tau'$ defined on $U$ extends
using the zero vector field on $M-U$ to a vector field $\tau$ which is
nonvanishing on $C\cap \partial M$. For $x$ in the union of the
$H$-translates of $U$, define $\omega_x=\sum_{h\in
H}\phi(h(x))\,h_*^{-1}(\tau_{h(x)})$. This is defined, nonsingular, and
equivariant on an $H$-invariant neighborhood of $\partial M$, and we use it
to define a collar $\partial M\times[0,2]$ equivariant in the sense that if
$(y,t)\in\partial M\times[0,2]$ then $h(y,t)= (h\vert_{\partial
  M}(y),t)$. Moreover, $N$ meets this collar in $\I$-fibers. On $\partial
M\times[0,2]$, choose an equivariant metric $R_1$ which is the product of
an equivariant metric on $\partial M$ and the standard metric on $[0,2]$,
and choose any equivariant metric $R_2$ defined on all of $M$. Using
Lemma~\ref{equivariant function}, choose $H$-equivariant functions $\phi_1$
and $\phi_2$ from $M$ to $[0,1]$ so that $\phi_1(x)= 1$ for all $x\in
\partial M\times [0,3/2]$ and the support of $\phi_1$ is contained in
$\partial M\times[0,2)$, and so that $\phi_2(x)= 1$ for $x\in M-\partial
  M\times [0,3/2]$ and the support of $\phi_2$ is contained in $M-\partial
  M\times[0,1]$.  Then, $\phi_1R_1+\phi_2R_2$ is $H$-equivariant and is a
  product near $\partial M$, and $N$ is vertical in $\partial M\times \I$.

Since $M/H$ is compact and $H$ acts as isometries, the metric must be
complete. For let $C$ be a compact subset of $M$ that maps
surjectively onto $M/H$. We may enlarge $C$ to a compact
codimension-zero submanifold $C'$ such that every point of $M$ has a
translate which lies in $C'$ at distance at least a fixed~$\epsilon$
from the frontier of $C'$. Then, any Cauchy sequence in $M$ can be
translated, except for finitely many terms, into a Cauchy sequence
in $C'$. Since $C'$ is compact, this converges, so the original sequence
also converged.
\end{proof}

\begin{proposition} Suppose that $H$ acts properly discontinuously on 
a locally compact connected Hausdorff space $X$, and that $X/H$ is
compact. Then $H$ is finitely generated.
\label{finitely generated}
\end{proposition}

\begin{proof} Using local compactness, there exists a compact set $C$
whose interior maps surjectively to $X/H$. Let $H_0$ be the subgroup
generated by the finitely many elements $h$ such that $h(C)\cap C$ is
nonempty. The union of the $H_0$-translates of $C$ is an open and closed
subset, so must equal $X$. This implies that $H= H_0$.
\end{proof}

\begin{definition}\label{def:controlled_equivariant}
Let $A$ be an $H$-invariant subset of $\widetilde{O}$. Define
\indexsymdef{CinfinityOHA}{$(\protect\Maps)_H^A(\protect\widetilde{\protect\O})$}$(\Maps)_H^A(\widetilde{\O})$ to be the elements of
$\MapsH(\widetilde{\O})$ that fix each point not in $A$, and
define 
\indexsymdef{DiffHAO}{$\protect\Diff_H^A(\protect\widetilde{\protect\mathcal{O}})$}$\Diff_H^A(\widetilde{\mathcal{O}})$ similarly. If $A$ is a neighborhood of
$\widetilde{\mathcal{W}}$, define
\indexsymdef{EmbHWOA}{$\protect\Imb_H^A(\protect\widetilde{\protect\mathcal{W}},\protect\widetilde{\protect\mathcal{\protect\O}})$}$\Imb_H^A(\widetilde{\mathcal{W}},\widetilde{\mathcal{\O}})$ to be the
elements of $\Imb_H(\widetilde{\mathcal{W}},\widetilde{\mathcal{\O}})$ that
carry $\mathcal{\widetilde{W}}$ into $A$. We use this
notation to extend our previous concepts to orbifolds. For example, if $K$
is a neighborhood of a suborbifold $\mathcal{W}$ in $\O$, then
$\Imb^K(\mathcal{W},\mathcal{O})$ is the subspace of
$\Imb(\mathcal{W},\mathcal{O})$ induced by elements of
$\Imb_H^{\pi^{-1}(K)}(\widetilde{\mathcal{W}},\widetilde{\mathcal{O}})$,
$\mathcal{X}^K(\mathcal{\O},T\mathcal{\O})$ is the subspace of elements
of $\mathcal{X}_H^K(\mathcal{\O},T\mathcal{\O})$ that equal $Z$ outside of
$\pi^{-1}(K)$, and so on.
\end{definition}

\begin{lemma}
Suppose that $H$ acts properly discontinuously as isometries on
$\widetilde{\O}$.  Let $\widetilde{K}$ be an $H$-invariant subset of
$\widetilde{\O}$ whose quotient in $\O$ is compact. Then there exists a
neighborhood $J$ of $1_{\widetilde{\O}}$ in
$(\Maps)_H^{\widetilde{K}}(\widetilde{\O})$ that consists of
diffeomorphisms.
\label{orblemmaB}
\end{lemma}

\begin{proof} Assume for now that $\O$ is compact and 
$\widetilde{\O}=\widetilde{K}$, and fix a compact set
$C$ in $\widetilde{\O}$ that maps surjectively to $\O$. 

We claim that if $f\in\MapsH(\widetilde{\O})$ is close enough to
$1_{\widetilde{\mathcal{O}}}$, then $f$ commutes with the $H$-action. By
Proposition~\ref{finitely generated}, $H$ is finitely generated.  Choose an
$x\in\widetilde{\O}$ which is not fixed by any nontrivial element of $H$.
Define $\Phi\colon (\Maps)_H^{\widetilde{K}}(\widetilde{\mathcal{O}})\to
\hbox{End}(H)$ by sending $f$ to $\phi_f$ where
$f(h(x))=\phi_f(h)f(x)$. This is independent of the choice of $x$, and
is a homomorphism. If $f$ is close enough to $1_{\widetilde{\O}}$ on
$\set{x,h_1(x),\ldots,h_n(x)}$, where $\set{h_1,\ldots,h_n}$ generates $H$,
then $\phi_f= 1_H$. This prove the claim.

%Next we will show that if $f$ is close enough to $1_{\widetilde{\O}}$ to
%ensure that $f$ commutes with the $H$-action, then $f^{-1}(S)$ is compact
%whenever $S$ is compact. Let $C$ be a compact set in $\widetilde{\O}$ which
%maps surjectively to $\O$. Suppose that $S$ is a closed set for which
%$f^{-1}(S)$ meets infinitely many translates of $C$. Since $f$ commutes
%with the $H$-action, this implies that $S$ meets infinitely many translates
%of $f(C)$, so $S$ cannot be compact.

For the remainder of the argument, we require $f$ to be close enough to
$1_{\widetilde{\O}}$ to ensure that $f$ commutes with the $H$-action. This
implies that $f^{-1}(S)$ is compact whenever $S$ is compact. For if $S$ is
a subset for which $f^{-1}(S)$ meets infinitely many translates of $C$,
then $S$ meets infinitely many translates of $f(C)$, so $S$ cannot be
compact.

Requiring in addition that $f$ be sufficiently
$\Cinf$-close to $1_{\widetilde{\O}}$, we have
$f_*$ nonsingular at each
point of $C$, hence on all of $\widetilde{\O}$. Since $f$ takes boundary
to boundary, it follows that $f$ is a local diffeomorphism. Since inverse images
of compact sets under $f$ are compact, $f$ is a covering map. And since
$\widetilde{\O}$ is simply-connected, $f$ is a diffeomorphism.

Now suppose that $\O$ is noncompact. Choose a compact
co\-di\-men\-sion-zero suborbifold $\mathcal{L}$ of $\O$ that contains
$\widetilde{K}/H$ in its topological interior. Each element of
$(\Maps)^{\widetilde{K}}_H(\widetilde{\mathcal{L}})$ extends to an element
of $(\Maps)^{\widetilde{K}}_H(\widetilde{\O})$ by using the identity on
$\widetilde{\O}-\widetilde{\mathcal{L}}$.  Applying the case when $\O$ is
compact, that is, using $\widetilde{\mathcal{L}}$ in place of $\O$, some
neighborhood of the identity in $\MapsH(\widetilde{\mathcal{L}})$
consists of maps which are diffeomorphisms
on~$\widetilde{\mathcal{L}}$. The intersection of this neighborhood with
$(\Maps)^{\widetilde{K}}_H(\widetilde{\mathcal{L}})$ consists of
diffeomorphisms, and their extensions to $\widetilde{\O}$ form the desired
neighborhood of the identity in
$(\Maps)^{\widetilde{K}}_H(\widetilde{\O})$.
\end{proof}

We now prove the analogues of Lemmas~\ref{logarithm} and~\ref{extension}
for vector fields on $\O $. Assume that $\mathcal{W}$ is a compact
suborbifold of $\O $.

\begin{lemma}[Equivariant Logarithm Lemma]\indexstate{Logarithm Lemma!equivariant}
There are a neighborhood $U$ of the inclusion $i_{\widetilde{\mathcal{W}}}$
of $\widetilde{\mathcal{W}}$
into $\widetilde{\O}$ in $\Imb_H(\widetilde{\mathcal{W}},\widetilde{\O})$
and a continuous map $X\colon U \to
\mathcal{X}_H(\widetilde{\mathcal{W}},T\widetilde{\O})$ such that for all
$j\in U$, $\Exp(X(j)(x))$ is defined for all $x\in \widetilde{\mathcal{W}}$ and
$\Exp(X(j)(x))= j(x)$ for all $x\in \widetilde{\mathcal{W}}$. Moreover,
$X(i_{\widetilde{\mathcal{W}}})=Z$.
\label{orblogarithm}
\end{lemma}

\begin{proof} Replacing $\O $ by a compact orbifold
neighborhood of $\mathcal{W}$ and using Lemma~\ref{covering isometries}, we
may assume that $H$ acts as isometries on $\widetilde{\O}$, that the metric
is a product near $\partial\widetilde{\O}$, and that
$\widetilde{\mathcal{W}}$ meets the collar
$\partial\widetilde{\mathcal{O}}\times \I$ in $\I$-fibers. The proof then
follows the argument of Lemma~\ref{logarithm}, working equivariantly in
$\widetilde{\O}$.
\end{proof}

\begin{lemma}[Equivariant Extension Lemma]\index{Extension Lemma!equivariant}
Let $\mathcal{W}$ be a compact 
suborbifold of $\O$. Let $L$ be a neighborhood of $\mathcal{W}$ in $\O$ and
let $S$ be a closed neighborhood in $\partial\mathcal{O}$ of
$S\cap\partial\mathcal{W}$.  Denote the inverse images in $\widetilde{\O}$ by
$\widetilde{L}$ and $\widetilde{S}$. Then there exists a continuous map
$k\colon \mathcal{X}_H(\widetilde{\mathcal{W}},T\widetilde{\O})\to
\mathcal{X}_H^{\widetilde{L}}(\widetilde{\O},T\widetilde{\O})$ such that
$k(X)(x)= X(x)$ for all $x$ in $\widetilde{\mathcal{W}}$. Moreover,
$k(Z)= Z$, and if $X(x)= Z(x)$ for all $x\in
\widetilde{S}\cap\partial\widetilde{\mathcal{W}}$, then $k(X)(x)= Z(x)$
for all $x\in\widetilde{S}$.
\label{orbextension}
\end{lemma}

\begin{proof} Assume first that $\mathcal{W}$ has positive codimension.
Replacing $\O$ by a compact orbifold neighborhood $\O'$ of $\mathcal{W}$,
$L$ by a compact neighborhood of $\mathcal{W}$ in $L\cap \O'$, and $S$ by
$S\cap \O'$, and using Lemma~\ref{covering isometries}, we may assume that
$H$ acts as isometries on $\widetilde{\O}$, that the metric is a product
near $\partial\widetilde{\O} $, and that $\widetilde{\mathcal{W}}$ meets
the collar $\partial\widetilde{\O} \times \I$ in $\I$-fibers. Let
$\nu(\widetilde{\mathcal{W}})$ be the normal bundle, regarded as a
subbundle of the restriction of $T\widetilde{\O}$ to
${\widetilde{\mathcal{W}}}$. For $\epsilon>0$, let
$\nu_\epsilon(\widetilde{\mathcal{W}})$ be the subspace of all vectors of
length less than $\epsilon$. Since $\mathcal{W}$ is compact and $H$ acts as
isometries on $\widetilde{L}$, $\Exp$ embeds
$\nu_\epsilon(\widetilde{\mathcal{W}})$ as a tubular neighborhood of
$\widetilde{\mathcal{W}}$ for sufficiently small $\epsilon$. By choosing
$\epsilon$ small enough, we may assume that
$\Exp(\nu_\epsilon(\widetilde{\mathcal{W}}))\subset \widetilde{L}$, that
the fibers at points in $\widetilde{S}$ map into $\widetilde{S}$, and that
the fibers at points in $\partial\widetilde{\mathcal{O}}-\widetilde{S}$ map
into $\partial\widetilde{\O}-\widetilde{S}$.

Now use Lemma~\ref{equivariant function} to choose an $H$-equivariant
smooth function $\alpha\colon\widetilde{\O}\to[0,1]$ which is
identically equal to~1 on $\widetilde{\mathcal{W}}$ and has support in
$\Exp(\nu_{\epsilon/2}(\widetilde{\mathcal{W}}))$. The extension $k(X)$ can
now be defined exactly as in Lemma~\ref{extension}. Note that since
$H$ acts as isometries, the parallel translation function $P$ is
$H$-equivariant, and the $H$-equivariance of $k(X)$ follows easily.

Assume now that $\mathcal{W}$ has codimension zero. The frontier $W$ of
$\widetilde{\mathcal{W}}$ is an equivariant properly embedded submanifold
of $\widetilde{\O}$. Since $H$ acts as isometries, we can select an
an equivariant tubular neighborhood of $W$ parameterized
as $W\times (-\infty,\infty)$ with
$\widetilde{\mathcal{W}}\cap (W\times (-\infty,\infty))=W\times
[0,\infty)$, and so that the action of $H$ respects the
$(-\infty,\infty)$-coordinate. By 
\index{Extension Lemma!Seeley}Lemma~\ref{lem:Seeley}, there is a
continuous linear extension operator carrying each vector field on
$\widetilde{\mathcal{W}}$ to a vector field on $\widetilde{\mathcal{W}}\cup
(W\times (-\infty,\infty))$. The extended vector fields are equivariant
since they are defined by a formula in terms of the coordinates of $W\times
[0,\infty)$. At points of $\partial \widetilde{\O}$, the component of each
vector in the direction perpendicular to $\partial \widetilde{\O}$ is $0$,
so the extended component is also $0$ and therefore the extended vector
fields are also tangent to the boundary. After multiplying by an
equivariant function on $\widetilde{\mathcal{W}}\cup (W\times
(-\infty,\infty))$ that is $1$ on $\widetilde{\mathcal{W}}$ and $0$ on
$W\times (-\infty,-1]$, these vector fields extend using $Z$ on
$\widetilde{\O}- (\mathcal{W}\cup (W\times (-\infty,\infty)))$.
\end{proof}

Now we are ready for the analogue of Theorem~B of~\cite{P}. Its statement
and proof use some notation explained in
Definition~\ref{def:controlled_equivariant}.

\begin{theorem} Let $\mathcal{W}$ be a compact suborbifold
of $\O $. Let $S$ be a closed neighborhood in $\partial\O $ of
$S\cap\partial\mathcal{W}$, and let $L$ be a neighborhood of $\mathcal{W}$
in $\O $. Then $\Imb^L(\mathcal{W},\O \rel S)$ admits local
$\Diff^L(\O \rel S)$ cross-sections.  
\par
\label{orbtheoremB}
\end{theorem}

\begin{proof} By Proposition~\ref{prop:inclusion}, it suffices to find a
local cross-section at the inclusion $i_\mathcal{W}$. Choose a compact
neighborhood $K$ of $\mathcal{W}$ with $K\subseteq L$. Using
\index{Logarithm Lemma!equivariant}Lemmas~\ref{orblogarithm}
and~\index{Extension Lemma!equivariant}\ref{orbextension}, we obtain 
an open neighborhood $\widetilde{V}$ of $i_{\widetilde{\mathcal{W}}}$ in
$\Imb_H(\widetilde{\mathcal{W}},\widetilde{\O})$ and continuous maps
$X\colon \widetilde{V}\to
\mathcal{X}_H(\widetilde{\mathcal{W}},T\widetilde{\O})$ and $k\colon
\mathcal{X}_H(\widetilde{\mathcal{W}},T\widetilde{\O})\to
\mathcal{X}_H^{\widetilde{L}}(\widetilde{O},T\widetilde{\O})$.  By
Lemma~\ref{orblemmaB}, there is a neighborhood $J$ of $1_{\widetilde{\O}}$
in $(\Maps)^{\widetilde{K}}_H(\widetilde{\O})$ that consists of
diffeomorphisms.

On a sufficiently small neighborhood $\widetilde{U}$ of
$i_{\widetilde{\mathcal{W}}}$, the function $\widetilde{\chi}\colon
\widetilde{U}\to \Diff_H^{\widetilde{K}}(\widetilde{\O})$ defined by
$\widetilde{\chi}(j)= \TExp \circ k \circ X(j)$ is defined and has
image in $J$. Let $U$ be the embeddings of $\mathcal{W}$ in $\O $ which
admit a lift to $\widetilde{U}$. By choosing $\widetilde{U}$ small enough,
we may ensure that the lift of an element of $U$ is unique. Define $\chi
\colon U \to \Diff^{K}(\mathcal{O})$ to be $\widetilde{\chi}$ applied to
the lift of an element of $U$ to $\widetilde{U}$, followed by the
projection of $\Diff_H^{\widetilde{K}}(\widetilde{\O})$ to
$\Diff^{K}(\mathcal{O})$.

For elements in $U\cap \Imb^K(\mathcal{W},\O \rel S)$, each lift to
$\widetilde{U}$ that is sufficiently close to $i_{\widetilde{\mathcal{W}}}$
must agree with $i_{\widetilde{\mathcal{W}}}$ on $\widetilde{S}$. So $U$
may be chosen small enough so that if $j\in U$ then its lift
$\widetilde{j}$ in $\widetilde{U}$ lies in $\Imb(\widetilde{\mathcal{W}},
\widetilde{\mathcal{O}}\rel\widetilde{S})$. Then, $X(\widetilde{j}(x))=
Z(x)$ for all $x\in \widetilde{S}$, so $k(X)(x)= Z(x)$ for all
$x\in\widetilde{S}$. It follows that $\chi(j)\in\Diff(\O \rel S)$.
\end{proof}

\indexstate{Restriction Theorem!orbifolds}%
\begin{corollary} Let $\mathcal{W}$ be a compact suborbifold
of $\O $, which is either properly embedded or codimension-zero.
Let $S$ be a closed neighborhood in $\partial\O $ of
$S\cap\partial\mathcal{W}$, and let $L$ be a neighborhood of $\mathcal{W}$
in $\O $. Then the restriction $\Diff^L(\O \rel S)\to
\Imb^L(\mathcal{W},\O \rel S)$ is locally trivial.
\label{orbcoro2}
\end{corollary}

\begin{corollary} Let $\mathcal{V}$ and $\mathcal{W}$ be
suborbifolds of $\O $, with $\mathcal{W}\subset \mathcal{V}$. Assume that
$\mathcal{W}$ compact, and is either properly embedded or
codimension-zero. Let $S$ be a closed neighborhood in $\partial\O $ of
$S\cap\partial\mathcal{W}$, and let $L$ be a neighborhood of $\mathcal{W}$
in $\O $. Then the restriction $\Imb^L(\mathcal{V},\O \rel S) \to
\Imb^L(\mathcal{W},{\mathcal{O}}\rel S)$ is locally trivial.
\label{orbcoro3}
\end{corollary}

\section{Singular fiberings}
\label{sfiber}

Throughout this section, $\Sigma$ and $\O$ denote compact connected
orbifolds, in the sense of Section~\ref{orbifold}.

\begin{definition}\label{def:singular_fibering}
A continuous surjection $p\colon \Sigma\to \O $ is called a
\indexdef{singular fiberings}\textit{singular fibering} if there exists a
commutative diagram
\[\vbox{\halign{\hfil$#$\hfil&\hfil$#$\hfil&\hfil$#$\hfil\cr
\widetilde{\Sigma}&\mapright{\widetilde{p}}&\widetilde{\O }\cr
\mapdown{\sigma}&&\mapdown{\tau}\cr
\Sigma&\mapright{p}&\O \cr}}\]
\noindent in which
\begin{enumerate}
\item[{\rm(i)}] $\widetilde{\Sigma}$ and $\widetilde{\O}$ are
manifolds, and $\sigma$ and $\tau$ are regular orbifold coverings with
groups of covering transformations $G$ and $H$ respectively,
\item[{\rm(ii)}] $\widetilde{p}$ is surjective and locally trivial,
and
\item[{\rm(iii)}] the fibers of $p$ and $\widetilde{p}$ are
path-connected.
\end{enumerate}
\end{definition}

\index{singular fiberings!and Seifert fiberings}The class of singular
fiberings includes many Seifert fiberings, for example all compact
3-dimensional Seifert manifolds $\Sigma$ except the \index{lens spaces}lens
spaces with one or two exceptional orbits (see for
example~\cite{Scott}). For some of those lens spaces, $\O $ fails to have
an orbifold covering by a manifold. On the other hand, it is a much larger
class than Seifert fiberings, because no structure as a homogeneous space
is required on the fiber.

For mappings there is a complete analogy with the bundle case, where now
\indexsymdef{DifffSigma}{$\Diff_f(\Sigma)$}$\Diff_f(\Sigma)$ is by definition the quotient of the group of
fiber-preserving $G$-equivariant diffeomorphisms
$(\Diff_G)_f(\widetilde{\Sigma})$ by its normal subgroup $G$, and so on.  A
suborbifold $W$ of $\Sigma$ is called
\indexdef{vertical!suborbifold}\textit{vertical} if it is a union of
fibers. In this case the inverse image $\widetilde{W}$ of $W$ in
$\widetilde{\Sigma}$ is a vertical submanifold, and we write
$\Imb_f(W,\Sigma)$ for embeddings induced by elements of
$(\Imb_G)_f(\widetilde{W},\widetilde{\Sigma})$, $\Imb_v(W,\Sigma)$ for
embeddings induced by elements of
$(\Imb_G)_v(\widetilde{W},\widetilde{\Sigma})$, and so on.

Following our usual notations, we put 
\indexsymdef{delvSigma}{$\partial_h\Sigma$, $\partial_v\Sigma$}$\partial_v\Sigma=
p^{-1}(\partial\O )$, 
\indexsym{delhW}{$\partial_hW$, $\partial_vW$}$\partial_vW= W\cap\partial_v\Sigma$,
$\partial_h\Sigma= \overline{\partial \Sigma - \partial_v\Sigma}$, and
$\partial_hW= \partial W\cap \partial_h\Sigma$.

Since $\O $ is compact, Lemma~\ref{covering isometries} shows that a
(complete) \index{Riemannian metric!equivariant}Riemannian metric on
$\widetilde{\O}$ can be chosen so that $H$ acts as isometries, and moreover
so that the metric on $\widetilde{\O}$ is a 
\index{product near the boundary}product near the boundary.
Next we will sketch how to obtain a $G$-equivariant metric which is a
product near $\partial_h\widetilde{\Sigma}$ and near
$\partial_v\widetilde{\Sigma}$. If $\partial_h\widetilde{\Sigma}$ is empty,
we simply apply Lemma~\ref{covering isometries}. Assume that
$\partial_h\widetilde{\Sigma}$ is nonempty. Construct a $G$-equivariant
collar of $\partial_h\widetilde{\Sigma}$, and use it to obtain a
$G$-equivariant metric such that the $\I$-fibers of
$\partial_h\widetilde{\Sigma}\times \I$ are vertical. If
$\partial_v\widetilde{\Sigma}$ is also nonempty, put
$Y=\partial_h\widetilde{\Sigma}\cap\partial_v\widetilde{\Sigma}$. We will
follow the construction in the last paragraph of
Section~\ref{exponent}. Denote the collar of $\partial_h\widetilde{\Sigma}$
by $\partial_h\widetilde{\Sigma}\times[0,2]_1$. Assume that the metric on
$\partial_h\widetilde{\Sigma}$ was a product on a collar $Y\times[0,2]_2$
of $Y$ in $\partial_h\widetilde{\Sigma}$. Next, construct a $G$-equivariant
collar $\partial_v\widetilde{\Sigma}\times [0,2]_2$ of
$\partial_v\widetilde{\Sigma}$ whose $[0,2]_2$-fiber at each point of
$Y\times [0,2]_1$ agrees with the $[0,2]_2$-fiber of the collar of $Y$ in
$\partial_h\widetilde{\Sigma}\times\set{t}$. Then, the product metric on
$\partial_v\widetilde{\Sigma}\times[0,2]_2$ agrees with the product metric
of $\partial_h\widetilde{\Sigma}\times[0,2]_1$ where they overlap, and the
$G$-equivariant patching can be done to obtain a metric which is a product
near $\partial_v\widetilde{\Sigma}$ without losing the property that it is
a product near~$\partial_h\widetilde{\Sigma}$.  We will always assume that
the metrics have been selected with these properties. By the first sentence
of the next lemma, $G$ preserves the vertical and horizontal parts of
vectors.

Some basic observations about singular fiberings will be needed.

\begin{lemma} The action of $G$ preserves the fibers of
$\widetilde{p}$. Moreover:
\begin{enumerate}
\item[\rm(i)] If $g\in G$, then there exists an element $h\in H$ such
that $\widetilde{p}g= h\widetilde{p}$.
\item[\rm(ii)] If $h\in H$, then there exists an element
$g$ of $G$ such that $\widetilde{p}g= h\widetilde{p}$.
\item[\rm(iii)] If $x\in \Sigma$, then $\tau^{-1}p(x)=
\widetilde{p}\sigma^{-1}(x)$.
\end{enumerate}
\label{lift}
\end{lemma}

\begin{proof}
Suppose that $\widetilde{p}(x)=\widetilde{p}(y)$. For $g\in G$, we
have $\tau\widetilde{p}(g(x))= p\sigma(g(x))= p\sigma(x)=
\tau\widetilde{p}(x)= \tau\widetilde{p}(y)=
\tau\widetilde{p}(g(y))$. Since the fibers of $\widetilde{p}$ are
path-connected, and the fibers of $\tau$ are discrete, this implies
that $g(x)$ and $g(y)$ lie in the same fiber of $\widetilde{p}$. For
(i), let $g\in G$. Since $g$ preserves the fibers of $\widetilde{p}$,
it induces a map $h$ on $\widetilde{\O }$. Given
$x\in\widetilde{\O }$, choose $y\in\widetilde{\Sigma}$ with
$\widetilde{p}(y)= x$. Then $\tau h(x)= \tau
\widetilde{p}(g(y))= p\sigma(g(y))= p\sigma(y)= \tau
\widetilde{p}(y)= \tau(x)$ so $h\in H$.

To prove (ii), suppose $h$ is any element of $H$. Let $\hbox{sing}(\O )$
denote the singular set of $\O $. Choose $a\in\widetilde{\O
}-\tau^{-1}(\hbox{sing}(\O ))$, choose $s\in \widetilde{\Sigma}$ with
$\widetilde{p}(s)= a$, and choose $s''\in \widetilde{\Sigma}$ with
$\widetilde{p}(s'')= h(a)$. Since $p\sigma(s)= \tau\widetilde{p}(s)=
\tau\widetilde{p}(s'')= p\sigma(s'')$, $\sigma(s)$ and $\sigma(s'')$ must
lie in the same fiber of $p$. Since the fiber is path-connected, there
exists a path $\beta$ in that fiber from $\sigma(s'')$ to $\sigma(s)$. Let
$\widetilde{\beta}$ be its lift in $\widetilde{\Sigma}$ starting at $s''$
and let $s'$ be the endpoint of this lift, so that $\sigma(s')=
\sigma(s)$. Note that $\widetilde{p}(s')= \widetilde{p}(s'')= h(a)$
since $\widetilde{\beta}$ lies in a fiber of $\widetilde{p}$. Since
$\sigma(s)= \sigma(s')$, there exists a covering transformation $g\in G$
with $g(s)= s'$.  To show that $\widetilde{p}g= h\widetilde{p}$, it is
enough to verify that they agree on the dense set
$\widetilde{p}^{-1}(\widetilde{\mathcal{O}}-\tau^{-1}(\hbox{sing}(\O
)))$. Let $t\in \widetilde{p}^{-1}(\widetilde{\O
}-\tau^{-1}(\hbox{sing}(\mathcal{O})))$ and choose a path $\gamma$ in
$\widetilde{p}^{-1}(\widetilde{\O }-\tau^{-1}(\hbox{sing}(\mathcal{O})))$
from $s$ to $t$. Since $g\in G$, we have $p\sigma\gamma= p\sigma
g\gamma$. Therefore $\tau\widetilde{p}\gamma= \tau\widetilde{p} g\gamma$,
and so $\widetilde{p}g\gamma$ is the unique lift of $p\sigma\gamma$
starting at $\widetilde{p}g(s)= h(a)$. But this lift equals
$h\widetilde{p}\gamma$, so $h\widetilde{p}(t)= \widetilde{p}g(t)$.

For (iii), fix $z_0\in \sigma^{-1}(x)$ and let $y_0=
\widetilde{p}(z_0)$. Suppose $y\in \widetilde{p}\sigma^{-1}(x)$.
Choose $z\in \sigma^{-1}(x)$ with $\widetilde{p}(z)= y$. Since
$\sigma$ is a regular covering, there exists $g\in G$ such that $g(z)=
z_0$. By~(i), $g$ induces $h$ on $\widetilde{\O}$, and $h(y)=
h\widetilde{p}(z)= \widetilde{p}g(z)= \widetilde{p}(z_0)= y_0$.
Therefore $\tau(y)= \tau(h(y))= \tau(y_0)=
\tau\widetilde{p}(z_0)= p\sigma(z_0)= p(x)$ so $y\in
\tau^{-1}(p(x))$. For the opposite inclusion, suppose that $y\in
\tau^{-1}p(x)$, so $\tau(y)= p(x)= \tau(y_0)$. Since $\sigma$ is
regular, there exists $h\in H$ such that $h(y_0)= y$. Let $g$ be as
in (ii). Then $y= h(y_0)= h\widetilde{p}(z_0)=
\widetilde{p}g(z_0)$, and $\sigma(g(z_0))= \sigma(z_0)= x$ so
$y\in\widetilde{p}(\sigma^{-1}(x))$.
\end{proof}

One consequence of Lemma~\ref{lift} is that there is a unique surjective
homomorphism $\phi\colon G\to H$ with respect to which $\widetilde{p}$ is
equivariant: $\widetilde{p}(gx)=\phi(g)(\widetilde{p}(x))$.

A second consequence of Lemma~\ref{lift} is that provided that $G$ acts as
isometries, the aligned exponential\index{aligned!exponential!equivariance}
$\Exp_a$ for the bundle $\widetilde{p}\colon \widetilde{\Sigma}\to
\widetilde{\mathcal{O}}$ is $G$-equivariant. Consequently, the aligned tame
exponential $\TExp_a$ takes $G$-equivariant vector fields on
$\widetilde{\Sigma}$ to $G$-equivariant smooth maps of
$\widetilde{\Sigma}$.

\begin{theorem} Let $S$ be a closed subset of
$\O$, and let $T= p^{-1}(S)$. Then $\Diff(\O \rel S)$ admits local
$\Diff_f(\Sigma\rel T)$ cross-sections.\par
\label{sftheorem1}
\end{theorem}

\begin{proof} By Proposition~\ref{prop:inclusion}, we only need a 
local $\Diff_f(\Sigma \rel T)$ cross-section at $1_{\O}$.

Applying \indexstate{Logarithm Lemma!equivariant}Lemma~\ref{orblogarithm} with
$\widetilde{\mathcal{W}}=\widetilde{\O}$ provides a neighborhood
$\widetilde{U}$ of $1_{\widetilde{\O}}$ in $\Diff_H(\widetilde{\O}\rel
\tau^{-1}(S))$ and $X\colon \widetilde{U}\to \mathcal{X}_H(\widetilde{\O},
T\widetilde{\O})$ such that $\Exp(X(j)(y))= j(y)$ for all $y\in
\widetilde{\O}$, and $X(j)(y)= Z(y)$ for all $y\in \tau^{-1}(S)$. Define
$\widetilde{X}\colon \widetilde{U}\to
\mathcal{X}(\widetilde{\Sigma},T\widetilde{\Sigma})$ by taking horizontal
lifts, that is,
\[ \widetilde{X}(j)(x) =
\big(\widetilde{p}\vert_{H_x}\big)_*^{-1}(X(j)(\widetilde{p}(x)))\ .\] We
claim that $\widetilde{X}(j)$ lies in
$\mathcal{A}_G(\widetilde{\Sigma}, T\widetilde{\Sigma})$. To verify the
boundary tangency conditions, we observe that $\widetilde{X}(j)$ must be
tangent to the vertical boundary since it is a lift of a vector tangent to
the boundary of $\widetilde{\O}$, and tangent to the horizontal boundary
since it is horizontal. Since
$\Exp(X(j)(y))$ is defined at all points of $\widetilde{\O}$, and
$\widetilde{X}(j)$ is horizontal, each
$\Exp_a(\widetilde{X}(j))(x))$ exists.
To check equivariance, let $g\in G$. By
Lemma~\ref{lift}, there exists $h\in H$ such that
$\widetilde{p}g=h\widetilde{p}$. We then have
\begin{align*}
\widetilde{X}(j)(g(x)) 
&= \big(\widetilde{p}\vert_{H_x}\big)_*^{-1}(X(j)(\widetilde{p}(g(x))))\\
&= \big(\widetilde{p}\vert_{H_x}\big)_*^{-1}(X(j)(h\widetilde{p}(x)))\\
&= \big(\widetilde{p}\vert_{H_x}\big)_*^{-1}(h_*(X(j)(\widetilde{p}(x))))\\
&= g_*\big(\widetilde{p}\vert_{H_x}\big)_*^{-1}(X(j)(h\widetilde{p}(x)))\\
&= g_*\widetilde{X}(j)(x)\ .
\end{align*}
Note also that $\widetilde{X}(j)(x)=Z(x)$ for every 
$x\in \widetilde{p}^{-1}\tau^{-1}(S)$,
since $X(j)(y)=Z(y)$ for every $y\in \tau^{-1}(S)$. 
Using Lemma~\ref{orblemmaB}, we
may pass to a smaller $\widetilde{U}$ if necessary to assume that
$\TExp_a\circ \widetilde{X}\colon \widetilde{U}\to
\Diff_G(\widetilde{\Sigma}\rel \widetilde{p}^{-1}\tau^{-1}(S))$.

Since $\tau\colon \widetilde{\O}\to \O$ is an orbifold covering map, there
exists a neighborhood of $1_{\widetilde{\O}}$ in
$\Diff_H(\widetilde{\O})$ such that no two elements in this neighborhood
induce the same diffeomorphism on $\O$. Intersecting this neighborhood with
$\widetilde{U}$, we may assume that $\widetilde{U}$ has the same property.

By definition, each $f\in \Diff(\O)$ has lifts to elements of
$\Diff_H(\widetilde{\O})$. If $f$ lies in some sufficiently small
neighborhood $U$ of $1_{\O}$, then it has a lift in $\widetilde{U}$. This
lift is unique, by our selection of $\widetilde{U}$, and we denote it by
$\widetilde{f}$. Define $\chi\colon U\to \Diff(\Sigma \rel T)$ by letting
$\chi(f)$ be the diffeomorphism induced on $\Sigma$ by $\TExp_a\circ
\widetilde{X}(\widetilde{f})$. Let $y\in \O$, choose $\widetilde{y}\in
\widetilde{\O}$ with $\tau(\widetilde{y})=y$, and $\widetilde{x}\in
\widetilde{\Sigma}$ with $\widetilde{p}(\widetilde{x})=\widetilde{y}$.
Then we have
\begin{gather*}
(\chi(f) \cdot 1_{\O}) (y)
=\overline{\chi(f)}(y) =\overline{\chi(f)}(\tau\circ \widetilde{p}(\widetilde{x}))
=\overline{\chi(f)}(p \circ \sigma (\widetilde{x}))\\
=p \circ \chi(f)(\sigma (\widetilde{x}))
= p\circ \sigma\circ \TExp_a\circ \widetilde{X}(\widetilde{f})(\widetilde{x})
= \tau\circ \widetilde{p} \circ \TExp_a\circ \widetilde{X}(\widetilde{f})(\widetilde{x})\\
= \tau\circ \Exp \circ \;\widetilde{p}_* \circ \widetilde{X}(\widetilde{f})(\widetilde{x})
= \tau\circ \Exp \circ X(\widetilde{f})(\widetilde{y})
= \tau\circ \widetilde{f}(y)
= f(y)\\
\end{gather*}
as required.
\end{proof}

Applying Proposition~\ref{theoremA}, we have immediately
\indexstate{Projection Theorem!singular fiberings}%
\begin{theorem} Let $S$ be a closed subset of
$\O $, and let $T= p^{-1}(S)$. Then
$\Diff_f(\Sigma\rel T)\to \Diff(\O \rel S)$ is
locally trivial.
\label{sfproject diffs}
\end{theorem}

We now examine 
\index{Logarithm Lemma!for fiber-preserving maps}Lemmas~\ref{lem:aligned_logarithm} 
and\index{Extension Lemma!for fiber-preserving maps}~\ref{lem:fiber_preserving_aligned_extension} 
in the singular fibered case. There is no difficulty in adapting 
Lemma~~\ref{lem:aligned_logarithm}
equivariantly:
\begin{lemma}[Logarithm Lemma for singular fiberings]%
\indexstate{Logarithm Lemma!for singular fiberings}
Let $W$ be a vertical suborbifold of $\Sigma$. Then there are an open
neighborhood $U$ of the inclusion $i_{\widetilde{W}}$ in
$(\Imb_f)_{G}(\widetilde{W},\widetilde{\Sigma})$ and a continuous map
$X\colon U \to\mathcal{A}_G(\widetilde{W},T\widetilde{\Sigma})$ such that
for all $j\in U$, $\Exp_a(X(j)(x))$ is defined for all $x\in
\widetilde{W}$ and $\Exp_a(X(j)(x))= j(x)$ for all $x\in
\widetilde{W}$. Also, $X(i_{\widetilde{W}})=Z$.
\label{sflemmaD}
\end{lemma}

\begin{lemma}[Extension Lemma for singular fiberings]\indexstate{Extension
Lemma!for singular fiberings}
Let $W$ be a vertical suborbifold of $\Sigma$, and $T$ a closed fibered
neighborhood in $\partial_v\Sigma$ of $T\cap \partial_vW$.  Then there
exists a continuous linear map $k\colon
\mathcal{A}_G(\widetilde{W},T\widetilde{\Sigma})\to
\mathcal{A}_G(\widetilde{\Sigma},T\widetilde{\Sigma})$ such that
$k(X)(x)= X(x)$ for all $X\in
\mathcal{A}_G(\widetilde{W},T\widetilde{\Sigma})$ and $x\in
\widetilde{W}$. If $X(x)= Z(x)$ for all $x\in
\widetilde{T}\cap\partial_v\widetilde{W}$, then $k(X)(x)= Z(x)$ for all
$x\in \widetilde{T}$. Moreover, $k(\mathcal{V}_G(\widetilde{W},
T\widetilde{\Sigma}))\subset
\mathcal{V}_G(\widetilde{\Sigma},T\widetilde{\Sigma})$.
\label{sflemmaC}
\end{lemma}

\begin{proof} We may assume that the metrics on $\widetilde{\Sigma}$ and
$\widetilde{\O}$ are $G$- and $H$-equivariant; in particular, $G$ takes
horizontal subspaces of $T\widetilde{\Sigma}$ to horizontal
subspaces. Notice that $\widetilde{p}_*$ carries $G$-invariant aligned
vector fields to $H$-invariant vector fields; this uses
Lemma~\ref{lift}(ii). It follows that the aligned exponential on
$\widetilde{\Sigma}$ is $G$-equivariant.  For let $X\in
\mathcal{A}_G(T\widetilde{\Sigma})$ and let $g\in G$. Let $x\in
\widetilde{\Sigma}$ and let $\widetilde{F}_x$ be the fiber of
$\widetilde{p}$ containing $x$. At $x$, $X(x)= X(x)_v+X(x)_h$. Since $g$ is
an isometry, $X(g(x))_v= g_*(X(x)_v)$ and $X(g(x))_h= g_*(X(x)_h)$. To find
$\Exp_a(X(x))$, we first find $\Exp_v(X(x)_v)$, that is, exponentiate
$X(x)_v$ using the metric induced on $\widetilde{F}_x$. This ends at a
point $x'\in \widetilde{F}_x$. Since $G$ acts as isometries,
$\Exp_v(g_*X(x)_v)= g\Exp_v(X(x)_v)=g(x')$.  Now, use Lemma~\ref{lift} to
obtain $\lambda\in H$ with $\lambda \widetilde{p} = \widetilde{p} g$. We
have $\lambda_*\widetilde{p}_*(X(x)_h) = \widetilde{p}_*(g_*(X(x)_h))=
\widetilde{p}_*(X(g(x))_h)$. Since $\lambda$ is an isometry, it carries the
geodesic in $\mathcal{O}$ determined by $\widetilde{p}_*(X(x)_h)$ to the
geodesic determined by $\widetilde{p}_*(X(g(x))_h)$. Therefore $g$ carries
the horizontal lift of $\widetilde{p}_*(X(x)_h)$ at $x'$ to the horizontal
lift of $\widetilde{p}_*(X(g(x))_h)$ at $g(x')$. So $g$ carries
$\Exp_a(X(x))$ to $\Exp_a(X(g(x)))$.

We can now proceed as in the proof of Lemma~\ref{lem:fiber_preserving_aligned_extension}.
Given a $G$-equivariant aligned section on $W$, extend the vertical part as
in Lemma~\ref{extension} and project the extension to the vertical
subspace. This process is equivariant since we use a $G$-equivariant metric
and $G$-equivariant functions to taper off the local extensions.  For the
horizontal part, project to $\widetilde{\O}$, extend $H$-equivariantly
using Lemma~\ref{orbextension}, and lift.
\end{proof}

\begin{theorem}
Let $W$ be a vertical suborbifold of $\Sigma$. Let $T$ be a closed
fibered neighborhood in $\partial_v\Sigma$ of $T\cap \partial_vW$.
Then
\begin{enumerate}
\item[{\rm (i)}] $\Imb_f(W,\Sigma\rel T)$ admits local
$\Diff_f(\Sigma\rel T)$ cross-sections, and
\item[{\rm (ii)}] $\Imb_v(W,\Sigma\rel T)$ admits local
$\Diff_v(\Sigma\rel T)$ cross-sections.
\end{enumerate}
\label{sftheorem2}
\end{theorem}

\begin{proof} By Proposition~\ref{prop:inclusion}, it suffices to find
local cross-sections at the inclusion $i_W$.

By \index{Logarithm Lemma!for singular fiberings}Lemma~\ref{sflemmaD},
there are an open neighborhood $\widetilde{U}$ of the inclusion
$i_{\widetilde{W}}$ in $(\Imb_f)_{G}(\widetilde{W},\widetilde{\Sigma})$ and
a continuous map $X\colon
\widetilde{U}\to\mathcal{A}_G(\widetilde{W},T\widetilde{\Sigma})$ such that
for all $j\in \widetilde{U}$, $\Exp_a(X(j)(x))$ is defined for all $x\in
\widetilde{W}$ and $\Exp_a(X(j)(x))= j(x)$ for all $x\in \widetilde{W}$. By
\index{Extension Lemma!for singular fiberings}Lemma~\ref{sflemmaC}, there
exists a continuous linear map $k\colon
\mathcal{A}_G(\widetilde{W},T\widetilde{\Sigma})\to
\mathcal{A}_G(\widetilde{\Sigma},T\widetilde{\Sigma})$ such that $k(X)(x)=
X(x)$ for all $X\in \mathcal{A}_G(\widetilde{W},T\widetilde{\Sigma})$ and
$x\in \widetilde{W}$. Additionally, $k(X)(x)= Z(x)$ for all $x\in
\widetilde{T}$, and $k(\mathcal{V}_G(\widetilde{W},
T\widetilde{\Sigma}))\subset
\mathcal{V}_G(\widetilde{\Sigma},T\widetilde{\Sigma})$.

\index{Exponentiation Lemma!for fiber-preserving maps}Lemma~\ref{lem:aligned_exponentiation} now gives a neighborhood
$\widetilde{U}_1$ of $Z$ in
$\mathcal{A}_G(\widetilde{\Sigma},T\widetilde{\Sigma})$ such $\Exp_a(X)$
is defined for all $X\in \widetilde{U}_1$, and $\TExp_a$ has image in
$\Diff_f^K(E)$. Putting $U=X^{-1}\circ k^{-1}(\widetilde{U}_1)$, the
composition $\TExp_a\circ k \circ X\colon \widetilde{U}\to \Diff_f^K(E)$ is
the desired cross-section for~(i).

Since $X$ carries $\Imb_v(W,\Sigma)$ into
$\mathcal{V}_G(\widetilde{W},T\widetilde{\Sigma})$, $k$ carries
$\mathcal{V}_G(\widetilde{W},T\widetilde{\Sigma})$ into
$\mathcal{V}_G(\widetilde{\Sigma},T\widetilde{\Sigma})$, and $\TExp_a$
carries $\widetilde{U}_1\cap
\mathcal{V}_G(\widetilde{\Sigma},T\widetilde{\Sigma})$ into
$\Diff_v(\widetilde{\Sigma})$, this cross-section restricts on
$\Imb_v(W,\Sigma \rel T)$ to a $\Diff_v^L(\Sigma\rel T)$ cross-section,
giving~(ii).
\end{proof}

As in Section~\ref{restrict}, we have the following immediate corollaries.
\indexstate{Restriction Theorem!singular fiberings}%
\begin{corollary} Let $W$ be a vertical suborbifold of
$\Sigma$. Let $T$ be a fibered neighborhood in $\partial_v\Sigma$ of
$T\cap\partial_vW$. Then the following restrictions are locally
trivial:
\begin{enumerate}
\item[{\rm(i)}] $\Diff_f(\Sigma\rel T)\to \Imb_f(W,\Sigma\rel T)$, and
\item[{\rm(ii)}] $\Diff_v(\Sigma\rel T)\to
\Imb_v(W,\Sigma\rel T)$.
\end{enumerate}
\label{sfcorollary2}
\end{corollary}

\begin{corollary} Let $V$ and $W$ be vertical
suborbifolds of $\Sigma$, with $W\subseteq V$. Let $T$ be a closed
fibered neighborhood in $\partial_v\Sigma$ of $T\cap\partial_vW$. Then
the following restrictions are locally trivial:
\begin{enumerate}
\item[{\rm(i)}] $\Imb_f(V,\Sigma\rel T)\to
\Imb_f(W,\Sigma\rel T)$, and
\item[{\rm(ii)}] $\Imb_v(V,\Sigma\rel T)\to
\Imb_v(W,\Sigma\rel W)$.
\end{enumerate}
\label{sfcorollary3}
\end{corollary}

\index{square!projection-restriction!singular fiberings}%
\index{projection-restriction square!singular fiberings}%
\begin{theorem} Let $W$ be a vertical suborbifold of
$\Sigma$. Let $T$ be a closed fibered neighborhood in
$\partial_v\Sigma$ of $T\cap\partial_vW$, and let $S= p(T)$. Then
all four maps in the following square are locally trivial:
$$\vbox{\halign{\hfil#\hfil\quad&#&\quad\hfil#\hfil\cr
$\Diff_f(\Sigma\rel T)$&$\longrightarrow$&%
$\Imb_f(W,\Sigma\rel T)$\cr
\noalign{\smallskip}
$\mapdown{}$&&$\mapdown{}$\cr
\noalign{\smallskip}
$\Diff(\O \rel S)$&$\longrightarrow$&%
$\Imb(p(W),\O \rel S)$\rlap{\ .}\cr}}$$\par
\label{sfsquare}
\end{theorem}

\section{Spaces of fibered structures}
\label{sec:fibered_structures}

In this section, we examine spaces of fibered structures.
\begin{definition}
Let $p\colon \Sigma\to \O$ be a singular fibering. The 
\indexdef{space of fibered structures}\textit{space of
  fibered structures}
isomorphic to $p$, 
(also called the
\indexdef{space of singular fiberings}\textit{space of singular fiberings}
isomorphic to $p$) is the space of
cosets $\Diff(\Sigma)/\Diff_f(\Sigma)$.\par
\label{def:fibered_structures}
\end{definition}

Our proof of the next theorem requires an additional condition, although we
do not know that it is necessary:
\begin{definition}\indexdef{very good!singular fibering}\indexdef{singular
fiberings!very good}
A singular fibering $p\colon \Sigma\to \O$ is called \textit{very good}
if $\widetilde{\Sigma}$ may be chosen to be compact. \par
\label{def:very_good_fibered_structure}
\end{definition}

The main result of this section is the following fibration theorem.
\begin{theorem}\index{space of singular fiberings!Frechet
structure@Fr\'echet structure}
Let $p\colon \Sigma\to \O$ be a very
good singular fibering. Then the space of fibered structures isomorphic to
$p$ is a Fr\'echet manifold locally modeled on the
quotient $\mathcal{X}_G(\widetilde{\Sigma},
T\widetilde{\Sigma})/\mathcal{A}_G(\widetilde{\Sigma},T\widetilde{\Sigma})$.
The quotient map $\Diff(\Sigma)\to \Diff(\Sigma)/\Diff_f(\Sigma)$ is a
locally trivial fibering.
\label{thm:fibered_structures}
\end{theorem}

\index{straightening fibers|(}Here is the basic idea of the proof. Roughly
speaking, finding a local $\Diff(\Sigma)$ cross-section for
$\Diff(\Sigma)/\Diff_f(\Sigma)$ boils down to the problem of taking an
$h\in \Diff(\Sigma)$ that carries fibers of $\Sigma$ to fibers that are
nearly vertical, and finding, for each fiber $F$ of $\Sigma$, a ``nearest''
vertical fiber to $h(F)$. It is not obvious that such a choice is uniquely
determined, but there is a way to make one when $h$ is sufficiently close
to a fiber-preserving diffeomorphism. For then each $p(h(F))$ lies a very
small open ball set in $B$, and $p(h(F))$ has a unique \textit{center of
  mass} $c_{p(h(F))}$. The natural choice for the nearest fiber to $h(F)$
is~$p^{-1}(c_{p(h(F))})$.

Before beginning the proof, we must clarify the idea of center of mass in
this context. A useful reference for this is
H.~\index{Karcher}Karcher~\cite{Karcher}, which we will follow here.

\indexdef{center of mass}Let $A$ be a measure space of volume $1$ and let
$B$ be an open ball in a compact Riemannian manifold $M$. By making its
radius small enough, we may ensure that the closure $\overline{B}$ is a
geodesically convex ball (that is, any two points in $\overline{B}$ are
connected by a unique geodesic that lies in $\overline{B}$). Let $f\colon
A\to M$ be a measurable map such that $f(A)\subset B$. Define $P_f\colon
\overline{B}\to \R$ by
\[P_f(m)=\frac{1}{2}\int_Ad(m,f(a))^2\,dA\ .\]
Various estimates on the gradient of $P_f$, detailed in~\cite{Karcher},
show that $P_f$ is a convex function that has a unique minimum in $B$, and
this minimum is defined to be the center of mass $C_f$ of~$f$. From its
definition, $C_f$ is independent of the choice of~$B$, although it is the
existence of such a $B$ that serves to ensure that it is uniquely defined.

\begin{proof}[Proof of
Theorem~\ref{thm:fibered_structures}]
Consider first the case of an ordinary bundle $p\colon E\to B$ with $E$
compact. For each $x\in E$, the fiber containing $x$ will denoted by $F_x$.
For each coset $h\Diff_f(E)$, the set of images $\{h(F_x)\}$ is independent
of the coset representative, and we will refer to these submanifolds as
``image fibers'', reserving ``fibers'' for the original fibers for
$p$. When the coset $h\Diff_f(E)$ is clear from the context, the image
fiber containing $x$ will be denoted by~$F_x'$.

Write $n$ for the dimension of $E$ and $k$ for the dimension of $B$. The
tangent bundle of $E$ has an associated bundle $G_k(TE)$ whose fiber is the
Grassmannian of $k$-planes in $\R^n$, and selecting the horizontal
$k$-plane at each point defines a section $s_0\colon E\to G_k(TE)$. The
normal subspaces for the image fibering of $h\Diff_f(E)$ determine another
section $s\colon E\to G_k(TE)$, defining a function $\Diff(E)/\Diff_f(E)\to
\Maps(E,G_k(TE))$. This function is injective, since distinct fiberings
must have different normal spaces at some points, so imbeds
$\Diff(E)/\Diff_f(E)$ into the Fr\'echet space of sections from $M$ into
$G_k(TE)$. This defines the topology on $\Diff(E)/\Diff_f(E)$. In
particular, we can speak of image fiberings as being $\Cinf$-close to
vertical, meaning that the section $s$ is $\Cinf$-close to~$s_0$.
\longpage

We will first produce local $\Diff(E)$ cross-sections, then examine the
Fr\'echet structure on $\Diff(E)/\Diff_f(E)$. Since $\Diff(E)$ acts
transitively on $\Diff(E)/\Diff_f(E)$, it is enough to produce a local
cross-section at the identity coset~$1_E\Diff_f(E)$.

For $\epsilon>0$, denote by $H_\epsilon(F_x)$ the space of horizontal
vectors in $TE|_{F_x}$ of length less than $\epsilon$ that are carried into
$E$ by the aligned exponential $\Exp_a$. By compactness, there exists an
$\epsilon_0>0$ such that for every $x\in E$, $\Exp_a$ carries
$H_{\epsilon_0}(F_x)$ diffeomorphically onto a tubular neighborhood
$N_{\epsilon_0}(F_x)$ of $F_x$ in~$E$. We may also choose $\epsilon_0$ so
that each ball in $B$ of radius at most $\epsilon_0$ has convex closure.

By compactness of $E$, there exists a neighborhood $U$ of $1_E\Diff_f(E)$
in $\Diff(E)/\Diff_f(E)$ such that for each $h\Diff_f(E)\in U$, the image
fibering of $h\Diff_f(E)$ has the following property: For each $y\in E$,
there exists a fiber $F_x$ such that $F_y'\subset N_{\epsilon_0}(F_x)$, and
moreover if $F_x$ is any such fiber, then $F_y'$ meets each normal fiber of
$N_{\epsilon_0}(F_x)$ transversely in a single point.

\index{figures!figure11@canonical straightening}%
\begin{figure}
\labellist
\pinlabel $E$ [B] at 24 127
\pinlabel $B$ [B] at 24 26
\pinlabel $y$ [B] at 34 62
\pinlabel $x=h(y)$ [B] at 82 117
\pinlabel $F'_x=h(F_y)$ [B] at 85 130
\pinlabel $V(F'_x)$ [B] at 107 99
\pinlabel $F_y$ [B] at 24 82
\pinlabel $p(F'_x)$ [B] at 78 23
\endlabellist
\begin{center}
\includegraphics[width=50ex]{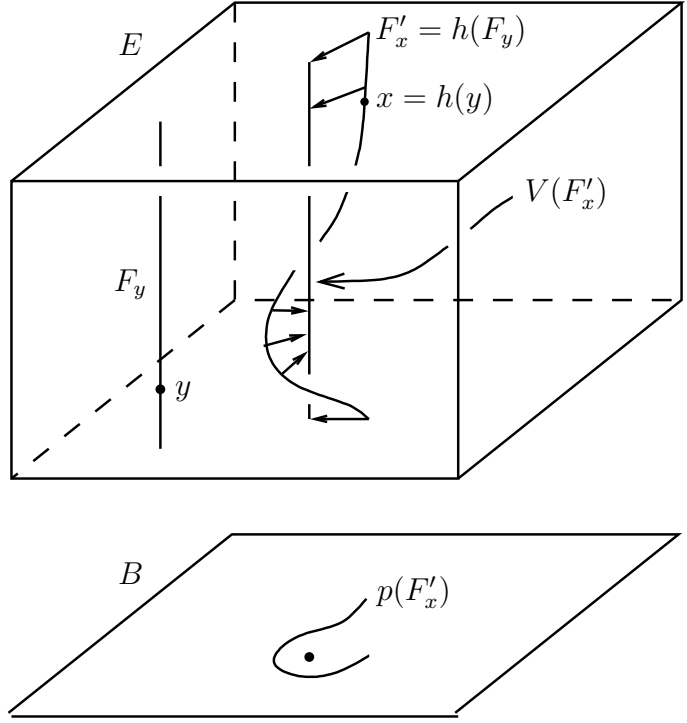}
\caption{Canonical straightening of a nearly vertical fiber. The dot in $B$
is the center of mass of the projection $p(F_x')$ of the image fiber
$F_x'$. The inverse image of the center of mass is the straightened fiber
$V(F_x')$, and some of the horizontal vector field $X$ is shown.}
\label{fig:straightening}
\end{center}
\end{figure}
Now we will set up the center-of-mass construction, illustrated in
Figure~\ref{fig:straightening}. Fix a coset
$h\Diff_f(E)$ and an image fiber $F_x'$, where $x=h(y)$ for some $y$. Let
$dF_x'$ be the volume form on $F_x'$ obtained from restriction of the
Riemannian metric on $TE$ to $TF_x'$, and define a measure $\mu_{F_x'}$ on
$F_x'$ of volume~$1$ by $\mu_{F_x'}(U)=\Vol(U)/\Vol(F_x')$.

Assume now that $h\,\Diff_f(E)$ is close enough to vertical that for each
image fiber $F_x'$, $p(F_x')$ lies in some $\epsilon_0$-ball.
The center of mass of $(F_x',m_{F_x'})$ is then defined, and we denote its
inverse image, a fiber of $p$, by $V(F_x')$.

For each $z\in E$, let $n(z)$ be the point of $V(F_z')$ such that the
normal fiber of $N_{\epsilon_0}(V(F_z'))$ at $n(z)$ contains $z$. There is
a unique horizontal vector $X(z)\in T_zE$ such that $\Exp_a(X(z))=n(z)$. To
see that the resulting horizontal vector field $X$ is smooth, we first
observe that changes of $z$ along the image fiber simply correspond to
changes of $n(z)$ along the fiber $V(F_z')$. As $z$ moves from image fiber
to image fiber, the projected images in $B$ of the image fibers are the
images of the original fibers of $E$ under the smooth map $p\circ h$. The
corresponding centers of mass change smoothly, and the remainder of the
construction presents no danger of loss of smoothness. Precomposing $h$ by
a fiber-preserving diffeomorphism does not change the image fibers, so
$X(h\Diff_f(E))$ is well-defined. If $h$ is fiber-preserving, then
$V(F_z')=F_z'$, $n(z)=z$, and $X(h)=Z$.

\newpage
For each image fibering $h\Diff_f(E)$ in some $\Cinf$-neighborhood
$U\Diff_f(E)$ of $1_E\Diff_f(E)$, we have defined a horizontal vector field
$X(h\Diff_f(E))$, for which applying the tame aligned exponential defines a
smooth map $g_{h\Diff_f(E)}$ that moves each image fiber onto a vertical
fiber. Since the coset $1_E\Diff_f(E)$ determines the zero vector field,
$g_{\Diff_f(E)}=1_E$. So by reducing the size of $U$, if necessary, each
$g_{h\Diff_f(E)}$ will be a diffeomorphism. A local $\Diff(E)$
cross-section $\chi\colon U\Diff_f(E)\to \Diff(E)$ is then defined by
sending $h\Diff_f(E)$ to $g_{h\Diff_f(E)}^{-1}$.

The aligned exponential has analogous local diffeomorphism properties to
the ordinary exponential, so we may use it to define a local chart for the
Fr\'echet manifold structure on $\Diff(E)$ at $1_E$, say $\TExp_a\colon
V\to \Diff(E)$, where $V$ is a neighborhood of $Z$ in
$\mathcal{X}(E,TE)$. Our cross-section $\chi\colon U\Diff_f(E)\to \Diff(E)$
takes $\Diff_f(E)$ to $1_E$, so by choosing $U$ small enough, we may assume
that $\chi$ has image in $V$. The local cross-section shows that every
fibering contained in $U\Diff_f(E)$ is the image of the vertical fibering
under a diffeomorphism $\chi(U\Diff_f(E))$ in $V$. For $X$ near $Z$, at
least, $\TExp_a(X)\in \Diff_f(E)$ if and only if $X\in \mathcal{A}(E,TE)$,
so the chart on $V$ descends to a chart for $\Diff(E)/\Diff_f(E)$, defined
on a neighborhood of $Z$ in $\mathcal{X}(E,TE)/\mathcal{A}(E,TE)$.

For the Fr\'echet space structure on $\mathcal{X}(E,TE)/\mathcal{A}(E,TE)$,
recall that the sections of a vector bundle over a smooth manifold form a
Fr\'echet space \cite[Example 1.1.5]{Hamilton}, and that a closed subspace
or quotient by a closed subspace of a Fr\'echet space is a Fr\'echet space
\cite[Section 1.2]{Hamilton}. As $\mathcal{X}(E,TE)$ is a closed subspace
of the space of all sections of $TE$, it is Fr\'echet. Since
$\mathcal{A}(E,TE)$ is a closed subspace,
$\mathcal{X}(E,TE)/\mathcal{A}(E,TE)$ is Fr\'echet as well.

In the case of a very good singular fibering $p\colon \Sigma\to \O$, we
carry out the previous construction working equivariantly in the bundle
$\widetilde{\Sigma}\to \widetilde{\O}$, which may be chosen with
$\widetilde{\Sigma}$ compact. Since we are using a $G$-equivariant
Riemannian metric on $\widetilde{\Sigma}$ and an $H$-equivariant one on
$\widetilde{\O}$, and $\widetilde{p}$ is equivariant, all parts of the
construction proceed equivariantly. Because $\widehat{\Sigma}$ is compact,
the image fibers of a $G$-equivariant diffeomorphism of
$\widetilde{\Sigma}$ will project under $\widetilde{p}$ to compact sets,
which are small when the fibers are nearly vertical, and consequently the
centers of mass will still be well-defined.
\end{proof}\index{straightening fibers|)}

\section{Restricting to the boundary or the basepoint}
\label{basept}

\index{suborbifold!of the boundary}Our restriction theorems deal with the
case when the suborbifold is properly embedded. By a simple doubling trick,
we can also extend to restriction to suborbifolds of the boundary.

\begin{proposition} Let $\Sigma\to\O $ be
a singular fibering. Let $S$ be a suborbifold of $\partial\O $,
and let $T= p^{-1}(S)$. Then
\begin{enumerate}
\item[{\rm(a)}] $\Imb(S,\partial\O )$ admits local
$\Diff(\O )$ cross-sections.
\item[{\rm(b)}] $\Imb_f(T,\partial_v{\Sigma})$ admits local
$\Diff_f(\Sigma)$ cross-sections.
\end{enumerate}
\label{restrict to boundary}
\end{proposition}

\begin{proof} We first show that $\Diff(\partial\O )$
admits local $\Diff(\O )$ cross-sections. Let $\Delta$ be the double of $\O
$ along $\partial\O $, and regard $\O $ as a suborbifold of $\Delta$ by
identifying it with one of the two copies of $\O $ in $\Delta$. By
Theorem~\ref{orbtheoremB}, $\Imb(\partial\O , \Delta)$ admits local
$\Diff(\Delta)$ cross-sections. We may regard $\Diff(\partial\O )$ as a
subspace of $\Imb(\partial\O ,\Delta)$. Suppose that $\chi\colon
U\to\Diff(\Delta)$ is a local cross-section at a point in $\Imb(\partial\O
,\Delta)$ that lies in $\Diff(\partial\mathcal{O})$. Elements of
$\Diff(\Delta,\partial \O)$ that interchange the sides of $\O$ are far from
elements that preserve the sides, so by making $U$ smaller if necessary, we
may assume that all elements $f\in U$ such that $\chi(f)$ lies in
$\Diff(\Delta,\O)$ either preserve the sides of $\O$ or interchange them.
In the latter case, we postcompose $\chi$ with the diffeomorphism of
$\Delta$ that interchanges the two copies of $\O$, to assume that all such
elements preserve the sides. Then, sending $g$ to $\chi(g)\vert_{\O}$
defines a local $\Diff(\O )$ cross-section on $U\cap\Diff(\partial\O )$.

By Proposition~\ref{prop:inclusion}, for (a) it suffices to produce local
cross-sections at the inclusion $i_S$. By Theorem~\ref{orbtheoremB},
there is a local $\Diff(\partial\O )$ cross-section $\chi_1$ for
$\Imb(S,\partial \O )$ at $i_S$. Let $\chi_2$ be a local
$\Diff(\O )$ cross-section for $\Diff(\partial \O )$ at
$\chi_1(i_S)$. On a neighborhood $U$ of $i_S$ in
$\Imb(S,\partial\O )$ small enough so that $\chi_2\chi_1$ is
defined, the composition is the desired $\Diff(\O )$
cross-section. For if $j\in U$, then
$\chi_2(\chi_1(j))\circ i_S=\chi_2(\chi_1(j))\circ \chi_1(i_S)\circ i_S=
\chi_1(j)\circ i_S = j$.

The proof of (b) is similar. Double $\Sigma$ along $\partial_v\Sigma$
and apply Theorem~\ref{sftheorem2}, obtaining local $\Diff_f(\Sigma)$
cross-sections for $\Diff_f(\partial_v\Sigma)$. Apply it again to
produce local $\Diff_f(\partial_v\Sigma)$ cross-sections for
$\Imb_f(T,\partial_v\Sigma)$. Their composition, where defined, is a
local $\Diff_f(\Sigma)$ cross-section for
$\Imb_f(T,\partial_v\Sigma)$.
\end{proof}

An immediate consequence is
\index{Restriction Theorem!to boundary of orbifold}%
\begin{corollary} For a singular fibering $\Sigma\to\O $,
let $S$ be a suborbifold of $\partial\O $, and let
$T=p^{-1}(S)$. Then $\Diff(\O )\to \Imb(S,\partial\O )$ and
$\Diff_f(\Sigma)\to \Imb_f(T,\partial_v\Sigma)$ are locally trivial.
In particular, $\Diff(\O )\to\Diff(\partial\O )$ and
$\Diff_f(\Sigma)\to\Diff_f(\partial_v\Sigma)$ are locally trivial.\par
\label{special1}
\end{corollary}

Another consequence is
\begin{corollary} Let $\mathcal{W}$ be a suborbifold of
$\mathcal{O}$. Then the restriction 
$\Imb(\mathcal{W},\O)\to\Imb(\mathcal{W}\cap \partial\mathcal{O},\partial\O )$ is locally
trivial.
\label{special2}
\end{corollary}

\begin{proof} By Theorem~\ref{orbtheoremB}, $\Imb(\mathcal{W}\cap
\partial\O ,\partial\O )$ admits local $\Diff(\partial\mathcal{O})$
cross-sections, and by Proposition~\ref{restrict to boundary},
$\Diff(\partial\O )$ admits local $\Diff(\O )$ cross-sections. Composing
them gives local $\Diff(\O )$ cross-sections for
$\Imb(\mathcal{W}\cap\partial\O ,\partial\mathcal{O})$.
\end{proof}

\begin{corollary} Let $W$ be a vertical suborbifold of
$\Sigma$. Then the restriction $\Imb_f(W,\Sigma)\to \Imb_f(W\cap
\partial_v\Sigma,\partial_v\Sigma)$ is locally trivial.
\label{special3}
\end{corollary}

\begin{proof} The map is $\Diff_f(\Sigma)$-equivariant, and
Proposition~\ref{restrict to boundary}(b) shows that
$\Imb_f(W\cap\partial_v\Sigma,\partial_v\Sigma)$ admits local
$\Diff_f(\Sigma)$ cross-sections.
\end{proof}

Some applications of the fibration $\Diff(M)\to\Imb(V,M)$ concern the case
when the submanifold is a single point. Since in the fibered case a single
point is not usually a vertical submanifold, this case is not directly
covered by our previous theorems. The next proposition allows nonvertical
suborbifolds that are contained in a single fiber, so applies when the
submanifold is a single point. To set notation, let $p\colon \Sigma\to \O $
be a singular fibering. Let $P$ be a (properly-imbedded) suborbifold of
$\Sigma$ which is contained in a single fiber~$F$. Let $T$ be a fibered
closed subset of $\partial_v\Sigma$ which does not meet $F$. By
$\Imb_t(P,\Sigma-T)$ we denote the orbifold embeddings whose image is
contained in a single fiber of~$\Sigma-T$, and which extend to elements of
$\Diff_f(\Sigma\rel T)$.

\begin{proposition} Let $P$ be a suborbifold of $\Sigma$ which is
contained in a single fiber~$F$. Let $T$ be a fibered closed subset of
$\partial_v\Sigma$, which does not meet $F$. Then $\Imb_t(P,\Sigma-T)$
admits local $\Diff_f(\Sigma\rel T)$ cross-sections.\par
\label{restrict to basepoint}
\end{proposition}

\begin{proof}
Let $S=p(T)$. Notice that $p(P)$ is a point and is a properly embedded
suborbifold of $\O$, with orbifold structure determined by the local group
at $p(P)$. Each embedding $i\in\Imb_t(P,\Sigma)$ induces an orbifold
embedding $\overline{\imath}\colon p(P)\to\O-S$.

By Proposition~\ref{prop:inclusion}, it suffices to produce a local
cross-section at the inclusion $i_P$. By Theorem~\ref{orbtheoremB},
$\Imb(p(P),\O - S)$ has local $\Diff(\O \rel S)$ cross-sections, and by
Proposition~\ref{sftheorem1}, $\Diff(\mathcal{O}\rel S)$ has local
$\Diff_f(\Sigma\rel T)$ cross-sections. A suitable composition of these
gives a local $\Diff_f(\Sigma\rel T)$ cross-section $\chi_1$ for
$\Imb(p(P),\O - S)$ at $\overline{\imath_P}$. As remarked in Section~\ref{palais}, we may
assume that $\chi_1(\overline{\imath_P})$ is the identity diffeomorphism of $\Sigma$. By
\index{Restriction Theorem!orbifolds}%
Corollary~\ref{orbcoro2}, there exists a local $\Diff(F)$ cross-section
$\chi_2$ for $\Imb(P,F)$ at $i_P$, and we may assume that $\chi_2(i_P)$ is
the identity diffeomorphism of $F$. Let $\chi_3$ be a local
$\Diff_f(\Sigma\rel T)$ cross-section for $\Imb_f(F,\Sigma - T)$ at $i_F$
given by 
\index{Restriction Theorem!singular fiberings}%
Corollary~\ref{sfcorollary2}. Regarding $\Diff(F)$ as a subspace
of $\Imb_f(F,\Sigma- T)$, we may assume that the composition $\chi_3\chi_2$
is defined. On a sufficiently small neighborhood of $i_P$ in
$\Imb_t(P,\Sigma- T)$ define $\chi(j)\in\Diff_f(\Sigma\rel T)$ by
$$\chi(j)=\chi_1(p(j))\,(\chi_3\chi_2)(\chi_1(p(j))^{-1}\circ j)\ .$$

\noindent Then for $x\in P$ we have
\begin{align*}\chi(j)\circ i_P(x)
&=\chi_1(p(j))\,(\chi_3\chi_2)(\chi_1(p(j))^{-1}\circ j)\circ i_P(x)\\
&=\chi_1(p(j))\,\chi_1(p(j))^{-1} \circ j(x)\\
&=j(x)
\end{align*}
\end{proof}

\noindent This yields immediately
\indexstate{Restriction Theorem!to point}%
\begin{corollary} Let $W$ be a vertical
suborbifold of $\Sigma$ containing $P$. Then $\Diff_f(\Sigma\rel
T)\to\Imb_t(P,\Sigma- T)$ and $\Imb_f(W,\Sigma\rel T)\to
\Imb_t(P,\Sigma- T)$ are locally trivial.
\label{restrict embeddings to S}
\end{corollary}

\section{The space of Seifert fiberings of a Haken 3-manifold}
\label{sfspace}

Let $p\colon\Sigma\to \O $ be a Seifert fibering of a Haken
manifold $\Sigma$. As noted in Section~\ref{sfiber}, $p$ is a singular
fibering. Denote by $\diff_f(\Sigma)$ the connected component of the
identity in $\Diff_f(\Sigma)$, and similarly for other spaces of
diffeomorphisms and embeddings. The main result of this section is the
following.
\begin{theorem} Let $\Sigma$ be a
Haken Seifert-fibered 3-manifold. Then the inclusion $\diff_f(\Sigma)\to
\diff(\Sigma)$ is a homotopy equivalence.
\label{space of fp homeos}
\end{theorem}

Before proving Theorem~\ref{space of fp homeos}, we will derive some
consequences. Each element of $\Diff(\Sigma)$ carries the given fibering to
an isomorphic fibering, and $\Diff_f(\Sigma)$ is precisely the stabilizer
of the given fibering under this action. Following
Definition~\ref{def:fibered_structures}, we define the 
\indexdef{space of Seifert fiberings}\textit{space of Seifert fiberings} 
isomorphic to the given fibering to be the space of cosets
$\Diff(\Sigma)/\Diff_f(\Sigma)$. Since $\Sigma$ is not a lens space with
one or two exceptional fibers, $\Sigma$ is a singular fibering. Moreover,
every Seifert fibering other than the exceptional lens space ones is
finitely covered by an $S^1$-bundle (because apart from these cases, the
quotient orbifold has a finite orbifold covering by a manifold), so is a
very good singular fibering. So Theorem~\ref{thm:fibered_structures}
ensures that the space of Seifert fiberings isomorphic to the given one is
a separable Fr\'echet manifold, and the map
\[ \Diff(\Sigma)\to \Diff(\Sigma)/\Diff_f(\Sigma) \]
is a fibration. Note that since $\Diff(\Sigma)/\Diff_f(\Sigma)$ is a
Fr\'echet manifold, each connected component is a path component, and since
$\Diff(\Sigma)$ acts transitively on $\Diff(\Sigma)/\Diff_f(\Sigma)$, any
two components are homeomorphic.

\begin{theorem} Let $\Sigma$ be a Seifert-fibered Haken 3-manifold. Then 
each component of the space of Seifert fiberings of $\Sigma$ is
contractible.
\label{space of sf's}
\end{theorem}

\begin{proof} As sketched on p.~85 of \cite{Waldhausen}, two
fiber-preserving diffeomorphisms of $\Sigma$ that are isotopic are isotopic
through fiber-preserving diffeomorphisms. That is, $\Diff_f(\Sigma)\cap
\diff(\Sigma))=\diff_f(\Sigma)$. Therefore the connected component of the
identity in $\Diff(\Sigma)/\Diff_f(\Sigma)$ is
$\diff(\Sigma)/(\Diff_f(\Sigma)\cap \diff(\Sigma))=
\diff(\Sigma)/\diff_f(\Sigma)$. Using Theorem~\ref{space of fp homeos}, the
latter is contractible.
\end{proof}

Theorem~\ref{space of sf's} shows that the space of Seifert fiberings of
$\Sigma$ is contractible when $\Diff_f(\Sigma)\to \Diff(\Sigma)$ is
surjective, that is, when every self-diffeomorphism of $\Sigma$ is isotopic
to a fiber-preserving diffeomorphism. Almost all Haken Seifert-fibered
$3$-manifolds have this property. The closed case is due to
F. Waldhausen~\cite{Waldhausen1}\index{Waldhausen} (see also \cite[Theorem
  8.1.7]{Orlik}), who showed that (among Haken manifolds) it fails only for
the $3$-torus, the double of the orientable $\I$-bundle over the Klein
bottle, and the \index{Hantsche-Wendt manifold}\textit{Hantsche-Wendt}
manifold, which is the manifold given by the Seifert invariants $\{-1;
(n_2,1); (2,1), (2,1)\}$ (see \cite[pp.\ 133, 138]{Orlik},
\cite[pp.\ 478-481]{Charlap-Vasquez}, \cite{Waldhausen1},
\cite{Hantsche-Wendt}). Topologically, the Hantsche-Wendt manifold is
obtained by taking two copies of the orientable $\I$-bundle over the Klein
bottle, one with the \index{meridional!fibering of $K_0$}meridional
fibering (the nonsingular fibering as an $S^1$-bundle over the M\"obius
band) and one with the \index{longitudinal!fibering of $K_0$}longitudinal
fibering (over the disk with two exceptional orbits of type $(2,1)$) and
gluing them together preserving the fibers on the boundary. It admits a
diffeomorphism interchanging the two halves, which is not isotopic to a
fiber-preserving diffeomorphism. For the bounded case, only $S^1$-bundles
over the disk, annulus or M\"obius band fail to have the property. This
appears as Theorem VI.18 of W. Jaco~\cite{Jaco}. We conclude:
\begin{theorem}\index{space of Seifert fiberings!contractibility} 
Let $\Sigma$ be a Seifert-fibered Haken $3$-manifold other 
than the Hantsche-Wendt manifold, the $3$-torus, the double of the
orientable $\I$-bundle over the Klein bottle, or an $S^1$-bundle over the
disk, annulus or M\"obius band. Then $\Diff_f(\Sigma)\to \Diff(\Sigma)$ is
a homotopy equivalence, that is, the space of Seifert fiberings of $\Sigma$
is contractible.\par
\label{thm:DiffSigmaHE}
\end{theorem}
\longpage

The remainder of this section will constitute the proof of
Theorem~\ref{space of fp homeos}.
\begin{lemma}
Let $\Sigma$ be a Seifert-fibered Haken
3-manifold, and let $C$ be a fiber of $\Sigma$. Then
each component of $\Diff_v(\Sigma\rel C)$ is contractible.
\label{rel fiber}
\end{lemma}

\begin{proof} Since $\Sigma$ is Haken, the base orbifold
of $\Sigma-C$ has negative Euler characteristic and is not
closed. It follows (see~\cite{Scott}) that $\Sigma-C$ admits an
$\H^2\times\R$ geometry. Thus there is an action of $\pi_1(\Sigma-C)$
on $\H^2\times\R$ such that every element preserves the $\R$-fibers
and acts as an isometry in the $\H^2$-coordinate.

It suffices to show that $\diff_v(\Sigma\rel C)$ is contractible. Let $N$
be a fibered solid torus neighborhood of $C$ in $\Sigma$. It is not
difficult to see (as in the argument below) that $\diff_v(\Sigma\rel C)$
deformation retracts to $\diff_v(\Sigma\rel N)$, which can be identified
with $\diff_v(\Sigma-C\rel N-C)$, so it suffices to show that the latter is
contractible. For $f\in\diff_v(\Sigma-C\rel N-C)$, let $F$ be a lift of $f$
to $\H^2\times\R$ that has the form $F(x,s)= (x,s+F_2(x,s))$, where
$F_2(x,s)\in\R$.

Since $f$ is vertically isotopic to the identity relative to $N-C$, we may
moreover choose $F$ so that $F_2(x,s)= 0$ if $(x,s)$ projects to
$N-C$. To see this, we choose the lift $F$ to fix a point in the inverse image
$W$ of $N-C$. Since $f$ is homotopic to the identity relative to $N-C$, $F$
is equivariantly homotopic to a covering translation relative to $W$. That
covering translation fixes the point in $W$, and therefore must be the
identity. Thus $F$ fixes $W$ and commutes with every covering translation.

Define $K_t$ by $K_t(x,s)= (x,s+(1-t)F_2(x,s))$. Since $K_0= F$
and $K_1$ is the identity, and each $K_t$ is the identity on the
inverse image of $N-C$, this will define a contraction of
$\Diff_v(\Sigma-C\rel N-C)$ once we have shown that each $K_t$ is
equivariant. Let $\gamma\in\pi_1(\Sigma-C)$. From~\cite{Scott},
$\Isom(\H^2\times\R)=\Isom(\H^2)\times\Isom(\R)$, so we can write
$\gamma(x,s)=(\gamma_1(x),\epsilon_\gamma s+\gamma_2)$, where
$\epsilon_\gamma= \pm 1$ and $\gamma_2\in\R$. Since $F\gamma=
\gamma F$, a straightforward calculation shows that
$$F_2(\gamma_1(x),\epsilon_\gamma s+\gamma_2)= \epsilon_\gamma
F_2(x,s)\ .$$

\noindent Now we calculate
\begin{align*}K_t\gamma(x,s)&= K_t(\gamma_1(x),\epsilon_\gamma
s+\gamma_2)\\
&= (\gamma_1(x),\epsilon_\gamma s +\gamma_2+(1-t)%
F_2(\gamma_1(x),\epsilon_\gamma s+\gamma_2))\\
&= (\gamma_1(x),\epsilon_\gamma s +\gamma_2+(1-t)%
\epsilon_\gamma F_2(x,s))\\
&= (\gamma_1(x),\epsilon_\gamma (s+(1-t)F_2(x,s)) +\gamma_2)\\
&= \gamma(x,s+(1-t)F_2(x,s))\\
&= \gamma K_t(x,s)\\
\end{align*}

\noindent showing that $K_t$ is equivariant.
\end{proof}

\begin{proof}[Proof of Theorem~\ref{space of fp homeos}] We first examine
$\diff_v (\Sigma)$. Choose a regular fiber $C$ and consider the restriction
$\diff_v (\Sigma)\to \imb_v(C,\Sigma)\cong\diff(C)\cong \diff(S^1)\simeq
\SO(2)$. By 
\index{Restriction Theorem!singular fiberings}%
Corollary~\ref{sfcorollary2}(ii), this is a fibration. By
Lemma~\ref{rel fiber}, each component of the fiber $\Diff_v(\Sigma\rel
C)\cap\diff_v(\Sigma)$ is contractible. It follows by the exact sequence
for this fibration that $\pi_q(\diff_v(\Sigma))\cong\pi_q(\SO(2))= 0$ for
$q\geq 2$, and for $q= 1$ we have an exact sequence 
\begin{small}
\[0\longrightarrow
\pi_1(\diff_v(\Sigma))\longrightarrow \pi_1(\diff(C))\longrightarrow
\pi_0(\Diff_v(\Sigma\rel C)\cap \diff_v(\Sigma))\longrightarrow 0\ .\]
\end{small}

We will first show that exactly one of the following holds.
\begin{enumerate}
\item[a)] $C$ is central and $\pi_1(\diff_v(\Sigma))\cong\Z$ generated by the
vertical $S^1$-action.
\item[b)] $C$ is not central and $\pi_1(\diff_v(\Sigma))$ is trivial.
\end{enumerate}

\noindent
Suppose first that the fiber $C$ is central in $\pi_1(\Sigma)$. Then
there is a vertical $S^1$-action on $\Sigma$ which moves the basepoint
(in $C$) once around $C$. This maps onto the generator of
$\pi_1(\diff(C))$, so $\pi_1(\diff_v(\Sigma))\to
\pi_1(\diff(C))$ is an isomorphism. Therefore
$\pi_1(\diff_v(\Sigma))$ is infinite cyclic, with generator
represented by the vertical $S^1$-action.

If the fiber is not central, then $\pi_1(\diff(C))\to
\pi_0(\Diff(\Sigma\rel C)\cap \diff_v(\Sigma))$ carries the
generator to a diffeomorphism of $\Sigma$ which induces an inner
automorphism of infinite order on $\pi_1(\Sigma,x_0)$, where $x_0$ is
a basepoint in $C$. Since elements of $\Diff(\Sigma\rel C)$ fix the
basepoint, this diffeomorphism (and its powers) are not in
$\diff(\Sigma\rel C)$. Therefore $\pi_1(\diff(C))\to
\pi_0(\Diff(\Sigma\rel C)\cap \diff_v(\Sigma))$ is injective, so
$\pi_1(\diff_v(\Sigma))$ is trivial.

Now let $\O$ be the quotient orbifold, and
consider the fibration of 
\index{Projection Theorem!singular fiberings}%
Theorem~\ref{sfproject diffs}:
\[\Diff_v(\Sigma)\cap \diff_f(\Sigma)
\longrightarrow \diff_f(\Sigma)\longrightarrow
\diff(\O )\ .\leqno{(*)}\]

Observe that $\diff(\O )$ is homotopy equivalent to the identity component
of the space of diffeomorphisms of the 2-manifold $\mathcal{O}-\Ecal $,
where $\Ecal$ is the exceptional set. Since $\Sigma$ is Haken, this
$2$-manifold is either a torus, annulus, disc with one puncture, M\"obius
band, or Klein bottle, or a surface of negative Euler
characteristic. Therefore $\diff(\O )$ is contractible unless $\chi(\O
-\mathcal{E})=0$, in which case $\mathcal{E}$ is empty and $\O$ is an
annulus or torus. Thus the higher homotopy groups of $\diff(\O)$ are all
trivial, and its fundamental group is isomorphic to the center of $\pi_1(\O
)$.  When this center is nontrivial, the elements of $\pi_1(\mathcal{O})$
are classified by their traces at a basepoint of $\O$. From the exact
sequence for the fibration $(*)$, it follows that
$\pi_q(\diff_f(\Sigma))= 0$ for~$q\geq 2$.

To complete the proof, we recall the result of \index{Hatcher}Hatcher~\cite{H} and
\index{Ivanov}Ivanov~\cite{I4}: for $M$ Haken, $\pi_q(\diff(M))$ is $0$ for $q\geq 2$ and
is isomorphic to the center of $\pi_1(M)$ for $q= 1$, and the elements of
$\pi_1(\diff(M))$ are classified by their traces at the basepoint. We
already have $\pi_q(\diff_f(\Sigma))= 0$ for $q\geq 2$, so it remains to
show that $\pi_1(\diff_f(\Sigma))\to \pi_1(\diff(\Sigma))$ is an
isomorphism.

\smallskip
\noindent {\sl Case I:} $\pi_1(\O )$ is centerless.
\smallskip

In this case $\diff(\O )$ is contractible, and either $C$ generates the
center or $\pi_1(\Sigma)$ is centerless. The exact sequence associated to
the fibration $(*)$ shows that
$\pi_1(\diff_v(\Sigma))=\pi_1(\Diff_v(\Sigma)\cap
\diff_f(\Sigma))\to\pi_1(\diff_f(\Sigma))$ is an isomorphism. Suppose $C$
generates the center. Since $\pi_1(\diff_v(\Sigma))$ is infinite cyclic
generated by the vertical $S^1$-action, Hatcher's theorem shows that the
composition
\[\pi_1(\diff_v(\Sigma)) \to
\pi_1(\diff_f(\Sigma))\to
\pi_1(\diff(\Sigma))\]

\noindent is an isomorphism. Therefore
$\pi_1(\diff_f(\Sigma))\to \pi_1(\diff(\Sigma))$ is an
isomorphism. If  $\pi_1(\Sigma)$ is
centerless, then $\pi_1(\diff(\Sigma))= 0$,
$\pi_1(\diff_f(\Sigma))\cong\pi_1(\diff_v(\Sigma))= 0$, and again
$\pi_1(\diff_f(\Sigma))\to \pi_1(\diff(\Sigma))$ is an
isomorphism.

\smallskip
\noindent {\sl Case II:} $\pi_1(\O )$ has center.

\smallskip
Assume first that $\O $ is a torus.  If $\Sigma$ is the $3$-torus, then by
considering the exact sequence for the fibration $(*)$, one can check
directly that the homomorphism
$\partial\colon\pi_1(\diff(\mathcal{O}))\rightarrow
\pi_0(\Diff_v(\Sigma)\cap\diff_f(\Sigma))$
is the zero map. We obtain the exact sequence
$$0\longrightarrow \Z
\longrightarrow \pi_1(\diff_f(\Sigma))\longrightarrow
\Z\times\Z\longrightarrow0\ .$$
\noindent
Since $\diff_f(\Sigma)$ is a topological group,
$\pi_1(\diff_f(\Sigma))$ is abelian and hence isomorphic to
$\Z\times\Z\times\Z$. The traces of the generating
elements generate the center of $\pi_1(\Sigma)$, which shows that
$\pi_1(\diff_f(\Sigma))\to\pi_1(\diff(\Sigma))$ is an
isomorphism.

Suppose that $\Sigma$ is not a $3$-torus. Then $\Sigma$ is a nontrivial
$S^1$-bundle over $\O$, $\pi_1(\Sigma)=\langle a,b,t \vbar tat^{-1}= a,
[a,b]=1,tbt^{-1}= a^nb\rangle$ for some integer $n$, and the fiber $a$
generates the center of~$\pi_1(\Sigma)$.

Let $b_0$ and $t_0$ be the image of the generators of $b$ and $t$
respectively in $\pi_1(\O )$. Now $\pi_1(\diff(\O ))\cong\Z\times\Z$
generated by elements whose traces represent the elements $b_0$ and
$t_0$. By lifting these isotopies we see that
$\partial\colon\pi_1(\diff(\mathcal{O}))\rightarrow\pi_0(\diff_v(\Sigma))$
is injective. Therefore $\pi_1(\diff_v(\Sigma))$ is isomorphic to
$\pi_1(\diff_f(\Sigma))$, and the result follows as in case~I.

Assume now that $\O $ is a Klein bottle. The $\Sigma$ is an $S^1$-bundle
over $\O$, $\pi_1(\Sigma)=\langle a,b,t \vbar tat^{-1}= a^{-1},
[a,b]=1,tbt^{-1}=a^{-n}b^{-1} \rangle$ for some integer $n$, with fiber
$a$, and $\pi_1(\O )=\langle b_0,t_0 \vbar t_0b_0t_0^{-1}=
b_0^{-1}\rangle$. Now $\pi_1(\diff(\O ))$ is generated by an isotopy whose
trace represents the generator of the center of $\pi_1(\diff(\O ))$, the
element $t_0^2$. Observe that $\pi_1(\Sigma)$ has center if and only if
$n= 0$. If $n=0$, then it follows that $\partial\colon\pi_1(\diff(\O
))\rightarrow \pi_0(\Diff_v(\Sigma)\cap\diff_f(\Sigma))$ is the zero map.
Hence $\pi_1(\diff_f(\Sigma))\rightarrow\pi_1(\diff(\O))$ is an isomorphism
and the generator of $\pi_1(\diff_f(\Sigma))$ is represented by an isotopy
whose trace represents the element $t^2$. By Hatcher's result,
$\pi_1(\diff_f(\Sigma))\rightarrow\pi_1(\diff(\Sigma))$ is an
isomorphism. If $n\neq 0$, then
$\partial\colon\pi_1(\diff(\mathcal{O}))\rightarrow
\pi_0(\Diff_v(\Sigma)\cap\diff_f(\Sigma))$ is injective. Since
$\pi_1(\Sigma)$ is centerless,
$\pi_1(\Diff_v(\Sigma)\cap\diff_f(\Sigma))=0$. This implies that
$\pi_1(\diff_f(\Sigma))=0$, and again Hatcher's result applies.

The cases where $\O$ is an annulus, disc with one puncture, or a
M\"obius band are similar to those of the torus and Klein bottle.
\end{proof}

\newpage\section{The Parameterized Extension Principle}
\label{sec:PEP}

As a final application of the methods of this section, we present a result
which will be used, often without explicit mention, in our later work. For
a parameterized family of diffeomorphisms $F\colon M\times W\to M$, we
denote the restriction $F\colon M\times\{u\}\to M$ by $F_u\in \Diff(M)$.
By a \indexdef{deformation!of parameterized family}\textit{deformation} 
of a parameterized family of diffeomorphisms
$F\colon M\times W\to M$, we mean a homotopy from $F$ to a parameterized
family $G\colon M\times W\to M$ of diffeomorphisms when $F$ and $G$ are
regarded as maps from $W$ to $\Diff(M,M)$.

\begin{theorem}[Parameterized Extension Principle]\indexstate{Parameterized Extension Principle}
Let $M$ and $W$ be compact smooth manifolds, let $M_0$ be a submanifold of
$M$, and let $U$ be an open subset of $M$ with $M_0\subset U$. Suppose that
$F\colon M\times W\to M$ is a parameterized family of diffeomorphisms of
$M$. If $g\in \Maps((M_0,M_0\cap \partial M)\times W,(M,\partial M))$ is
sufficiently close to $F\vert_{M_0\times W}$, then there is a deformation
$G$ of $F$ such that $G\vert_{M_0\times W}=g$, and $G=F$ on $(M-U)\times
W$. By selecting $g$ sufficiently close to $F\vert_{M_0\times W}$, $G$ may
be selected arbitrarily close to~$F$.
\label{thm:PEP}
\end{theorem}

\begin{proof}
We may assume that each $F_u$ is the identity on $M$. Provided that $g$ is
sufficiently close to $F\vert_{M_0\times W}$, the \index{Logarithm Lemma}Logarithm
Lemma~\ref{logarithm} gives sections $X(g_u)\in \mathcal{X}(M_0,TM)$ such
that $\Exp(X(g_u))(x)=g_u(x)$. Applying the Extension Lemma~\ref{extension}
gives a continuous linear map $k\colon \mathcal{X}(M_0,TM)\to
\mathcal{X}(M,TM)$ with $k(X)(x)=X(x)$ for $x\in M_0$. Finally, the
Exponentiation Lemma~\ref{lem:exponentiation} shows that for $g$ in some
neighborhood $U$ of the inclusion family (that is, the parameterized family
with each $g_u$ the inclusion of $M_0$ into $M$), each $\TExp\circ \,k\circ
X$ carries $U$ into parameterized families of diffeomorphisms.  By local
convexity, after making $U$ smaller, if necessary, the resulting
diffeomorphisms $G_u$ will be isotopic to the original $F_u$ by moving
along the unique geodesic between $G_u(x)$ and $F_u(x)$, giving the
required deformation.
\end{proof}

\chapter{Elliptic $3$-manifolds containing one-sided Klein bottles}
\label{ch:one-sided}

In this chapter, we will prove Theorem~\ref{thm:one-sided}.
Section~\ref{sec:M(m,n)} gives a construction of the elliptic $3$-manifolds
that contain a one-sided geometrically incompressible Klein bottle; they
are described as a family of manifolds $M(m,n)$ that depend on two integer
parameters $m,n\geq 1$. Section~\ref{sec:results} is a section-by-section
outline of the entire proof, which constitutes the remaining sections of
the chapter.

\section[The manifolds $M(m,n)$]
{The manifolds $M(m,n)$}
\label{sec:M(m,n)}

Let \indexsym{K0}{$K_0$}$K_0$ be a Klein bottle, which will later 
be the special ``base'' Klein bottle in $M(m,n)$, and write 
\index{pi1K0@$\pi_1(K_0)$!standard presentation}$\pi_1(K_0)=
\langle\,a,\,b\,\,\vert\,\,bab^{-1}= a^{-1}\, \rangle$. The four isotopy
(as well as homotopy) classes of unoriented essential simple closed curves
on $K_0$ are $b$, $ab$, $a$, and $b^2$, with $b$ and $ab$
orientation-reversing and $a$ and $b^2$ orientation-preserving.

Let \indexsymdef{P}{$P$ ($\protect\I$-bundle)}$P$ be the 
orientable $\I$-bundle over $K_0$. 
The free abelian group $\pi_1(\partial P)$ is generated by elements 
homotopic in $P$ to) $a$ and~$b^2$.

Let \indexsymdef{R}{$R$}$R$ be a solid torus containing a meridional 
$2$-disk with boundary $C$, a circle in $\partial R$. For a pair 
$(m,n)$ of relatively prime
integers, the $3$-manifold \indexsymdef{M(m,n)}{$M(m,n)$}$M(m,n)$ is 
formed by identifying $\partial R$ and $\partial P$ in such a way 
that $C$ is attached along a simple closed
curve representing the element $a^mb^{2n}$.  If $m=0$, the resulting
manifold is $\RP^3\#\RP^3$, while if $n=0$ it is $S^2\times S^1$. In
these cases $K_0$ is compressible, so from now on we will assume that
neither $m$ nor $n$ is zero. Since $M(-m,n)= M(m,n)$ and $M(-m,-n)=
M(m,n)$, we can and always will assume that both $m$ and $n$ are positive.

Each fibering of $K_0$ extends to a Seifert fibering of $M(m,n)$, for which
$P$ and $R$ are fibered submanifolds. If $K_0$ has the 
\index{longitudinal!fibering of $K_0$}longitudinal fibering, then in 
$\partial P$ the fiber represents $b^2$. The meridian
circle $C$ of $R$ equals $ma+nb^2$.  Choosing $p$ and $q$ so that
$mp-nq=1$, the element $L=qa+pb^2$ is a longitude of $R$, since the
intersection number $C\cdot L=mp-nq=1$. We find that $b^2=mL-qC$, so on $R$
the Seifert fibering has an exceptional fiber of order $m$, unless
$m=1$. If instead $K_0$ has the 
\index{meridional!fibering of $K_0$}meridional fibering, then the fiber
represents $a$ in $\partial R$, and since $a=pC-nL$, $R$ has an exceptional
fiber of order $n$, unless $n=1$. In terms of $m$ and $n$, then, the cases
discussed in Section~\ref{sec:SCtoday} are as follows: I is $m>1$ and
$n>1$, II is $m=1$ and $n>1$, III is $m>1$ and $n=1$, and IV is $m=n=1$.

The fundamental group of $M(m,n)$ has a presentation\indexsym{pi1M(m,n)}{$\pi_1(M(m,n))$}
\[\langle\,a,\,b\,\,\vert\,\,bab^{-1}= 
a^{-1},\, a^mb^{2n}= 1\,\rangle\ .\] 
\noindent Note that $a^{2m}=1$ and~$b^{4n}=1$.

If $n$ is odd, then 
\indexsym{Dstar4m}{$D^*_{4m}$}$\pi_1(M(m,n))\cong C_n\times D^*_{4m}$, where $C_n$ is
cyclic and
$$D^*_{4m}=\langle x,y \vbar x^2=y^m=(xy)^2\rangle$$
is the binary dihedral group. The $C_n$ factor is generated by $b^4$ and
the $D^*_{4m}$ factor by $x=b^n$ and $y=a$.

If $n$ is even, write $C_{4n}=\langle t\vbar t^{4n}=1\rangle$.  Let
$\Delta$ be the \index{diagonal subgroup}diagonal subgroup of index 
$2$ in $C_{4n}\times
D^*_{4m}$. That is, there is a unique homomorphism from $C_{4n}$ onto
$C_2$, and, since $m$ is odd, a unique homomorphism from $D^*_{4m}$ onto
$C_2$. The latter sends $y$ to $1$. Combining these homomorphisms sends
$C_{4n}\times D^*_{4m}$ onto $C_2$ with kernel $\Delta$. The element
$(t^{2n},y^m)$ is a central involution in $\Delta$, and $\pi_1(M(m,n))$ is
isomorphic to $\Delta/\langle (t^{2n},y^m)\rangle$. The correspondence is
that $a=(1,y)$ and $b=(t,x)$.

When $m=1$, the groups reduce in both cases to a cyclic group of order
$4n$. From \cite{B-W} or \cite{R2}, $M(1,n)=L(4n,2n-1)$.\
\index{M(m,n)@$M(m,n)$!is lens space when $m=1$|(}This
homeomorphism
can be seen directly as follows. Let $T$ be a solid torus with
$H_1(\partial T)$ the free abelian group generated by $\lambda$, a
longitude, and $\mu$, the boundary of a meridian disk. Let $C_1$ and $C_2$
be disjoint loops in $\partial T$, each representing $2\lambda+\mu$. There
is a M\"obius band $M$ in $T$ with boundary $C_2$. The double of $T$ is an
$S^2\times S^1$ in which $M$ and the other copy of $M$ form a one-sided
Klein bottle. The double has a Seifert fibering which is longitudinal on
the Klein bottle, nonsingular on its complement, and in which $C_1$ is a
fiber. If the attaching map in the doubling is changed by Dehn twists about
$C_1$, the resulting manifolds are of the form $M(1,n)$, since they still
have fiberings which are longitudinal on the Klein bottle and nonsingular
on its complement. Since $\mu$ intersects $C_1$ twice, the image of $\mu$
under $k$ Dehn twists about $C_1$ is
$\mu+2k(\mu+2\lambda)=4k\lambda+(2k+1)\mu$, so the resulting manifold is
$L(4k,2k+1)=L(4k,2k-1)$. It must equal $M(1,k)$ since $M(1,k)$ is the only
manifold of the form $M(1,n)$ with fundamental group $C_{4k}$.\index{M(m,n)@$M(m,n)$!is lens space when $m=1$|)}

As we have seen, with the longitudinal fibering the manifolds $M(m,n)$ have
fibers of orders $2$, $2$, and $m$, so in the terminology of \cite{M},
$M(2,n)$ is a \index{quaternionic manifold}quaternionic manifold, while 
for $m>2$, $M(m,n)$ is a (nonquaternionic) \index{prism manifold}prism manifold.

\newpage

\section[Outline of the proof]
{Outline of the proof}
\label{sec:results}\longpage

By Theorem~\ref{thm:pi0 Smale}, the inclusion $\Isom(M(m,n))\to
\Diff(M(m,n))$ is a bijection on path components, so we need only prove
that the inclusion $\isom(M(m,n))\to \diff(M(m,n))$ of the connected
components of the identity induces isomorphisms on all homotopy groups.
The rest of this chapter establishes this when at least one of $m$ or $n$
is greater than~$1$, that is, for Cases~I, II, and~III in
Section~\ref{sec:SCtoday}. The remaining possibility $M(1,1)$ is the lens
space $L(4,1)$, for which the Smale Conjecture holds by
Theorem~\ref{thm:SCforLensSpaces} proven in Chapter~\ref{ch:lens}.

In Section~\ref{isometry}, we give a calculation of the connected
components of the identity in the isometry groups of the $M(m,n)$, in the
process establishing the viewpoint and notation needed in
Section~\ref{fiberings}.

The first task in Section~\ref{fiberings} is to observe that the elements
of $\pi_1(M(m,n))$ preserve the fibers of the \index{Hopf fibering!of $S^3$}Hopf 
fibering of $S^3$. Consequently there is an induced Seifert fibering of 
the $M(m,n)$, which we call the \index{Hopf fibering!of $M(m,n)$}Hopf Seifert 
fibering of $M(m,n)$. A certain torus $T_0$
in $S^3$, vertical in the Hopf fibering, descends to a vertical Klein
bottle $K_0$ in $M(m,n)$ which we call the \index{base Klein bottle}base Klein 
bottle. On $K_0$, the
Hopf fibering of $M(m,n)$ restricts to the 
\index{longitudinal!fibering of $K_0$}longitudinal fibering in Cases~I
and~II and the 
\index{meridional!fibering of $K_0$}meridional fibering in Case~III. 
In Section~\ref{fiberings},
we also check that the isometries of $M(m,n)$ are fiber-preserving and act
isometrically on the quotient orbifold.

Most of Section~\ref{fiberings} is devoted to verifying two facts:
\begin{enumerate}
\item[(a)] The map from $\isom(M(m,n))$ to the space of fiber-preserving
isometric embeddings of $K_0$ into $M(m,n)$, defined by restriction to
$K_0$, is a homeomorphism onto the connected component of the inclusion
(Lemma~\ref{restriction}).
\item[(b)] The inclusion of the latter space into the space of all
fiber-preserving embeddings of $K_0$ into $M(m,n)$ that are isotopic to the
inclusion is a homotopy equivalence (Lemma~\ref{lem:isom_to_imb}).
\end{enumerate}
The big picture of what is going on here can be seen by consideration of
the three types of quotient orbifolds shown in table~\ref{Kleintable2}, which
correspond to Cases~I, II, and~III respectively. For the first two types,
$K_0$ is the inverse image of a geodesic arc connecting the two order~$2$ cone
points, and for the third type, $K_0$ is the inverse image of a ``great circle''
geodesic in $\RP^2$.  The inverse images of such geodesics are the images of
the fiber-preserving isometric embeddings isotopic to the inclusion, the
so-called \index{special!Klein bottle}``special'' Klein bottles. They 
are the translates of $K_0$ under
$\isom(M(m,n))$ (which also contains ``vertical'' isometries that take each
fiber to itself, so preserve each special Klein bottle). Our precise
description of $\isom(M(m,n))$ allows us to examine its effects on these
Klein bottles and establish fact~(a). For the first two types of quotient
orbifold, a fiber-preserving embedding of $K_0$ that is isotopic to the
inclusion carries $K_0$ onto the inverse image of an arc connecting two
order~$2$-cone points and isotopic (avoiding the third cone point, if there
is one) to a geodesic arc, and for the third type they carry $K_0$ onto the
inverse image of an essential circle in $\RP^2$. Fact~(b) for the third type
of orbifold boils down to the fact that the space of all essential
embeddings of the circle in $\RP^2$ is homotopy equivalent to the space of
geodesic embeddings (which is $L(4,1)$), and analogous properties of arcs
in the other two types of orbifolds. 

The reader who is comfortable with this summary of Sections~\ref{isometry}
and~\ref{fiberings} has little need to wade through their details.\longpage

The Smale Conjecture for the $M(m,n)$ reduces to Theorem~\ref{main}, which
says that the inclusion of the space of \textit{fiber-preserving}
embeddings of $K_0$ into $M(m,n)$ into the space of \textit{all} embeddings
of $K_0$ into $M(m,n)$ is a homotopy equivalence (on the connected
components of the inclusion $K_0\to M(m,n)$). This is the main content of
Section~\ref{mainthms}, and is obtained using the results of
Section~\ref{fiberings} and routine manipulation of exact sequences arising
from fibrations of various spaces of mappings.

The final three sections are the proof of Theorem~\ref{main}. One must
start with a family of embeddings of $K_0$ into $M(m,n)$ parameterized by
$D^k$, and change it by homotopy as an element of
$\maps(D^k,\imb(K_0,M(m,n)))$ to a family of fiber-preserving
embeddings. The embeddings are fiber-preserving at parameters in $\partial
D^k$, and this property must be retained so during the homotopy.
Sections~\ref{sec:generic position} and~\ref{sec:generic position families}
are auxiliary results needed for the main argument in
Section~\ref{parameterization}.

In Section~\ref{sec:generic position}, we analyze the situation when an
embedded Klein bottle $K$ meets $K_0$ in 
\index{generic position}``generic position,'' meaning that
all tangencies are of \index{finite multiplicity type}finite multiplicity type.
In $M(m,n)$, $K_0$ has a
standard neighborhood which is a twisted $\I$-bundle $P$, and $P-K_0$ has a
product structure $T\times (0,1]$ with each 
\indexsym{Tu}{$T_u$}$T_u=T\times\{u\}$ a fibered
``level'' torus.  The key result of the analysis is
Proposition~\ref{prop:circles}, which says for all $u$ sufficiently close
to $0$, each circle of $K\cap T_u$ is either inessential in $T_u$, or
represents $a$ or $b^2$ in $\pi_1(T_u)$. We will see below where this
critical fact is needed.

The proof of Proposition~\ref{prop:circles} uses a technique which may seem
surprising in our differentiable context. Since we may not have full
transversality, we go ahead and make the situation much less transverse, by
a process called \index{flattening}\textit{flattening}. It moves $K$ 
to a PL-embedded Klein
bottle that intersects $K_0$ in a $2$-complex, but still meets torus levels
for $u$ near $0$ in loops isotopic to their original intersection circles
\newpage
with $K$. For these flattened surfaces, combinatorial arguments can be used
to establish that those intersection circles are $a$- and
$b^2$-curves. Proposition~\ref{prop:circles} fails for $M(1,1)$, as we show
by example.

Section~\ref{sec:generic position families} recalls 
\index{Ivanov}Ivanov's idea~\cite{I2}
of perturbing a parameterized family of embeddings of $K_0$ into $M(m,n)$
so that each image meets $K_0$ in generic position. A bit of extra work is
needed to ensure that during a homotopy from our original family to the
generic position family, the embeddings remain fiber-preserving at
parameters in~$\partial D^k$.

Section~\ref{parameterization} is the argument to make a parameterized
family of embeddings $K_0\to K_t\subset M(m,n)$, $t\in D^k$,
fiber-preserving for the Hopf fibering on $M(m,n)$. The first step is a
minor technical trick needed to ensure that no $K_t$ equals $K_0$; this
allows Section~\ref{sec:generic position families} to be applied to assume
that the $K_t$ meet $K_0$ in generic position. Next, we use
\index{Hatcher}Hatcher's methods to simplify the intersections of the Klein
bottles $K_t$ with the torus levels \indexsym{Tu}{$T_u$}$T_u$. Each $K_t$
has finitely many associated torus levels $T_u$, obtained using
Proposition~\ref{prop:circles}. First, we eliminate intersections that are
contractible in $K_t$ (and hence in $T_u$). This part of the argument,
called Step~2, is a straightforward adaptation of Hatcher's arguments
from~\cite{Hold, Hnew}, but we give a fair amount of detail since these
methods are not widely used.\longpage

Step~3 is where the hard work from Section~\ref{sec:generic position} comes
into play. From our analysis of generic position configurations,
specifically Proposition~\ref{prop:circles}, we know that $K_t$ meets its
associated levels $T_u$ in circles that represent $a$ or $b^2$ in
$\pi_1(T_u)$. Now, $T_u$ separates $M(m,n)$ into a twisted $\I$-bundle
$P_u$, containing $K_0$, and a solid torus $R_u$. Some homological
arguments (which again break down for $M(1,1)$) show that a circle of
$K_t\cap T_u$ is a longitude of $R_u$ only if it is isotopic in $T_u$ to a
fiber. Hence any circles of $K_t\cap T_u$ that are \textit{not} isotopic to
fibers are also \textit{not longitudes of $R_u$}, and consequently
\textit{the annuli of $K_t\cap R_u$ that contain them are uniquely
boundary-parallel in $R_u$.} This allows us to once again apply 
\index{Hatcher}Hatcher's
parameterized methods to pull the annuli of $K_t\cap R_u$ whose boundary
circles are not isotopic in $T_u$ to fibers out of $R_u$, achieving that
\textit{every loop of $K_t\cap T_u$ is isotopic in $T_u$ to a fiber.}

Two tasks remain:
\begin{enumerate}
\item[(1)] Make $K_t$ intersect its associated levels $T_u$ in circles that
are fibers and are the images of fibers of $K_0$ under the embedding
$K_0\to M(m,n)$.
\item[(2)] Make the embeddings fiber-preserving on the intersections of
$K_t$ with the other pieces of $M(m,n)$, which are topologically either
twisted $\I$-bundles over $K_0$, product regions between levels, or
solid tori that are complements of twisted $\I$-bundles over~$K_0$.
\end{enumerate}
\newpage\noindent The underlying facts about fiber-preserving embeddings
needed for this are given in Step~4. The final part of the argument,
Step~5, applies these facts, working up the skeleta of a triangulation of
$D^k$, to complete the deformation.

\section[Isometries of elliptic $3$-manifolds]
{Isometries of elliptic $3$-manifolds}
\label{isometry}

In Section~\ref{sec:isometries}, we recalled the isometry groups of
elliptic $3$-manifolds. We will now present the calculations of these
groups--- actually, only the connected component $\isom(M)$ of the
identity--- for the elliptic $3$-manifolds that contain a geometrically
incompressible Klein bottle. Besides giving an opportunity to revisit the
beautiful interaction between the structure of $S^3$ as the unit
quaternions and the structure of $\SO(4)$, which will provide the setting
for some key technical results in Section~\ref{fiberings}.

\index{quaternions|(}%
Fix coordinates on $S^3$ as $\set{(z_0,z_1)\;|\;z_i\in\C, z_0\overline{z_0}
+ z_1\overline{z_1}=1}$. Its group structure as the unit quaternions can
then be given by writing points in the form $z_0+z_1j$, where $j^2=-1$ and
$jz_i=\overline{z_i}j$. The unique element of order~$2$ in $S^3$ is $-1$,
and it generates the center of $S^3$.

By $S^1$ we will denote the subgroup of points in $S^3$ with $z_1=0$, that
is, all quaternions of the form $z_0$, where $z_0$ lies in the unit circle
in $\C$.  Let \indexsymdef{xi}{$\xi_k$}$\xi_k=\exp(2\pi i/k)$, which
generates a cyclic subgroup $C_k\subset S^1$.  The elements $S^1\cup
S^1\,j$ form a subgroup \indexsym{O(2)}{$\Or(2)$}$\Or(2)^*\subset S^3$,
which is exactly the normalizer of $C_k$ if $k>2$.  Also contained in
$\Or(2)^*$ is the binary dihedral group 
\indexsym{Dstar4m}{$D^*_{4m}$}$D_{4m}^*$ generated by $x=j$ and
$y=\xi_{2m}$; its normalizer is $D_{8m}^*$.  By 
\indexsymdef{J}{$J$ (subgroup of $S^3$)}$J$ we denote the subgroup
of $S^3$ consisting of the elements with both $z_0$ and $z_1$ real. It is
the centralizer of~$j$.

The real part 
\indexsymdef{Re(z0,z1)}{$\Re(z_0+z_1j)$}$\Re(z_0+z_1j)$ is the real 
part $\Re(z_0)$ of the complex
number $z_0$, and the imaginary part 
\indexsymdef{Im(z0,z1)}{$\Im(z_0+z_1j)$}$\Im(z_0+z_1j)$ is $\Im(z_0)+z_1j$.
The usual inner product on $S^3$ is given by $z\cdot w=\Re(zw^{-1})$, where
$\Re(z_0+z_1j)=\Re(z_0)$. Consequently, left multiplication and right
multiplication by elements of $S^3$ are orthogonal transformations of
$S^3$, and there is a homomorphism 
$F\colon S^3\times S^3\to \SO(4)$
defined by $F(q_1,q_2)(q)=q_1 q q_2^{-1}$. It is surjective and has kernel
$\set{(1,1),(-1,-1)}$.

The quaternions with real part $0$ are the 
\index{pure imaginary quaternions}pure imaginary quaternions, and
form a subspace \indexsymdef{P(quaternions)}{$P$ (quaternions)}$P\subset S^3$ 
homeomorphic to $S^2$. In fact, $P$ is exactly the orthogonal complement of
$1$. Conjugation by elements of $S^3$ preserves $P$, defining a surjective
homomorphism $S^3\to \SO(3)$ with kernel~$\langle \pm1\rangle$.

Suppose that $G$ is a finite subgroup of $\SO(4)$ acting freely on
$S^3$. Since $\SO(4)$ is the full group of orientation-preserving
isometries of $S^3$, the orienta\-tion-preserving isometries 
\indexsymdef{Isom+S3/G}{$\Isom_+(S^3/G)$}$\Isom_+(S^3/G)$
are the quotient $\Norm(G)/G$, where 
\indexsymdef{NormG}{$\Norm(G)$}$\Norm(G)$ is the 
\index{normalizer}normalizer of $G$ in
$\SO(4)$. Assuming that the group $G$ is clear from the context, we denote
the isometry that an element $F(q_1,q_2)$ of $\Norm(G)$ induces on $S^3/G$
by $f(q_1,q_2)$.
\index{quaternions|)}

\index{isometry groups!calculation|(}%
\index{isometries!calculation of $\protect\isom(\pi_1(M(m,n)))$|(}Let
$G^*= F^{-1}(G)$, and let $G_L$ and $G_R$ be the projections of $G^*$
into the left and right factors of $S^3\times S^3$.  Notice that
$\Norm(G)/G\cong \Norm(G^*)/G^*$. The connected component of the identity
in $\Norm(G^*)$ is denoted by $\normalizer(G^*)$. Since $G^*$ is discrete,
these elements centralize $G^*$. Consequently, $\normalizer(G^*)$ is the product
$\normalizer(G_L)\times\normalizer(G_R)$ of the corresponding connected
normalizers of $G_L$ and $G_R$ in the $S^3$ factors. The connected
component of the identity in the isometry group of $S^3/G$ is then
$\isom(M)=\normalizer(G^*)/(G^*\cap\normalizer(G^*))$. We now compute
$\isom(M(m,n))$ for the four cases listed in Section~\ref{sec:SCtoday}:

\medskip
\noindent {\sl Case II and IV.} $m=1$.

The element $F(\xi_{4n}^{n-1},i)$ acts on $S^3$ by
$$F(\xi_{4n}^{n-1},i)(z_0+z_1j)=\xi_{4n}^{n-1}z_0(-i)+
\xi_{4n}^{n-1}z_1j(-i)=\xi_{4n}^{-1}z_0+\xi_{4n}^{2n-1}z_1j\ .$$ 
\noindent Consequently the quotient of $S^3$ by the subgroup generated by
$F(\xi_{4n}^{n-1},i)$ is $L(4n,2n+1)=L(4n,2n-1)=M(1,n)$. For some work in
Section~\ref{fiberings}, however, it is more convenient to use a conjugate
of this generator. Conjugation by
$F(1,\frac{1}{\sqrt{2}}i+\frac{1}{\sqrt{2}}j)$ moves $F(\xi_{4n}^{n-1},i)$
to $F(\xi_{4n}^{n-1},j)$. The latter will be our standard generator for
$G=\pi_1(M(1,n))$.

Letting $G$ be the group $C_{4n}$ generated by $F(\xi_{4n}^{n-1},j)$, $G_R$
is the cyclic subgroup of order $4$ generated by $j$, so
$\normalizer(G_R)=J$, and $G_L$ is generated by $\{\xi_{4n}^{n-1}, -1\}$.
If $n=1$, then $\xi_{4n}^{n-1}=1$ and $G_L=C_2$. If $n>1$ then
$\xi_{4n}^{n-1}$ has order $4n/\gcd(4n,n-1)=4n/\gcd(4,n-1)$, so $G_L$ is
$C_{4n}$ if $n$ is even, $C_{2n}$ if $n\equiv 3\mod 4$, and $C_{n}$ if
$n\equiv 1\mod 4$.

\begin{enumerate}
\item If $n=1$, then $\normalizer(G_L)=S^3$, and $\isom(M(1,1))\cong \SO(3)\times
S^1$, consisting of all isometries of the form $f(q,x)$ with $(q,x)\in
S^3\times J$.
\item If $n>1$, then $\normalizer(G_L)=S^1$, so $\isom(M(1,n))\cong S^1\times
S^1$, consisting of all isometries of the form $f(x_1,x_2)$ with
$(x_1,x_2)\in S^1\times J$.
\end{enumerate}

\noindent {\sl Case III.} $m>1$ and $n=1$.

We embed $G=D_{4m}^*$ in $\SO(4)$ as the subgroup
$F(D_{4m}^*\times\set{1})$. We have $G_L=D_{4m}^*$ and $G_R= C_2$, so
$\normalizer(G_L)\times \normalizer(G_R)=\{1\}\times S^3$. Therefore
$\isom(M(m,1))\cong \SO(3)$, and consists of all isometries of the form
$f(1,q)$.

\smallskip\noindent {\sl Case I.} $m>1$ and $n>1$.

If $n$ is odd, then $G= C_n\times D_{4m}^*$, and we embed $G$ in $\SO(4)$ as
$F(C_{2n}\times D_{4m}^*)$, so $G_L=C_{2n}$ and $G_R=D_{4m}^*$. If $n$ is
even, then $G$ is the image in $\SO(4)$ of the unique 
\index{diagonal subgroup}diagonal subgroup of
index $2$ in $C_{4n}\times D_{4m}^*$, so $G_L=C_{4n}$ and
$G_R=D_{4m}^*$. In either case, we have
$\normalizer(G_L)\times\normalizer(G_R)=S^1\times\set{1}$. Therefore $\isom(M(m,n))\cong
S^1$, and consists of all isometries of the form $f(x,1)$ with 
$x\in S^1$.\index{isometries!calculation of $\protect\isom(\pi_1(M(m,n)))$|)}
\index{isometry groups!calculation|)}
\smallskip

\index{table!$\isom(M(m,n))$ calculations}Table~\ref{Kleintable1} 
summarizes our calculations of $\Isom(M(m,n))$.
\begin{table}
\begin{scriptsize}
\renewcommand{\arraystretch}{1.5}
\setlength{\fboxsep}{0pt}
\fbox{%
\begin{tabular}{c|c|c}
$m,n$ values&$M$&$\isom(M)$\\
\hline
\hline
$m= n= 1$&$L(4,1)$&%
\makebox[13 ex][c]{$\SO(3)\times S^1$}$=$%
\makebox[27 ex][l]{\ $\set{f(q,x)\vbar (q,x)\in S^3\times J}$}\\
\hline
$m= 1,\,n>1$&$L(4n,2n-1)$&%
\makebox[13 ex][c]{$S^1\times S^1$}$=$%
\makebox[27 ex][l]{\ $\set{f(x,y)\vbar (x,y)\in S^1\times J}$}\\
\hline
$m>1,\,n=
1$&\begin{minipage}[m]{3.24cm}\begin{center}\rule[2.4ex]{0mm}{0mm}quaternionic 
($m=2$)\\
or \rule[-1.3ex]{0mm}{0mm}prism ($m>2$)\end{center}\end{minipage}&%
\makebox[13 ex][c]{$\SO(3)$}$=$%
\makebox[27 ex][l]{\ $\set{f(1,q)\vbar q\in S^3}$}\\
\hline
$m>1,\,n> 1$&\begin{minipage}[m]{3.24cm}\begin{center}\rule[2.4ex]{0mm}{0mm}quaternionic 
($m=2$)\\
or \rule[-1.3ex]{0mm}{0mm}prism ($m>2$)\end{center}\end{minipage}&%
\makebox[13 ex][c]{$S^1$}$=$%
\makebox[27 ex][l]{\ $\set{f(x,1)\vbar x\in S^1}$}\\
\end{tabular}}
\end{scriptsize}
\bigskip\medskip
\caption{Isometry groups of the $M(m,n)$}
\label{Kleintable1}
\end{table}

\section[Hopf fiberings and special Klein bottles]
{The Hopf fibering of $M(m,n)$ and special Klein bottles}
\label{fiberings}

From now on, we use $M$ to denote one of the manifolds $M(m,n)$ with $m>1$
or $n>1$. In this section, we construct certain Seifert fiberings of these
$M$, which we will call their Hopf fiberings, and examine the effect of
$\isom(M)$ on them. Also, we define certain vertical Klein bottles in $M$,
called special Klein bottles, are deeply involved in the reductions carried
out in Section~\ref{mainthms}. A certain special Klein bottle $K_0$, called
the base Klein bottle, will play a key role.

We will regard the $2$-sphere $S^2$ as $\C\cup\set{\infty}$. We speak of
antipodal points and orthogonal transformations on $S^2$ by transferring
them from the unit $2$-sphere using the stereographic projection that
identifies the point $(x_1,x_2,x_3)$ with $(x_1+x_2i)/(1-x_3)$. For
example, the antipodal map $\alpha$ is defined by
$\alpha(z)=-1/\overline{z}$.

As is well-known, the Hopf fibering on $S^3$ is an $S^1$-bundle structure
with projection map $H\colon S^3\to S^2$ defined by
$H(z_0,z_1)=z_0/z_1$. The left action of $S^1$ on $S^3$ takes each Hopf
fiber to itself, so preserves the Hopf fibering. The element $F(j,1)$ also
preserves it.  For $j(z_0+z_1j)=-\overline{z_1}+\overline{z_0}j$, so
$H(F(j,1)(z_0+z_1j))=-1\,/\,\overline{z_0/z_1}$. Right multiplication by
elements of $S^3$ commutes with the left action of $S^1$, so it preserves
the Hopf fibering, and there is an induced action of $S^3$ on $S^2$. In
fact, it acts orthogonally. For if we write $x=x_0+x_1j$ and $z=z_0+z_1j$,
we have $zx^{-1}=z_0\overline{x_0}+z_1\overline{x_1} +(z_1x_0-z_0x_1)j$, so
the induced action on $S^2$ carries $z_0/z_1$ to $
(z_0\overline{x_0}+z_1\overline{x_1})/(z_1x_0-z_0x_1)
=\displaystyle\frac{\overline{x}_0\left(\displaystyle\frac{z_0}{z_1}\right) +
  \overline{x}_1}
{-x_1\left(\displaystyle\frac{z_0}{z_1}\right) + x_0}
=\begin{pmatrix}\phantom{-}\overline{x_0}&\overline{x_1}\\ 
-x_1&x_0\end{pmatrix}(z_0/z_1)$.
The trace of this linear fractional transformation is real and lies between
$-2$ and $2$ (unless $x=\pm 1$, which acts as the identity on $S^2$), so it
is elliptic.  Its fixed points are
$\big(\,(x_0-\overline{x_0})\pm\sqrt{(x_0-\overline{x_0})^2 -
  4x_1\overline{x_1}\,}\,\,\big)/(2x_1)$, which are antipodal, so it is an
orthogonal transformation. Combining these observations, we see that the
action induced on $S^2$ via $H$ determines a surjective homomorphism
\indexsym{O(2)star}{$\Or(2)^*$}%
$h\colon \Ostar\times S^3\to\Or(3)$, given by $h(x_0,1)=1$ for $x_0\in
S^1$, $h(j,1)=\alpha$, and $h(1,x_0+x_1j)=
\begin{pmatrix}\phantom{-}\overline{x_0}&\overline{x_1}\\
-x_1&x_0\end{pmatrix}$. The kernel of $h$ is $S^1\times\set{\pm 1}$.

With the explicit embeddings selected in Section~\ref{isometry},
each of our groups $G=\pi_1(M)$ lies in $F(\Ostar\times S^3)$, so
preserves the Hopf fibering, and descends to a Seifert fibering on
$M(m,n)=S^3/G$.
\begin{definition}\indexdef{Hopf fibering!of $M(m,n)$}
The \textit{Hopf fibering} of $M(m,n)$ is the image of the Hopf fibering of
$S^3$ under the quotient map $S^3\to S^3/\pi_1(M(m,n))$. We will always use
the Hopf fibering on the
manifolds~$M(m,n)$.\par\label{def:Hopf_fibering}\end{definition}

The Hopf fibering $H\colon S^3\to S^2$ induces the orbit map $M(m,n)\to
S^2/h(G)$, and the orbit map is induced by the composition of $H$ followed
by the quotient map from $S^2$ to the quotient orbifold $S^2/h(G)$. The
quotient orbifolds for our fiberings are easily calculated using the
explicit embeddings of $G$ into $\SO(4)$ given in Section~\ref{isometry},
together with the facts that $h(j,1)=\alpha$, $h(1,\xi_{2m})=r_m$, the
(clockwise) rotation through an angle $2\pi/m$ with fixed points $0$ and
$\infty$, defined by $r_m(z)=\xi_m^{-1}z$, and $h(1,j)=t$, the rotation
through an angle $\pi$ with fixed points $\pm i$, defined by
$t(z)=-1/z$. Table~\ref{Kleintable2} lists the various cases, where
$(F;n_1,\dots,n_k)$ denotes the $2$-orbifold with underlying topological
space the surface $F$ and $k$ cone points of orders $n_1,\dots,n_k$.

\index{table!quotient orbifolds}
\begin{table}
\begin{center}
\labellist
\pinlabel $2$ [B] at -7 90
\pinlabel $2$ [B] at 68 89
\pinlabel $m$ [B] at 44 18
\pinlabel $2$ [B] at 151 135
\pinlabel $2$ [B] at 151 18
\pinlabel $(S^2;2,2,m)$ [B] at 32 -6
\pinlabel $(S^2;2,2)$ [B] at 138 -6
\pinlabel $(\RP^2;)$ [B] at 268 -6
\endlabellist
\includegraphics[width=54ex]{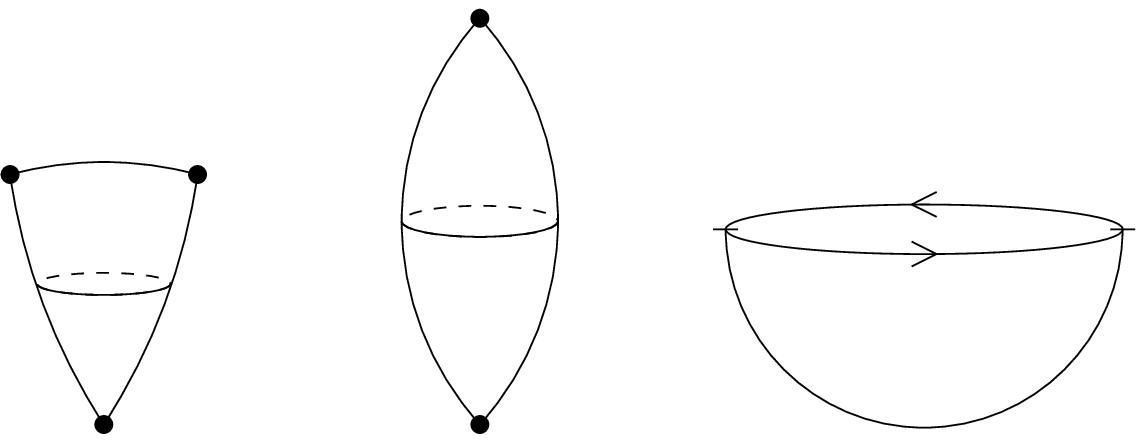}
\end{center}
\vspace{3 ex}
\begin{small}
\renewcommand{\arraystretch}{1.5}
\setlength{\fboxsep}{0pt}
\fbox{%
\begin{tabular}{c|c|c|c}
$m,n$ values&$h(\pi_1(M))$&$\orb$&$\isom(\orb)$\\
\hline
\hline
$m>1$, $n>1$&$D_{2m}=\langle r_m,t\rangle$&$(S^2;2,2,m)$&$\set{1}$\\
\hline
$m=1$, $n>1$&$C_2=\langle t\rangle$&$(S^2;2,2)$&$\SO(2)$\\
\hline
$m>1$, $n=1$&$C_2=\langle \alpha\rangle$&$(\RP^2;)$&$\SO(3)$\\
\end{tabular}}
\end{small}
\bigskip\medskip
\caption{Quotient orbifolds for the Hopf fiberings}
\label{Kleintable2}
\end{table}

Since $m>1$ or $n>1$, we have $\normalizer(\pi_1(M))\subset F(\Ostar\times
S^3)$, so $\isom(M)$ preserves the Hopf fibering. Since the quotient
orbifolds are the quotients of orthogonal actions on $S^2$, they have
metrics of constant curvature $1$, except at the cone points, where the
cone angle at an order $k$ cone point is $2\pi/k$. Table~\ref{Kleintable2} shows
the quotient orbifolds with shapes that suggest the symmetries for this
constant curvature metric. The isometry group of each orbifold $\orb$ is
the normalizer of its orbifold fundamental group $h(G)$ in the isometry
group $\Or(3)$ of $S^2$. The homomorphism $h$ induces a homomorphism
$\isom(M)\to \isom(\orb)$, and from the explicit description of $\isom(M)$
from Table~\ref{Kleintable1} we can use $h$ to compute the image. In each case,
all isometries in the connected component of the identity, $\isom(\orb)$,
are induced by elements of $\isom(M)$. (The groups $\isom(\orb)$ are
computed as $\normalizer(G)/(G\cap \normalizer(G))$ where $\normalizer(G)$
is the connected component of the identity in the normalizer of $G$ in
$\isom(S^2)=\SO(3)$. In particular, $\isom(\RP^2)=\SO(3)$, which can be
seen directly by noting that each isometry of $\RP^2$ lifts to an unique
orientation-preserving isometry of $S^2$.)

Our next task is to understand the fibered Klein bottles in $M$.
\begin{definition}
A torus $T\subset S^3$ \indexdef{special!torus}\textit{special} if its 
image in $S^2$ under $H$ is a
great circle. Klein bottles in $M$ that are the images of special tori in
$S^3$ are called 
\indexdef{special!Klein bottle}\textit{special Klein bottles.} 
A suborbifold in $\orb$ is
called \indexdef{special!suborbifold of $\protect\mathcal{O}$}\textit{special} 
when it is either
\begin{enumerate}
\item[\textup(i)] a one-sided geodesic circle (when $\orb=\RP^2$), or 
\item[\textup(ii)] a geodesic arc connecting two order-$2$ cone points (in
the other two cases).
\end{enumerate}
\label{def:special}
\end{definition}
\noindent Clearly special tori are vertical in the Hopf fibering. We remark
that special tori are Clifford tori, that is, they have induced curvature
zero in the usual metric on~$S^3$.

A Klein bottle in $M$ is special if and only if its image in $\orb$ is a
special suborbifold. To see this, consider a special torus $T$ in $S^3$. If
its image in $\orb$ is special, then its image in $M$ is a one-sided
submanifold, so must be a Klein bottle. Conversely, the projection of $T$
to $\orb$ must always be a geodesic, and if its image in $M$ is a
submanifold, then the projection to $\orb$ cannot have any self
intersections or meet a cone point of order more than $2$. And if the
projection is a circle, it is one-sided if and only if the image of $T$ in
$M$ is one-sided.

Note that the fibering on a special Klein bottle is meridional
(i.~e.~an $S^1$-bundle over $S^1$) in case (i),
and longitudinal (two exceptional fibers that are center circles of
M\"obius bands) in case (ii). From Table~\ref{Kleintable2}, we see that:
\begin{enumerate}
\item When $n=1$, special Klein bottles have the meridional fibering.
\item When $n>1$, special Klein bottles have the longitudinal fibering.
\end{enumerate}
\shortpage

Let $T_0$ be the fibered torus $H^{-1}(U)$, where $U$ is the unit circle in
$S^2$. Explicitly, $T_0$ consists of all $z_0+z_1j$ for which
$\abs{z_0}=\abs{z_1}=\frac{1}{\sqrt{2}}$. Observe that the isometries
\indexsym{OstartimesOstar}{$\protect\Ostar\protect\ttimes\protect\Ostar$}$F(\Ostar\times \Ostar)$ of $S^3$ leave $T_0$ invariant.  The action of
$F(\Ostar\times \Ostar)$ on $T_0$ can be calculated using the normalized
coordinates $[x_0,y_0]\in S^1\times S^1/\langle (-1,-1)\rangle$, where
$[x_0,y_0]$ corresponds to the point
$x_0\bigl(\frac{1}{\sqrt{2}}\,\overline{y_0}\,i+\frac{1}{\sqrt{2}}\,y_0\,j\bigr)$.
For $(z_0,w_0)\in S^1\times S^1$, we have $F(z_0,w_0)[x_0,y_0]=
[z_0x_0,w_0y_0]$. Also:
\begin{enumerate}
\item[(a)] $F(j,1)[x_0,y_0]=[-\overline{x_0},i\,y_0]$. Viewed in the
fundamental domain $\Im(x_0)\geq0$ for the involution on
$T_0=\{(x_0,y_0)\}$ that multiplies by $(-1,-1)$, this rotates the
$y_0$-coordinate through $\pi/2$, and reflects in the $x_0$-coordinate
fixing the point $i$.
\item[(b)] $F(1,j)[x_0,y_0]=[i\,x_0,-\overline{y_0}]$. Again viewing in the
fundamental domain $\Im(y_0)\geq0$ in $T_0$, this rotates the
$x_0$-coordinate through $\pi/2$, and reflects in the $y_0$-coordinate
fixing the point $i$.
\end{enumerate}
In fact, the restriction of 
\indexsym{OstartimesOstar}{$\protect\Ostar\protect\ttimes\protect\Ostar$}$F(\Ostar\times \Ostar)$ to $T_0$ is exactly
the group of all fiber-preserving isometries $\Isom_f(T_0)$.  The Hopf
fibers are the orbits of the action of $F(S^1\times \set{1})$ on~$T_0$, so
are the circles with constant $y_0$-coordinate. Using (a) and (b), we find
that $F(j,i)[x_0,y_0]=[\overline{x_0}, y_0]$ and
$F(i,j)[x_0,y_0]=[x_0, \overline{y_0}]$. The elements $F(z_0,w_0)$ act
transitively on $T_0$, and only the two reflections $F(i,j)$ and $F(j,i)$
and their composition fix $[1,1]$ and preserve the fibers, so together they
generate all the fiber-preserving isometries.

Since $F(\Ostar\times \Ostar)$ contains (each of our groups) $G$, the image
of $T_0$ in $M$ is a fibered submanifold $K_0$.
When $n=1$, the image of $T_0$ in $\orb$ is a geodesic circle which is the
center circle of a M\"obius band, and when $n>1$ its image is a geodesic
arc connecting two cone points of order $2$, so $K_0$ is a special Klein
bottle. 
\begin{definition}
The special Klein bottle $K_0$ is called the 
\indexdef{base Klein bottle}\textit{base Klein bottle} of
$M(m,n)$.
\label{def:baseKleinbottle}
\end{definition}
\noindent Since $K_0$ is special, it has the meridional or longitudinal
fibering according as $n=1$ or $n>1$.

Since $G$ acts by isometries on $S^3$, the subspace metric on $T_0$ induces
a metric on $K_0$ such that the inclusion of $K_0$ into $M$ is
isometric. Denote by $\isom_f(K_0,M)$ the connected component of the
inclusion in the space of all fiber-preserving isometric embeddings of
$K_0$ into~$M$. Since the isometries of $M$ are fiber-preserving, their
compositions with the inclusion determine a map
$\isom(M)\to\isom_f(K_0,M)$. By composition with the inclusion, we may
regard $\isom_f(K_0)$, the connected component of the identity in the group
of fiber-preserving isometries of~$K_0$, as a subspace of~$\isom_f(K_0,M)$.

\begin{lemma} 
If $m>1$ or $n>1$, then $\isom(M)\to\isom_f(K_0,M)$ is a
homeomorphism. Moreover,
\begin{enumerate}
\item[(i)] If $n=1$, then the elements $f(1,w_0)$ for $w_0\in S^1$ preserve
$K_0$, and restriction of this subgroup of $\isom(M)$ gives a homeomorphism
$S^1\to \isom_f(K_0)$.
\item[(ii)] If $n>1$, then the elements $f(x_0,1)$ for $x_0\in S^1$ preserve
$K_0$, and restriction of this subgroup of $\isom(M)$ gives a homeomorphism
$S^1\to \isom_f(K_0)$.
\end{enumerate}
\label{restriction}
\end{lemma}

\begin{proof} 
For injectivity, suppose that an element of $\isom(M)$ fixes each point of
$K_0$. Since it is isotopic to the identity, it cannot locally interchange
the sides of $K_0$. Since it is an isometry, this implies it is the
identity on all of~$M$.

For surjectivity, we first examine the action of $\isom(M)$ on special
Klein bottles in~$M$.

For the quotient orbifolds of the form $(S^2;2,2)$, the special
suborbifolds are the portions of great circles running between the two cone
points, and for those of the form $(\RP^2;)$, they are the images of great
circles under $S^2\to \RP^2$. For those of the form $(S^2;2,2,m)$ with
$m>2$, the geodesic running between the two order-$2$ cone points is the
unique special suborbifold. In all of these cases, $\isom(\orb)$ acts
transitively on the special suborbifolds. In the remaining case of
$(S^2;2,2,2)$, there are three special suborbifolds corresponding to the
three nonisotopic special Klein bottles in $M$, and $\isom(\orb)$ acts
transitively on the special suborbifolds isotopic to~$K_0$. Since all
elements of $\isom(\orb)$ are induced by elements of $\isom(M)$, it follows
that in all cases, $\isom(M)$ acts transitively on the space of special
Klein bottles in $M$ that are isotopic to~$K_0$.

Since $\isom(M)$ acts transitively on the space of special Klein bottles
isotopic to $K_0$, it remains to check that any element of $\isom_f(K_0,M)$
that carries $K_0$ to $K_0$ is the restriction of an element of~$\isom(M)$.

Consider first the case when $m>1$ and $n=1$, so $G$ is
$F(D_{4m}^*\times\set{1})$ and $K_0$ has the meridional fibering. The
fiber-preserving isometry group $\Isom_f(K_0)$ is $\Norm(G)/G$ where
$\Norm(G)$ is the normalizer of $G$ in $\Isom_f(T_0)$. The elements in
$F(C_{2m}\times\set{1})$ rotate in the $x_0$-coordinate, while the element
$F(j,1)$ is as described in (a). So each element of $G-C_{2k}$ leaves
invariant a pair of circles each having constant $x_0$-coordinate. The
union of these invariant circles for all the elements of $G-C_{2k}$ must be
invariant under the action of $\Norm(G)$ on $T_0$, so the identity
component of $\Norm(G)$ consists only of $F(\set{1}\times
S^1)$. Consequently the elements $f(1,w_0)$ of $\isom(M)$ induce all
elements of $\isom_f(K_0)$, proving the surjectivity of
$\isom(M)\to\isom_f(K_0,M)$ and verifying assertion~(i).

For $m=1$ and $n>1$, $G$ is cyclic generated by $F(\xi_{4n}^{n-1},j)$ and
$K_0$ has the longitudinal fibering. Since $F(1,j)$ is as described in (b),
there is a pair of circles in $T_0$, each having constant $y_0$-coordinate
and each invariant under all elements of $G$ (these circles become the
exceptional fibers in $K_0$). Since these circles must be invariant under
the normalizer of $G$ in $\Isom(T_0)$, the identity component of $\Norm(G)$
consists only of $F(S^1\times \set{1})$. Therefore the isometries
$f(x_0,1)$ with $x_0\in S^1$ of $\isom(M)$ induce all elements of
$\isom_f(K_0)$, proving the surjectivity of $\isom(M)\to\isom_f(K_0,M)$ and
verifying assertion~(ii) for this case.

Finally, if both $m>1$ and $n>1$, then $G$ contains $F(1,j)$ and $K_0$ has
the longitudinal fibering. Again, the identity component of $\Norm(G)$ is
$F(S^1\times \set{1})$, and the isometries $f(x_0,1)$ with $x_0\in S^1$
induce all of $\isom_f(K_0)$.
\end{proof}

In the space of all smooth fiber-preserving embeddings of $K_0$ in $M$ (for
the appropriate fibering on $K_0$), let $\imb_f(K_0,M)$ denote the
connected component of the inclusion.

\begin{lemma} 
If either $m>1$ or $n>1$, then the inclusion
\[\isom_f(K_0,M)\to\imb_f(K_0,M)\] 
is a homotopy equivalence.
\label{lem:isom_to_imb}
\end{lemma}

\begin{proof}
Let $\mathcal{K}_0$ be the image of $K_0$ in the quotient orbifold
$\mathcal{O}$ of the fibering on $M$. As we have seen, when $K_0$ has the
meridional fibering, $\mathcal{K}_0$ is a one-sided geodesic circle in
$\orb$, and when $K_0$ has the longitudinal fibering, $\mathcal{K}_0$ is a
geodesic arc connecting two order $2$ cone points of $\orb$. Let
$\imb(\mathcal{K}_0,\orb)$ denote the connected component of the inclusion
in the space of orbifold embeddings, and $\isom(\mathcal{K}_0,\orb)$ its
subspace of isometric embeddings, and let a subscript $v$ as in
$\Diff_v(K_0)$ indicate the vertical maps--- those that take each fiber to
itself.  Consider the following diagram, which we call the main diagram:
\begin{equation*}
\begin{CD}
\Isom_v(K_0)\cap \isom_f(K_0) @>>>  \isom_f(K_0,M) @>>>
\isom(\mathcal{K}_0,\mathcal{O})\\
@VVV @VVV @VVV\\
\Diff_v(K_0)\cap \diff_f(K_0) @>>>  \imb_f(K_0,M) @>>>
\imb(\mathcal{K}_0,\mathcal{O})
\end{CD}
\end{equation*}
\noindent in which the vertical maps are inclusions. The left-hand
horizontal arrows are inclusions, and the right-hand horizontal arrows take
each embedding to the embedding induced on the quotient objects. By
Theorem~\ref{sfsquare}, the bottom row is a fibration. We will now examine
the top row.

Suppose first that $n=1$, so that $\orb=(\RP^2;)$ and $\mathcal{K}_0$ is
the image of the unit circle $U$ of $S^2$. For this case,
$\isom(\mathcal{K}_0,\mathcal{O})$ can be identified with the unit tangent
space of $\RP^2$. For if we fix a unit tangent vector of $\mathcal{K}_0$,
the image of this vector under an isometric embedding is a unit tangent
vector to $\RP^2$, and each unit tangent vector of $\RP^2$ corresponds to
a unique isometric embedding of $\mathcal{K}_0$. To understand this unit
tangent space, note first that the unit tangent space of $S^2$ is $\RP^3$,
since each unit tangent vector to $S^2$ corresponds to a unique element of
$\SO(3)=\RP^3$. The unit tangent space of $S^2$ double covers the unit
tangent space of $\RP^2$, so the latter must be~$L(4,1)$.

Since the isometries of $M$ are all fiber-preserving,
there is a commutative diagram
\begin{equation*}
\begin{CD}
\isom(M) @>\overline{h}>> \isom(\mathcal{O}) \\
@V{\widetilde{\rho}}VV @V{\rho}VV\\
\isom_f(K_0,M) @>>> \isom(\mathcal{K}_0,\mathcal{O})
\end{CD}
\end{equation*}
\noindent
where $\overline{h}$ is induced by the homomorphism $h\colon \Ostar\times
S^3\to \Or(3)$ defined near the beginning of this section.  By
Lemma~\ref{restriction}, the restriction $\widetilde{\rho}$ is a
homeomorphism, and from Table~\ref{Kleintable2}, $\overline{h}$ is a
homeomorphism.  The restriction $\rho$ is a $2$-fold covering map, since
there are two isometries that restrict to the inclusion on $\mathcal{K}_0$:
the identity and the reflection across $\mathcal{K}_0$. This identifies the
second map of the top row of the main diagram as the $2$-fold covering map
from $\RP^3$ to $L(4,1)$, with fiber the vertical elements of
$\isom_f(K_0)$.  We will identify $\Isom_v(K_0)\cap\isom_f(K_0)$ as the
fiber of this covering map, by checking that it is $C_2$, generated by the
isometry $f(1,i)$. By part~(i) of Lemma~\ref{restriction}, the elements of
$\isom_f(K_0)$ are induced by the isometries $f(1,w_0)$ for $w_0\in
S^1$. Such an isometry is vertical precisely when $\overline{h}(1,w_0)$
acts as the identity or the antipodal map on $U$, since each fiber of $K_0$
is the image of the circles in $S^3$ which are the inverse images of antipodal
points of $U$ (since these are exactly the fibers of $T_0$ that are
identified by elements of $G=F(D_{4m}^*\times\set{1})$).  For $x_0\in U$,
we have
$\overline{h}(1,w_0)(x_0)=\begin{pmatrix}\overline{w_0}&0\\ 0&w_0\end{pmatrix}(x_0)=\overline{w_0^2}x_0$.
So $\overline{h}(1,w_0)$ is the identity or antipodal map of $U$ exactly
when $w_0=\pm 1$ or $\pm i$.  The cases $w_0=\pm 1$ give $f(1,1)$ and
$f(1,-1)$, which are the identity on $M$ since $F(-1,1)=F(1,-1)\in
G$. Since $f(1,-1)$ is already in $G$, $f(1,i)$ and $f(1,-i)$ are the same
isometry on $K_0$ and give the unique nonidentity element of
$\Isom_v(K_0)\cap\isom_f(K_0)$.

Suppose now that \mbox{$m= 1$}. This time, both $\widetilde{\rho}$ and
$\rho$ are homeomorphisms, since $\mathcal{K}_0$ is just a geodesic arc
connecting the two order-$2$ cone points of $\orb=(S^2;2,2)$. From
Tables~\ref{Kleintable1} and~\ref{Kleintable2}, $\overline{h}\colon \isom(M)\to \isom(\orb)$ is
just the projection from $S^1\times S^1$ to its second coordinate. The
first coordinate is left multiplication of $S^3$ by elements of $S^1$,
which by part~(ii) of Lemma~\ref{restriction} give exactly the elements of
$\isom_f(K_0)$. Since $\overline{h}(x_0,1)$ is the identity on $S^2$ for
all these $x_0$, $\isom_f(K_0)=\Isom_v(K_0)\cap \isom_f(K_0)$.  So the top
row of the main diagram is simply the product fibration $S^1\to S^1\times
S^1\to S^1$, where the second map is projection to the second
coordinate. 

Finally, if both $m>1$ and $n>1$, the quotient orbifold is $(S^2;2,2,m)$
and as seen in the proof of Lemma~\ref{restriction},
$\isom(\mathcal{K}_0)$ is a single point. Again part~(ii) of
Lemma~\ref{restriction} identifies $\isom_f(K_0,M)$ with the vertical
isometries $\Isom_v(K_0)$ that are isotopic to the identity. So the top row
of the main diagram is $S^1\to S^1\to\set{1}$.

In all three cases, the top row of the main diagram is a fibration. The
proof will be completed by showing that the rightmost and leftmost vertical
arrows of the main diagram are homotopy equivalences.

Suppose first that $n=1$. We have a commutative diagram whose vertical maps
are inclusions:
\begin{equation*}
\begin{CD}
\Isom(\orb \rel \mathcal{K}_0) @>>> \isom(\orb) @>>>
\isom(\mathcal{K}_0,\mathcal{O})\\
@VVV @VVV @VVV\\
\Diff(\orb \rel \mathcal{K}_0) @>>>  \diff(\orb) @>>>
\imb(\mathcal{K}_0,\mathcal{O})
\end{CD}
\end{equation*}
The bottom row is a fibration by 
\index{Restriction Theorem!orbifolds}%
Corollary~\ref{orbcoro2}, and we have
already seen how to identify the top row with the covering fibration
$C_2\to \RP^3\to L(4,1)$. Each component of $\Diff(\orb \rel
\mathcal{K}_0)$ can be identified with $\Diff(D^2\rel \partial D^2)$, which
is contractible by \cite{Smale}, so the left vertical arrow is a homotopy
equivalence. The middle arrow is a homotopy equivalence by the main result
of \cite{G}. Consequently the right vertical arrow is a homotopy
equivalence, which is also the right vertical arrow of the main diagram.

We have already seen that part~(i) of Lemma~\ref{restriction} identifies
$S^1$, the group of isometries of the form $f(1,w_0)$, with $\isom_f(K_0)$,
so that $f(1,i)$ is the nontrivial element of $\Isom_v(K_0)\cap
\isom_f(K_0)$.  The group $\Diff_v(K_0)\cap \diff_f(K_0)$ consists of two
contractible components, one in which the diffeomorphisms preserve the
orientation of each fiber and the other in which they reverse it
($\Diff_v(K_0)$ consists of four contractible components, these two and two
others represented by the same maps composed with a single Dehn twist about
a vertical fiber). The identity map and $f(1,i)$ are points in these two
components, so the left vertical arrow of the main diagram is also a
homotopy equivalence.

A detailed analysis of $\Diff_v(K_0)\cap \diff_f(K_0)$ can proceed by
regarding $K_0$ as a circle bundle over $S^1$, letting $s_0$ be a basepoint
in $S^1$ and $C$ be the fiber in $K_0$ which is the inverse image of $s_0$, and
examining the commutative diagram
\begin{footnotesize}
\begin{equation*}
\begin{CD}
\Diff_v(K_0\rel C)\cap \diff_f(K_0) @>>> \Diff_v(K_0)\cap \diff_f(K_0) @>>>
\Diff(C)\\
@VVV @VVV @VVV\\
\Diff_f(K_0\rel C)\cap \diff_f(K_0) @>>>  \diff_f(K_0) @>>>
\imb_f(C,K_0)\\
@VVV @VVV @VVV\\
\diff(S^1\rel s_0) @>>>  \diff(S^1) @>>>
\imb(s_0,S^1)
\end{CD}
\end{equation*}
\end{footnotesize}%
whose rows and columns are all fibrations (the first and middle rows using
Theorem~\ref{square}, the third row by the Palais-Cerf Restriction
Theorem, the first and middle columns by 
\index{Projection Theorem!singular fiberings}%
Theorem~\ref{sfproject diffs}, and
the third column by Theorem~\ref{square}). The spaces in this diagram are
homotopy equivalent to the spaces shown here:
\begin{equation*}
\begin{CD}
\Z @>>> C_2\times \R @>>>
C_2\times S^1\\ 
@VVV @VVV @VVV\\ 
\Z @>>> S^1 \times \R @>>> S^1\times S^1\\
@VVV @VVV @VVV\\ 
1 @>>> S^1 @>>> S^1
\end{CD}
\end{equation*}
\longpage

When $n>1$, the situation is quite a bit simpler. If $m=1$,
$\imb(\mathcal{K}_0,\orb)$ is just the embeddings of an arc in $S^2$
relative to two points, which is homotopy equivalent to
$\isom(\mathcal{K}_0,\orb)$. For the left vertical arrow,
$\Diff_v(\mathcal{K}_0)\cap \diff_f(K_0)$ has only one component, since a
vertical diffeomorphism which reverses the direction of the fibers induces
a nontrivial outer automorphism on $\pi_1(K_0)$. To see
that $\diff_v(\mathcal{K}_0)$ is homotopy equivalent to a circle, we can
fix a generic fiber $C$ and a point $c_0$ in $C$, then lift a vertical
diffeomorphism to a covering of $\mathcal{K}_0$ by $S^1\times\R$ and
equivariantly deform it to the isometry of $S^1\times \R$ that has the same
effect on a lift of $c_0$. This can be carried out canonically using the
$\R$-coordinate, so actually gives a deformation retraction to
$\isom_v(K_0)$. When $m>1$, the situation is the same except that
$\isom(\mathcal{K}_0,\orb)$ is a point and $\imb(\mathcal{K}_0,\orb)$ is
contractible.
\end{proof}

\section[Homotopy type of the space of diffeomorphisms]
{Homotopy type of the space of diffeomorphisms}
\label{mainthms}

We continue to use the notation of Section~\ref{fiberings}.  Our main
technical result shows that parameterized families of embeddings of the
base Klein bottle $K_0$ in $M$ can be deformed to families of
fiber-preserving embeddings:

\indexstate{Smale Conjecture!for incompressible Klein bottle case}%
\begin{theorem}
If either $m>1$ or $n>1$, then the inclusion
\[\imb_f(K_0,M)\rightarrow\imb(K_0,M)\] 
\noindent is a homotopy equivalence.
\label{main}
\end{theorem}

\noindent Its proof will be given in Sections \ref{sec:generic position},
\ref{sec:generic position families}, and \ref{parameterization}. From
Theorem~\ref{main}, we can deduce the Smale Conjecture for our
$3$-manifolds for all cases except $M(1,1)$.

\begin{theorem} If $m>1$ or $n>1$, then
the inclusion
\[\Isom(M(m,n))\to \Diff(M(m,n))\] 
\noindent is a homotopy equivalence.
\label{smale1}
\end{theorem}

\begin{proof}[Proof of Theorem \ref{smale1} assuming 
Theorem \ref{main}.] 
By 
\newline Theorem \ref{thm:pi0 Smale}, the inclusion is a
bijection on path com\-po\-nents, so we will restrict attention to the
connected components of the identity map.

By \index{Restriction Theorem!singular fiberings}Corollary~\ref{sfcorollary2}, 
restriction of diffeomorphisms to
embeddings defines a fibration
\[ \Diff_f(M\rel K_0) \cap \diff_f(M) \to \diff_f(M) \to \imb_f(K_0,M)\ .\]

\noindent Since any diffeomorphism in this fiber is
orientation-preserving, it cannot locally interchange the sides of
$K_0$. Therefore the fiber may be identified with a subspace
consisting of path components of $\Diff_f(S^1\times D^2\rel S^1\times
\partial D^2)$. By 
\index{Projection Theorem!singular fiberings}%
Theorem~\ref{sfproject diffs}, there
is a fibration
\begin{footnotesize}
\begin{equation*}
\Diff_v(S^1\times D^2\rel S^1\times \partial D^2)\to \Diff_f(S^1\times
D^2\rel S^1\times \partial D^2)\to \Diff(D^2\rel\partial D^2)\ ,
\end{equation*}
\end{footnotesize}%
\noindent whose fiber is the group of vertical diffeomorphisms that take
each fiber to itself. The base is contractible by~\cite{Smale}, and it is
not difficult to show that the fiber is contractible, so the restriction
fibration becomes
\[ \diff_f(M\rel K_0) \to \diff_f(M) \to \imb_f(K_0,M)\]
\noindent with contractible fiber. Similarly there is a fibration
\[ \diff(M\rel K_0) \to \diff(M) \to \imb(K_0,M)\ .\]
\noindent The fact that it is a fibration is the Palais-Cerf
Restriction Theorem, and the contractibility of the fiber uses~\cite{H}. We
can now fit these into a diagram
\begin{equation*}
\begin{CD}
\diff_f(M\rel K_0) @>>> \diff_f(M) @>>>  \imb_f(K_0,M)\\
@VVV @VVV @VVV\\
\diff(M\rel K_0) @>>>  \diff(M) @>>>  \imb(K_0,M)\makebox[0pt][l]{\ .}
\end{CD}
\end{equation*}

\noindent The vertical maps are inclusions. By Theorem~\ref{main}, the
right hand vertical arrow is a homotopy equivalence. Since the fibers
are both contractible, it follows that $\diff_f(M)\to \imb_f(K_0,M)$,
$\diff(M)\to \imb(K_0,M)$, and $\diff_f(M)\to \diff(M)$ are homotopy
equivalences.

The right-hand square of the previous diagram is the bottom square of the
following diagram, whose vertical arrows are inclusions and whose
horizontal arrows are obtained by restriction of maps to $K_0$:
\begin{equation*}
\begin{CD}
\isom(M) @>>> \isom_f(K_0,M)\\
@VVV @VVV\\
\diff_f(M) @>>> \imb_f(K_0,M)\\
@VVV @VVV\\
\diff(M) @>>> \imb(K_0,M)
\end{CD}
\end{equation*}

\noindent 
From Lemma~\ref{restriction}, $\isom(M)\to \isom_f(K_0,M)$ is a
homeomorphism, and from Lemma~\ref{lem:isom_to_imb},
$\isom_f(K_0,M)\to\imb_f(K_0,M)$ is a homotopy equivalence.  We
conclude that $\isom(M)\to \diff_f(M)$ is a homotopy equivalence,
hence so is the composite $\isom(M) \to \diff(M)$.
\end{proof}

\section [Generic position configurations] {Generic position
configurations}
\label{sec:generic position}

Let $S$ and $T$ be smoothly embedded closed surfaces in a closed
$3$-manifold $M$. A point $x$ in $S\cap T$ is called a \indexdef{singular
  point!of ST@of $S\cap T$}\textit{singular} point if $S$ is not transverse
to $T$ at $x$. There is a concept of finite multiplicity of such singular
points, as described in Section~5 of \cite{I2} (another useful reference
for these ideas is~\cite{Bruce}). For a singular point $x$ of \index{finite
  multiplicity}finite multiplicity, either $x$ is an isolated point of
$S\cap T$, or $S\cap T$ meets a small disc neighborhood $D^2$ of $x$ in $T$
in a finite even number of smooth arcs running from $x$ to $\partial D^2$,
which are transverse intersections of $S$ and $T$ except at $x$
(cf.~Fig.\ 3, p.\ 1653 of \cite{I2}). Singular points are isolated on $T$,
so by compactness $S\cap T$ will have only finitely many singular points.

We say that the surfaces are \indexdef{generic position}\textit{in generic
position} if all singular points of intersection are of finite
multiplicity. Notice that $S\cap T$ is then a graph (with components that
may be circles or isolated points) whose vertices are the singular points.
Each vertex has even valence (possibly $0$), since along each of the arcs
of $S\cap T$ that emanates from the singular point, $S$ crosses over to the
(locally) other side of~$T$.

Finite multiplicity intersections have the additional property that if
$D^2\times [-1,1]$ is a product neighborhood of $x$ which meets $T$ in
$D^2\times\set{0}$, then for some $u_0>0$, $S$ meets $D^2\times\set{u}$
transversely for each $u$ with $0<\abs{u}\leq u_0$ \cite[Lemma~(5.4)]{I2}.
Consequently, if $T_u$ are the horizontal levels of a tubular neighborhood
of $T$, with $u\in(-1,0)\cup (0,1)$ if $T$ is two-sided, and $u\in(0,1)$ if
$T$ is one-sided, then $S$ is transverse to $T_u$ for all $u$ sufficiently
close to $0$.

Now we specialize to the base Klein bottle $K_0\subseteq M$, where as usual
$M$ denotes an $M(m,n)$ with either $m>1$ or $n>1$. To set notation, let
$T$ be the torus and fix a $2$-fold covering from $T\times [-1,1]$ to the
twisted $\I$-bundle neighborhood $P$ of $K_0$, so that $T\times \set{0}$ is
a $2$-fold covering of $K_0$, and so that for $0<u<1$, the image $T_u$ of
$T\times \{u\}$ is a fibered torus. We call the
\indexsymdef{Tu}{$T_u$}$T_u$~\indexdef{levels}\textit{levels.}

As usual, we write $\pi_1(K_0)= \langle\,a,\,b\,\,\vert\,\,bab^{-1}=
a^{-1}\, \rangle$. For the \index{meridional!fibering of $K_0$}meridional
fibering, the fiber represents $a$, and for the 
\index{meridional!fibering of $K_0$}longitudinal fibering, the
exceptional fibers represent $b$, and the generic fiber~$b^2$. We also
recall from Section~\ref{fiberings} that as a fibered submanifold of
$M(m,n)$, $K_0$ has the meridional or longitudinal fibering according as
$n=1$ or $n>1$.

Each $T_u$ is the boundary of a tubular neighborhood $P_u$ of $K_0$, and
also bounds the solid torus $\overline{M-P_u}$, which we denote
by~$R_u$. For each $u>0$, the elements $a$ and $b^2$ generate the free
abelian group $\pi_1(T_u)$, a subgroup of $\pi_1(P_u)$.\longpage

By a \indexdef{meridian}\textit{meridian} in $T_u$ we mean a simple loop in
$T_u$ which is essential in $T_u$ but contractible in $R_u$. The meridians
represent $(a^mb^{2n})^{\pm1}$ in $\pi_1(T_u)$. By a
\indexdef{longitude!in level}\textit{longitude} in $T_u$ we mean a simple loop in $T_u$
which represents a generator of the infinite cyclic group $\pi_1(R_u)$. The
longitudes represent elements of $\pi_1(T_u)$ of the form
$(a^pb^{2q}(a^mb^{2n})^k)^{\pm1}$, where $pn-qm=\pm1$, since these are
precisely the elements whose intersection number with the meridians is
$\pm1$. This leads us to the following observation.
\begin{lemma} Let $\ell$ be a loop in $T_u$ which represents $a$ or $b^2$ 
in $\pi_1(T_u)$. Then $\ell$ is not a meridian of $R_u$. If $(m,n)\neq
(1,1)$, and $\ell$ is a longitude of $R_u$, then $\ell$ is isotopic in
$T_u$ to a fiber of the Seifert fibering of $M(m,n)$.
\label{lem:longitudes are fibers}
\end{lemma}

\begin{proof}
Since neither of $m$ nor $n$ is $0$, $\ell$ cannot be a meridian of~$R_u$.

Suppose that $\ell$ represents $a$. If $n= 1$, then $\ell$ is a fiber of
$M(m,n)$. If $n>1$, then the longitudes are of the form
$(a^pb^{2q}(a^mb^{2n})^k)^{\pm1}$, where $pn-qm= \pm1$. If $a$ is a
longitude, then $q+kn=0$. But $q$ and $n$ are relatively prime, so this is
impossible.

Suppose now that $\ell$ represents $b^2$. If $n>1$, then $\ell$ is a fiber of
$M(m,n)$. If $n=1$, then the longitudes are of the form $(a(a^mb^2)^k)^{\pm
1}$. If $b^2$ is a longitude, then $1+km=0$. But when $n=1$, we have $m>1$,
so this is impossible.
\end{proof}
\noindent The lemma fails for $M(1,1)$, for in that case an $a$-circle is a
longitude of $R_u$ which is not isotopic to a fiber of the longitudinal
fibering, while a $b^2$-circle is a longitude not isotopic to a fiber of
the meridional fibering.

If $K$ is a Klein bottle in $M$ that meets $K_0$ in generic position, then
the intersection of $K$ with the nearby levels is restricted by the next
proposition, which is the main result of this section.
\begin{proposition} Suppose that $M= M(m,n)$ with $(m,n)\neq (1,1)$,
and let $K$ be a Klein bottle in $M$ which is isotopic to $K_0$ and meets
$K_0$ in generic position. Then there exists $u_0>0$ so that for each
$u\leq u_0$, $K$ is transverse to $T_u$, and each circle of $K\cap T_u$ is
either inessential in $T_u$, or represents $a$ or $b^2$ in $\pi_1(T_u)$.
\label{prop:circles}
\end{proposition}
\longpage

In order to prove Proposition~\ref{prop:circles}, we introduce a special
kind of isotopy. Suppose that $L_0$ is an embedded surface in a closed
$3$-manifold $N$. A piecewise-linearly embedded surface $S$ in $N$ is said
to be \indexdef{flattened surface}\textit{flattened} (with respect to $L_0$
and the choice of the $L_u$) if it satisfies the following conditions.
\begin{enumerate}
\item
There is a $4$-valent graph \indexsymdef{Gamma}{$\Gamma$ (graph in flattened surface)}$\Gamma$ (possibly
with components which are circles) contained in $L_0$ such that $S\cap L_0$
consists of the closures of some of the connected components of
$L_0-\Gamma$.
\item 
Each point $p$ in the interior of an edge of $\Gamma$ has a neighborhood
$U$ for which the quadruple $(U,U\cap L_0,U\cap S,p)$ is PL homeomorphic to
the configuration $(\R^3,\{(x,y,z)\;\vert\;z= 0\},
\{(x,y,z)\vbar\text{either\ } z= 0 \text{\ and\ } x\geq0,
\text{\ or\ } x= 0 \text{\ and\ } z\geq0\},\{0\})\,$ (see
Figure~\ref{fig:flattened surfaces}(a)).
\item
Each vertex $v$ of $\Gamma$ has a neighborhood $U$ for which the quadruple
$(U,U\cap L_0, U\cap S,v)$ is PL homeomorphic to the configuration
$(\R^3,\{(x,y,z)\;\vert\;z= 0\}, \{(x,y,z)\vbar \text{either\ }
z=0 \text{\ and\ } xy\leq0, \text{\ or\ } x=0 \allowbreak
\text{\ and\ }z\geq0,\text{\ or\ }y= 0 \text{\ and\ }z\leq0\},
\{0\})\,$ (see Figure~\ref{fig:flattened surfaces}(b)).
\end{enumerate}
\index{figures!figure2@flattened surfaces, local picture}%
\begin{figure}
\labellist
\pinlabel $(a)$ [B] at 8 93
\pinlabel $(b)$ [B] at 178 93
\pinlabel $x$ [B] at 30 10
\pinlabel $x$ [B] at 193 10
\pinlabel $y$ [B] at 144 37
\pinlabel $y$ [B] at 306 37
\pinlabel $z$ [B] at 91 103
\pinlabel $z$ [B] at 253 103
\endlabellist
\begin{center}
\includegraphics[width=8.8cm]{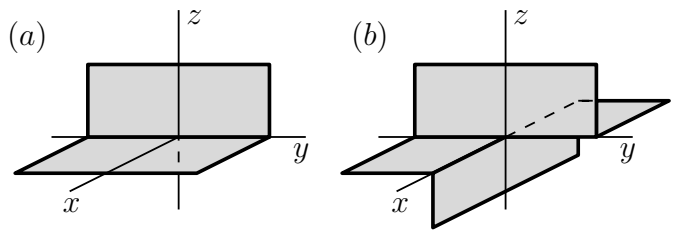}
\caption{Flattened surfaces, local picture.}
\label{fig:flattened surfaces}
\end{center}
\end{figure}
In Figure~\ref{fig:flattened surfaces}(a), the graph $\Gamma$ is the
intersection of $S$ with the $y$-axis. In Figure~\ref{fig:flattened
  surfaces}(b), $\Gamma$ is the intersection of $S$ with the union of the
$x$-~and $y$-axes, and in Figure~\ref{fig:flat_singularity}, it is the
intersection of the horizontal portion in $S_{1/2}\cap L_0$ with the four
vertical bands of $S_{1/2}$. The vertices of $\Gamma$ are exactly the
points that appear as the origin in a local picture as in
Figure~\ref{fig:flattened surfaces}(b) (or
Figure~\ref{fig:flattened_process}(b) below).

\begin{lemma} Let $S_0$ be a smoothly embedded surface in $N$ which
meets the one-sided surface $L_0$ in generic position. Denote by $L_u$ the
level surfaces in a tubular neighborhood $(L\times [-1,1])/((x,u)\sim
(\tau(x),-u))$ for some free involution $\tau$ of $L$ with quotient $L_0$
(so $L_u=L_{-u}$).
Then for some $u_0>0$, there is a PL isotopy $S_t$ from $S_0$ to a PL
embedded surface $S_1$ such that
\begin{enumerate}
\item[(i)] each $S_t$ is transverse to $L_u$ for $0< u \leq u_0$,
and
\item[(ii)] $S_1$ is flattened with respect to $L_0$.
\end{enumerate}
\label{flatten}
\end{lemma}

\begin{proof} Initially, $S_0$ meets $L_0$ in a graph, with tangencies at
the vertices and transverse intersections on the open edges. We have
already noted that there is a $u_0>0$ so that $S_0$ is transverse to $L_u$
for all $0< u\leq u_0$. The isotopy will only move $S_0$ in the
region where $0<u<u_0$. For $0\leq t\leq 1$, we denote by $S_t$ the image
of $S_0$ at time $t$. With respect to $u$, points of $S_t$ must move
monotonically toward~$L_0$, in such a way that the transversality required
by condition~(i) in the lemma is achieved.

Figure~\ref{fig:flat_singularity} illustrates the first portion of the
isotopy, near a singular point $x$ of $S_0\cap L_0$. During the time $0\leq
t\leq 1/2$, a $2$-disk neighborhood of $x$ in $S_0$ moves onto a $2$-disk
neighborhood of $x$ in $L_0$. In a neighborhood $U$ of $x,\;S_0\cap L_0$
consists of $x$ together with a (possibly empty) collection of arcs
$\alpha_1,\alpha_2,\ldots,\alpha_{2n}$ emanating from $x$, at which $S_0$
crosses alternately above and below $L_0$ as one travels around $x$ on
$S_0$. There is a neighborhood of $x$ for which the angles of intersection
of $S_0$ with $L_0$ are close to $0$; the isotopy moves points only within
such a neighborhood. At the end of the initial isotopy, there is a
neighborhood $U$ of $x$ for which $S_{1/2}\cap L_0\cap U$ is a regular
neighborhood in $L_0$ of $\cup_{i=1}^{2n}\alpha_i$. Near interior points of
the $\alpha_i$, $S_{1/2}$ is positioned as in
Figure~\ref{fig:flattened_process}(a), where $L_0$ is the horizontal plane
and $S_{1/2}$ travels ``up'' on one side of $\alpha_i$ and ``down'' on the
other. These isotopies may be performed simultaneously near each singular
point of intersection.
\index{figures!figure3@flattened saddle tangency}
\begin{figure}
\labellist
%\pinlabel $(a)$ [B] at 5 93
%\pinlabel $(b)$ [B] at 175 93
%\pinlabel $x$ [B] at 30 10
\pinlabel $S_{1/2}$ [B] at 15 110
\pinlabel $S_{1/2}$ [B] at 293 206
\pinlabel $S_{1/2}\cap L_0$ [B] at 139 86
\pinlabel $S_{1/2}$ [B] at 188 10
%\pinlabel $y$ [B] at 320 45
%\pinlabel $z$ [B] at 93 103
%\pinlabel $z$ [B] at 255 103
%\pinlabel $\alpha$ [B] at -7 36
%\pinlabel $\alpha$ [B] at 355 53
\endlabellist
\begin{center}
\includegraphics[width=6cm]{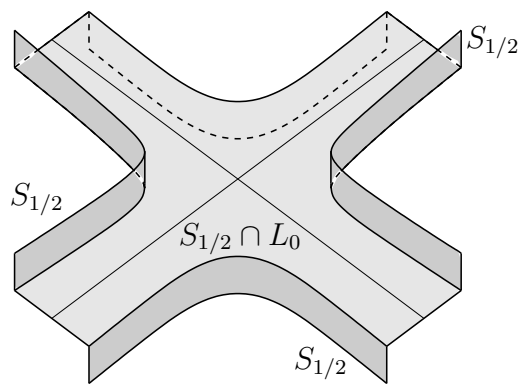}
\caption{A portion of the partially flattened surface $S_{1/2}$ near a point
of $S_0\cap L_0$ that was an ordinary saddle tangency. The intersecting 
diagonal lines are in $S_0\cap L_0$. The horizontal surface is in $S_{1/2}\cap
L_0$, while the darker vertical strips are in $S_{1/2}$ but not $L_0$.}
\label{fig:flat_singularity}
\end{center}
\end{figure}

\index{figures!figure4@flattened surface near original intersection arc}
\begin{figure}
\labellist
\pinlabel $(a)$ [B] at 5 93
\pinlabel $(b)$ [B] at 175 93
\pinlabel $x$ [B] at 30 10
\pinlabel $x$ [B] at 193 10
\pinlabel $y$ [B] at 158 45
\pinlabel $y$ [B] at 320 45
\pinlabel $z$ [B] at 93 103
\pinlabel $z$ [B] at 255 103
\pinlabel $\alpha$ [B] at -7 35
\pinlabel $\alpha$ [B] at 356 52
\endlabellist
\begin{center}
\includegraphics[width=9.8cm]{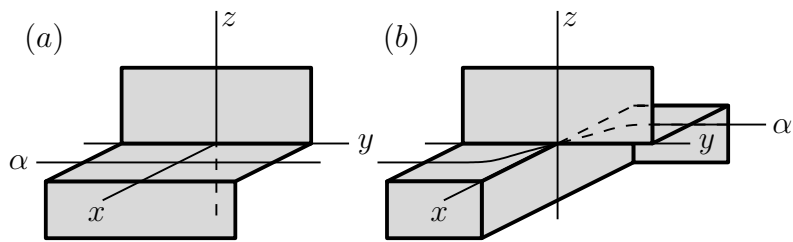}
\caption{Portions of a flattened surface near an original intersection arc
$\alpha$.}
\label{fig:flattened_process}
\end{center}
\end{figure}
The remainder of the isotopy will move points only in a small neighborhood
of the original (open) edges of $S_0\cap L_0$. Consider the closure
$\alpha$ of such an edge. Initially, $S_0$ and $L_0$ were tangent at its
endpoints (which may coincide), and nearly tangent near its endpoints, and
$S_{1/2}$ actually coincides with $L_0$ at points of $\alpha$ near the
endpoints. On the remainder, we continue to flatten $S_{1/2}$ so that it
meets $L_0$ is a neighborhood of $\alpha$, as shown locally in
Figure~\ref{fig:flattened_process}(a).

When the two flattenings from the ends of $\alpha$ meet somewhere in the
middle, it might happen that both go ``up'' on the same side of $\alpha$,
so that the flattening may be continued to achieve
Figure~\ref{fig:flattened_process}(a) at all points of the interior of
$\alpha$. It might happen, however, that one flattening goes ``up'' while
the other goes ``down'' on the same side of $\alpha$. In that case, we
flatten to the local configuration in
Figure~\ref{fig:flattened_process}(b), adding one such crossover point on
each such edge $\alpha$. 

This process starting with $S_{1/2}$ may be carried out simultaneously for
all intersection edges, giving the desired isotopy ending with a flattened
surface~$S_1$.
\end{proof}

We call an isotopy as in Lemma~\ref{flatten}, or the resulting PL surface,
a \indexdef{flattening!isotopy}\textit{flattening} of $S_0$. By property (i)
of the lemma, the collection of intersection circles in $L_u$ for $0<u\leq
u_0$ is changed only by isotopy in $L_u$. After flattening, each of these
circles projects through $S_1$ (i.~e.~vertical projection to the $xy$-plane
in Figure~\ref{fig:flattened surfaces}) to an immersed circle lying in
$\Gamma$, having a transverse self-intersection at each of its double
points (which can occur only at vertices of $\Gamma$.)

\begin{proof}[Proof of Proposition~\ref{prop:circles}.]
Suppose first that the intersection $K\cap K_0$ is transverse. Since $K$
must meet every nearby level $T_u$ transversely, it intersects $P_u$ in
M\"obius bands and annuli, which after isotopy of $K$ (keeping it
transverse to level tori) may be assumed to be vertical in the $\I$-bundle
structure. The projection of $T_u$ onto $K_0$ maps circles of $K\cap T_u$
onto circles of $K\cap K_0$ either homeomorphically or by two-fold
coverings. Only inessential and $a$- and $b^2$-circles can be inverse
images of embedded circles in~$K_0$. For suppose that a loop representing
$a^kb^{2\ell}$ covers an embedded circle. Then it must have zero
intersection number with its image under the covering transformation $\tau$
of $T_u$ over $K_0$. Since $a$ and $b^2$ have intersection number $1$ in
$T_u$, and $\tau(a)=a^{-1}$ and $\tau(b^2)=b^2$, the image represents
$a^{-k}b^{2\ell}$ and the intersection number is $2k\ell$. Therefore the
proposition holds when $K$ meets $K_0$ transversely.

Suppose now that $K\cap K_0$ contains singular points. By
Lemma~\ref{flatten}, we can flatten $K$ near $K_0$, without changing the
isotopy classes in $T_u$ of the loops $K\cap T_u$. After the flattening,
$K\cap K_0$ consists of a valence 4 graph $\Gamma$, which is the image of
the collection of disjoint simple closed curves $K\cap T_u$ under a
$2$-fold covering projection, together with some of the complementary
regions of $\Gamma$ in $K_0$, which we will call the \textit{faces.} Each
edge of $\Gamma$ lies in the closure of exactly one face. It is convenient
to choose an $\I$-fibering of $P_{u_0}$ so that $K\cap P_{u_0}$ lies in the
union of $K\cap K_0$ and the $\I$-fibers that meet~$\Gamma$.

Suppose for contradiction that one of the circles in $K\cap T_u$ represents
$a^kb^{2\ell}$ with $k\ell\neq0$. Since $K$ is geometrically incompressible
(if not, then $M$ would contain an embedded projective plane), there is an
isotopy of $K$ in $M$ which eliminates the circles of $K\cap T_u$ that are
inessential in $T_u$, without altering the remaining circles or destroying
the flattened position of $K\cap P_u$. 

At this point none of the components of $\Gamma$ can be a circle. If so,
then it would lie in a vertical annulus or M\"obius band in $K\cap P_u$,
and be the image of a $1$ or $2$-fold covering of a circle of $K\cap T_u$,
but we have seen that only inessential, $a$-, and $b^2$-circles in $T_u$
project along $\I$-fibers to imbedded circles in $K_0$. So we may assume
that $K\cap T_u$ consists of disjoint circles each representing
$a^kb^{2\ell}$. Since $K$ is isotopic to $K_0$, each loop in $T_u$ has zero
mod~$2$-intersection number in $M$ with $K$, and hence has even algebraic
intersection number with $K\cap T_u$. Therefore $K\cap T_u$ consists of an
even number of these circles; denote them by $A_1$, $A_2,\ldots\,$,
$A_{2r}$.

The vertices of $\Gamma$ are the images of the intersections of $\cup A_i$
with $\cup\tau(A_i)$ (note that by the properties of $\Gamma$, $\cup A_i$
and $\cup \tau(A_i)$ meet transversely). As above, we compute the
intersection number to be
\[\left(\cup A_i\right)\,\cdot\,\left(\cup \tau(A_i)\right)
   = (2r\,a^kb^{2\ell})\,\cdot\,(2r\,a^{-k}b^{2\ell})
   = 4r^2\,2k\ell\ .\]

\noindent Since $(\cup A_i) \cup (\cup \tau(A_i))$ is $\tau$-invariant,
each vertex of $\Gamma$ is covered by two intersections, so $\Gamma$ has at
least $4r^2\abs{k\ell}$ vertices.

\index{figures!figure5@removal of a bigon by isotopy}%
\begin{figure}
\labellist
\pinlabel $(a)$ [B] at 18 190
\pinlabel $(b)$ [B] at 533 190
\endlabellist
\begin{center}
\includegraphics[width=0.85\textwidth]{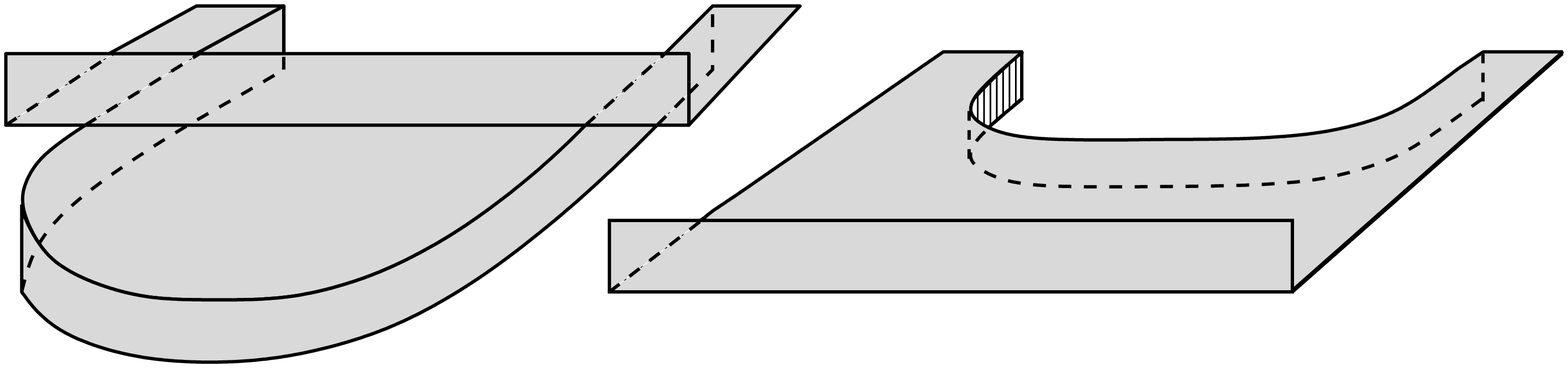}
\caption{Removal of a bigon by isotopy. The picture shows a portion of $K$ 
near a bigon face of $K\cap K_0$, and $K_0$ is the horizontal plane
containing the bigon. During the isotopy, the top vertical portion of $K$ 
in (a) moves forward and the bottom one moves backward, ending with $K$ in
the position shown in~(b). The bigon is eliminated from $K\cap K_0$, while
the other two portions of intersection faces seen in~(a) (which might be
portions of the same face) are joined by a new horizontal band in $K\cap
K_0$ seen in~(b).}
\label{fig:2-gon isotopy}
\end{center}
\end{figure}

We claim that each edge of $\Gamma$ runs between two distinct vertices of
$\Gamma$. Supposing to the contrary, we would see a crossing configuration
as Figure~\ref{fig:flattened surfaces}(b), for which starting at the origin
and traveling along one of the four edges of $\Gamma$ that meet there
returns to the origin along one of the other three edges without passing
through another vertex. Suppose, for example, that the edge starts with the
positive $y$-axis in Figure~\ref{fig:flattened surfaces}(b).  Consider the
right-hand orientation $(\vec{\jmath},-\vec{\imath},\vec{k})$ at the origin
in Figure~\ref{fig:flattened surfaces}(b). Travel out the edge $e$ along
the positive $y$-axis. On the edge, we can make a continuous choice of
local orientation $(\vec{\jmath}_t,-\vec{\imath}_t,\vec{k}_t)$ where
$\vec{\jmath}_t$ is a tangent vector to the edge, $-\vec{\imath}_t$ is the
inward normal of $K\cap K_0$, and $\vec{k}_t$ is the inward normal of
$\overline{K-K_0}$. Returning to the initial point of the edge, one
approaches the origin along either the negative $y$-axis or the positive or
negative $x$-axis, but on each of these axes the orientation
$(\vec{\jmath}_t,-\vec{\imath}_t,\vec{k}_t)$ is left-handed in
Figure~\ref{fig:flattened surfaces}(b), contradicting the fact that $M$ is
orientable.

A similar argument shows that each face of $K\cap K_0$ has an even number
of edges. Successive edges of a face meet at configurations as in
Figure~\ref{fig:flattened surfaces}(b), and the orientations described in
the previous paragraph change to the opposite orientation of $M$ each time
one passes to a new edge.

Consider a face that is a bigon. Since no edge has equal endpoints, the
face must have two distinct vertices, as in Figure~\ref{fig:2-gon
isotopy}(a). The isotopy of $K$ described in Figure~\ref{fig:2-gon
isotopy} eliminates this bigon and adds a band to $K\cap K_0$; this band
is a ($2$-dimensional) $1$-handle attached onto previous faces of $K\cap
K_0$, and either combines two previous faces or is added onto a single
previous face. Repeating, we move $K$ by isotopy (not changing the isotopy
classes of the loops of $K\cap T_u$) to eliminate all faces that are
bigons. No component of $\Gamma$ can be a circle, since as before this
would force $K\cap T_u$ to have a component that is inessential or is an
$a$-~or $b^2$-curve. So each face of $K\cap K_0$ now contains at least 
$4$ vertices.

The Euler characteristic of $K\cap P_u$ is at least $-2r$, since
$\chi(K)=0$ and $K\cap P_u$ has exactly $2r$ boundary components. Letting
$V$, $E$, and $F$ denote the number of vertices, edges, and faces of $K\cap
K_0$, we have $E=2V$ and $F\leq V/2$ (since each edge lies in exactly one
face and each face has at least $4$ edges). Therefore $-2r\leq \chi(K\cap
P_u)=\chi(K\cap K_0)\leq-V/2$ (note that the latter estimate does not
require that the faces themselves have Euler characteristic $1$). Since
$V\geq 4r^2\abs{k\ell}$, it follows that $r\abs{k\ell}\leq1$, forcing
$r=\abs{k\ell}=1$, $\chi(K\cap K_0)= -2$, $V= 4$, and $F=2$.
That is, $K\cap K_0$ consists of two faces, each a $4$-gon, meeting at
their four vertices. Since $\abs{k\ell}=1$, $\Gamma$ is the image of two
embedded circles of $T_u$ each representing $a^{\pm1}b^{\pm2}$. Since
$\chi(K\cap P_u)=\chi(K\cap K_0)=-2$ and $\chi(K)=0$, each of these circles
must bound a disk in $R_u$.  This contradicts the hypothesis that
$(m,n)\neq (1,1)$.
\end{proof}
\index{figures!figure6@an exceptional $K\cap K_0$ for $M(1,1)$}%
\begin{figure}
\begin{center}
\includegraphics[width=0.5\textwidth]{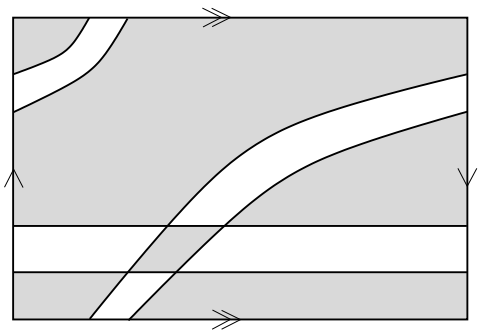}
\caption{The shaded region shows $K\cap K_0$ in $K_0$ for a case when
$r=|k\ell|=1$. Necessarily $M=M(1,1)$, which is excluded by hypothesis.}
\label{fig:intersection}
\end{center}
\end{figure}

Figure~\ref{fig:intersection} shows $K\cap K_0$ for a Klein bottle $K$ in
$M(1,1)$ that \textit{is} the flattening of a Klein bottle that meets every
$T_u$ close to $K_0$ in longitudes not homotopic to fibers, i.~e.~in loops
representing $ab^2$.

\newpage

\section [Generic position families] {Generic position families}
\label{sec:generic position families}

In this section, we achieve the necessary generic position for a
parameterized family. As usual, $M=M(m,n)$ with at least one of $m>1$ or
$n>1$. 

\begin{proposition}\index{generic position}
Let $F\colon D^k\to \Imb(K_0,M)$ be a parameterized family of embeddings of
the standard Klein bottle $K_0$ into $M$. Then every open neighborhood of
$F$ in $\Maps(D^k,\Imb(K_0,M))$ contains a map $G\colon D^k\to\Imb(K_0,M)$
for which $G(t)(K_0)$ is in generic position with respect to $K_0$ for all
$t\in D^k$. Moreover, we may select $G$ to be homotopic to $F$ within the
given neighborhood.\par
\label{perturb 2-manifolds}
\end{proposition}

\begin{proof}
From Lemma~(5.2) of~\cite{I2} (see also~\cite{Bruce}), a $G$ with each
$G(t)(K_0)$ in generic position exists, and we need only verify that it may
also be selected to be homotopic to $F$ within the given neighborhood~$V$.

Each $F(t)$ determines a bundle map from the restriction $E$ of the tangent
bundle of $M$ to $K_0$ to the restriction $E(t)$ of the tangent bundle of
$M$ to $F(t)(K_0)$; in the directions tangent to $K_0$, it is the
differential of $F(t)$, and it takes unit normals to unit normals. At
each~$t$, the Fr\'echet manifold of $\Maps$-sections of $E(t)$ whose image
vectors all have length less than some sufficiently small $\epsilon$
corresponds, using the exponential map, to a neighborhood $W_\epsilon(t)$
of $F(t)$ in~$\Imb(K_0,M)$. In particular, the zero section corresponds
to~$F(t)$. Since $D^k$ is compact, we may fix a uniform value of $\epsilon$
for all~$F(t)$.

An $\epsilon$-small section $s(t)$, corresponding to an embedding $L(t)$,
is isotopic to the zero section by sending each $s(x)$ to $(1-s)s(x)$. Via
the exponential, this becomes a homotopy $L_s$ with each $L_0(t)=L(t)$ and
$L_1(t)=F(t)$, that is, a homotopy from $L$ to $F$ as elements of
$\Maps(D^k,\Imb(K_0,M))$. With respect to coordinates on $E(t)$, all
partial derivatives of $s(t)$ move closer to zero $s$ goes from $0$ to $1$,
so those of $L_s$ in $M$ move closer to those of $F$.  Consequently,
provided that $\epsilon$ was small enough and $G$ was selected so that each
$G(t)$ was in $W_\epsilon(t)\cap V$, the $G_s$ will remain in~$V$.
\end{proof}

We are now ready for the main result of Sections~\ref{sec:generic position}
and~\ref{sec:generic position families}.

\begin{theorem} Let $F\colon D^k\to \Imb(K_0,M)$ be a parameterized family 
of Klein bottles in $M$. Assume that if $t\in \partial D^k$, then $F(t)$ is
fiber-preserving and $F(t)(K_0)\neq K_0$. Then $F$ is homotopic relative to
$\partial D^k$ to a family $G$ such that for each $t\in D^k$, there exists
$u>0$ so that $G(t)(K_0)$ meets $T_u$ transversely and each circle of
$G(t)(K_0)\cap T_u$ is either inessential in $T_u$, or represents $a$ or
$b^2$ in $\pi_1(T_u)$.
\label{thm:circles}
\end{theorem}

\begin{proof}
We first note that any embedded Klein bottle in $M$ must meet $K_0$, since
otherwise it would be embedded in the open solid torus $\overline{M-K_0}$,
so would admit an embedding into $3$-dimensional Euclidean space.

Recall that $M(m,n)$ is constructed from the $\I$-bundle $P$ over $K_0$ by
attaching a solid torus. There is a $u$-coordinate on $P$, $0\leq u\leq 1$,
such that the points with $u=0$ are $K_0$ and for each $0<u\leq 1$, the
points with $u$-coordinate equal to $u$ are a ``level'' torus $T_u$.
Fixing a parameter $t\in\partial D^k$, let $f\colon F(t)^{-1}(P-K_0)\to \I$
be the composition of $F(t)$ with projection to the $u$-coordinate of
$P-K_0$. Since $F(t)(K_0)$ must meet $K_0$, and by hypothesis, $F(t)(K_0)$
does not equal $K_0$, the image of $f$ contains an interval. By Sard's
Theorem, almost all values of $u$ are regular values of $f$, so there is
some level $T_u$ such that $F(t)(K_0)$ meets $T_u$ transversely. 
%In particular, at parameters in $\partial D^k$, $F(t)$ is fiber-preserving
%so $F(t)(K_0)$ meets $T_u$ in curves that represent $a$ or $b^2$ in
%$\pi_1(T_u)$.

Since transversality is an open condition and $\partial D^k$ is compact,
there is a finite collection of open sets in $D^k$ whose union contains
$\partial D^k$ and such that on each open set, there is a corresponding
level $T_u$ such that $F(t)(K_0)$ meets $T_u$ transversely for every $t$ in
the open set. At points of $\partial D^k$, the intersection curves are
fibers, so must be either $a$- or $b^2$-circles in~$\pi_1(T_u)$. Choose a
collar neighborhood $U=\partial D^k\times \I$ of $\partial D^k$, with
$\partial D^k\times\{0\}=\partial D^k$, such that the closure
$\overline{U}$ is contained in the union of these open sets. Since
transversality is an open condition, there is an open neighborhood $V$ of
$F$ in $\Maps(D^k,\Maps(K_0,M))$ such that for any map $G$ in $V$ and any
$t\in U$, $G(t)$ is transverse to one of the corresponding levels, and
$G(t)(K_0)$ intersects that level in loops representing either $a$
or~$b^2$.

Apply Proposition~\ref{perturb 2-manifolds} to obtain a homotopy $G_s'$
from $F$ to a map $G_1'$ for which $G_1'(t)$ meets $K_0$ in generic
position for every $t\in D^k$, and such that each $G_s'$ lies in
$V$. Define a new homotopy $G_s$ that equals $G_s'$ on $\overline{D^k-U}$
and carries out only the portion of $G'_s$ from $s=0$ to $s=r$ on each
$\partial D^k\times \{r\}\subset U$. In particular, $G_s$ is a homotopy
relative to~$\partial D^k$.

At each point $t$ of $U$, $G_s(t)$ lies in $V$ so is transverse to some
level $T_u$ and $G_s(K_0)$ intersects that level in loops representing
either $a$ or~$b^2$. On $\overline{D^k-U}$, $G_s(K_0)$ meets $K_0$ in
generic position, so by Proposition~\ref{prop:circles}, it meets all $T_u$,
for $u$ sufficiently close to $0$, transversely in loops which are either
inessential in $T_u$ or represent $a$ or $b^2$ in~$\pi_1(T_u)$.
\end{proof}

%We can taper this homotopy off in the collar
%neighborhood $U$, obtaining a homotopy relative to $\partial D^k$ from $F$
%to a map $G$ which agrees with $G'$ on $D^k-U$, and so that for each $t\in
%U$, $G(t)(K_0)$ meets some $T_u$ transversely in loops representing $a$ or
%$b^2$. For $t\in D^k-U$, each $G(t)$ meets $K_0$ in generic position. By
%Proposition~\ref{prop:circles}, the latter $G(t)(K_0)$ meet all $T_u$, for
%$u$ sufficiently close to $0$, transversely in loops which are either
%inessential in $T_u$, or represent $a$ or $b^2$ in~$\pi_1(T_u)$.

\section[Parameterization]{Parameterization}
\label{parameterization}

In this section we will complete the proof of Theorem~\ref{main}. By definition,
both $\imb(K_0,M)$ and $\imb_f(K_0,\allowbreak M)$ are connected, so
$\pi_0(\imb(K_0,M),\imb_f(K_0,M))=0$. To prove that the higher relative
homotopy groups vanish, we begin with a smooth map $F\colon D^k\rightarrow
\imb(K_0,M)$, where $k\geq1$, which takes all points of $\partial D^k$ to
$\imb_f(K_0,M)$. We will deform $F$, possibly changing the embeddings at
parameters in $\partial D^k$ but retaining the property that they are
fiber-preserving, to a family which is fiber-preserving at every
parameter. In fact, all deformations will be relative to $\partial D^k$,
except for the first step.

\smallskip
\noindent\textsl{Step 1:} Obtain generic position\index{generic position}
\smallskip

In order to apply Theorem~\ref{thm:circles}, we must ensure that no
$F_t(K_0)$ equals $K_0$ for $t\in \partial D^k$. Select a smooth isotopy
$J_s$ of $M$, $0\leq s\leq 1$, with the following properties:
\begin{enumerate}
\item[(a)]
$J_0$ is the identity of $M$.
\item[(b)]
Each $J_s$ is fiber-preserving.
\item[(c)]
$J_1(K_0)\neq F(t)(K_0)$ for any $t\in \partial D^k$.
\end{enumerate}
One construction of $J_s$ is as follows. As elaborated in
Section~\ref{fiberings}, the image of $K_0$ in the quotient orbifold
$\mathcal{O}$ of the Hopf fibering is either a geodesic arc or a geodesic
circle. Let $A$ denote the image, let $S$ be the endpoints of $A$ if $A$ is
an arc and the empty set if $A$ is a circle, and let $T$ be the inverse
image of $S$ in $M$. Consider an isotopy $j_t$ of $\mathcal{O}$, relative
to $S$, that moves $A$ to an arc or circle $A'$ of large length.

By 
\index{Projection Theorem!singular fiberings}%
Theorem~\ref{sfproject diffs}, the map $\Diff_f(M \rel T)\to \Diff(\O
\rel S)$ induced by projection is a fibration, so $j_t$ lifts to an isotopy
$J_t$ of $M$ with $J_0$ the identity map of $M$. The image of $J_1(K_0)$ is
$A'$.  Since $F(t)$ is fiber-preserving for each $t\in \partial D^k$, its
image is an arc or circle, and by compactness of $\partial D^k$, there is a
maximum value for the lengths of these images. Provided that $A'$ was
selected to have length larger than this maximum, $J_1(K_0)$ cannot equal
any $F(t)(K_0)$ for $t\in \partial D^k$.

%For example, fix Riemannian metrics and take an isotopy that moves $K_0$ so
%that it has a point with Gaussian curvature larger than the maximum
%curvature of any of the $F_t(K_0)$. Define a homotopy of $F$ by
%$F_s(t)=J_s^{-1}\circ F(t)$. Each $F_1(t)(K_0)\neq K_0$,
%otherwise we would have $J_1(K_0)=J_1 F_1(t)(K_0)= F(t)(K_0)$. Since each
%$J_s$ is fiber-preserving, $F_1$ represents the same element of
%$\pi_k(\imb(K_0,M),\imb_f(K_0,M))$ that $F$ did, so we may assume that
%$F(t)(K_0)\neq K_0$ for $t\in\partial D^k$.

We now perform a deformation $F_s$ of $F$ such that each
$F_s(t)=J_s^{-1}\circ F(t)$. At each $t\in \partial D^k$, each $F_s(t)$ is
fiber-preserving, and $F_1(t)(K_0)=J_1^{-1}(F(t)(K_0))\neq K_0$. We can now
apply Theorem~\ref{thm:circles} to further deform $F$ relative to $\partial
D^k$ so that for each $t$, there is a value $u$, $0<u\leq 1$, so that
\begin{enumerate}
\item[(1)] $F(t)$ is transverse to $T_u$.
\item[(2)] Every circle of $K_t \cap T_u$ is either inessential in $T_u$,
or represents either $a$ or $b^2$ in $\pi_1(T_u)$.
\end{enumerate}

From now on, we will write $K_t$ for~$F(t)(K_0)$.

\smallskip
\noindent\textsl{Step 2:} Eliminate inessential intersection circles
\smallskip

The next step is to get rid of inessential intersections. Consider a single
$K_t$ and its associated level $T_u$. Notice first that each circle $c$ of
$K_t\cap T_u$ that bounds a (necessarily unique) $2$-disk $D_T(c)$ in $T_u$
also bounds a unique $2$-disk $D_K(c)$ in $K_t$, since $K_t$ is
geometrically incompressible. We claim that if $D_K(c)$ is innermost among
all such disks on $K_t$, then the interior of $D_K(c)$ is disjoint from
$T_u$. If not, then there is a smaller disk $E$ in $D_K(c)$ such that
$\partial E$ is essential in $T_u$ and the interior of $E$ is disjoint from
$T_u$. Now $E$ cannot be contained in $P_u$, since $T_u$ is incompressible
in $P_u$, so $E$ must be a meridian disk of $R_u$. But then, $\partial E$
is a circle of $K_t\cap T_u$ which is essential in $T_u$ but does not
represent either $a$ or $b^2$, contradicting~(2) and establishing the
claim. We conclude that if $D_K(c)$ is innermost, then $D_K(c)$ and
$D_T(c)$ bound a unique $3$-ball $B(c)$ in $M$ that meets $T_u$ only
in~$D_T(c)$.

We now follow the procedure of \index{Hatcher}Hatcher described in
\cite{Hold,Hnew} to deform the family $F$ to eliminate the circles of
$K_t\cap T_u$ that are inessential in $T_u$. It is not difficult to adapt
the procedure to our situation, in fact a few simplifications occur, but
since this is a crucial part of our argument and these methods are
unfamiliar to many, we will navigate through the details. We will follow
\cite{Hnew}, as it is an easier read than~\cite{Hold}, and its numbered
formulas are convenient for referencing in our discussion. Start at the
proof of the main theorem on p.~2 of~\cite{Hnew}. Our $K_t$ and $T_u$ are
in the role of the surfaces $M_t$ and $N_i$ in~\cite{Hnew}. We ignore the
points called $p_t$ there, which are irrelevant for us (since a loop can
bound at most one disk in $K_t$ or $T_u$). Only notational substitutions
are needed to obtain the conditions (1)-(3) and (5)-(6), $(5_\epsilon)$,
and $(6_\epsilon)$ in~\cite{Hnew} (condition~(4) there concerns the
irrelevant $p_t$), and the conditions called~(a) and~(b) there are assumed
inductively as before. We have already seen that the disks $D_K(c_t)$ and
$D_T(c_t)$ bound a unique $3$-ball $B(c_t)$--- this replaces the hypothesis
of the main theorem of~\cite{Hnew} that any two essential $2$-spheres in
$M$ are isotopic. The argument that the boundary of $B(c_t)$ has a corner
with angle less than $\pi$ along $c_t$ applies in our case, since $T_u$
cannot be contained in the $3$-ball $B(c_t)$.

The individual isotopies that make up the deformations called $M_{tu}$
in~\cite{Hnew} (so they would be called $K_{tu}$ for us)
are constructed as before. A crucial point in Hatcher's method is that if
the isotopies pushing $D_K(c_t)$ and $D_K(c_t')$ across $B(c_t)$ and
$B(c_t')$ overlap in time, then the balls $B(c_t)$ and $B(c_t')$ must be
disjoint, ensuring that the isotopies have disjoint support and do not
interfere with each other.  The verification that such a $B(c_t)$ and
$B(c_t')$ must be disjoint proceeds as in~\cite{Hnew}: If the isotopies
overlap in time, then
\begin{enumerate}
\item[(i)] $D_K(c_t)$ is disjoint from $D_K(c_t')$ (condition~$(5_\epsilon)$),
\item[(ii)] $D_T(c_t)$ and $D_T(c_t')$ lie in different levels $T_u$ and $T_{u'}$
(condition~$(6_\epsilon)$), and
\item[(iii)] $D_K(c_t)$ is disjoint from $T_{u'}$ and $D_K(c_t')$ is
disjoint from $T_u$ (condition~(b)).
\end{enumerate}
Conditions (i)-(iii) show that the boundaries of $B(c_t)$ and $B(c_t')$ are
disjoint, so $B(c_t)$ and $B(c_t')$ are either disjoint or nested. But
neither $T_u$ nor $T_{u'}$ is contained in a $3$-ball in $M$, so nesting
would contradict~(iii). The remainder of the proof is completed without
significant modification.

At the end of this process, for each $t\in D^k$, there is a value $u>0$ so
that in place of~(2) above we have

\begin{enumerate}
\item[(2$^{\prime}$)] Every intersection circle of $K_t$ with $T_u$
represents either $a$ or $b^2$ in $\pi_1(T_u)$.
\end{enumerate}

\smallskip
\noindent\textsl{Step 3:} Make the intersection circles fibers
\smallskip

Since $a$ and $b^2$ are nontrivial elements of $\pi_1(M)$, the circles of
$K_t\cap T_u$ are essential in $K_t$ as well, so each component of $K_t\cap
R_u$ must be either an annulus or a M\"obius band. In fact, M\"obius bands
cannot occur.  For the center circle of such a M\"obius band would have
intersection number $1$ with $K_t$ and intersection number $0$ with $K_0$,
contradicting the fact that $K_t$ is isotopic to~$K_0$.

We will use the procedure of \cite{Hold,Hnew}, this time to pull the annuli
$K_t\cap R_u$ out of the $R_u$. The details of adapting~\cite{Hold,Hnew} are not
quite as straightforward as in Step~2.  Again, the $K_t$ and $T_u$ are in
the role of $M_t$ and $N_t$ respectively, and the points $p_t$ are
irrelevant. Setting notation, for each parameter $t$ in a ball $B_i$ in
$D^k$, $K_t$ is transverse to a level $T_{u_i}$, and $K_t\cap R_{u_i}$ is a
collection of annuli. These annuli are in the role of the disks $D_M(c)$
of~\cite{Hnew}. Denote by $C_t^i$ the annuli of $K_t\cap R_{u_i}$ whose
boundary circles are not isotopic in $T_{u_i}$ to fibers. Notice that
$C_t^i$ is empty for parameters $t$ in $\partial D^k$. By condition
(2$^{\prime}$) and Lemma~\ref{lem:longitudes are fibers}, any circle of
$K_t\cap R_u$ that is not isotopic in $T_u$ to a fiber is also not a
longitude of $R_u$. So each such annulus $a_t$ is parallel across a region
$W(a_t)$ in $R_{u_i}$ to a uniquely determined annulus $A_t$ in~$T_{u_i}$.

Let $C_t$ be the union of the $C_t^i$ for which $t\in B_i$.  Any two annuli
of $C_t$ are either nested or disjoint on~$K_t$. Again we consider
functions $\varphi_t\colon C_t\to (0,1)$ such that
$\varphi_t(a_t)<\varphi_t(a_t')$ whenever $a_t\subset a_t'$; this is the
version of condition~(5) needed for our case. For example, we may take
$\varphi_t(a_t)$ to be the area in $K_0$ of the inverse image of $a_t$ with
respect to the embedding $K_0\to K_t$, where the area of $K_0$ is
normalized to $1$. The transversality trick of~\cite{Hnew} achieves
condition~(6) as before, conditions~$(5_\epsilon)$ and~$(6_\epsilon)$ are
again true by compactness, and conditions~(a) and~(b) are assumed
inductively.

The angles of the regions $W(a_t)$ along the circles $a_t\cap T_{u_i}$ are
less than $\pi$, this time simply because $a_t$ is contained in~$R_{u_i}$.

Again, the key point in defining the isotopies that pull the $a_t$ across
the regions $W(a_t)$ and out of the $R_{u_i}$ is that is two of the
isotopies on such regions $W(a_t)$ and $W(a_t')$ overlap in time,
then $W(a_t)$ and $W(a_t')$ must be disjoint. When they do overlap in time,
we have
\begin{enumerate}
\item[(i)] $a_t$ is disjoint from $a_t'$ (condition~$(5_\epsilon)$), and
\item[(ii)] $A_t$ and $A_t'$ lie in different levels $T_u$ and $T_{u'}$
(condition~$(6_\epsilon)$).
\end{enumerate}
We also have
\begin{enumerate}
\item[(iii)] $a_t$ is disjoint from $T_{u'}$, and $a_t'$ is disjoint from
$T_u$. 
\end{enumerate}
To see this, choose notation so that $T_{u'}\subset R_u$. Then $a_t'$ is
disjoint from $T_u$ since $a_t'\subset R_u$. If $a_t$ were to meet $T_{u'}$,
then there would be a circle of $a_t\cap T_{u'}$ that is parallel in
$\overline{R_u-R_{u'}}$ to a circle of $a_t\cap T_u$. The latter is not a
longitude of $R_u$, so the circle of $a_t\cap T_{u'}$ is not a longitude of
$R_{u'}$. So $a_t$ contains an annulus of $K_t\cap R_{u'}$ that is in
$C_t$, contradicting condition~$(5_\epsilon)$ (that is, such an annulus
would already have been eliminated earlier in the isotopy).

Conditions (i) and (ii) show that the boundaries of $W(a_t)$ and $W(a_t')$
are disjoint, so $W(a_t)$ and $W(a_t')$ are either disjoint or
nested. Suppose for contradiction that they are nested. Again we choose
notation so that $R_{u'}\subset R_u$. Since $a_t$ is disjoint from
$T_{u'}$, and $a_t'\subset W(a_t)$, we must have $R_{u'}\subset W(a_t)$. It
follows that $W(a_t)$ contains a loop that generates $\pi_1(R_{u'})$ and
hence generates $\pi_1(R_u)$. But $W(a_t)$ is a regular neighborhood of the
annulus $a_t$, and the boundary circles of $a_t$ were not longitudes of
$R_u$, so this is contradictory. The remaining steps of the argument
require no non-obvious modifications.

At the end of this process, we have in addition to (1) that
\begin{enumerate}
\item[(2$^{\prime\prime}$)] Every circle of $K_t \cap T_u$ is isotopic in
       $T_u$ to a fiber of the Seifert fibering on $M$.
\end{enumerate}

\smallskip
\noindent\textsl{Step 4:} Establish lemmas needed for the final step
\smallskip

To complete the argument, we require two technical lemmas.

\begin{lemma} Let $T$ be a torus with a fixed
$S^1$-fibering, and let $C_n=\cup_{i=1}^nS_i$ be a union of $n$ distinct
fibers. Then $\imb_f(C_n,T)\to\imb(C_n,T)$ is a  homotopy equivalence.
The same holds for the Klein bottle with either the meridional fibering or
the longitudinal singular fibering.\par
\label{lem:fiberpreserving1}
\end{lemma}

\begin{proof} 
First consider a surface $F$ other than the $2$-sphere, the disk, or the
projective plane, with a base point $x_0$ in the interior of $F$ and an
embedding $S^1\subset F$ with $x_0\in S^1$ which does not bound a disk in
$F$. In the next paragraph, we will sketch an argument using \cite{G} that
$\imb((S^1,x_0),(\interior(F), x_0))$ has trivial homotopy groups. The
approach is awkward and unnatural, but we have found no short, direct way
to deduce this fact from \cite{G} or other sources.

By the\index{Palais-Cerf Restriction Theorem}
Palais-Cerf Restriction Theorem, there is a fibration
\[ \Diff(F \rel S^1)\cap \diff(F,x_0) \to  \diff(F,x_0)\to 
\imb((S^1,s_0),(\interior(F), x_0)) \ .\] Since $F$ is not the $2$-sphere,
disk, or projective plane, Proposition~2 of \cite{G} shows that
$\diff(F,x_0)$ has the same homotopy groups as $\diff_1(F,x_0)$, the
subgroup of diffeomorphisms that induce the identity on the tangent space
at $x_0$, and by Theorem~2 of \cite{G}, the latter is contractible. So we
have isomorphisms
\[ \pi_{q+1}(\imb((S^1,s_0),(\interior(F), x_0))) \cong
\pi_q(\Diff(F \rel S^1)) \]
for $q\geq 1$, and
\[ \pi_1(\imb((S^1,s_0),(\interior(F), x_0))) \cong
\pi_0(\Diff(F \rel S^1) \cap \diff(F,x_0))\ . \] 
Proposition~6 of \cite{G} shows that the components of $\Diff(F \rel S^1)$
are contractible, so it remains to see that only one component of $\Diff(F
\rel S^1)$ is contained in $\diff(F,x_0)$. That is, if $h\in \Diff(F \rel
S^1)\cap \diff(F,x_0)$, then $h$ is isotopic to the identity relative to
$S^1$. This is an exercise in surface theory, using Lemma~1.4.2
of~\cite{Waldhausen}.

We now start with the torus case of the lemma. Choose notation so that the
$S_i$ lie in cyclic order as one goes around $T$, and fix basepoints $s_i$
in $S_i$ for each $i$.  Consider the diagram
\begin{equation*}
\begin{CD}
\imb_f(S_n,T\rel s_n) @>>> \imb_f(S_n,T) @>>>  \imb(s_n,T)\\
@VVV @VVV @VV=V\\
\imb(S_n,T\rel s_n) @>>>  \imb(S_n,T) @>>>  \imb(s_n,T)\makebox[0pt][l]{\ .}
\end{CD}
\end{equation*}

\noindent 
The first row is a fibration by 
\index{Restriction Theorem!to point}%
Corollary~\ref{restrict embeddings to S}
and the second by the Palais-Cerf Restriction Theorem. The fiber of the
top row is homeomorphic to $\Diff_+(S_n\rel s_n)$, the group of
orientation-preserving diffeomorphisms, which is contractible. We have
already seen that the fiber of the second row is contractible.
Therefore the middle vertical arrow is a homotopy equivalence. For
$n= 1$, this completes the proof, so we assume that $n\geq 2$.

Let $A$ be the annulus that results from cutting $T$ along $S_n$, and let
$A_0$ be the interior of $A$. Consider the diagram
\begin{equation*}
\begin{CD}
\imb_f(S_{n-1}, A_0\rel s_{n-1}) @>>> \imb_f(S_{n-1},A_0) 
@>>> \imb(s_{n-1},A_0)\\
@VVV @VVV @VV=V\\
\imb(S_{n-1}, A_0\rel s_{n-1}) @>>> \imb(S_{n-1},A_0)
@>>> 
\imb(s_{n-1},A_0)
\makebox[0pt][l]{\ .}
\end{CD}
\end{equation*}
\noindent As in the previous diagram, the rows are fibrations. 
As before, the fibers are contractible, so the middle vertical
arrow is a homotopy equivalence. Now consider the diagram
\begin{small}
\begin{equation*}
\begin{CD}
\imb_f(C_{n-1},A_0\rel S_{n-1}) @>>> \imb_f(C_{n-1},A_0) @>>>
\imb_f(S_{n-1},A_0)\\
@VVV @VVV @VVV\\
\imb(C_{n-1},A_0\rel S_{n-1}) @>>> \imb(C_{n-1},A_0) @>>>  \imb(S_{n-1},A_0)
\makebox[0pt][l]{\ .}
\end{CD}
\end{equation*}
\end{small}

\noindent 
The top row is a fibration by Corollary~\ref{corollary3}, and the bottom by
the Palais-Cerf Restriction Theorem. The right vertical arrow was shown to
be a homotopy equivalence by the previous diagram.  For $n=2$, both fibers
are points, so the middle vertical arrow is a homotopy equivalence. But
$\imb_f(C_{n-1},A_0\rel S_{n-1})$ can be identified with
$\imb_f(C_{n-2},A_0-S_{n-1})$, and similarly for the non-fiber-preserving
spaces. So induction on $n$ shows that the middle vertical arrow is a
homotopy equivalence for any value of~$n$.

To complete the proof, we use the diagram
\begin{equation*}
\begin{CD}
\imb_f(C_n,T\rel S_n) @>>> \imb_f(C_n,T) @>>>  \imb_f(S_n,T)\\
@VVV @VVV @VVV\\
\imb(C_n,T\rel S_n) @>>>  \imb(C_n,T) @>>>  \imb(S_n,T)\makebox[0pt][l]{\ .}
\end{CD}
\end{equation*}
The rows are fibrations, as in the previous diagram. The right-hand
vertical arrow is the case $n=1$, already proven, and the map between
fibers can be identified with $\imb_f(C_{n-1}, A_0)\to \imb(C_{n-1}, A_0)$,
which has been shown to be a homotopy equivalence for all $n$.

For the Klein bottle case, the proof is line-by-line the same in the case
of the meridional fibering. For the longitudinal singular fibering, the only
difference is that rather than an annulus $A$, the first cut along $S_n$
produces either one or two M\"obius bands.
\end{proof}

\begin{lemma}
Let $\Sigma$ be a compact $3$-manifold with nonempty boundary and having a
fixed Seifert fibering, and let $F$ be a vertical $2$-manifold properly
embedded in $\Sigma$. Let $\imb_{\partial f}(F,\Sigma)$ be the connected
component of the inclusion in the space of (proper) embeddings for which
the image of $\partial F$ is a union of fibers. Then $\imb_f(F,\Sigma)\to
\imb_{\partial f}(F,\Sigma)$ is a homotopy equivalence.
\label{lem:fiberpreserving2}
\end{lemma}

To prove Lemma~\ref{lem:fiberpreserving2}, we need a preliminary result.

\begin{lemma} The following maps
induced by restriction are fibrations.
\begin{enumerate}
\item[{\rm(i)}] $\imb(F,\Sigma)\to \imb(\partial F,\partial \Sigma)$
\item[{\rm(ii)}] $\imb_{\partial f}(F,\Sigma)\to \imb_f(\partial
F,\partial \Sigma)$
\item[{\rm(iii)}] $\imb_f(F,\Sigma)\to \imb_f(\partial F,\partial \Sigma)$.
\end{enumerate}
\label{sublemma for fiberpreserving2}
\end{lemma}

\begin{proof} Parts~(i) and~(iii) are cases of 
Corollaries~\ref{special2} and~\ref{special3}. Part (ii) follows from part
(i) since $\imb_{\partial f}(F,\Sigma)$ is the inverse image of $\imb_f(\partial
F,\partial\Sigma)$ under the fibration of part~(i).
\end{proof}

\begin{proof}[Proof of Lemma~\ref{lem:fiberpreserving2}.]
First we use the following fibration from 
\index{Projection Theorem!singular fiberings}%
Theorem~\ref{sfproject diffs}:
\[\Diff_v(\Sigma\rel\partial\Sigma)\cap
\diff_f(\Sigma\rel\partial\Sigma)
\to\diff_f(\Sigma\rel\partial\Sigma)
\to \diff(\mathcal{O}\rel\partial \mathcal{O})\]

\noindent where $\mathcal{O}$ is the quotient orbifold of $\Sigma$ and as
usual $\Diff_v$ indicates the diffeomorphisms that take each fiber to
itself. The full orbifold diffeomorphism group of $\mathcal{O}$ can be
identified with a subspace consisting of path components of the
diffeomorphism group of the $2$-manifold $B$ obtained by removing the cone
points from~$\mathcal{O}$ (the subspace for which the permutation of
punctures respects the local groups at the cone points). Since $\partial B$
is nonempty, $\diff(B\rel\partial B)$ and therefore
$\diff(\mathcal{O}\rel\partial \mathcal{O})$ are contractible. Since
$\pi_1(\diff(\mathcal{O}\rel\partial \mathcal{O}))$ is trivial, the exact
sequence of the fibration shows that $\Diff_v(\Sigma\rel\partial\Sigma)\cap
\diff_f(\Sigma\rel\partial\Sigma)$ is connected, so is equal to
$\diff_v(\Sigma\rel\partial\Sigma)$. It is not difficult to see that each
component of $\Diff_v(\Sigma\rel\partial\Sigma)$ is contractible (see
Lemma~\ref{rel fiber} for a similar argument), so we conclude that
$\diff_f(\Sigma\rel\partial\Sigma)$ is contractible.

Next, consider the diagram
\begin{footnotesize}
$$\begin{array}{@{}c@{\;\,}c@{\;\,}c@{\;\,}c@{\;\,}c@{}}
\Diff_f(\Sigma\rel F\cup \partial \Sigma)\cap\diff_f(\Sigma\rel\partial\Sigma) &
\longrightarrow                              &
\diff_f(\Sigma\rel\partial\Sigma)            &
\longrightarrow                              &
\imb_f(F,\Sigma\rel\partial F)\\
\Big\downarrow & & \Big\downarrow & & \Big\downarrow\\
\Diff(\Sigma\rel F\cup\partial\Sigma)\cap\diff(\Sigma\rel\partial\Sigma) &
\longrightarrow                              &
\diff(\Sigma\rel\partial\Sigma)              &
\longrightarrow                              &
\imb(F,\Sigma\rel\partial F)
\end{array}$$
\end{footnotesize}

\noindent where the rows are fibrations by 
\index{Restriction Theorem!Palais-Cerf!relative version}%
\index{Restriction Theorem!singular fiberings}%
Corollaries~\ref{sfcorollary2}
and~\ref{palaiscoro2}. We have already shown that the components of
$\Diff_f(\Sigma\rel\partial\Sigma)$ and (by cutting along $F$) the
components of $\Diff_f(\Sigma\rel F\cup\partial\Sigma)$ are
contractible. By~\cite{H} (which, as noted in \cite{H}, extends to $\Diff$
using \cite{HSmale}), the components of $\Diff(\Sigma\rel\partial\Sigma)$
and $\Diff(\Sigma\rel F\cup\partial\Sigma)$ are contractible. Therefore to
show that $\imb_f(F,\Sigma\rel\partial F)\to\imb(F,\Sigma\rel\partial F)$
is a homotopy equivalence it is sufficient to show that
$\pi_0(\Diff_f(\Sigma\rel F\cup \partial \Sigma)\cap
\diff_f(\Sigma\rel\partial\Sigma))\to \pi_0(\Diff(\Sigma\rel
F\cup\partial\Sigma)\cap \diff(\Sigma\rel\partial\Sigma))$ is bijective. It
is surjective because every diffeomorphism of a Seifert-fibered
$3$-manifold which is fiber-preserving on the (non-empty) boundary is
isotopic relative to the boundary to a fiber-preserving diffeomorphism
(Lemma~VI.19 of W. Jaco~\cite{Jaco}). It is injective because
fiber-preserving diffeomorphisms that are isotopic are isotopic through
fiber-preserving diffeomorphisms (see
Waldhausen~\cite{Waldhausen}\index{Waldhausen}).

The proof is completed by the following diagram in which the rows
are fibrations by parts (iii) and (ii) of Lemma~\ref{sublemma for
fiberpreserving2}, and we have verified that the left vertical arrow
is a homotopy equivalence.
\begin{equation*}
\begin{CD}
\imb_f(F,\Sigma \rel\partial F) @>>>  \imb_f(F,\Sigma) @>>>
\imb_f(\partial F,\partial \Sigma)\\
@VVV @VVV @VV=V\\
\imb(F,\Sigma \rel\partial F) @>>>  \imb_{\partial f}(F,\Sigma) @>>>
\imb_f(\partial F,\partial \Sigma)
\end{CD}
\end{equation*}
\end{proof}

\smallskip
\noindent\textsl{Step 5:} Complete the proof
\smallskip

We can now complete the proof of Theorem~\ref{main} by deforming the family
$F$ to a fiber-preserving family. Since~(1) and~(2$^{\prime\prime}$) are
open conditions, we can cover $D^k$ by convex $k$-cells $B_j$, $1\leq j\leq
r$, having corresponding levels $T_{u_j}$ for which (1) and
(2$^{\prime\prime}$) hold throughout $B_j$. Also, we may slightly change
the $u$-values, if necessary, to assume that the $u_i$ are distinct. It is
convenient to rename the $B_j$ so that $u_1<u_2<\dots<u_r$, that is, so
that the levels $T_{u_j}$ sit farther away from $K_0$ as $j$ increases.

Choose a PL triangulation $\Delta$ of $D^k$ sufficiently fine so that each
$i$-cell lies in at least one of the $B_j$. The deformation of $F$ will
take place sequentially over the $i$-skeleta of $\Delta$. It will never be
necessary to change $F$ at points of~$\partial D^k$.

Suppose first that $\tau$ is a $0$-simplex of $\Delta$. Let $j_1<j_2<\dots
<j_s$ be the values of $j$ for which $\tau\subseteq B_j$.  By condition
(2$^{\prime\prime}$), each intersection circle of $K_\tau$ with each
$T_{j_q}$ is isotopic in $T_{j_q}$ to a fiber of the Seifert fibering. We
claim that they are also isotopic on $K_\tau$ to an image of a fiber of
$K_0$ under $F(\tau)$.  Since $K_\tau$ is isotopic to $K_0$ and the
intersection circles are two-sided in $K_\tau$, each intersection circle is
isotopic in $M$ to an $a$-loop or a $b^2$-loop in $K_0$. When $m= 1$,
$b^2$ is the generic fiber of $M$, and $a$ is not isotopic in $M$ to $b^2$
since $a= b^{2n}$ and $n\neq 1$. When $n= 1$, $a$ is the fiber of $M$,
and $b^2$ is not isotopic to $a$ since $a^m= b^2$ and $m>1$. So the
isotopy from $K_\tau$ to $K_0$ carries the intersection loops to loops in
$K_0$ representing the fiber. But $a$-loops are nonseparating and
$b^2$-loops are separating, so the intersection loops must be isotopic in
$K_\tau$ to the image of the fiber of $K_0$ under $F(\tau)$.

We may deform the parameterized family near $\tau$, retaining transverse
intersection with each $T_{u_j}$ for which $\tau\in B_j$, so that the
intersection circles of $K_\tau$ with these $T_{u_j}$ are fibers and images of
fibers. To accomplish this, first change $F(\tau)$ by an ambient isotopy of
$M$ that preserves levels and moves the intersection circles onto fibers in
the $T_{u_j}$. Now, consider the inverse images of these circles in $K_0$. We
have seen that there is an isotopy that moves them to be fibers, changing
$F(\tau)$ by this isotopy (and tapering it off in a small neighborhood of
$\tau$ in $D^k$) we may assume that the intersection circles are fibers of
$K_\tau$ as well. Now, using Lemma~\ref{lem:fiberpreserving2} successively
on the solid torus $R_{u_s}$, the product regions
$\overline{R_{u_{j-1}}-R_{u_j}}$ for $j= j_s,j_{s-1},\dots,j_2$, and the
twisted $\I$-bundle $P_{u_{j_1}}$, deform $F(\tau)$ to be
fiber-preserving. These isotopies preserve the levels $T_{u_j}$ for which
$\tau\in B_j$, so may be tapered off near $\tau$ so as not to alter any
other transversality conditions.

Inductively, suppose that $F(t)$ is fiber-preserving for each $t$ lying in
any $i$-simplex of $\Delta$. Let $\tau$ be an $(i+1)$-simplex of
$\Delta$. For each $t\in\partial\tau$, $F(t)$ is fiber-preserving.
Consider a level $T_{u_j}$ for which $\tau\subset B_j$. For each $t\in
\tau$, the restriction of $F(t)$ to the inverse image of $T_{u_j}$ is a
parameterized family of embeddings of a family of circles into $T_{u_j}$,
which embeds to fibers at each $t\in \partial \tau$. By
Lemma~\ref{lem:fiberpreserving1}, there is a deformation of $F|_\tau$,
relative to $\partial \tau$, which makes each $K_t\cap T_{u_j}$ consist of
fibers in $T_{u_j}$. We may select the deformation so as to move image
points of each $F(t)$ only very near $T_{u_j}$, and thereby not alter
transversality with any other $T_{u_\ell}$. Now, the restriction of the
$F(t)^{-1}$ to the intersection circles is a family of embeddings of a
collection of circles into $K_0$, which are fibers at points in
$\partial\tau$. Using Lemma~\ref{lem:fiberpreserving1} we may alter
$F\vert_\tau$, relative to $\partial \tau$ and without changing the images
$F(t)(K_0)$, so that the intersection circles are fibers of $K_0$ as
well. We repeat this for all $\ell$ such that $\tau\subset B_\ell$. Using
Lemma~\ref{lem:fiberpreserving2} as before, proceeding from $R_{u_{j_s}}$
to $P_{u_{j_1}}$, deform $F$ on $\tau$, relative to $\partial\tau$, to be
fiber-preserving for all parameters in $\tau$. This completes the induction
step and the proof of Theorem~\ref{main}.

\chapter{Lens spaces}
\label{ch:lens}

Recall that we always use the term \textit{lens space} will mean a
$3$-dimensional lens space \indexsym{L(m,q)}{$L(m,q)$}$L(m,q)$ with 
$m\geq 3$. In addition, we always select $q$ so that $1\leq q<m/2$.

In this chapter, we will prove Theorem~\ref{thm:SCforLensSpaces}, the Smale
Conjecture for Lens Spaces.  The argument is regrettably quite lengthy. It
uses a lot of combinatorial topology, but draws as well on some mathematics
unfamiliar to many low-dimensional topologists. We have already seen some
of that material in earlier chapters, but we will also have to use the
\index{Rubinstein-Scharlemann!method}Rubinstein-Scharlemann method, reviewed in
Section~\ref{sec:Rubinstein-Scharlemann}, and some results from singularity
theory, presented in Section~\ref{sec:generalposition}.

The next section is a comprehensive outline of the entire proof. We hope
that it will motivate the various technical complications that ensue.
\longpage

\section[Outline of the proof]
{Outline of the proof}
\label{sec:outline}

Some initial reductions, detailed in Section~\ref{sec:reduction}, reduce
the Smale Conjecture for Lens Spaces to showing that the inclusion
$\diff_f(L)\to \diff(L)$ is an isomorphism on homotopy groups. Here,
$\diff(L)$ is the connected component of the identity in $\Diff(L)$, and
$\diff_f(L)$ is the connected component of the identity in the group of
diffeomorphisms that are fiber-preserving with respect to a Seifert
fibering of $L$ induced from the Hopf fibering of its universal cover,
$S^3$. To simplify the exposition, most of the argument is devoted just to
proving that $\diff_f(L)\to \diff(L)$ is surjective on homotopy groups,
that is, that a map from $S^d$ to $\diff(L)$ is homotopic to a map into
$\diff_f(L)$. The injectivity is obtained in Section~\ref{sec:Dk} by a
combination of tricks and minor adaptations of the main program.

Of course, a major difficulty in working with elliptic $3$-manifolds is
their lack of incompressible surfaces. In their place, we use another
structure which has a certain degree of essentiality, called a
\index{sweepout}\textit{sweepout.} This means a structure on $L$ as a
quotient of $P\times \I$, where $P$ is a torus, in which $P\times \{0\}$ and
$P\times \{1\}$ are collapsed to core circles of the solid tori of a
genus~1 Heegaard splitting of $L$. For $0<u\leq 1$, $P\times\{t\}$ becomes a
Heegaard torus in $L$, denoted by $P_u$, and called a
\index{levels}\textit{level.} The sweepout is chosen so that each $P_u$ is a
union of fibers. Sweepouts are examined in Section~\ref{sec:RSgraphic}.

Start with a parameterized family of diffeomorphisms $f\colon L\times
S^d\to L$, and for $u\in S^d$ denote by $f_u$ the restriction of $f$ to
$L\times\{u\}$.  The procedure that deforms $f$ to make each $f_u$
fiber-preserving has three major steps.

Step~1 (``finding good levels'') is to perturb $f$ so that for each $u$,
there is some pair $(s,t)$ so that $f_u(P_u)$ intersects $P_t$
transversely, in a collection of circles each of which is either essential
in both $f_u(P_s)$ and $P_t$ (a
\index{biessential}\textit{biessential}\index{intersection!biessential}intersection),
or inessential in both (a
\index{discal}\index{intersection!discal}\textit{discal}intersection), and at
least one intersection circle is biessential. This pair is said to
intersect in \index{good position}\textit{good position,} and if none of the
intersections is discal, in \index{very good!position}\textit{very good
position.}  These concepts are developed in Section~\ref{sec:very good
position}, after a preliminary examination of annuli in solid tori in
Section~\ref{sec:preliminaries}.

To accomplish Step~1, the methodology of Rubinstein and Scharlemann in
\cite{RS} is adapted. This is reviewed in 
\index{Rubinstein-Scharlemann!graphic}Section~\ref{sec:Rubinstein-Scharlemann}. 
First, one perturbs $f$ to be in ``general position,'' as defined in
Section~\ref{sec:generalposition}. The intersections of the $f_u(P_s)$ and
$P_t$ are then sufficiently well-controlled to define a
\index{graphic}\textit{graphic} in the square $\I^2$. That is, the pairs
$(s,t)$ for which $f_u(P_s)$ and $P_t$ do not intersect transversely form a
graph embedded in the square. The complementary regions of this graph in
$\I^2$ are labeled according to a procedure in \cite{RS}, and in
Section~\ref{sec:goodregions} we show that the properties of general
position salvage enough of the combinatorics of these labels developed in
\cite{RS} to deduce that at least one of the complementary regions consists
of pairs in good position.
\longpage

Perhaps the hardest work of the proof, and certainly the part that takes us
furthest from the usual confines of low-dimensional topology, is the
verification that sufficient \index{general position}``general position''
can be achieved. Since we use parameterized families, we must allow
$f_u(P_s)$ and $P_t$ to have large numbers of tangencies, some of which may
be of high order. It turns out that to make the combinatorics of \cite{RS}
go through, we must achieve that at each parameter \textit{there are at
  most finitely many pairs $(s,t)$ where $f_u(P_s)$ and $P_t$ have multiple
  or high-order tangencies} (at least, for pairs not extremely close to the
boundary of the square).  To achieve the necessary degree of general
position, we use results of a number of people, notably J. W. \index{Bruce}Bruce
\cite{Bruce} and F. \index{Sergeraert}Sergeraert \cite{Sergeraert}.

The need for this kind of general position is indicated in
Section~\ref{sec:examples}, where we construct a pair of sweepouts of
$S^2\times S^1$ with all tangencies of Morse type, but having no pair of
levels intersecting in good position. Although we have not constructed a
similar example for an $L(m,q)$, we see no reason why one could not exist.
\newpage

Step~2 (``from good to very good'') is to deform $f$ to eliminate the
discal intersections of $f_u(P_s)$ and $P_t$, for certain pairs in good
position that have been found in Step~1, so that they intersect in very
good position. This is an application of \index{Hatcher}Hatcher's
parameterization methods~\cite{H}. One must be careful here, since an
isotopy that eliminates a discal intersection can also eliminate a
biessential intersection, and if all biessential intersections were
eliminated by the procedure, the resulting pair would no longer be in very
good position.\index{biessential!after basic isotopy}
Lemma~\ref{lem:no biessential elimination} ensures that not
all biessential intersections will be eliminated.

Step~3 (``from very good to fiber-preserving'') is to use the pairs in very
good position to deform the family so that each $f_u$ is fiber-preserving.
This is carried out in Sections~\ref{sec:lemmas}
and~\ref{sec:fiber-preserving families}. The basic idea is first to use the
biessential intersections to deform the $f_u$ so that $f_u(P_s)$ actually
equals $P_t$ (for certain $(s,t)$ pairs that originally intersected in good
position), then use known results about the diffeomorphism groups of
surfaces and Haken $3$-manifolds to make the $f_u$ fiber-preserving on
$P_s$ and then on its complementary solid tori. This process is technically
complicated for two reasons. First, although a biessential intersection is
essential in both tori, it can be contractible in one of the complementary
solid tori of $P_t$, and $f_u(P_s)$ can meet that complementary solid torus
in annuli that are not parallel into $P_t$. So one may be able to push the
annuli out from only one side of $P_t$. Secondly, the fitting together of
these isotopies requires one to work with not just one level but many
levels at a single parameter.

Two natural questions are whether 
\index{Bonahon!method for $\pi_0(\protect\Diff(L))$}Bonahon's original 
method for determining
the mapping class group $\pi_0(\Diff(L))$ \cite{Bonahon} can be adapted to
the parameterized setting, and whether our methodology can be used to
recover his results. Concerning the first question, we have had no success
with this approach, as we see no way to perturb the family to the point
where the method can be started at each parameter. For the second, the
answer is yes. In fact, the key geometric step of \cite{Bonahon} is the
isotopy uniqueness of genus-one Heegaard surfaces in $L$, which was already
reproven in Rubinstein and \index{Scharlemann}Scharlemann's original
work~\cite[Corollary~6.3]{RS}.
\longpage\longpage

\section[Reductions]
{Reductions}
\label{sec:reduction}

In this section, we carry out some initial reductions. The Conjecture
will be reduced to a purely topological problem of deforming
parameterized families of diffeomorphisms to families of
diffeomorphisms that preserve a certain Seifert fibering of~$L$.

By Theorem~\ref{thm:pi0 Smale}, it is sufficient to prove that
$\isom(L)\to\diff(L)$ is a homotopy equivalence. And we have seen that this
follows once we prove that $\isom(L)\to\diff(L)$ is a homotopy
equivalence.

Section~1.4 of \cite{M} gives a certain way to embed $\pi_1(L)$ into
$\SO(4)$ so that its action on $S^3$ is fiber-preserving for the
fibers of the Hopf bundle structure of $S^3$. Consequently, this bundle
structure descends to a Seifert fibering of $L$, which we call the
\index{Hopf fibering!of lens space}\textit{Hopf fibering} of $L$. 
If $q=1$, this Hopf fibering is actually an $S^1$-bundle structure, 
while if $q>1$, it has two exceptional fibers with
invariants of the form $(k,q_1)$, $(k,q_2)$ where $k=m/\gcd(q-1,m)$ (see
Table~4 of \cite{M}). We will always use the Hopf fibering as the
Seifert-fibered structure of $L$.\longpage

Theorem~2.1 of \cite{M} shows that (since $m>2$) every
orientation-preserving isometry of $L$ preserves the Hopf fibering on
$L$. In particular, $\isom(L)\subset \diff_f(L)$, so there are inclusions
\[\isom(L)\to \diff_f(L)\to \diff(L)\ .\]

\begin{theorem}
The inclusion $\isom(L)\to \diff_f(L)$ is a homotopy equivalence.
\label{thm:reduce to fiber-preserving}
\end{theorem}

\begin{proof}
The argument is similar to the latter part of the proof of 
Theorem~\ref{smale1}, so we only give a sketch. There is a diagram
\begin{equation*}
\begin{CD}
S^1 @>>> \isom(L) @>>>  \isom(L_0)\\ 
@VVV @VVV @VVV { }\\
\vertical(L) @>>>  \diff_f(L) @>>>  \diff_{orb}(L_0)
\end{CD}
\end{equation*}
\noindent
where $L_0$ is the quotient orbifold and $\diff_{orb}(L_0)$ is the group of
orbifold diffeomorphisms of $L_0$, and $\vertical(L)$ is the group of
vertical diffeomorphisms. The first row is a fibration, in fact an
$S^1$-bundle, and the second row is a fibration by Theorem~\ref{sfproject
diffs}. The vertical arrows are inclusions. When $q=1$, $L_0$ is the
$2$-sphere and the right-hand vertical arrow is the inclusion of $\SO(3)$
into $\diff(S^2)$, which is a homotopy equivalence by \cite{Smale}. When
$q>1$, $L_0$ is a $2$-sphere with two cone points, $\isom(L_0)$ is
homeomorphic to $S^1$, and $\diff_{orb}(L_0)$ is essentially the connected
component of the identity in the diffeomorphism group of the annulus. Again
the right-hand vertical arrow is a homotopy equivalence. The left-hand
vertical arrow is a homotopy equivalence in both cases, so the middle
arrow is as well.
\end{proof}

Theorem~\ref{thm:reduce to fiber-preserving} reduces the Smale Conjecture
for Lens Spaces to proving that the inclusion $\diff_f(L)\to \diff(L)$ is a
homotopy equivalence. For this it is sufficient to prove that for all
$d\geq 1$, any map $f\colon (D^d,S^{d-1})\to (\diff(L),\diff_f(L))$ is
homotopic, through maps taking $S^{d-1}$ to $\diff_f(L)$, to a map from
$D^d$ into $\diff_f(L)$. To simplify the exposition, we work
until the final section with a map $f\colon S^d\to \diff(L)$ and show
that it is homotopic to a map into $\diff_f(L)$. In the final section, we
give a trick that enables the entire procedure to be adapted to maps
$f\colon (D^d,S^{d-1})\to (\diff(L),\diff_f(L))$, completing the proof.

\section[Annuli in solid tori]
{Annuli in solid tori}
\label{sec:preliminaries}

Annuli in solid tori will appear frequently in our work. Incompressible
annuli present little difficulty, but we will also need to examine
compressible annuli, whose behavior is more complicated. In this section,
we provide some basic definitions and lemmas.

A loop $\alpha$ in a solid torus $V$ is called a \indexdef{longitude!in
  solid torus}\textit{longitude} if its homotopy class is a generator of the
infinite cyclic group $\pi_1(V)$.  If in addition there is a product
structure $V=S^1\times D^2$ for which $\alpha=S^1\times\set{0}$, then
$\alpha$ is called a \indexdef{core region}\textit{core circle} of $V$. A
subset of a solid torus $V$ is called a \indexdef{core region}\textit{core
  region} when it contains a core circle of $V$. An embedded circle in
$\partial V$ which is essential in $\partial V$ and contractible in $V$ is
called a \indexdef{meridian!of solid torus}\textit{meridian} of $V$; a
properly embedded disk in $V$ whose boundary is a meridian is called a
\indexdef{meridian disk}\textit{meridian disk} of~$V$.

Annuli in solid tori will always be assumed to be properly embedded, which
for us includes the property of being transverse to the boundary, unless
they are actually contained in the boundary. The next three results are
elementary topological facts, and we do not include proofs.
\begin{proposition}\label{prop:boundary-parallel annuli}
Let $A$ be a boundary-parallel annulus in a solid torus $V$, which
separates $V$ into $V_0$ and $V_1$, and for $i=0,1,$ let $A_i=V_i\cap
\partial V$.  Then $A$ is parallel to $A_i$ if and only if $V_{1-i}$ is a
core region.\par
\end{proposition}
\noindent 

\begin{proposition} 
Let $A$ be a properly embedded annulus in a solid torus $V$, which
separates $V$ into $V_0$ and $V_1$, and let $A_i=V_i\cap \partial V$. The
following are equivalent:
\begin{enumerate}
\item
$A$ contains a longitude of $V$.
\item
$A$ contains a core circle of $V$.
\item
$A$ is parallel to both $A_0$ and $A_1$.
\item
Both $V_0$ and $V_1$ contain longitudes of $V$.
\item
Both $V_0$ and $V_1$ are core regions of $V$.
\end{enumerate}
\label{prop:longitudinal annuli}
\end{proposition}
\noindent An annulus satisfying the conditions in
Proposition~\ref{prop:longitudinal annuli} is said to be
\indexdef{longitudinal!annulus}\textit{longitudinal.} A 
longitudinal annulus must be incompressible.

\begin{proposition} Let $V$ be a solid torus and let $\cup A_i$ be a union 
of disjoint boundary-parallel annuli in $V$. Let $C$ be a core circle of
$V$ that is disjoint from $\cup A_i$. For each $A_i$, let $V_i$ be the
closure of the complementary component of $A_i$ that does not contain $C$,
and let $B_i=V_i\cap \partial V$. Then $A_i$ is parallel to $B_i$.
Furthermore, either
\begin{enumerate}
\item
no $A_i$ is longitudinal, and exactly one component of $V-\cup A_i$
is a core region, or
\item
every $A_i$ is longitudinal, and every component of $V-\cup A_i$ is a core
region.
\end{enumerate}
\label{prop:core circles}
\end{proposition}

There are various kinds of compressible annuli in solid tori. For example,
there are boundaries of regular neighborhoods of properly embedded arcs,
possibly knotted. Also, there are annuli with one boundary circle a
meridian and the other a contractible circle in the boundary torus. When
both boundary circles are meridians, we call the annulus 
\indexdef{meridional!annulus}\textit{meridional.}
As shown in Figure~\ref{fig:meridional annuli}, meridional annuli are not
necessarily boundary-parallel.
\index{figures!figure7@meridional annuli in a solid torus}%
\begin{figure}
\includegraphics[width=7cm]{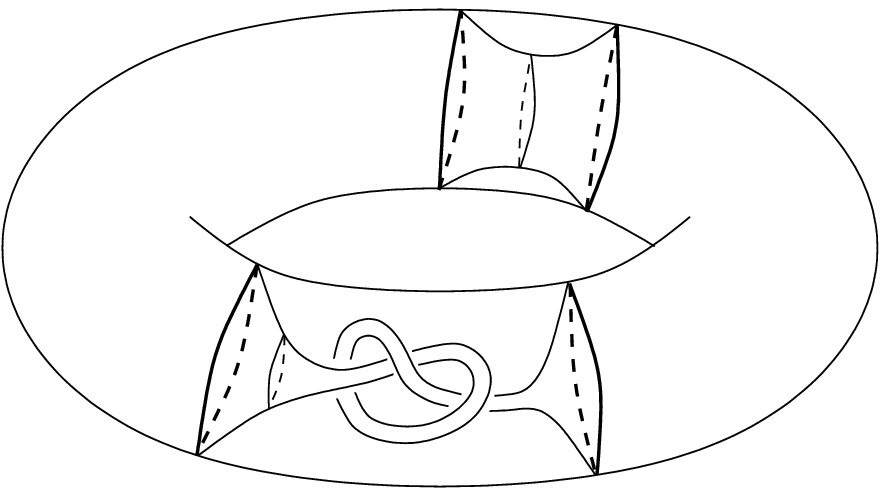}
\caption{\index{meridional!annulus!figure}Meridional annuli in a solid torus.}
\label{fig:meridional annuli}
\end{figure}\index{examples!example2@meridional annuli in a solid torus}

Although meridional annuli need not be boundary-parallel, they behave
homologically as though they were, and as a consequence any family of
meridional annuli misses some longitude.
\begin{lemma}
Let $A_1,\ldots\,$, $A_n$ be disjoint meridional annuli in a solid torus
$V$. Then:
\begin{enumerate}
\item
Each $A_i$ separates $V$ into two components, $V_{i,0}$ and $V_{i,1}$, for
which $A_i$ is incompressible in $V_{i,0}$ and compressible in $V_{i,1}$.
\item $V_{i,1}$ contains a meridian disk of $V$.
\item $\pi_1(V_{i,0})\to \pi_1(V)$ is the zero homomorphism.
\item The intersection of the $V_{i,1}$ is the unique
component of the complement of $\cup A_i$ that contains a longitude of $V$.
\end{enumerate}
\label{lem:meridional annuli}
\end{lemma}

\begin{proof}
For each $i$, every loop in $V$ has even algebraic intersection with $A_i$,
since every loop in $\partial V$ does, so $A_i$ separates $V$. Since $A_i$
is not incompressible, it must be compressible in one of its complementary
components, $V_{i,1}$, and since $V$ is irreducible, $A_i$ must be
incompressible in the other complementary component, $V_{i,0}$.

Notice that $V_{i,1}$ must contain a meridian disk of $V$. Indeed, if $K$ is
the union of $A_i$ with a compressing disk in $V_{i,1}$, then two of the
components of the frontier of a regular neighborhood of $K$ in $V$ are
meridian disks of $V_{i,1}$. Consequently, $\pi_1(V_{i,0})\to \pi_1(V)$ is
the zero homomorphism. The Mayer-Vietoris sequence shows that $H_1(A_i)\to
H_1(V_{i,0})$ and $H_1(V_{i,1})\to H_1(V)$ are isomorphisms.

Let $V_1$ be the intersection of the $V_{i,1}$, and let $V_0$ be the union
of the $V_{i,0}$. The Mayer-Vietoris sequence shows that $V_1$ is
connected, and that $H_1(V_1)\to H_1(V)$ is an isomorphism, so $V_1$
contains a longitude of $V$. For any $i,j$, either $V_{i,1}\subseteq
V_{j,1}$ or $V_{j,1}\subseteq V_{i,1}$, since otherwise $H_1(V_{i,1})\to
H_1(V_{j,0})\to H_1(V)$ would be the zero homomorphism. Therefore the
intersection $V_1=\cap V_{i,1}$ is equal to some $V_{k,1}$, and in
particular it contains a longitude of $V$. No other complementary
component of $\cup A_i$ contains a longitude, since each such component
lies in $V_{k,0}$, all of whose loops are contractible in~$V$.
\end{proof}

\section[Heegaard tori in very good position]
{Heegaard tori in very good position}
\label{sec:very good position}

A \indexdef{Heegaard torus}\textit{Heegaard torus} in a lens space $L$ is a
torus that separates $L$ into two solid tori.  In this section we will
develop some properties of Heegaard tori.  Also, we introduce the concepts
of discal and biessential intersection circles, good position, and very
good position, which will be used extensively in later sections.

When $P$ is a Heegaard torus bounding solid tori $V$ and $W$, and $Q$ is a
Heegaard torus contained in the interior of $V$, $Q$ need not be parallel
to $\partial V$. For example, start with a core circle in $V$, move a small
portion of it to $\partial V$, then pass it across a meridian disk of $W$
and back into $V$.  This moves the core circle to its band-connected sum in
$V$ with an $(m,q)$-curve in $\partial V$. By varying the choice of band---
for example, by twisting it or tying knots in it--- and by iterating this
construction, one can construct complicated knotted circles in $V$ which
are isotopic in $L$ to a core circle of $V$. The boundary of a regular
neighborhood of such a circle is a Heegaard torus of $L$. But here is one
restriction on Heegaard tori

\begin{proposition}
Let $P$ be a Heegaard torus in a lens space $L$, bounding solid tori $V$
and $W$. If a loop $\ell$ embedded in $P$ is a core circle for a solid
torus of some genus-$1$ Heegaard splitting of $L$, then $\ell$ is a
longitude for either $V$ or $W$.
\label{prop:Heegaard loops}
\end{proposition}

\begin{proof}
Since $L$ is not simply-connected, $\ell$ is not a meridian for either $V$
or $W$, consequently $\pi_1(\ell)\to \pi_1(V)$ and $\pi_1(\ell)\to
\pi_1(W)$ are injective. So $P-\ell$ is an open annulus separating
$L-\ell$, making $\pi_1(L-\ell)$ a free product with amalgamation
$\Z*_{\Z}\Z$. Since $\ell$ is a core circle, $\pi_1(L-\ell)$ is infinite
cyclic, so at least one of the inclusions of the amalgamating subgroup to
the infinite cyclic factors is surjective.
\end{proof}

Let $F_1$ and $F_2$ be transversely intersecting embedded surfaces in the
interior of a $3$-manifold $M$. A component of $F_1\cap F_2$ is called
\indexdef{discal}\textit{discal} when it is contractible in both $F_1$ and
$F_2$, and \indexdef{biessential}\textit{biessential} when it is essential in
both. We say that $F_1$ and $F_2$ are \indexdef{good position}\textit{in good
position} when every component of their intersection is either discal or
biessential, and at least one is biessential, and we say that they are
\indexdef{very good!position}\textit{in very good position} when they are in
good position and every component of their intersection is biessential.
\longpage\longpage

Later, we will go to considerable effort to obtain pairs of Heegaard tori
for lens spaces that intersect in very good position. Even then, the
configuration can be complicated. Consider a Heegaard torus $P$ bounding
solid tori $V$ and $W$, and another Heegaard torus $Q$ that meets $P$ in
very good position. When the intersection circles are not meridians for
either $V$ or $W$, the components of $Q\cap V$ and $Q\cap W$ are annuli
that are incompressible in $V$ and $W$, and must be as described in
Proposition~\ref{prop:core circles}. But if the intersection circles are
meridians for one of the solid tori, say $V$, then $Q\cap V$ consists of
meridional annuli, and as shown in Figure~\ref{fig:bad torus}, they need
not be boundary-parallel. To obtain that configuration, one starts with a
torus $Q$ parallel to $P$ and outside $P$, and changes $Q$ by an isotopy
that moves a meridian $c$ of $Q$ in a regular neighborhood of a meridian
disk of $P$. First, $c$ passes across a meridian in $P$, then shrinks down
to a small circle which traces around a knot. Then, it expands out to
another meridian in $P$ and pushes across. The resulting torus meets $P$ in
four circles which are meridians for $V$, and meets $V$ in two annuli, both
isotopic to the non-boundary-parallel annulus in Figure~\ref{fig:meridional
annuli}. The next lemma gives a small but important restriction on
meridional annuli of $Q\cap V$.
\index{figures!figure8@Heegaard torus with bad meridional annuli}%
\begin{figure}
\labellist
\pinlabel $P$ [B] at 239 125
\pinlabel $Q$ [B] at 140 78
\endlabellist
\includegraphics[width=0.55\textwidth]{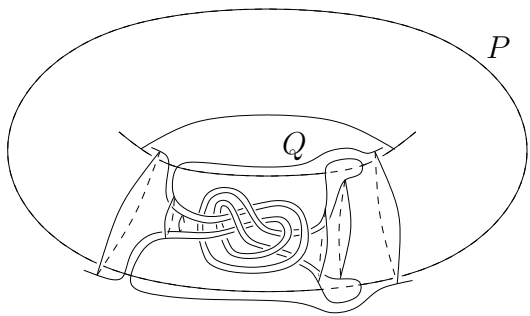}
\caption{Heegaard tori in very good
position with non-boundary-parallel meridional annuli.}
\label{fig:bad torus}
\end{figure}

\begin{lemma}
Let $P$ be a Heegaard torus which separates a lens space into two solid
tori $V$ and $W$.  Let $Q$ be another Heegaard torus whose intersection
with $V$ consists of a single meridional annulus $A$. Then $A$ is
boundary-parallel in $V$.
\label{lem:one annulus}
\end{lemma}

\begin{proof}
From Lemma~\ref{lem:meridional annuli}, $A$ separates $V$ into two
components $V_0$ and $V_1$, such that $A$ is compressible in $V_1$ and
$V_1$ contains a longitude of $V$. Suppose that $A$ is not
boundary-parallel in~$V$.

Let $A_0=V_0\cap \partial V$. Of the two solid tori in $L$ bounded by $Q$,
let $X$ be the one that contains $A_0$, and let $Y$ be the other one. Since
$Q\cap V$ consists only of $A$, $Y$ contains $V_1$, and in particular
contains a compressing disk for $A$ in $V_1$ and a longitude for $V$.

Suppose that $A_0$ were incompressible in $X$. Since $A_0$ is not
parallel to $A$, it would be parallel to $\overline{\partial
X-A}$. So $V_0$ would contain a core circle of $X$. Since $\pi_1(V_0)\to
\pi_1(V)$ is the zero homomorphism, this implies that $L$ is
simply-connected, a contradiction. So $A_0$ is compressible in $X$. A
compressing disk for $A_0$ in $X$ is part of a $2$-sphere that meets $Y$
only in a compressing disk of $A$ in $V_1$. This $2$-sphere has algebraic
intersection $\pm 1$ with the longitude of $V$ in $V_1$,
contradicting the irreducibility of $L$.
\end{proof}

Regarding $D^2$ as the unit disk in the plane, for $0<r<1$ let 
\indexsymdef{rD2}{$rD^2$}$rD^2$
denote $\{(x,y)\;|\;x^2+y^2\leq r^2\}$.  A solid torus $X$ embedded in a
solid torus $V$ is called 
\indexdef{concentric solid torus}\textit{concentric in $V$} if there is some product
structure $V=D^2\times S^1$ such that $X=rD^2\times S^1$. Equivalently, $X$
is in the interior of $V$ and some (hence every) core circle of $X$ is a
core circle of $V$.

The next lemma shows how we will use Heegaard tori that meet in very good
position.
\begin{lemma}
Let $P$ be a Heegaard torus which separates a lens space into two solid
tori $V$ and $W$.  Let $Q$ be another Heegaard torus, that meets $P$ in
very good position, and assume that the annuli of $Q\cap V$ are
incompressible in $V$. Then at least one component $C$ of $V-(Q\cap V)$
satisfies both of the following:
\begin{enumerate}
\item
$C$ is a core region for $V$.
\item
Suppose that $Q$ is moved by isotopy to a torus $Q_1$ in $W$, by pushing
the annuli of $Q\cap V$ one-by-one out of $V$ using isotopies that move
them across regions of $V-C$, and let $X$ be the solid torus bounded by
$Q_1$ that contains $V$. Then $V$ is concentric in $X$.
\item
After all but one of the annuli have been pushed out of $V$, the
image $Q_0$ of $Q$ is isotopic to $P$ relative to $Q_0\cap P$.
\end{enumerate}
\label{lem:pushout}
\end{lemma}

\begin{proof}
Assume first that $Q\cap V$ has only one component $A$. Then $\partial A$
separates $P$ into two annuli, $A_1$ and $A_2$. Since $A$ is incompressible
in $V$, it is parallel in $V$ to at least one of the $A_i$, say $A_1$. Let
$A'=Q\cap W$.

If $A'$ is longitudinal, then $A'$ is parallel in $W$ to $A_2$. So pushing
$A$ across $A_1$ moves $Q$ to a torus in $W$ parallel to $P$, and the lemma
holds, with $C$ being the region between $A$ and $A_2$. An isotopy from $Q$
to $P$ can be carried out relative to $Q\cap P$, giving the last statement
of the lemma. Suppose that $A'$ is not longitudinal. If $A'$ is
incompressible, then it is boundary parallel in $W$. If $A'$ is not
incompressible, then since $P$ and $Q$ meet in very good position, $A'$ is
meridional, and by Lemma~\ref{lem:one annulus} it is again
boundary-parallel in $W$. If $A'$ is parallel to $A_2$, then we are
finished as before. If $A'$ is parallel to $A_1$, but not to $A_2$, then
there is an isotopy moving $Q$ to a regular neighborhood of a core circle
of $A_1$. By Proposition~\ref{prop:Heegaard loops}, $A$ is longitudinal, so
must also be parallel in $V$ to $A_2$. In this case, we take $C$ to be the
region between $A$ and~$A_1$.

Suppose now that $Q\cap V$ and hence also $Q\cap W$ consist of $n$ annuli,
where $n>1$. By isotopies pushing outermost annuli in $V$ across $P$, we
obtain $Q_0$ with $Q_0\cap V$ consisting of one annulus $A$. At least one
of its complementary components, call it $C$, satisfies the lemma. Let $Z$
be the union of the regions across which the $n-1$ annuli were
pushed. Since $C$ is a core region, $C\cap (V-Z)$ is also a core region
(since a core circle of $V$ in $C$ can be moved, by the reverse of the
pushout isotopies, to a core circle of $V$ in $C\cap (V-Z)$).  So $C\cap
(V-Z)$ satisfies the conclusion of the lemma.
\end{proof}

Here is a first consequence of Lemma~\ref{lem:pushout}.
\begin{corollary}
Let $P$ be a Heegaard torus which separates a lens space into two solid
tori $V$ and $W$, and let $Q$ be another Heegaard torus separating it into
$X$ and $Y$. Assume that $Q$ meets $P$ in very good position. If the
circles of $P\cap Q$ are meridians in $X$ or in $Y$ (respectively, in $X$
and in $Y$), then they are meridians in $V$ or in~$W$ (respectively, in $V$
and in $W$). An analogous assertion holds for longitudes. \par
\label{coro:comeridian}
\end{corollary}

\begin{proof}
We may choose notation so that the annuli of $Q\cap V$ are incompressible
in $V$.  Use Lemma~\ref{lem:pushout} to move $Q$ out of $V$. After all but
one annulus has been pushed out, the image $Q_0$ of $Q$ is isotopic to $P$
relative to $Q_0\cap P$. That is, the original $Q$ is isotopic to $P$ by an
isotopy relative to $Q_0\cap P$. If the circles of $Q\cap P$ were
originally meridians of $X$ or $Y$, then in particular those of $Q_0\cap P$
are meridians of $X$ or $Y$ after the isotopy, that is, of $V$ or $W$.  The
``and'' assertion and the case of longitudes are similar.
\end{proof}

\newpage
\section[Sweepouts, and levels in very good positions]
{Sweepouts, and levels in very good position}
\label{sec:RSgraphic}

In this section we will define sweepouts and related structures. Also, we
will prove an important technical lemma concerning pairs of sweepouts
having levels that meet in very good position.
\longpage\longpage

By a \indexdef{sweepout}\textit{sweepout} of a closed orientable
$3$-manifold, we mean a smooth map $\tau\colon P\times [0,1]\to M$, where
$P$ is a closed orientable surface, such that
\begin{enumerate}
\item
$T_0=\tau(P\times\set{0})$ and $T_1=\tau(P\times \set{1})$ are disjoint
graphs with each vertex of valence $3$.
\item 
Each $T_i$ is a union of a collection of smoothly embedded arcs and circles
in $M$.
\item
$\tau\vert_{P\times(0,1)}\colon P\times (0,1)\to M$ is a diffeomorphism
onto $M-(T_0\cup T_1)$.
\item
Near $P\times\partial I$, $\tau$ gives a mapping cylinder neighborhood of
$T_0\cup T_1$.
\end{enumerate}
Associated to any $t$ with $0<t<1$, there is a Heegaard splitting
$M=V_t\cup W_t$, where 
\indexsymdef{Vt}{$V_t$, $W_t$}$V_t=\tau(P\times [0,t])$ and 
$W_t=\tau(P\times
[t,1])$. For each $t$, $T_0$ is a deformation retract of $V_t$ and $T_1$ is
a deformation retract of $W_t$.  We denote $\tau(P\times \set{t})$ by
$P_t$, and call it a \indexdef{level!of sweepout}\textit{level} of
$\tau$. Also, for $0<s< t<1$ we denote $\tau(P\times [s,t])$ by
\indexsymdef{R(s,t)}{$R(s,t)$}$R(s,t)$. Note that any genus-$1$ Heegaard
splitting of $L$ provides sweepouts with $T_0$ and $T_1$ as core circles of
the two solid tori, and the Heegaard torus as one of the levels.

A sweepout $\tau\colon P\times [0,1]\to M$ induces a continuous projection
function $\pi\colon M\to [0,1]$ by the rule $\pi(\tau(x,t))=t$.  By
composing this with a smooth bijection from $[0,1]$ to $[0,1]$ all of whose
derivatives vanish at $0$ and at $1$, we may reparameterize $\tau$ to
ensure that $\pi$ is a smooth map. We always assume that $\tau$ has been
selected to have this property.

By a \indexdef{spine}\textit{spine} for a closed connected surface $P$, we
mean a $1$-dimensional cell complex in $P$ whose complement consists of
open disks.

The next lemma gives very strong restrictions on levels of two different
sweepouts of a lens space that intersect in very good position.
\begin{lemma}
Let $L$ be a lens space. Let $\tau \colon T\times[0,1]\to L$ be a sweepout
as above, where $T$ is a torus. Let $\sigma\colon T\times[0,1]\to L$ be
another sweepout, with levels $Q_s=\sigma(T\times\set{s})$.  Suppose that
for $t_1<t_2$, $s_1\neq s_2$, and $i=1,2$, $Q_{s_i}$ and $P_{t_i}$
intersect in very good position, and that $Q_{s_1}$ has no discal
intersections with $P_{t_2}$. If $Q_{s_1}$ has nonempty intersection with
$P_{t_2}$, then either
\begin{enumerate}
\item
every intersection circle of $Q_{s_1}$ with $P_{t_2}$ is biessential, and
consequently $Q_{s_1}\cap R(t_1,t_2)$ contains an annulus with one boundary
circle essential in $P_{t_1}$ and the other essential in $P_{t_2}$, or
\item
for $i=1,2$, $Q_{s_i}\cap P_{t_i}$ consists of meridians of $W_{t_i}$, and
$Q_{s_1}\cap R(t_1,t_2)$ contains a surface $\Sigma$ which is a homology
from a circle of $Q_{s_1}\cap P_{t_1}$ to a union of circles in $P_{t_2}$.
\end{enumerate}
\label{lem:hitting levels}
\end{lemma}
\noindent Figure~\ref{fig:sigma} illustrates case (2) of
Lemma~\ref{lem:hitting levels}.

We mention that to apply Lemma~\ref{lem:hitting levels} when $t_1>t_2$, we
interchange the roles of $V_{t_i}$ and $W_{t_i}$. The intersection circles
in case (2) are then meridians of the $V_{t_i}$ rather than the $W_{t_i}$.
\index{figures!figure9@complicated configuration of intersecting levels}%
\begin{figure}
\labellist
\pinlabel $P_{t_1}$ at -20 5
\pinlabel $P_{t_2}$ at -20 130
\pinlabel $Q_{s_1}$ at 370 230
\pinlabel $\Sigma$ [B] at 147 56
\endlabellist
\includegraphics[width=7cm]{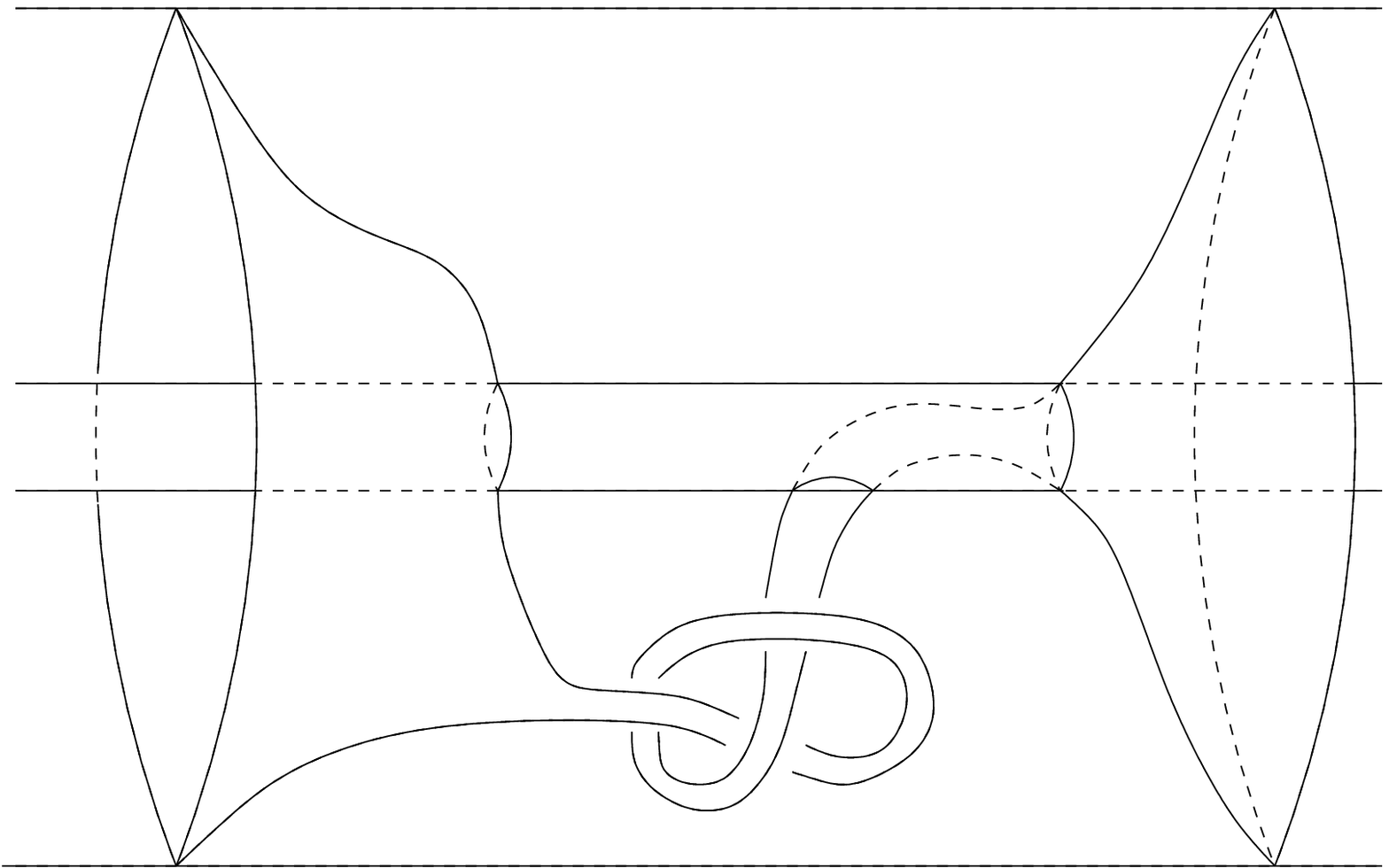}
\caption{Case (2) of Lemma \ref{lem:hitting levels}}
\label{fig:sigma}
\end{figure}

\begin{proof}[Proof of Lemma~\ref{lem:hitting levels}]
Assume for now that the circles of $Q_{s_2}\cap P_{t_2}$ are not meridians
of $W_{t_2}$.

We first rule out the possibility that there exists a circle of
$Q_{s_1}\cap P_{t_2}$ that is inessential in $Q_{s_1}$. If so, there would
be a circle $C$ of $Q_{s_1}\cap P_{t_2}$, bounding a disk $D$ in $Q_{s_1}$
with interior disjoint from $P_{t_2}$. Since $Q_{s_1}$ and $P_{t_2}$ have
no discal intersections, $C$ is essential in $P_{t_2}$, so $D$ is a
meridian disk for $V_{t_2}$ or $W_{t_2}$. It cannot be a meridian disk for
$V_{t_2}$, for then some circle of $D\cap P_{t_1}$ would be a meridian of
$V_{t_1}$, contradicting the fact that $Q_{s_1}$ and $P_{t_1}$ meet in very
good position. But $D$ cannot be a meridian disk for $W_{t_2}$, since $D$
is disjoint from $Q_{s_2}$ and the circles of $Q_{s_2}\cap P_{t_2}$ are not
meridians of~$W_{t_2}$.

We now rule out the possibility that there exists a circle of $Q_{s_1}\cap
P_{t_2}$ that is essential in $Q_{s_1}$ and inessential in $P_{t_2}$.
There is at least one biessential intersection circle of $Q_{s_1}$ with
$P_{t_1}$, hence also an annulus $A$ in $Q_{s_1}$ with one boundary circle
inessential in $P_{t_2}$ and the other essential in either $P_{t_1}$ or
$P_{t_2}$, with no intersection circle of the interior of $A$ with
$P_{t_1}\cup P_{t_2}$ essential in $A$.  The interior of $A$ must be
disjoint from $P_{t_1}$, since $Q_{s_1}$ meets $P_{t_1}$ in very good
position. It must also be disjoint from $P_{t_2}$, by the previous
paragraph. So, since $A$ has at least one boundary circle in $P_{t_2}$, it
is properly embedded either in $R(t_1,t_2)$ or in $W_{t_2}$. It cannot be
in $R(t_1,t_2)$, since it has one boundary circle inessential in $P_{t_2}$
and the other essential in $P_{t_1}\cup P_{t_2}$. So $A$ is in $W_{t_2}$,
and since one boundary circle is inessential in $P_{t_2}$, the other must
be a meridian, contradicting the assumption that no circle of $Q_{s_2}\cap
P_{t_2}$ is a meridian of $W_{t_2}$. Thus conclusion~(1) holds when circles
of $Q_{s_2}\cap P_{t_2}$ are not meridians of~$W_{t_2}$.

Assume now that the circles of $Q_{s_2}\cap P_{t_2}$ are meridians
of~$W_{t_2}$. We will achieve conclusion~(2).

Suppose first that some circle of $Q_{s_1}\cap P_{t_2}$ is essential in
$Q_{s_1}$. Then there is an annulus $A$ in $Q_{s_1}$ with one boundary
circle essential in $P_{t_1}$, the other essential in $P_{t_2}$, and all
intersections of the interior of $A$ with $P_{t_1}\cup P_{t_2}$ inessential
in $A$. Since $Q_{s_1}$ meets $P_{t_1}$ in very good position, the interior
of $A$ must be disjoint from $P_{t_1}$. So $A\cap R(t_1,t_2)$ contains a
planar surface $\Sigma$ with one boundary component a circle of
$Q_{s_1}\cap P_{t_1}$ and the other boundary components circles in
$P_{t_2}$ which are meridians in $W_{t_2}$, giving the conclusion~(2) of
the lemma.

Suppose now that every circle of $Q_{s_1}\cap P_{t_2}$ is contractible in
$Q_{s_1}$. We will show that this case is impossible. An intersection
circle innermost on $Q_{s_1}$ bounds a disk $D$ in $Q_{s_1}$ which is a
meridian disk for $W_{t_2}$, since $\partial D$ is essential in $P_{t_2}$
and disjoint from $Q_{s_2}\cap P_{t_2}$.  Now, use Lemma~\ref{lem:pushout}
to push $Q_{s_2}\cap V_{t_2}$ out of $V_{t_2}$ by an ambient isotopy of
$L$.  Suppose for contradiction that one of these pushouts, say, pushing an
annulus $A_0$ in $Q_{s_2}$ across an annulus in $P_{t_2}$, also eliminates a
circle of $Q_{s_1}\cap P_{t_1}$. Let $Z$ be the region of parallelism
across which $A_0$ is pushed. Since $Z$ contains an essential loop of
$Q_{s_1}$, and each circle of $Q_{s_1}\cap P_{t_2}$ is contractible in
$Q_{s_1}$, $Z$ contains a spine of $Q_{s_1}$. This spine is isotopic in $Z$
into a neighborhood of a boundary circle of $A_0$. Since this boundary
circle is a meridian of $W_{t_2}$, every circle in the spine is
contractible in $L$. This contradicts the fact that $Q_{s_1}$ is a Heegaard
torus. So the pushouts do not eliminate intersections of $Q_{s_1}$ with
$P_{t_1}$, and after the pushouts are completed, the image of $Q_{s_1}$
still meets $P_{t_1}$.

During the pushouts, some of the intersection circles of $Q_{s_1}$ with
$P_{t_2}$ may disappear, but not all of them, since the pushouts only move
points into $W_{t_2}$. So after the pushouts, there is a circle of
$Q_{s_1}\cap P_{t_2}$ that bounds a innermost disk in $Q_{s_1}$ (since all
the original intersection circles of $Q_{s_1}$ with $P_{t_2}$ bound disks
in $Q_{s_1}$, and the new intersection circles are a subset of the old
ones). Since the boundary of this disk is a meridian of $W_{t_2}$, the disk
it bounds in $Q_{s_1}$ must be a meridian disk of $W_{t_2}$. The image of
$Q_{s_2}$ lies in $W_{t_2}$ and misses this meridian disk, contradicting
the fact that $Q_{s_2}$ is a Heegaard torus.
\end{proof}

\newpage
\section[The Rubinstein-Scharlemann graphic]
{The Rubinstein-Scharlemann graphic}
\label{sec:Rubinstein-Scharlemann}

\index{Rubinstein-Scharlemann!method|(}The
purpose of this section is to present a number of definitions, and to
sketch the proof of Theorem~\ref{thm:RS} below, originally from~\cite{RS}.
It requires the hypothesis that two sweepouts meet in general position in a
strong sense that we call Morse general position. In
Section~\ref{sec:goodregions}, this proof will be adapted to the weaker
concept of general position developed in Section~\ref{sec:generalposition}.

Consider a smooth function $f\colon (\R^2,0)\to (\R,0)$. A critical point
of $f$ is 
\indexdef{stable!critical point}\indexdef{critical point!stable}\textit{stable} 
when it is locally equivalent under smooth
change of coordinates of the domain and range to $f(x,y)=x^2+y^2$ or
$f(x,y)=x^2-y^2$. The first type is called a 
\indexdef{center!critical point type}%
\indexdef{critical point!center type}\textit{center,} and the second a 
\indexdef{saddle!critical point type}%
\indexdef{critical point!saddle type}\textit{saddle.} An unstable critical
point is called a 
\indexdef{birth-death point}%
\indexdef{critical point!birth-death type}%
\textit{birth-death} point if it is locally $f(x,y)=x^2+y^3$.

Let $\tau\colon P\times [0,1]\to M$ be a sweepout as in
Section~\ref{sec:RSgraphic}. As in that section, we denote
$\tau(P\times\set{0,1})$ by $T$, $\tau(P\times\set{t})$ by $P_t$,
$\tau(P\times[0,t])$ by $V_t$, and $\tau(P\times[t,1])$ by $W_t$.  For
a second sweepout $\sigma\colon Q\times [0,1]\to M$, we denote
$\sigma(Q\times\set{0,1})$ by $S$, $\sigma(Q\times\set{s})$ by $Q_s$,
$\sigma(Q\times[0,s])$ by $X_s$, and $\sigma(Q\times[s,1])$ by $Y_s$.  We
call 
\indexsymdef{Qs}{$Q_s$}$Q_s$ a 
\indexdef{sigmalevel@$\sigma$-level}\textit{$\sigma$-level} and 
\indexsymdef{Pt}{$P_t$}$P_t$ 
a~\indexdef{taulevel@$\tau$-level}\textit{$\tau$-level.}

A tangency of $Q_s$ and $P_t$ at a point $w$ is said to be 
\indexdef{Morse type}\textit{of Morse type} at $w$ if in some local 
$xyz$-coordinates with origin at $w$, $P_t$
is the $xy$-plane and $Q_s$ is the graph of a function which has a stable
critical point or a birth-death point at the origin.  Note that this
condition is symmetric in $Q_s$ and $P_t$. We may refer to a tangency as
stable or unstable, and as a center, saddle, or birth-death point.

A tangency of $S$ with a $\tau$-level is said to be 
\indexdef{stable!tangency}\textit{stable} if there
are local $xyz$-coordinates in which the $\tau$-levels are the planes
$\R^2\times\{z\}$ and $S$ is the graph of $z=x^2$ in the $xz$-plane. In
particular, the tangency is isolated and cannot occur at a vertex of
$S$. There is an analogous definition of stable tangency of $T$ with a
$\sigma$-level.

We will say that $\sigma$ and $\tau$ are in 
\indexdef{Morse general position}\textit{Morse general position}
when the following hold:
\begin{enumerate}
\item
$S$ is disjoint from $T$, 
\item 
all tangencies of $S$ with $\tau$-levels and of $T$ with $\sigma$-levels
are stable,
\item 
all tangencies of $\sigma$-levels with $\tau$-levels are of Morse type, and
only finitely many are birth-death points,
\item
each pair consisting of a $\sigma$-level and a $\tau$-level has at most two
tangencies, and
\item
there are only finitely many pairs consisting of
a $\sigma$-level and a $\tau$-level with two tangencies, and for each of
these pairs both tangencies are stable.
\end{enumerate}

Suppose that $P$ is a Heegaard surface in $M$, bounding a handlebody
$V$. We define a 
\indexdef{precompression}\textit{precompression} or 
\indexdef{precompressing disk}\textit{precompressing disk} for
$P$ in $V$ to be an embedded disk $D$ in $M$ such that
\begin{enumerate}
\item $\partial D$ is an essential loop in $P$, 
\item $D$ meets $P$ transversely at $\partial D$,
and $V$ contains a neighborhood of $\partial D$,
\item the interior of $D$ is transverse to $P$, and its intersections with
$P$ are discal.
\end{enumerate}
Provided that $M$ is irreducible, a precompression for $P$ in $V$ is
isotopic relative to a neighborhood of $\partial D$ to a compressing disk
for $P$ in $V$. In particular, if the Heegaard splitting is strongly
irreducible, then the boundaries of a precompression for $P$ in $V$ and a
precompression for $P$ in $\overline{M-V}$ must intersect.

The following concept due to A. Casson and C. McA.\ Gordon \cite{CG} is a
crucial ingredient in \cite{RS}. A Heegaard splitting $M=V\cup_P W$ is
called \indexdef{strongly irreducible!Heegaard splitting}\textit{strongly irreducible} when every compressing disk for $V$
meets every compressing disk for $W$. A sweepout is called 
\indexdef{strongly irreducible!sweepout}\textit{strongly
irreducible} when the associated Heegaard splittings are strongly
irreducible. We can now state the main technical result of \cite{RS}.
\begin{theorem}[Rubinstein-Scharlemann]
Let $M\neq S^3$ be a closed orientable $3$-manifold, and let
$\sigma,\tau\colon F\times[0,1]\to M$ be strongly irreducible sweepouts of
$M$ which are in Morse general position. Then there exists $(s,t)\in
(0,1)\times (0,1)$ such that $Q_s$ and $P_t$ meet in good position.
\label{thm:RS}
\end{theorem}

We will now review the proof of Theorem~\ref{thm:RS}. The closure in $\I^2$
of the set $(s,t)$ for which $Q_s$ and $P_t$ have a tangency is a graph
$\Gamma$. On $\partial I^2$, it can have valence-$1$ vertices corresponding
to valence-$3$ vertices of $S$ or $T$, and valence-$2$ vertices
corresponding to points of tangency of $S$ with a $\tau$-level or $T$ with
a $\sigma$-level (see p.~1008 of \cite{RS}, see also \cite{KS} for an
exposition with examples). In the interior of $\I^2$, it can have
valence-$4$ vertices which correspond to a pair of levels which have two
stable tangencies, and valence-$2$ vertices which correspond to pairs of
levels having a birth-death tangency.

The components of the complement of $\Gamma$ in the interior of $\I^2$ are
called 
\indexdef{regions!of graphic}\textit{regions.}  
Each region is either unlabeled or bears a label
consisting of up to four letters. The labels are determined by the
following conditions on $Q_s$ and $P_t$, which by transversality hold
either for every $(s,t)$ or for no $(s,t)$ in a region.\index{labels!of regions}
\begin{enumerate}
\item
If $Q_s$ contains a precompression for $P_t$ in $V_t$ (respectively, in
$W_t$), the region receives the letter $A$ (respectively, $B$).
\label{item:AB}
\item
If $P_t$ contains a precompression for $Q_s$ in $X_s$ (respectively, in
$Y_s$), the region receives the letter $X$ (respectively, $Y$).
\label{Item:XY}
\item
If the region has neither an $A$-label nor a $B$-label, and $V_t$
(respectively, $W_t$), contains a spine of $Q_s$, the region receives the
letter $b$ (respectively, $a$).
\item
If the region has neither an $X$-label nor a $Y$-label, and $X_s$
(respectively, $Y_s$), contains a spine of $P_t$, the region receives the
letter $y$ (respectively, $x$).
\end{enumerate}

With these conventions, $Q_s$ and $P_t$ are in good position if and only if
the region containing $(s,t)$ is unlabeled.  To check this, assume first
that they are in good position.  Since all intersections are biessential or
discal, neither surface can contain a precompressing disk for the other,
and since there is a biessential intersection circle, the complement of one
surface cannot contain a spine for the other. For the converse, an
intersection circle which is not biessential or discal leads to a
precompression as in~(1) or~(2), so assume that all intersections are
discal. Then the complement of the intersection circles in $Q_s$ contains
a spine, so the region has either an $a$- or $b$-label, and by the same
reasoning applied to $P_t$ the region has either an $x$- or $y$-label. This
verifies the assertion, as well as the following lemma.

\begin{lemma}
If the label of a region contains the letter $a$ or $b$, then it must also
contain either $x$ or $y$. Similarly, if it contains $x$ or $y$, then it
must also contain $a$ or $b$.
\label{lem:no lone small letters}
\end{lemma}

We call the data consisting of the graph $\Gamma\subset I^2$ and the
labeling of a subset of its regions the 
\indexdef{Rubinstein-Scharlemann!graphic}%
\indexdef{graphic!Rubinstein-Scharlemann}\textit{Rubinstein-Scharlemann
graphic} associated to the sweepouts. Regions of the graphic are called
\indexdef{adjacent regions in graphic}\textit{adjacent} if there is an 
edge of $\Gamma$ which is contained in both of their closures.

At this point, we begin to make use of the fact that the sweepouts are
strongly irreducible. The labels will then have the following properties,
where\index{labels!key properties}\indexdef{RS1@(RS1), (RS2), (RS3)}
\indexsymdef{sA}{$sA$, $sB$, $sX$, $sY$}$\sA$ stands for either of $A$ and $a$, and 
$\sB$, $\sX$, and $\sY$ are defined similarly.  
\newcounter{RScounter}
\begin{list}{(RS\arabic{RScounter})}
{\usecounter{RScounter}
\setlength{\labelwidth}{7 ex}
\setlength{\leftmargin}{10 ex}
\setlength{\labelsep}{2 ex}
\setlength{\parsep}{5pt}
}
\item
A label cannot contain both an $\sA$ and a $\sB$, or both an
$\sX$ and a $\sY$ (direct from the labeling rules and the
definition of strong irreducibility).
\label{item:noAB}
\item
If the label of a region contains $\sA$, then the label of any
adjacent region cannot contain $\sB$. Similarly for $\sX$ and
$\sY$ (Corollary~5.5 of~\cite{RS}).
\label{item:edge labeling}
\item
If all four letters $\sA$, $\sB$, $\sX$, and $\sY$ appear in the labels of
the regions that meet at a valence-$4$ vertex of $\Gamma$, then two
opposite regions must be unlabeled (Lemma~5.7 of \cite{RS}).
\label{item:2-cell labeling}
\end{list}

Property (RS\ref{item:edge labeling}) warrants special comment, since it
will play a major role in our later work. The analysis of labels of
adjacent regions given in Section~5 of \cite{RS} uses only the fact that
for the points $(s,t)$ in an open edge of $\Gamma$, the corresponding $Q_s$
and $P_t$ have a single stable tangency. The open edges of the more general
graphics we will use for the diffeomorphisms in parameterized families in
general position will still have this property, so the labels of their
graphics will still satisfy property~(RS\ref{item:edge labeling}). 
They will not satisfy property~(RS\ref{item:2-cell labeling}), indeed
the $\Gamma$ for their graphics can have vertices of high valence, so
property~(RS\ref{item:2-cell labeling}) will not even be meaningful.

We now analyze the labels of regions whose closures meet $\partial I^2$, as
on p.~1012 of \cite{RS}. Consider first a region whose closure meets the
side $s=0$ (we consider $s$ to be the horizontal coordinate, so this is the
left-hand side of the square).  The region must contains points $(s,t)$
with $s$ arbitrarily close to~$0$. These correspond to $Q_s$ which are
extremely close to $S_0$. For almost all $t$, $S_0$ is transverse to $P_t$,
and for sufficiently small $s$ any intersection of such a $P_t$ with $Q_s$
must be an essential circle of $Q_s$ bounding a disk in $P_t$ that lies in
$X_s$, in which case the region must have an $X$-label. If $P_t$ is
disjoint from $Q_s$, then $P_t$ lies in $Y_s$ so the region has an
$x$-label. That is, all such regions have an $\sX$-label. Similarly, the
label of any region whose closure meets the edge $t=0$ (respectively,
$s=1$, $t=1$) contains $\sA$ (respectively, $\sY$, $\sB$).

We will set up some of the remaining steps a bit differently from those of
\cite{RS}, so that their adaptation to our later arguments will be more
transparent. We have seen that it is sufficient to prove that there exists
an unlabeled region in the graphic defined by the sweepouts. To accomplish
this, Rubinstein and Scharlemann use the shaded subset of the square shown
in Figure~\ref{fig:Diagram}. It is a simplicial complex in which each of
the four triangles is a $2$-simplex. Henceforth we will refer to it as
\indexdef{Diagram}\textit{the Diagram.}
\index{figures!figure91@the Diagram}%
\begin{figure}
\labellist
\pinlabel $\sB\sX$ at -8 154
\pinlabel $\sB$ at 72 154
\pinlabel $\sB\sY$ at 158 154
\pinlabel $\sX$ at -11 74
\pinlabel $\sY$ at 157 74
\pinlabel $\sA\sX$ at -8 -9
\pinlabel $\sA$ at 74 -9
\pinlabel $\sA\sY$ at 158 -9
\endlabellist
\includegraphics[width=3.2cm]{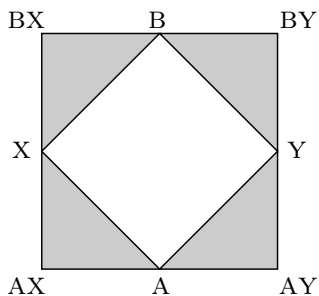}
\caption{The Diagram.}
\label{fig:Diagram}
\end{figure}

Suppose for contradiction that every region in the
Rubin\-stein-Schar\-le\-mann graphic is labeled.  Let $\Delta$ be a
triangulation of $\I^2$ such that each vertex of $\Gamma$ and each corner of
$\I^2$ is a $0$-simplex, and each edge of $\Gamma$ is a union of
$1$-simplices. Let $K$ be $\I^2$ with the structure of a regular $2$-complex
dual to $\Delta$. We observe the following properties 
of~$K$:\indexdef{K1@(K1), (K2), (K3)}
\newcounter{Kcounter}
\begin{list}{(K\arabic{Kcounter})}
{\usecounter{Kcounter}
\setlength{\labelwidth}{7 ex}
\setlength{\leftmargin}{10 ex}
\setlength{\labelsep}{2 ex}
\setlength{\parsep}{5pt}
}
\item
Each $0$-cell of $K$ lies in the interior of a side of
$\partial I^2$ or in a region.\par
\label{item:K1}
\item
Each $1$-cell of $K$ either lies in $\partial I^2$, or is disjoint from
$\Gamma$, or crosses one edge of $\Gamma$ transversely in one point.
\label{item:K2}
\item
Each $2$-cell of $K$ either contains no vertex of $\Gamma$, in which case
all of its $0$-cell faces that are not in $\partial I^2$ lie in one region
or in two adjacent regions, or contains one vertex of $\Gamma$, in which
case all of its $0$-cell faces which do not lie in $\partial I^2$ lie in
the union of the regions whose closures contain that vertex.
\label{item:K3}
\end{list}

We now construct a map from $K$ to the Diagram. First, each $0$-cell in
$\partial K$ is sent to one of the single-letter $0$-simplices of the
diagram: if it lies in the side $s=0$ (respectively, $t=0$, $s=1$, $t=1$)
then it is sent to the $0$-simplex labeled $\sX$ (respectively, $\sA$,
$\sY$, $\sB$).  Similarly, any $1$-cell in a side of $\partial K$ is sent
to the $0$-simplex that is the image of its endpoints, and the four
$1$-cells in $\partial K$ dual to the original corners are sent to the
$1$-simplex whose endpoints are the images of the endpoints of the
$1$-cell. Notice that $\partial K$ maps essentially onto the circle
consisting of the four diagonal $1$-simplices of the Diagram.

We will now show that if there is no unlabeled region, this map extends to
$K$, a contradiction. Since an unlabeled region produces pairs $Q_s$ and
$P_t$ that meet in good position, this will complete the proof sketch of
Theorem~\ref{thm:RS}.

Now we consider cells of $K$ that do not lie entirely in $\partial K$. Each
$0$-cell in the interior of $K$ lies in a region.  
By\index{RS1@(RS1), (RS2), (RS3)} (RS\ref{item:noAB}),
the label of each $0$-cell has a form associated to one of the
$0$-simplices of the Diagram, and we send the $0$-cell to that $0$-simplex.

Consider a $1$-cell of $K$ that does not lie in $\partial K$. Suppose it
has one endpoint in $\partial K$, say in the side $s=0$ (the other cases
are similar). The other endpoint lies in a region whose closure meets the
side $s=0$, so its label contains $\sX$. Therefore the images of the
endpoints of the $1$-cell both contain $\sX$, so lie either in a
$0$-simplex or a $1$-simplex of the Diagram. We extend the map to the
$1$-cell by sending it into that $0$- or $1$-simplex.  Suppose the $1$-cell
lies in the interior of $K$.  Its endpoints lie either in one region or in
two adjacent regions. If the former, or the latter and the labels of the
regions are equal, we send the $1$-cell to the $0$-simplex for that
label. If the latter and the labels of the regions are different, then
property (RS\ref{item:edge labeling}) shows that the labels span a unique
$1$-simplex of the Diagram, in which case we send the $1$-cell to that
$1$-simplex.

Assuming that the map has been extended to the $1$-cells in this way,
consider a $2$-cell of $K$. Suppose first that it has a face that meets
the side $s=0$ (the other cases are similar). Then each of its $0$-cell
faces lies in one of the sides $s=0$, $t=0$, or $t=1$, or in a region whose
closure meets $s=0$. In the latter case, we have seen that the label of the
region must contain $\sX$, so it cannot contain $\sY$, and in particular it
cannot be a single letter $\sY$. In no case does the $0$-cell map to the
vertex $\sY$ of the Diagram, so the image of the boundary of the $2$-cell
maps into the complement of that vertex in the Diagram. Since that
complement is contractible, the map extends over the $2$-cell.

Suppose now that the $2$-cell lies entirely in the interior of $K$.  If it
is dual to a $0$-simplex of $\Delta$ that lies in a region or in the
interior of an edge of $\Gamma$, then all its $0$-cell faces lie in a
region or in two adjacent regions. In this case, all of its $1$-dimensional
faces map into some $1$-simplex of the Diagram, so the map on the faces
extends to a map of the $2$-cell into that $1$-simplex. Suppose the
$2$-cell is dual to a vertex of $\Gamma$. Its faces lie in the union of
regions whose closures contain the vertex. If the vertex has valence $2$,
then all $0$-cell faces lie in two adjacent regions (actually, in this
case, the regions must have the same label) and the map extends to the
$2$-cell as before. If the vertex has valence $4$, then by
(RS\ref{item:2-cell labeling}), the labels of the four regions whose
closures contain the vertex must all avoid at least one of the four
letters. This implies that the boundary of the $2$-cell of $K$ maps into a
contractible subset of the Diagram. So again the map can be extended over
the $2$-cell, giving us the desired contradiction.

We emphasize that the map from $K$ to the Diagram carries each $1$-cell of
$K$ to a $0$-simplex or a $1$-simplex of the Diagram, principally due to
property~(RS\ref{item:edge labeling})\index{Rubinstein-Scharlemann!method|)}.

\section[Graphics having no unlabeled region]
{Graphics having no unlabeled region}
\label{sec:examples}\index{examples!example3@graphic with no unlabeled region|(}

One cannot hope to perturb a parameterized family of sweepouts to be in
Morse general position. One must allow for the possibility of levels having
tangencies of high order, and having more than two tangencies. We will see
in Section~\ref{sec:generalposition} that all such phenomena can be
isolated at the vertices of the graph $\Gamma$ in the graphic. In
particular, the $(s,t)$ that lie on the open edges of $\Gamma$ will still
correspond to pairs of levels that have a single stable tangency, and
therefore their associated graphics will still have property
(RS\ref{item:edge labeling}). Achieving this property for the edges of
$\Gamma$ will require considerable effort, so before beginning the task, we
will show that the hard work really is necessary. We will give here an
example of a pair of sweepouts on $S^2\times S^1$ (that is, on $L(0,1)$)
which have a graphic with no unlabeled region. It will be clear that what
goes wrong is the existence of edges of $\Gamma$ that consist of pairs
having multiple tangencies, and the corresponding failure of the graphic to
have property~(RS\ref{item:edge labeling}).

We do not have an explicit counterexample of this kind on a lens space,
which would be even more complicated to describe, but we think it is fairly
clear that the construction, which starts with a simple pair of sweepouts
and ``closes up'' a good region in their graphic, could be carried out on a
typical pair of sweepouts.

This section is not part of the proof of the Smale Conjecture for Lens
Spaces, and can be read independently (provided that one is familiar with
Rubinstein-Scharlemann graphics and their labeling scheme).

The first step is to construct a pair of sweepouts of $S^2\times S^1$, with
the graphic shown on the left in Figure~\ref{fig:graphics}.  In
Figure~\ref{fig:graphics}, the edges of pairs for which the corresponding
levels have a single center tangency are shown as dotted. The four corner
regions are not labeled, since their labels are the same as the regions
that are adjacent to them along an edge of centers.

After constructing the sweepouts that produce the first graphic, we will
see how to move one of the sweepouts by isotopy to ``collapse'' the
unlabeled region. Two edges of the first graphic are moved to coincide,
producing the graphic on the right in Figure~\ref{fig:graphics}. The three
open edges that lie on the diagonal $y=x$ consist of pairs of levels which
have two saddle tangencies. The two vertices where the edges labeled $1$
and $4$ cross the diagonal at points corresponding to pairs having three
saddle tangencies.
\index{figures!figure92@graphic with no unlabeled region}%
\begin{figure}
\labellist
\pinlabel $B$ at 224 290
\pinlabel $B$ at 600 290
\pinlabel $by$ at 285 285
\pinlabel \begin{scriptsize}$by$\end{scriptsize} at 660 291
\pinlabel $Y$ at 290 225
\pinlabel $Y$ at 655 232
\pinlabel $bx$ at 132 220
\pinlabel $bx$ at 512 220
\pinlabel $4$ at 277 185
\pinlabel $4$ at 648 185
\pinlabel $3$ at 215 170
\pinlabel $2,3$ at 510 170
\pinlabel $2$ at 120 170
\pinlabel $P_t$ at -15 170
\pinlabel $P_t$ at 356 170
\pinlabel $1$ at 56 155
\pinlabel $1$ at 426 155
\pinlabel $X$ at 45 107
\pinlabel $X$ at 418 98
\pinlabel $ay$ at 210 115
\pinlabel $ay$ at 564 115
\pinlabel $ax$ at 55 55
\pinlabel \begin{scriptsize}$ax$\end{scriptsize} at 415 45
\pinlabel $A$ at 107 43
\pinlabel $A$ at 466 47
\pinlabel $Q_s$ at 172 -10
\pinlabel $Q_s$ at 542 -10
\endlabellist
\includegraphics[width=12cm]{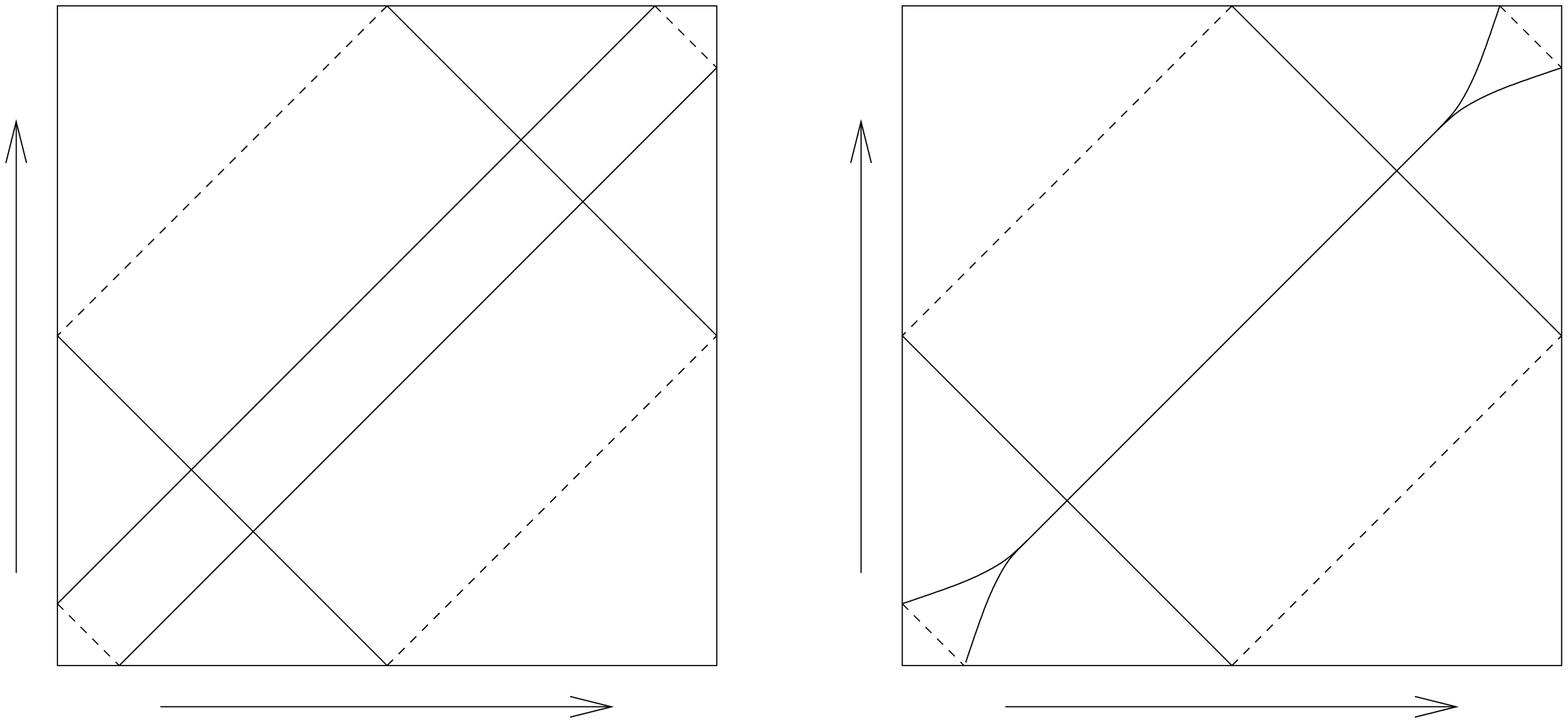}
\caption{Graphics before and after deformation.}
\label{fig:graphics}
\end{figure}

As it is rather difficult to visualize the sweepouts directly, we describe
them by level pictures for various $P_t$.  The $Q_s$ appear as level curves
in each $P_t$. Here are some general conventions:
\begin{enumerate}
\item[(i)] A solid dot is a center tangency.
\item[(ii)] An open dot (i.~e.~a tiny circle) is a point in one of the singular
circles $S_i$ of the $Q_s$-sweepout.
\item[(iii)] Double-thickness lines are intersections with a $Q_s$ that
have more than one tangency.
\item[(iv)] Dashed lines are biessential intersection circles.
\end{enumerate}

In a picture of a $P_t$, the level curves $P_t\cap Q_s$ that contain
saddles appear as curves with self-crossings, and we label the crossings
with $1$, $2$, $3$, or $4$ to indicate which edge of the graphic in
Figure~\ref{fig:graphics} contains that $(s,t)$-pair. For a fixed $t$,
$s(n)$ will denote the $s$-level of saddle~$n$. That is, in the graphic the
edge of $\Gamma$ labeled $n$ contains the point $(s(n),t)$.

Figure~\ref{fig:Morse} shows some $P_t$ with $t\leq 1/2$, for a sweepout of
$S^2\times S^1$ whose graphic is the one shown in the left of
Figure~\ref{fig:graphics}. Here are some notes on Figure~\ref{fig:Morse}.
\begin{enumerate}
\item In (a)-(f), the circles $x=\text{constant}$ are longitudes of $V_t$,
and the circles $y=\text{constant}$ are meridians.
\item The point represented by the four corners is the point of $P_t$ with
largest $s$-level. In (a) it is a tangency of $P_{1/2}$ with $S_1$, and in
(b)-(f) it is a center tangency of $P_t$ with $Q_{t+1/2}$.
\item The open dots in the interior of the squares are intersections of
$P_t$ with $S_0$. In (a) it is a tangency of $P_{1/2}$ with $S_0$, in
(b)-(e) they are transverse intersections. In (f), $P_t$ is disjoint from
$S_0$.
\item In (b), saddle~1 has appeared. Circles of $Q_s\cap P_t$ with $s<s(1)$
are essential in $Q_s$, and these $(s,t)$ lie in the region labeled $X$ in
the graphic. Circles of $Q_s\cap P_t$ with $s(1)<s<s(2)$ enclose the
figure-$8$ in (b), which is $P_t\cap Q_{s(1)}$. They are inessential in
both $Q_s$ and $P_t$, and these $(s,t)$ lie in the region labeled $bx$. The
vertical dotted lines are biessential intersections corresponding to a pair
in the unlabeled region. Finally, one crosses $Q_{s(3)}$, and eventually
reaches the center tangency.
\item The horizontal level curves shown in (f) are meridians of $V_t$ that
bound disks in the $Q_s$ that contain them. This $(s,t)$ lies in the region
labeled $A$ in the graphic.
\end{enumerate}

For $t>1/2$, the intersection pattern of $P_t$ with the $Q_s$ is isomorphic
to the pattern for $P_{1-t}$, by an isomorphism for which $Q_s$ corresponds
to $Q_{1-s}$. As one starts $t$ at $1/2$ and moves upward through
$t$-levels, saddle $4$ appears inside the component of $P_t-Q_{s(3)}$ that
is an open disk, and expands until the level where $s(3)=s(4)$. The
biessential intersection circles in (a)-(d) are again longitudes in $V_t$
and in $W_t$, and the horizontal intersection circles in (f) are meridians
of $W_t$. These $(s,t)$ lie in the region labeled $B$ in the graphic. This
completes the description of the sweepouts in Morse general position.

Figure~\ref{fig:nogood} shows some $P_t$ for a sweepout of $S^2\times S^1$
whose graphic is the one shown in the right of Figure~\ref{fig:graphics}.
This sweepout is obtained from the previous one by an isotopy that moves
parts of the $Q_s$ levels down (to lower $t$-levels) near saddle $2$ and
up near saddle $3$. Again, the portion that is shown fits together with a
similar portion for $1/2\leq t\leq 1$. As $t$ increases past $1/2$, saddle
$4$ appears in the component of $P_t-S_{s(2)}$ that contains the point
which appears as the four corners.
\index{figures!figure93@sweepout with no unlabeled region}%
\begin{center}
\begin{figure} %[p] %[ht]
\labellist
\pinlabel $(a)$ at 111 588
\pinlabel $(b)$ at 398 588
\pinlabel $(c)$ at 111 300
\pinlabel $(d)$ at 398 300
\pinlabel $(e)$ at 111 12
\pinlabel $(f)$ at 398 12
\scriptsize
\pinlabel $2$ at 98 662
\pinlabel $3$ at 168 715
\pinlabel $2$ at 388 662
\pinlabel $3$ at 457 715
\pinlabel $1$ at 390 742
\pinlabel $2$ at 98 375
\pinlabel $3$ at 169 429
\pinlabel $1$ at 100 474
\pinlabel $3$ at 457 427
\pinlabel $1$ at 388 482
\pinlabel $2$ at 390 402
\pinlabel $2$ at 101 122
\pinlabel $3$ at 170 151
\pinlabel $1$ at 120 199
\pinlabel $3$ at 430 140
\pinlabel $1$ at 400 199
\endlabellist
\includegraphics[height=0.81\textheight]{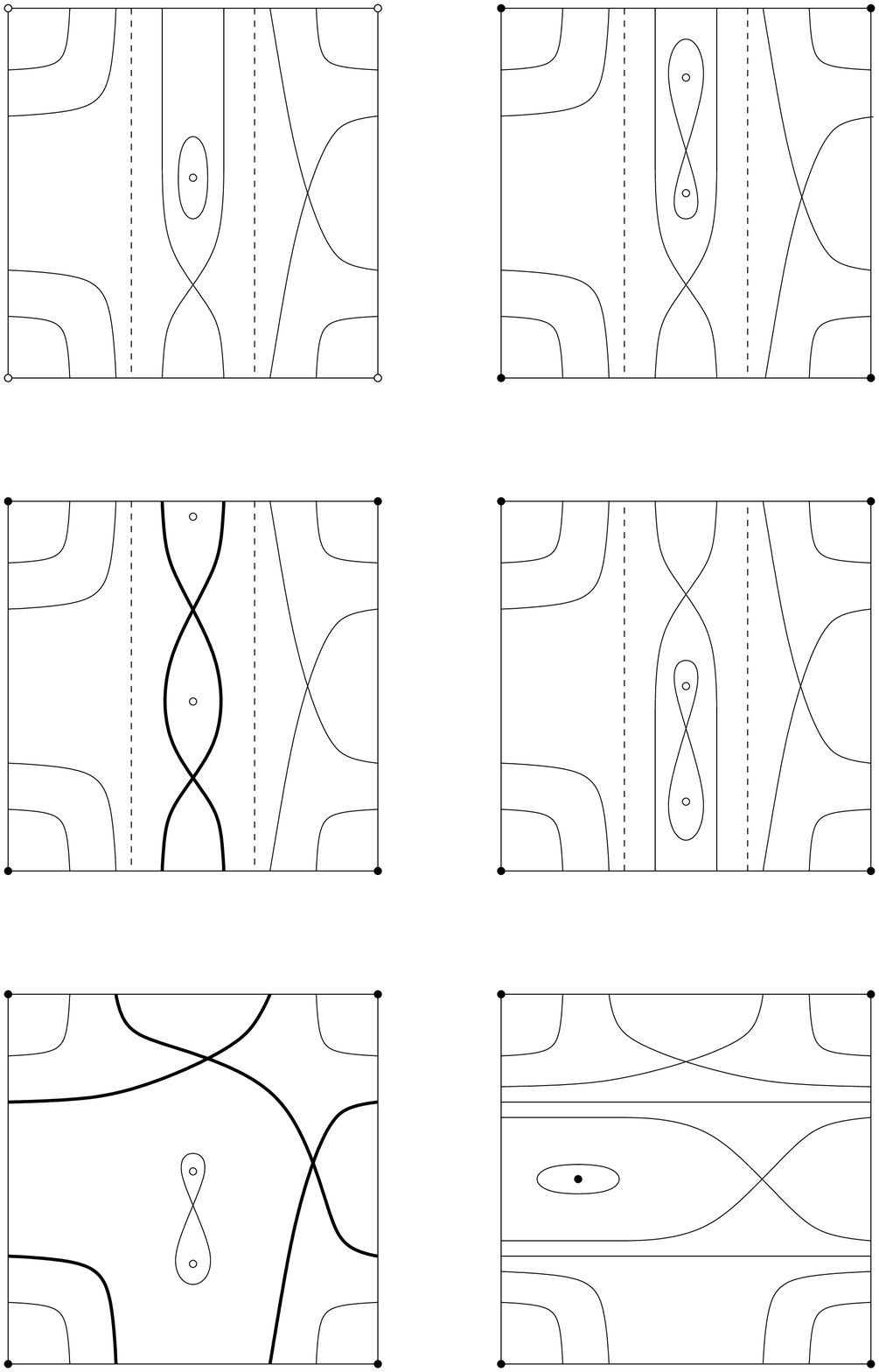}
\caption[sweepout2]{Intersections of the $Q_s$ with fixed $P_t$ as $t$
decreases from $1/2$ to $0$, for the sweepouts with an unlabeled region.
\begin{enumerate}
\item[(a)] $P_{1/2}$.
\item[(b)] $P_t$ where $s(1)<s(2)<s(3)$.
\item[(c)] $P_t$ where $s(1)=s(2)$.
\item[(d)] $P_t$ where $s(2)<s(1)<s(3)$.
\item[(e)] $P_t$ where $s(1)=s(3)$.
\item[(f)] $P_t$ where $s(3)<s(1)$, and after saddle 2 changes to a center.
\end{enumerate}}
\label{fig:Morse}
\end{figure}
\end{center}

\index{figures!figure93@sweepout with no unlabeled region}%
\begin{figure} %[p] %[ht]
\labellist
\pinlabel $(a)$ at 111 588
\pinlabel $(b)$ at 398 588
\pinlabel $(c)$ at 111 300
\pinlabel $(d)$ at 398 300
\pinlabel $(e)$ at 111 12
\pinlabel $(f)$ at 398 12
\scriptsize
\pinlabel $2$ at 86 644
\pinlabel $3$ at 173 715
\pinlabel $2$ at 374 645
\pinlabel $3$ at 461 715
\pinlabel $1$ at 384 784
\pinlabel $2$ at 97 373
\pinlabel $3$ at 146 427
\pinlabel $1$ at 99 482
\pinlabel $3$ at 431 438
\pinlabel $1$ at 399 487
\pinlabel $2$ at 375 377
\pinlabel $2$ at 97 85
\pinlabel $3$ at 166 119
\pinlabel $1$ at 113 201
\pinlabel $3$ at 430 139
\pinlabel $1$ at 399 199
\endlabellist
\includegraphics[height=0.85\textheight]{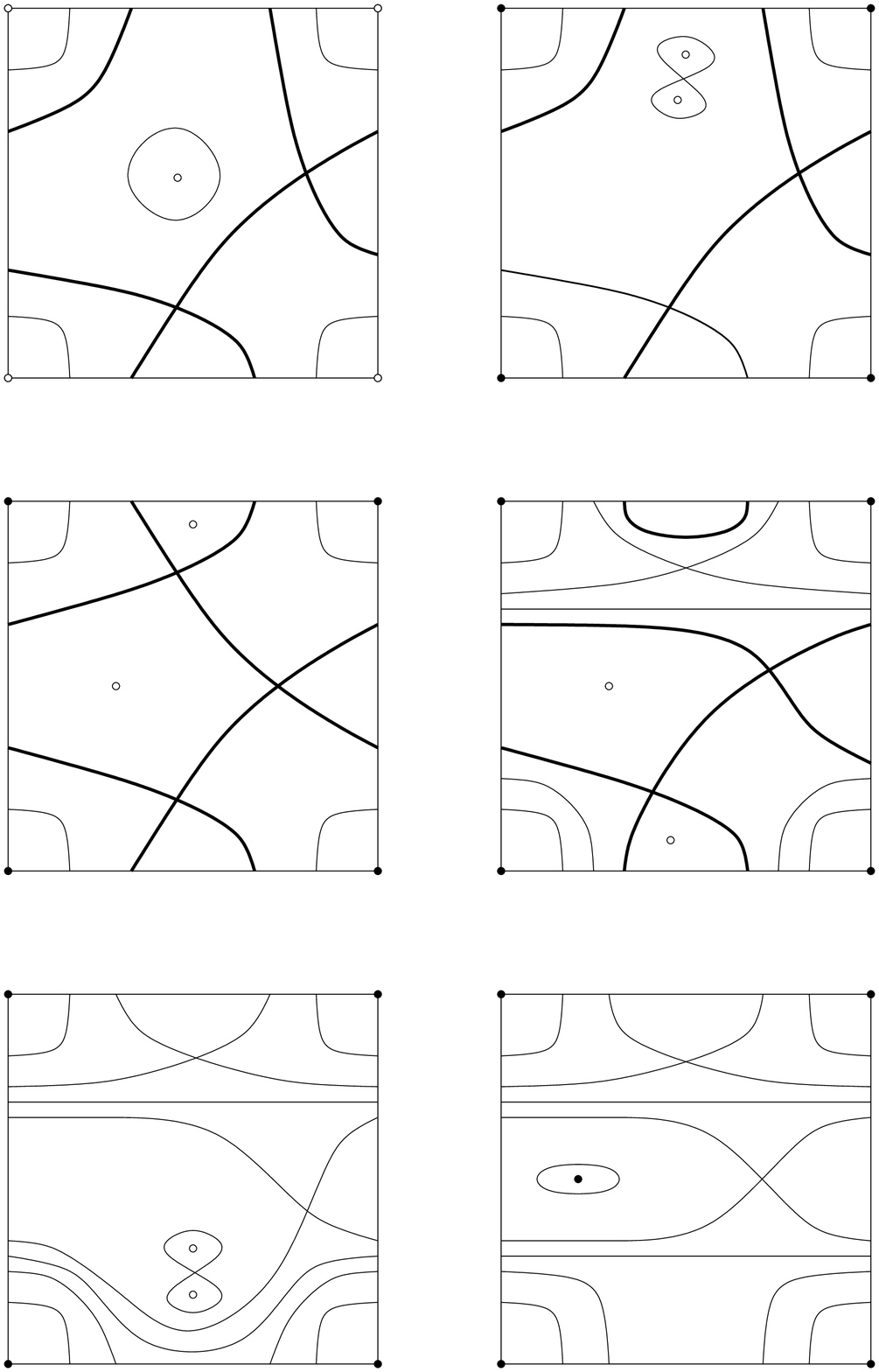}
\caption[sweepout2]{Intersections of the $Q_s$ with fixed $P_t$ as $t$
decreases from $1/2$ to $0$, for the sweepouts with no unlabeled region.
\begin{enumerate}
\item[(a)] $P_{1/2}$.
\item[(b)] $P_t$ where $s(1)<s(2)=s(3)$.
\item[(c)] $P_t$ where $s(1)=s(2)=s(3)$.
\item[(d)] $P_t$ where $s(2)=s(3)<s(1)$.
\item[(e)] $P_t$ where $s(2)<s(3)<s(1)$.
\item[(f)] $P_t$ where $s(3)<s(1)$, and after saddle 2 changes to a center.
\end{enumerate}}
\label{fig:nogood}
\end{figure}\index{examples!example3@graphic with no unlabeled region|)}

\section[Graphics for parameterized families]
{Graphics for parameterized families}
\label{sec:generalposition}

In this section we prove that a parameterized family of sweepouts can be
perturbed so that a suitable graphic exists at each parameter. As discussed
in Section~\ref{sec:examples}, in a parameterized family one must allow for
the possibility of levels having tangencies of high order, and having more
than two tangencies.

Additional complications arise because one cannot avoid having parameters
where the singular sets of the sweepouts intersect, or where the singular
sets have high-order tangencies with levels. We sidestep these
complications by working only with sweepout parameters that lie in an
interval $[\epsilon,1-\epsilon]$. The graphic is only considered to exist
on the square $[\epsilon,1-\epsilon]\times[\epsilon,1-\epsilon]$, which we
call $I^2_\epsilon$. The number $\epsilon$ is chosen so that the labels of
regions whose closure meets a side of $I^2_\epsilon$ will be known to
include certain letters. Just as before, this will ensure that the map to
the Diagram be essential on the boundary of the dual complex~$K$.

These considerations motivate our definition of a general position family
of diffeomorphisms. As usual, let $M$ be a closed orientable $3$-manifold
and $\tau\colon P\times [0,1]\to M$ a sweepout with singular set $T=T_0\cup
T_1$ and level surfaces $P_t$ bounding handlebodies $V_t$ and $W_t$. Let
$f\colon M\times W\to M$ be a parameterized family of diffeomorphisms,
where $W$ is a compact manifold. For $u\in W$ we denote the restriction of
$f$ to $M\times\set{u}$ by $f_u$. When a choice of parameter $u$ has been
fixed, we denote $f_u(P_s)$ by $Q_s$, and $f_u(V_s)$ and $f_u(W_s)$ by
$X_s$ and $Y_s$ respectively. When $Q_s$ meets $P_t$ transversely, a label
is assigned to $(s,t)$ as in Section~\ref{sec:Rubinstein-Scharlemann}.

A preliminary definition will be needed. We say that a positive number
$\epsilon$ 
\indexdef{border label control}\textit{gives border label control} for 
$f$ if the following hold
at each parameter $u$:
\begin{enumerate}
\item
If $t\leq 2\epsilon$, then there exists $r$ such that $Q_r$ meets $P_t$
transversely and contains a compressing disk of $V_t$.
\item
If $t\geq 1-2\epsilon$, then there exists $r$ such that $Q_r$ meets $P_t$
transversely and contains a compressing disk of $W_t$.
\item
If $s\leq 2\epsilon$, then there exists $r$ such that $P_r$ meets $Q_s$
transversely and contains a compressing disk of $X_s$.
\item
If $s\geq 1-2\epsilon$, then there exists $r$ such that $P_r$ meets $Q_s$
transversely and contains a compressing disk of $Y_s$.
\end{enumerate}

Throughout this section, a \indexdef{graph}\textit{graph} is a compact
space which is a disjoint union of a CW-complex of dimension $\leq 1$ and
circles. The circles, if any, are considered to be open edges of the graph.

We say that $f$ is 
\indexdef{general position!with respect to sweepout}\textit{in general position} 
(with respect to the
sweepout $\tau$) if there exists $\epsilon >0$ such that $\epsilon$ gives
border label control for $f$ and such that the following hold for each
parameter $u\in W$.\indexdef{GP1@(GP1), (GP2), (GP3)}
\newcounter{GPcounter}
\begin{list}{(GP\arabic{GPcounter})}
{\usecounter{GPcounter}
\setlength{\labelwidth}{7 ex}
\setlength{\leftmargin}{9 ex}
\setlength{\labelsep}{2 ex}
\setlength{\parsep}{5pt}
}
\item 
For each $(s,t)$ in $I^2_\epsilon$, $Q_s\cap P_t$ is a graph. At each
point in an open edge of this graph, $Q_s$ meets $P_t$ transversely. At
each vertex, they are tangent.
\item
The $(s,t)\in I^2_\epsilon$ for which $Q_s$ has a tangency with $P_t$ form
a graph $\Gamma_u$ in $I^2_\epsilon$.
\item
If $(s,t)$ lies in an open edge of $\Gamma_u$, then $Q_s$ and $P_t$ have a
single stable tangency.
\end{list}

The next lemma is immediate from the definition of border label control and
the labeling rules for regions.  It does not require that we be working
with lens spaces, so we state it as a lemma with weaker hypotheses.
\begin{lemma}\indexstate{border label control!and labels in graphic}%
\indexstate{labels!and border label control in graphics}
Suppose that $f\colon M\times W\to M$ is in general position with respect
to $\tau$. Assume that $M\neq S^3$ and that the Heegaard splittings
associated to $\tau$ are strongly irreducible. Suppose that $\epsilon$
gives border label control for $f$.
\begin{enumerate}
\item
If $t\leq \epsilon$, then the label of $(s,t)$ contains $A$.
\item
If $t\geq 1-\epsilon$, then the label of $(s,t)$ contains $B$.
\item
If $s\leq \epsilon$, then the label of $(s,t)$ contains $X$.
\item
If $s\geq 1-\epsilon$, then the label of $(s,t)$ contains $Y$.
\end{enumerate}
\label{lem:borderlabelconstraints}
\end{lemma}

Here is the main result of this section.
\begin{theorem}\index{general position!of parameterized family}
Let $f\colon M\times W\to M$ be a parameterized family of diffeomorphisms.
Then by an arbitrarily small deformation, $f$ can be put into general
position with respect to $\tau$.
\label{thm:generalposition}
\end{theorem}
\noindent The proof of Theorem~\ref{thm:generalposition} will constitute
the remainder of this section. Since the argument is rather long, we will
break it into subsections. Until Subsection~\ref{sec:borderlabel}, $M$ can
be a closed manifold of arbitrary dimension~$m$.
\shortpage

\subsection{Weak transversality}\index{weak transversality|(}
Although individual maps may be put transverse to a submanifold of the
range, it is not possible to perturb a parameterized family so that each
individual member of the family is transverse. But a very nice result 
J. W. Bruce, Theorem~1.1 of~\cite{Bruce}, allows one to simultaneously
improve the members of a family.
\begin{theorem}[J. W. Bruce]\indexstate{Bruce!weak general position theorem}
Let $A$, $B$ and $U$ be smooth manifolds and $C\subset B$ a submanifold.
There is a residual family of mappings $F\in \Maps(A\times U,B)$ such
that:
\begin{enumerate}
\item[(a)] For each $u\in U$, the restriction $F_u=F|_{A\times \{u\}}\colon
A\to B$ is transverse to $C$ except possibly on a discrete set of points.
\item[(b)] For each $u\in U$, the set $F_u^{-1}(C)$ is a smooth submanifold
of codimension equal to the codimension of $C$ in $B$, except possibly at a
discrete set of points. At each of these exceptional points $F_u^{-1}(C)$
is locally diffeomorphic to the germ of an algebraic variety, with the
exceptional point corresponding to an isolated singular point of the
variety.
\end{enumerate}
\label{thm:Bruce1}
\end{theorem}
\noindent
That is, $F_u^{-1}(C)$ is smooth except at isolated points where it has
topologically a nice cone-like structure.  It is not assumed that any of
the manifolds involved is compact.

Theorem 1.3 of \cite{Bruce} is a version of Theorem~\ref{thm:Bruce1} in
which $C$ is replaced by a bundle $\phi\colon B\to D$. The statement is:
\begin{theorem}[J. W. Bruce]\indexstate{Bruce!weak general position theorem!for parameterized family}
For a residual family of mappings $F\in \Maps(A\times U,B)$, the
conclusions of Theorem~\ref{thm:Bruce1} hold for all submanifolds
$C=\phi^{-1}(d)$, $d\in D$.
\label{thm:Bruce2}
\end{theorem}

We should comment on the significance of the residual subset in these two
theorems. The method of proof of these theorems is to define, in an
appropriate jet space, a locally algebraic subset which contains the jets
of all the maps that fail these weak transversality conditions. These
subsets have increasing codimension as higher-order jets are taken. A
variant of \index{Thom transversality}Thom transversality (Lemma~1.6 of
\cite{Bruce}) allows one to perturb a parameterized family of maps so that
these jets are avoided and the conclusion holds.  When $A$ and $W$ are
compact, the image of $A\times W$ will lie in the open complement of the
locally algebraic sets of sufficiently high codimension. Consequently, any
map sufficiently close to the perturbed map will also satisfy the
conclusions of the theorems. In all of our applications, the spaces
involved will be compact, and \textit{we tacitly assume that the result of
  any procedure holds on an open neighborhood of the perturbed map.}
\shortpage

We now adapt the methodology of Bruce to prove a version of
Theorem~\ref{thm:Bruce1} in which the submanifold $C$ is replaced by the
zero set of a nontrivial polynomial. We will prove it only for the case
when $A=I$, although a more general version should be possible.

\begin{proposition}
Let $P\colon \R^n\to \R$ be a nonzero polynomial and put $V=P^{-1}(0)$.
Let $W$ be compact. Then for all $G$ in an open dense subset of
$\Maps(I\times W,\R^n)$, each $G_u^{-1}(V)$ is finite.
\label{prop:polynomial_Bruce}
\end{proposition}

\begin{proof}
Let $J_0^k(1,n)$ be the space of germs of degree-$k$ polynomials from
$(\R,0)$ to $\R^n$; an element of $J_0^k(1,n)$ can be written as
$(a_{1,0}+a_{1,1}t+\cdots+a_{1,k}t^k,\ldots,a_{n,0}+a_{n,1}t+\cdots+a_{n,k}t^k)$,
so that $J_0^k(1,n)$ can be identified with $\R^{(k+1)n}$. Note that the
jet space $J^k(I,\R^n)$ can be regarded as $I\times J_0^k(1,n)$, by
identifying the jet of $\alpha\colon I\to \R^n$ at $t_0$ with the jet of
$\alpha(t-t_0)$ at $0$.

Define a polynomial map $P_*\colon J_0^k(1,n)\to J_0^k(1,1)$ by applying
$P$ to the $n$-tuple $(a_{1,0}+a_{1,1}t+\cdots+a_{1,k}t^k,
\ldots,a_{n,0}+a_{n,1}t+\cdots+a_{n,k}t^k)$, and then taking only the terms
up to degree~$k$. The inverse image $P_*^{-1}(0)$ is the set of $k$-jets
$\alpha$ in $\R^n$ such that $P(\alpha(0))=0$ and the first $k$ derivatives
of $P\circ \alpha$ vanish at $t=0$, that is, the set of germs of paths that
lie in $V$ up to $k^{th}$-order.

\begin{lemma} If $P$ is nonconstant, then the codimension of $P_*^{-1}(0)$
goes to $\infty$ as $k\to \infty$.
\label{lem:Jacobian}
\end{lemma}

\begin{proof}[Proof of Lemma~\ref{lem:Jacobian}]
It suffices to show that the rank of the Jacobian of $P_*$ goes to $\infty$
as $k\to \infty$. For notational simplicity, we will give the proof for
$P(X,Y)$, and it will be evident how the argument extends to the general
case. 

Write $a=a_0+a_1t+a_2t^2+\cdots$ and $b=b_0+b_1t+b_2t^2+\cdots$, and
examine~$P(a,b)$. We have $P_*(a,b) = Q_0+ Q_1t + Q_2t^2+\cdots$ where each
$Q_i$ is a (finite) polynomial in $\R[a_0,b_0,a_1,b_1,\ldots]$.  Notice
that $Q_j=\dfrac{1}{j!}\dfrac{\partial^j P_*}{\partial
t^j}\bigg\vert_{t=0}$.

\newpage
It is instructive to calculate a few derivatives of $P_*(a,b)$. We have
\begin{align*}
\dfrac{\partial P_*}{\partial t} &= a'P_X + b'P_Y\\
\dfrac{\partial^2 P_*}{\partial t^2} &= a''P_X + b''P_Y + (a')^2P_{XX} +
2a'b'P_{XY} + (b')^2P_{YY}\\
\dfrac{\partial^3 P_*}{\partial t^3} &= a'''P_X + b'''P_Y + a''a'P_{XX} +
(a''b'+ a'b'')P_{XY} + b''b'P_{YY}\\
&+ 2a'a''P_{XX} + (2a''b' + 2a'b'')P_{XY} + 2b'b''P_{YY}\\
&+ (a')^3P_{XXX} + 3(a')^2 b' P_{XXY} + 3a'(b')^2 P_{XYY} + (b')^3P_{YYY}\\
&= a'''P_X + b'''P_Y + 3a''a'P_{XX} +
3(a''b'+ a'b'')P_{XY} + 3b''b'P_{YY}\\
&+ (a')^3P_{XXX} + 3(a')^2 b' P_{XXY} + 3a'(b')^2 P_{XYY} + (b')^3P_{YYY}
\end{align*}
and at $t=0$ these become
\begin{align*}
Q_1 &= a_1P_X(a_0,b_0) + b_1P_Y(a_0,b_0)\\
2! Q_2 &= 2a_2P_X(a_0,b_0) + 2b_2P_Y(a_0,b_0)\\ 
&+ a_1^2P_{XX}(a_0,b_0) + 2a_1b_1P_{XY}(a_0,b_0) + b_1^2P_{YY}(a_0,b_0)\\
3! Q_3 &= 6a_3P_X(a_0,b_0) + 6b_3P_Y(a_0,b_0) + 6a_1a_2P_{XX}(a_0,b_0)\\ 
&+ 6(a_2b_1+ a_1b_2)P_{XY}(a_0,b_0) + 6b_1b_2P_{YY}(a_0,b_0) +
a_1^3P_{XXX}(a_0,b_0)\\ 
&+ 3a_1^2 b_1 P_{XXY}(a_0,b_0) + 3a_1b_1^2 P_{XYY}(a_0,b_0) + b_1^3P_{YYY}(a_0,b_0)
\end{align*}
Induction shows that in general, writing $K_{rX,sX}$ for 
$\dfrac{\partial^{r+s} P}{\partial^rX\partial^sY}(a_0,b_0)$, there are
positive constants $c_{i_1\cdots i_rj_1\cdots j_s}$ so that for large $N$,
\begin{equation} Q_N = \sum K_{rX,sY}
\Big(\sum c_{i_1\cdots i_rj_1\cdots j_s} a_{i_1}\cdots a_{i_r}
b_{j_1}\cdots b_{j_s}\Big)\ .
\label{eqn:QN}
\end{equation}
For large $N$, all partial derivatives $K_{rX,sY}$ of $P$ at $(a_0,b_0)$
appear, and some must be nonzero since $P$ is a polynomial. Notice also
that there is no cancellation due to values of the $K_{rX,sY}$, since each
monomial term $a_{i_1}\cdots a_{i_r}b_{j_1}\cdots b_{j_s}$ appears just
once.

Any given $a_i$ appears in some of the monomial terms of $Q_N$ 
for all sufficiently large $N$. On the other hand, $Q_N$ contains no $a_i$
or $b_i$ with $i>N$, so $\dfrac{\partial Q_i}{\partial a_j}$ vanishes for
$j>i$, and similarly for $\dfrac{\partial Q_j}{\partial b_i}$. Therefore if
we truncate at $t^k$, the Jacobian $\begin{bmatrix}\Big(\dfrac{\partial
A_i}{\partial a_j}\Big) & \Big(\dfrac{\partial A_i}{\partial b_j}\Big)\\
\end{bmatrix}$ is a $(k+1)\times 2(k+1)$ matrix consisting of (two, since 
we are in the case of a two-variable $P(X,Y)$) upper triangular blocks:
\[\begin{bmatrix} * & 0 & 0 & \cdots & 0 & * & 0 & 0 & \cdots & 0 \\
* & * & 0 & \cdots & 0 & * & * & 0 & \cdots & 0 \\
&&&\ddots&&&&&\ddots\\
* & * & & \cdots & 0 & * & * & & \cdots & 0\\
* & * & & \cdots & * & * & * & & \cdots & *\\
\end{bmatrix}\ .
\]

If the lemma is false, then there is some maximal rank of these Jacobians
as $k\to \infty$. That is, there are, say, $m$ rows such that every row is
an $\R$-linear combination of these rows. For values of $k$ much larger
than $m$, all of these $m$ rows have zeros in the upper triangular part of
the two blocks. On the other hand, Equation~\ref{eqn:QN} and the
observations that follow it show that for each fixed $j$, $\dfrac{\partial
A_i}{\partial a_j}$ is nonzero for sufficiently large $i$. This completes
the proof of Lemma~\ref{lem:Jacobian}.
\end{proof}

For each $k$, put $Z_k=P_*^{-1}(0)$.  Lemma~\ref{lem:Jacobian} shows that
the codimension of $Z_k$ in $J^k_0(1,n)$ goes to $\infty$ as $k\to
\infty$. If $\alpha\colon(\R,0)\to \R^n$ is a germ of a smooth map, and $0$
is a limit point of $\alpha^{-1}(V)$, then all derivatives of $P\circ
\alpha$ vanish at $0$. That is, the $k$-jet of $\alpha$ at $t=0$ is
contained in $Z_k$ for every~$k$.

By Lemma~1.6 of \cite{Bruce}, there is a residual set of maps $G\in
\Maps(I\times W,\R^n)$ such that the jet extensions $j^kG\colon I\times
W\to J^k(I,\R^n)$ defined by $j^kG(t,u)=j^kG_u(t)$ are transverse to $I\times
Z_k$. For $k+1$ larger than the dimension of $I\times W$, this says that
no point of $G_u^{-1}(0)$ is a limit point, so each $G_u^{-1}(0)$ is
finite.
\end{proof}\index{weak transversality|)}

\subsection{Finite singularity type}
For our later work, we will need some ideas from singularity theory.  Let
$g\colon (\R^m,0)\to (\R^p,0)$ be a germ of a smooth map. There is a
concept of 
\index{finite singularity type|(}\index{singularity!finite type|(}\textit{finite singularity type} for $g$, whose definition is
readily available in the literature (for example, \cite[p.~117]{Bruce}).
The basic idea of the proof of Theorem~\ref{thm:Bruce1} (given as
Theorem~1.1 in \cite{Bruce}) is to regard the submanifold $C$ locally as
the inverse image of $0$ under a submersion $s$, then to perturb $f$ so that for
each $u$, the critical points of $s\circ f_u$ are of finite singularity
type. In fact, this is exactly the definition of what it means for $f_u$ to
be weakly transverse to $C$. In particular, when $C$ is a point, the
submersion can be taken to be the identity, so we have:
\begin{proposition} Let $f\colon M\to \R$ be smooth. If $f$ is weakly
transverse to a point $r\in \R$, then at each critical point in
$f^{-1}(r)$, the germ of $f$ has finite singularity type.
\label{prop:FST_weaktransverse}
\end{proposition}

Let $f$ and $g$ be 
\index{germs of smooth maps!types of equivalence}%
\index{equivalence!of germs of smooth maps}germs of smooth maps 
from $(\R^m,a)$ to $(\R^p,f(a))$.
They are said to be 
\indexsymdef{Aequivalent}{$\protect\mathcal{A}$-equivalent}\textit{$\mathcal{A}$-equivalent} 
if there exist a germ
$\varphi_1$ of a diffeomorphism of $(\R^m,a)$ and a germ $\varphi_2$ of a
diffeomorphism of $(\R^p,f(a))$ such that $g=\varphi_2\circ f\circ
\varphi_1$. If $\varphi_2$ can be taken to be the identity, then $f$ and
$g$ are called 
\indexsymdef{Requivalent}{$\protect\mathcal{R}$-equivalent}\textit{$\mathcal{R}$-equivalent} (for
\textit{right-equivalent}).
There is also a notion of 
contact 
equivalence,
denoted by $\mathcal{K}$-equivalence, whose definition is readily
available, for example in \cite{Wall}. It is implied by
$\mathcal{A}$-equivalence.

We use 
\indexsymdef{jkf}{$j^kf$}$j^kf$ to denote the \index{jets}$k$-jet of $f$; 
for fixed coordinate systems at
points $a$ and $f(a)$ this is just the Taylor polynomial of $f$ of degree
$k$. For $\mathcal{G}$ one of $\mathcal{A}$, $\mathcal{K}$, or
$\mathcal{R}$, one says that $f$ is 
\indexdef{finitely determined}finitely $\mathcal{G}$-determined if
there exists a $k$ so that any germ $g$ with $j^kg=j^kf$ must be
$\mathcal{G}$-equivalent to $f$.  In particular, if $f$ is finitely
$\mathcal{G}$-determined, then for any fixed choice of coordinates at $a$
and $f(a)$, $f$ is $\mathcal{G}$-equivalent to a polynomial.

The elaborate theory of singularities of maps from $\R^m$ to $\R^p$
simplifies considerably when $p=1$.
\begin{lemma} Let $f$ be the germ of a map from $(\R^m,0)$ to $(\R,0)$,
with $0$ is a critical point of $f$. The following are equivalent.
\begin{enumerate}
\item[(i)] $f$ has finite singularity type.
\item[(ii)] $f$ is finitely $\mathcal{A}$-determined.
\item[(iii)] $f$ is finitely $\mathcal{R}$-determined.
\item[(iv)] $f$ is finitely $\mathcal{K}$-determined.
\end{enumerate}
\label{lem:FST}
\end{lemma}

\begin{proof}
In all dimensions, $f$ is finitely $\mathcal{K}$-determined if and only if
it is of finite singularity type (Corollary~III.6.9 of~\cite{GWLP}, or
alternatively the definition of finite singularity type of
J. Bruce\index{Bruce} \cite[p.~117]{Bruce} is exactly the condition given
in Proposition~(3.6)(a) of J. Mather\index{Mather}~\cite{Mather} for $f$ to
be finitely $\mathcal{K}$-determined).  Therefore (i) is equivalent to
(iv).  Trivially (ii) implies (iii), and (iii) implies (iv), and by
Corollary~2.13 of \cite{Wall}, (iv) implies~(ii).
\end{proof}

\subsection{Semialgebraic sets}
Recall (see for example Chapter~I.2 of \cite{GWLP}) that the class of
\index{semialgebraic sets|(}\textit{semialgebraic} subsets 
of $\R^m$ is defined to be the smallest
Boolean algebra of subsets of $\R^m$ that contains all sets of the form
$\{x\in \R^m\;\vert\;p(x)>0\}$ with $p$ a polynomial on $\R^m$. The
collection of semialgebraic subsets of $\R^m$ is closed under finite
unions, finite intersections, products, and complementation. The inverse
image of a semialgebraic set under a polynomial mapping is semialgebraic. A
nontrivial fact is the 
\index{Tarski-Seidenberg Theorem}Tarski-Seidenberg Theorem (Theorem~II.2(2.1) of
\cite{GWLP}), which says that a polynomial image of a semialgebraic set is
a semialgebraic set. Here is an easy lemma that we will need later.

\begin{lemma} Let $S$ be a semialgebraic subset of $\R^n$. If $S$ has empty
interior, then $S$ is contained in the zero set of a nontrivial polynomial
in $\R^n$.
\par
\label{lem:semialgebraic}
\end{lemma}

\begin{proof}
Since the union of the zero sets of two polynomials is the zero set of
their product, it suffices to consider a single semialgebraic set of the
form $(\cap_{i=1}^r \{x\;\vert\;p_i(x)\geq 0\})\cap (\cap_{j=1}^s
\{x\;\vert\;q_j(x) > 0\})$ where $p_i$ and $q_j$ are nontrivial
polynomials. We will show that if $S$ is of this form and has empty
interior, then $r\geq 1$ and $S$ is contained in the zero set of
$\prod_{i=1}^r p_i$. Suppose that $x\in S$ but all $p_i(x)>0$. Since all
$q_j(x)>0$ as well, there is an open neighborhood of $x$ on which all $p_i$
and all $q_j$ are positive. But then, $S$ has nonempty interior.
\end{proof}\index{finite singularity type|)}%
\index{singularity!finite type|)}\index{semialgebraic sets|)}

\subsection{The codimension of a real-valued function}
\label{subsec:Sergeraert1}\index{codimension!of real-valued function|(}

It is, of course, fundamentally important that the Morse functions form an
open dense subset of $\Maps(M,\R)$. But a great deal can also be said
about the non-Morse functions. There is a ``natural'' stratification of
$\Maps(M,\R)$ by subsets $\mathcal{F}_i$, where 
\indexdef{stratification}stratification here
means that the $\mathcal{F}_i$ are disjoint subsets such that for every $n$
the union $\cup_{i=0}^n \mathcal{F}_i$ is open. The functions in
$\mathcal{F}_n$ are those of ``codimension'' $n$, which we will define
below. In particular, $\mathcal{F}_0$ is exactly the open dense subset of
Morse functions.

The union $\cup_{i=0}^\infty \mathcal{F}_i$ is not all of
$\Maps(M,\R)$. However, the residual set $\Maps(M,\R)-\cup_{i=0}^\infty
\mathcal{F}_i$ is of ``infinite codimension,'' and any parameterized family
of maps $F\colon M\times U\to \R$ can be perturbed so that each $F_u$ is of
finite codimension. In fact, by applying Theorem~\ref{thm:Bruce2} to the
trivial bundle $1_{\R}\colon \R\to \R$ and noting
Proposition~\ref{prop:FST_weaktransverse}, we may perturb any parameterized
family so that each $F_u$ is of finite singularity type at each of its
critical points. The definition of $f\in \Maps(M,\R)$ being of finite
codimension, given below, is exactly equivalent to the algebraic condition
given in (3.5) of Mather\index{Mather}~\cite{Mather} for $f$ to be finitely
$\mathcal{A}$-determined at each of its critical points (as noted in
\cite{Mather}, this part of (3.5) was first due to Tougeron
\cite{Tougeron1}, \cite{Tougeron2}). By Lemma~\ref{lem:FST}, this is
equivalent to $f$ having finite singularity type at each of its critical
points.  We summarize this as
\begin{proposition}\index{codimension!and singularity type}%
\index{singularity type!and finite codimension}
A map $f\in \Maps(M,\R)$ is of finite codimension 
if and only if it has finite singularity type at each of its critical
points.\par
\label{prop:FST_codimension}
\end{proposition}

\index{Sergeraert}We now recall material from Section~7 of
\cite{Sergeraert}. Denote the smooth sections of a bundle $E$ over $M$ by
$\Gamma(E)$. Until we reach Theorem~\ref{thm:Morse weak transversality}, we
will denote $\Maps(M,\R)$ by $C(M)$. For a compact subset $K\subset \R$,
define $\Diff_K(\R)$ to be the diffeomorphisms of $\R$ supported on $K$.

Fix an element $f\in C(M)$ and a compact subset $K\subset \R$ for which
$f(M)$ lies in the interior of $K$. Define $\Phi\colon \Diff(M)\times
\Diff_K(\R)\to C(M)$ by $\Phi(\varphi_1,\varphi_2)=\varphi_2\circ
f\circ\varphi_1$. The differential of $\Phi$ at $(1_M,1_{\R})$ is defined
by $D(\xi_1,\xi_2)=f_*\xi_1+\xi_2\circ f$. Here, $\xi_1\in \Gamma(TM)$,
which is regarded as the tangent space at $1_M$ of $\Diff(M)$, $\xi_2\in
\Gamma_K(T\R)$, similarly identified with the tangent space at $1_{\R}$ of
$\Diff_K(\R)$, and $f_*\xi_1+\xi_2\circ f$ is regarded as an element of
$\Gamma(f^*T\R)$, which is identified with $C(M)$.  The 
\textit{codimension}
\indexsymdef{cdimf}{$\cdim(f)$}$\cdim(f)$ of $f$ 
is defined to be the real codimension of the image of $D$
in~$C(M)$. As will be seen shortly, the codimension of $f$ tells the real
codimension of the $\Diff(M)\times \Diff_K(\R)$-orbit of $f$ in~$C(M)$.

Suppose that $f$ has finite codimension $c$. In Section~7.2 of
\cite{Sergeraert}, a method is given for computing $\cdim(f)$ using the
critical points of $f$.  Fix a critical point $a$ of $f$, with critical
value $f(a)=b$.  Consider $D_a\colon \Gamma_a(TM) \times C_b(\R)\to
C_a(M)$, where a subscript as in $\Gamma_a(TM)$ indicates the germs at $a$
of $\Gamma(TM)$, and so on. Notice that the codimension of the image
of $D_a$ is finite, indeed it is at most $c$.

Let $A$ denote the ideal $f_*\Gamma_a(TM)$ of $C_a(M)$. This can be
identified with the ideal in $C_a(M)$ generated by the partial derivatives
of $f$. An argument using Nakayama's Lemma \cite[p.~645]{Sergeraert} shows
that $A$ has finite codimension in $C_a(M)$, and that some power of
$f(x)-f(a)$ lies in $A$. Define $\cdim(f,a)$ to be the dimension of
$C_a(M)/A$, and $\dim(f,a,b)$ to be the smallest $k$ such that
$(f(x)-f(a))^k\in A$.

Here is what these are measuring. The ideal $A$ tells what local
deformations of $f$ at $a$ can be achieved by precomposing $f$ with a
diffeomorphism of $M$ (near $1_M$), thus $\cdim(f,a)$ measures the
codimension of the $\Diff(M)$-orbit of the germ of $f$ at $a$.  The
additional local deformations of $f$ at $a$ that can be achieved by
postcomposing with a diffeomorphism of $\R$ (again, near $1_{\R}$) reduce
the codimension by $k$, basically because Taylor's theorem shows that the
germ at $a$ of any $\xi_2(f(x))$ can be written in terms of the powers
$(f(x)-f(a))^i$, $i<k$, plus a remainder of the form $K(x)(f(x)-f(a))^k$,
which is an element of the ideal $A$. Thus $\cdim(f,a)-\dim(f,a,b)$ is the
codimension of the image of $D_a$. For a noncritical point or a stable
critical point such as $f(x,y)=x^2-y^2$ at $(0,0)$, this local codimension
is $0$, but for unstable critical points it is positive.

Now, let $\dim(f,b)$ be the maximum of $\dim(f,a,b)$, taken over the
critical points $a$ such that $f(a)=b$ (put $\dim(f,b)=0$ if $b$ is not a
critical value). The codimension of $f$ is then $\sum_{a\in
M}\cdim(f,a)-\sum_{b\in\R}\dim(f,b)$.

Here is what is happening at each of the finitely many critical values $b$
of $f$. Let $a_1,\ldots\,$, $a_\ell$ be the critical points of $f$ with
$f(a_i)=b$, and for each $i$ write $f_i$ for the germ of $f-f(a_i)$ at
$a_i$. Consider the element $(f_1,\ldots,f_\ell)\in
C_{a_1}(M)/A_1\oplus\cdots\oplus C_{a_\ell}(M)/A_\ell$. The integer
$\dim(f,b)$ is the smallest power of $(f_1,\ldots,f_\ell)$ that is trivial
in $C_{a_1}(M)/A_1\oplus\cdots\oplus C_{a_\ell}(M)/A_\ell$.  The sum
$\sum_i \cdim(f,a_i)$ counts how much codimension of $f$ is produced by the
inability to achieve local deformations of $f$ near the $a_i$ by
precomposing with local diffeomorphisms at the $a_i$. This codimension is
reduced by $\dim(f,b)$, because the germs of the additional deformations
that can be achieved by postcomposition with diffeomorphisms of $\R$ near
$b$ are the linear combinations of $(1,\ldots,1)$, $(f_1,\ldots,f_\ell)$,
$(f_1^2,\ldots,f_\ell^2),\ldots\,$, $(f_1^{k-1},\ldots,f_\ell^{k-1})$. Thus
the contribution to the codimension from the critical points that map to
$b$ is $\sum_i\cdim(f,a_i)-\dim(f,b)$, and summing over all critical values
gives the codimension of $f$.\index{codimension!of real-valued function|)}

\subsection{The stratification of $\Maps(M,\R)$ by codimension}
\label{subsec:Sergeraert2}\index{stratification!of space of real-valued functions|(}

The functions whose codimension is finite and equal to $n$ form the stratum
$\mathcal{F}_n$. In particular, $\mathcal{F}_0$ are the Morse functions,
$\mathcal{F}_1$ are the functions either having all critical points stable
and exactly two with the same critical value, or having distinct critical
values and all critical points stable except one which is a birth-death
point. Moving to higher strata occurs either from more critical points
sharing a critical value, or from the appearance of more singularities of
positive but still finite local codimension.

We use the natural notations $\mathcal{F}_{\geq n}$ for $\cup_{i\geq
n}\mathcal{F}_i$, $\mathcal{F}_{> n}$ for $\cup_{i > n}\mathcal{F}_i$,
and so on. In particular, $\mathcal{F}_{\geq 0}$ is the set of all
elements of $C(M)$ of finite codimension, and $\mathcal{F}_{> 0}$ is the
set of all elements of finite codimension that are not Morse functions.

The main results of \cite{Sergeraert} (in particular, Theorem~8.1.1 and
Theorem~9.2.4) show that the Sergeraert stratification is locally trivial,
in the following sense.
\begin{theorem}[Sergeraert]\indexstate{Sergeraert!Stratification Theorem} Suppose that
$f\in \mathcal{F}_n$. Then there is a neighborhood $V$ of $f$ in
$C(M)$ of the form $U\times \R^n$, where
\begin{enumerate}
\item $U$ is a neighborhood of $1$ in $\Diff(M)\times \Diff_K(\R)$, and
\item
there is a stratification $\R^n=\cup_{i=0}^n F_i$, such that
$\mathcal{F}_i\cap V=U\times F_i$.
\end{enumerate}
\label{thm:Sergeraert}
\end{theorem}

The inner workings of this result are as follows. Select elements
$f_1,\ldots,\,$ $f_n\in C(M)$ that represent a basis for the quotient of
$C(M)$ by the image of the differential $D$ of $\Phi$ at $(1_M,1_{\R})$.  For
$x=(x_1,\ldots,x_n)\in \R^n$, the function $g_x=f+\sum_{i=1}^n x_i f_i$ is
an element of $C(M)$. If the $x_i$ are chosen in a sufficiently small ball
around $0$, which is again identified with $\R^n$, then these $g_x$ form a
copy $E$ of $\R^n$ ``transverse'' to the image of $\Phi$. Then, $F_i$ is
defined to be the intersection $E\cap \mathcal{F}_i$. A number of subtle
results on this local structure and its relation to the action of
$\Diff(M)\times \Diff_K(\R)$ are obtained in \cite{Sergeraert}, but we will
only need the local structure we have described here.

We remark that $F_n$ is not necessarily just $\{0\}\in \R^n$, that is, the
orbit of $f$ under $\Diff(M)\times \Diff_K(\R)$ might not fill up the
stratum $\mathcal{F}_n$ near $f$. This result, due to H. Hendriks
\index{Hendriks}\cite{Hendriks}, has been interpreted as saying that the Sergeraert
stratification of $C(M)$ is \textit{not} locally trivial (a source of some
confusion), or that it is ``pathological'' (which we find far too
pejorative). 

Denoting $\cup_{i\geq 1}F_i$ by $F_{\geq 1}$, we have the following key
technical result.

\begin{proposition}
For some coordinates on $E$ as $\R^n$, there are a neighborhood $L$ of
$0$ in $\R^n$ and a nonzero polynomial $p$ on $\R^n$ such that $p(L\cap
F_{\geq 1})=0$.
\label{prop:Sergeraert}
\end{proposition}
\begin{proof}
We will begin with a rough outline of the proof. Using Lemma~\ref{lem:FST},
we may choose local coordinates at the critical points of $f$ for which $f$
is polynomial near each critical point. We will select the $f_i$ in the
construction of the transverse slice $E=\R^n$ to be polynomial on these
neighborhoods.  Now $F_{\geq 1}$ consists exactly of the choices of
parameters $x_i$ for which $f+\sum x_if_i$ is not a Morse function, since
they are the intersection of $E$ with $\mathcal{F}_{\geq 1}$. We will show
that they form a semialgebraic set. But $F_{\geq1}$ has no interior, since
otherwise (using Theorem~\ref{thm:Sergeraert}) the subset of Morse
functions $\mathcal{F}_0$ would not be dense in $C(M)$. So
Lemma~\ref{lem:semialgebraic} shows that $F_{\geq 1}$ lies in the zero set
of some nontrivial polynomial.

Now for the details. Recall that $m$ denotes the dimension of $M$. Consider
a single critical value $b$, and let $a_1,\ldots\,$, $a_\ell$ be the
critical points with $f(a_i)=b$. Fix coordinate neighborhoods $U_i$ of the
$a_i$ with disjoint closures, so that $a_i$ is the origin $0$ in $U_i$.  By
Lemma~\ref{lem:FST}, $f$ is finitely $\mathcal{R}$-determined near each
critical point, so on each $U_i$ there is a germ $\varphi_i$ of a
diffeomorphism at $0$ so that $f\circ \varphi_i$ is the germ of a
polynomial. That is, by reducing the size of the $U_i$ and changing the
local coordinates, we may assume that on each $U_i$, $f$ is a polynomial
$p_i$. As explained in Subsection~\ref{subsec:Sergeraert1}, the
contribution to the codimension of $f$ from the $a_i$ is the dimension of
the quotient
\[ Q_b=\big(\oplus_{i=1}^\ell C_{a_i}(U_i)/A_i\big)/B\]
where $B$ is the vector subspace spanned by $\{1,
(p_1(x)-b,\ldots,p_\ell(x)-b),\ldots,\allowbreak
((p_1(x)-b)^{k-1},\ldots,(p_\ell(x)-b)^{k-1})\}$. Choose $q_{i,j}$, $1\leq
j\leq n_i$, where $q_{i,j}$ is a polynomial on $U_i$, so that the germs of
the $q_{i,j}$ form a basis for $Q_b$. Fix vector spaces $\Lambda_i\cong
\R^{n_i}=\{(x_{i,1},\ldots,x_{i,n_i})\}$; these will eventually be some of
the coordinates on~$E$.

In each $U_i$, select round open balls $V_i$ and $W_i$ centered at $0$ so
that $W_i\subset\overline{W_i} \subset V_i\subset \overline{V_i}\subset
U_i$.  We select them small enough so that the closures in $\R$ of their
images under $f$ do not contain any critical value except for $b$. Fix a
smooth function $\mu\colon M\to [0,1]$ which is $1$ on $\cup
\overline{W_i}$ and is $0$ on $M-\cup V_i$, and put $f_{i,j}=\mu\cdot
q_{i,j}$, a smooth function on all of $M$. Now choose a product $L=\prod_i
L_i$, where each $L_i$ is a round open ball centered at $0$ in $\Lambda_i$,
small enough so that if each $(x_{i,1},\ldots, x_{i,n_i})\in L_i$, then
each critical point of $f+\sum x_{i,j} f_{i,j}$ either lies in $\cup W_i$,
or is one of the original critical points of $f$ lying outside of~$\cup
U_i$.

We repeat this process for each of the finitely many critical values of
$f$, choosing additional $W_i$ and $L_i$ so small that all critical points
of $f+\sum x_{i,j} f_{i,j}$ lie in $\cup W_i$. That is, these perturbations
of $f$ are so small that each of the original critical points of $f$ breaks
up into critical points that lie very near the original one and far from
the others.

The sum of all $n_i$ is now $n$. We again use $\ell$ for the number of
$U_i$, and write $\Lambda$ and $L$ for the direct sum of all the
$\Lambda_i$ and the product of all the $L_i$ respectively. For $x\in
L$, write $g_x=f+\sum x_{i,j}f_{i,j}$. It remains to show that the set of
$x$ for which $g_x$ is not a Morse function--- that is, has a critical
point with zero Hessian or has two critical points with the same value---
is contained in a union of finitely many semialgebraic sets.

Denote elements of $W_i$ by $\overline{u_i} = (u_{i,1},\ldots,u_{i,m})$,
and similarly for elements $\overline{x_i}$ of $L_i$.  Define $G_i\colon
W_i\times L_i\to \R$ by $G_i(\overline{u_i},\overline{x_i}) =
p_i(\overline{u_i})+\sum_{j=1}^{n_i}x_{i,j}q_{i,j}(\overline{u_i})$. Note
that for $x=(\overline{x_1},\ldots,\overline{x_\ell})$,
$(G_i)_{\overline{x_i}}$ is exactly the restriction of $g_x$ to $W_i$.

We introduce one more notation that will be convenient.  For $X\subseteq
L_i$ define $E(X)$ to be the set of all
$(\overline{x_1},\ldots,\overline{x_\ell})$ in $L$ such that
$\overline{x_i}\in X$. When $X$ is a semialgebraic subset of $L_i$, $E(X)$
is a semialgebraic subset of $L$. Similarly, if $X\times Y\subseteq
L_i\times L_j$, we use $E(X\times Y)$ to denote its extension to a subset
of $L$, that is, $E(X)\cap E(Y)$.

For each $i$, let $S_i$ be the set of all $(\overline{u_i},\overline{x_i})$
in $W_i\times L_i$ such that $\partial G_i/\partial u_{i,j} $ for
$1\leq j\leq n_i$ all vanish at $(\overline{u_i},\overline{x_i})$, that is,
the pairs such that $\overline{u_i}$ is a critical point of
$(G_i)_{\overline{x_i}}$.  Since $S_i$ is the intersection of an algebraic
set in $\R^m\times \R^{n_i}$ with $W_i\times L_i$, and the latter are round
open balls, $S_i$ is semialgebraic. Let $H_i$ be the set of all
$(\overline{u_i},\overline{x_i})$ in $W_i\times L_i$ such that the Hessian
of $(G_i)_{\overline{x_i}}$ vanishes at $\overline{u_i}$, again a
semialgebraic set. The intersection $H_i\cap S_i$ is the set of all
$(\overline{u_i},\overline{x_i})$ such that $(G_i)_{\overline{x_i}}$ has an
unstable critical point at $\overline{u_i}$. By the 
\index{Tarski-Seidenberg Theorem}Tarski-Seidenberg
Theorem, its projection to $L_i$ is a semialgebraic set, which we will
denote by $X_i$. The union of the $E(X_i)$, $1\leq i\leq \ell$, is
precisely the set of $x$ in $L$ such that $g_x$ has an unstable critical
point.

\newpage
Now consider $G_i\times G_i\colon S_i\times S_i-\Delta_i\to \R^2$, where
$\Delta_i$ is the diagonal in $S_i\times S_i$.  Let
$\widetilde{Y_i}=(G_i\times G_i)^{-1}(\Delta_{\R^2})$, where
$\Delta_{\R^2}$ is the diagonal of $\R^2$. Now, let $\Delta_i'$ be the set
of all $((\overline{u_i},\overline{x_i}),
(\overline{u_i}',\overline{x_i}'))$ in $W_i\times L_i\times W_i\times L_i$
such that $\overline{x_i}=\overline{x_i}'$. Then the projection of
$\widetilde{Y_i}\cap \Delta_i'$ to its first two coordinates is the set of
all $(\overline{u_i},\overline{x_i})$ in $W_i\times L_i$ such that
$\overline{u_i}$ is a critical point of $(G_i)_{\overline{x_i}}$ and
$(G_i)_{\overline{x_i}}$ has another critical point with the same value.
The projection to the second coordinate alone is the set $Y_i$ of
$\overline{x_i}$ for which $(G_i)_{\overline{x_i}}$ has two critical points
with the same value.

Finally, for $i\neq j$, consider $G_i\times G_j\colon S_i\times S_j\to
\R^2$ and let $\widetilde{Y_{i,j}}$ be the inverse image of $\Delta_{\R^2}$. Let
$Y_{i,j}$ be the projection of $\widetilde{Y_{i,j}}$ to a subset of
$L_i\times L_j$.  The union of the $E(Y_i)$ and the $E(Y_{i,j})$ is
precisely the set of all $x$ such that $g_x$ has two critical points with
the same value. Since these are semialgebraic sets, the proof is complete.
\end{proof}

Here is the main result of this subsection.
\begin{theorem}
Let $M$ and $W$ be compact smooth manifolds. Then for a residual set of
smooth maps $F$ from $I\times W$ to $\Maps(M,\R)$, the following hold.
\begin{enumerate}
\item[(i)] $F(I\times W)\subset \mathcal{F}_{\geq 0}$.
\item[(ii)] Each $F_u^{-1}(\mathcal{F}_{>0})$ is finite.
\end{enumerate}
\label{thm:Morse weak transversality}
\end{theorem}

\begin{proof}
Start with a smooth map $G\colon I\times W\to \Maps(M,\R)$.  Regarding
it as a parameterized family of maps $M\times (I\times W)\to \R$, we apply
Theorem~\ref{thm:Bruce2} to perturb $G$ so that each $G_u$ is weakly
transverse to the points of $\R$. By
Proposition~\ref{prop:FST_codimension}, this implies that $G(I\times
W)\subset \mathcal{F}_{\geq 0}$. Since $I\times W$ is compact, $G(I\times
W)\subset \mathcal{F}_{\leq n}$ for some~$n$.

For each $f\in \mathcal{F}_{>0}$, choose a neighborhood $V_f=U_f\times
\R^n$ as in Theorem~\ref{thm:Sergeraert}. Using
Proposition~\ref{prop:Sergeraert}, we may select a neighborhood $L_f$ of
$0$ in $\R^n$ and a nonzero polynomial $p_f\colon L_f\to \R$ such that
$p_f(L\cap F_{i\geq 1})=0$.
\longpage\longpage

Now, partition $\I$ into subintervals and triangulate $W$ so that for each
subinterval $J$ and each simplex $\Delta$ of maximal dimension in the
triangulation, $G(J\times \Delta)$ lies either in $\mathcal{F}_0$ or in
some $U_f\times L_f$. Fix a particular $J\times \Delta$.  If $G(J\times
\Delta)$ lies in $\mathcal{F}_0$, do nothing. If not, choose $f$ so that
$G(J\times \Delta)$ lies in $U_f\times L_f$.  Let $\pi\colon U_f\times
L_f\to L_f$ be the projection, so that $p_f\circ \pi(U_f\times
F_{\geq 1})=0$. By Proposition~\ref{prop:polynomial_Bruce}, we
may perturb $G\vert_{J\times \Delta}$ (changing only its $L_f$-coordinate
in $U_f\times L_f$) so that for each $u\in \Delta$,
$G_u\vert_J^{-1}(\mathcal{F}_{i\geq 1})$ is finite, and any map
sufficiently close to $G\vert_{J\times \Delta}$ on $J\times \Delta$ will
have this same property. As usual, of course, this is extended to a
perturbation of~$G$.

This process can be repeated sequentially on the remaining $J\times
\Delta$. The perturbations must be so small that the property of having
each $G_u\vert_J^{-1}(\mathcal{F}_{i\geq 1})$ finite is not lost on
previously considered sets. When all $J\times \Delta$ have been considered,
each $G_u^{-1}(\mathcal{F}_{i\geq 1})$ is finite.
\end{proof}\index{stratification!of space of real-valued functions|)}

\subsection{Border label control}\index{border label control|(}
\label{sec:borderlabel}
We now return to the case when $M$ is a closed $3$-manifold, as in the
introduction of Section~\ref{sec:generalposition}. In this subsection, we
will obtain a deformation of $f\colon M\times W\to M$ for which some
$\epsilon$ gives border label control.

We begin by ensuring that no $f_u$ carries a component of the singular set
$T$ of $\tau$ into $T$. Consider two circles $C_1$ and $C_2$ embedded in
$M$. By Theorem~\ref{thm:Bruce1}, applied with $A=C_1\times W$, $B=M$, and
$C=C_2$, we may perturb $f\vert_{C_1\times W}$ so that for each $u\in W$,
$f_u\vert_{C_1}$ meets $C_2$ in only finitely many points.

Recall that $T$ consists of smooth circles and arcs in $M$. Each arc is
part of some smoothly embedded circle, so $T$ is contained in a union
$\cup_{i=1}^n C_i$ of embedded circles in $M$. By a sequence of
perturbations as above, we may assume that at each $u$, each $f_u(C_i)$
meets each $C_j$ in a finite set (including when $i=j$), so that $f_u(T)$
meets $T$ in a finite set.

The next potential problem is that at some $u$, $f_u(T_0)$ or $f_u(T_1)$
might be contained in a single level $P_t$. Recall that the notation
$R(s,t)$, introduced in Section~\ref{sec:RSgraphic}, means
$\tau^{-1}([s,t])$.  For some $\delta>0$, every $f_u(T_0)$ meets
$R(3\delta,1-3\delta)$, since otherwise the compactness of $W$ would lead
to a parameter $u$ for which $f_u(T_0)\subset T$.  Let $\phi\colon
R(\delta,1-\delta)\to [\delta,1-\delta]$ be the restriction of the map
$\pi(\tau(x,t))=t$. This $\phi$ makes $R(\delta,1-\delta)$ a bundle with
fibers that are level tori. As before, let $C_1$ be one of the circles
whose union contains $T$. Only the most superficial changes are needed to
the proof of Theorem~\ref{thm:Bruce2} given in \cite{Bruce} so that it
applies when $\phi$ is a bundle map defined on a codimension-zero
submanifold of $B$ rather than on all of $B$; the only difference is that
the subsets of jets which are to be avoided are defined only at points of
the subspace rather than at every point of $B$. Using this slight
generalization of Theorem~\ref{thm:Bruce2} (and as usual, the Parameterized
Extension Principle), we perturb $f$ so that each $f_u\vert_{C_1}$ is
weakly transverse to each $P_t$ with $\delta\leq t\leq 1-\delta$. Since
$C_1$ is $1$-dimensional, weakly transverse implies that $f_u(C_1)$ meets
each such $P_t$ in only finitely many points.  Repeating for the other
$C_i$, we may assume that each $f_u(T_0)$ meets the $P_t$ with $\delta\leq
t \leq 1-\delta$ in only finitely many points. We also choose the
perturbations small enough so that each $f_u(T_0)$ still meets
$R(2\delta,1-2\delta)$. So $f_u^{-1}(P_t)\cap T_0$ is nonempty and finite
at least some~$t$.  In particular, $\pi(f_u(T_0))$ contains an open set, so
by \index{Sard's Theorem}Sard's Theorem applied to $\pi\circ
f_u\vert_{T_0}$, for each $u$, there is an $r$ so that $f_u(T_0)$ meets
$P_r$ transversely in a nonempty set (we select $r$ so that $P_r$ does not
contain the image of a vertex of $T_0$). For a small enough $\epsilon$, a
component of $X_s\cap P_r$ will be a compressing disk of $X_s$ whenever
$s\leq 2\epsilon$, and by compactness of $W$, there is an $\epsilon$ such
that for every $u$, there is a level $P_r$ such that some component of
$X_s\cap P_r$ contains a compressing disk of $X_s$ whenever $s\leq
2\epsilon$.

Applying the same procedure to $T_1$, we may assume that for every $u$,
there there is a level $P_r$ such that some component of $Y_s\cap P_r$ is a
compressing disk of $Y_s$ whenever $s\geq 1-2\epsilon$.

Let $h\colon M\times W\to M$ be defined by $h(x,u)=f_u^{-1}(x)$.  Fix new
sweepouts on the $M\times \{u\}$, given by $f_u\circ \tau$, so that $h_u$
carries the levels of this sweepout to the original $P_t$. Applying the
previous procedure to $h$, making sure that all perturbations are small
enough to preserve the conditions developed for $f$, and perhaps making
$\epsilon$ smaller, we may assume that for each $u$, there is a level $Q_r$
such that $V_t\cap Q_r$ is a compressing disk of $V_t$ whenever $t\leq
2\epsilon$, and a similar $Q_r$ for $W_t$ with $t\geq 1-2\epsilon$. Thus
the number $\epsilon$ gives border label control for $f$. Since border
label control holds, with the same $\epsilon$, for any map sufficiently
close to $f$, we may assume it is preserved by all future
perturbations.\index{border label control|)}

\subsection{Building the graphics}
It remains to deform $f$ to satisfy conditions\index{GP1@(GP1), (GP2), (GP3)}
(GP1), (GP2), and (GP3).  As
before, let $i\colon I\to \R$ be the inclusion, and consider the smooth map
$i\circ \pi \circ f\circ (\tau\times 1_W)\colon P\times \I\times W\to \R$.
Regard this as a family of maps from $\I$ to $\Maps(P,\R)$, parameterized
by $W$. Apply Theorem~\ref{thm:Morse weak transversality} to obtain a
family $k\colon P\times \I\times W\to \R$. For each $I\times \{u\}$, there
will be only finitely many values of $s$ in $\I$ for which the restriction
$k_{(s,u)}$ of $k$ to $P\times \{s\}\times \{u\}$ is not a Morse
function. At these levels, the projection from $Q_s$ into the transverse
direction to $P_t$ is an element of some $\mathcal{F}_n$, so each tangency
of $Q_s$ with $P_t$ looks like the graph of a critical point of finite
multiplicity. This will ultimately ensure that condition (GP1) is attained
when we complete our deformations of~$f$.

We will use $k$ to obtain a deformation of the original $f$, by moving
image points vertically with respect to the levels of the range. This would
not make sense where the values of $k$ fall outside $(0,1)$, so the motion
will be tapered off so as not to change $f$ at points that map near $T$. It
also would not be well-defined at points of $T\times W$, so we taper off
the deformation so as not to change $f$ near $T\times W$. The fact that $f$
is unchanged near $T\times W$ and near points that map to $T$ will not
matter, since border label control will allow us to ignore these regions in
our later work.

Regard $P\times \I\times W$ as a subspace of $P\times \R\times W$. For
each $(x,r,u)\in P\times \I\times W$, let $w_{(x,r,u)}'$ be
$k(x,r,u)-i\circ \pi \circ f_u\circ \tau(x,r)$, regarded as a tangent vector to
$\R$ at $i\circ \pi\circ f_u\circ \tau(x,r)$.

We will taper off the $w_{(x,r,u)}'$ so that for each fixed $u$ they will
produce a vector field on $M$. Fix a number $\epsilon$ that gives border
label control for $f$, and a smooth function $\mu\colon \R\to I$ which
carries $(-\infty,\epsilon/4]\cup [1-\epsilon/4,\infty)$ to $0$ and carries
$[\epsilon/2,1-\epsilon/2]$ to $1$. Define $w_{(x,r,u)}$ to be
$\mu(r)\,\mu(i\circ \pi \circ f_u\circ \tau(x,r))\,w_{(x,r,u)}'$. These
vectors vanish whenever $r\notin [\epsilon/4,1-\epsilon/4]$ or $i\circ
\pi\circ f_u\circ \tau(x,r,u)\notin [\epsilon/4,1-\epsilon/4]$, that is,
whenever $\tau(x,r)$ or $f_u\circ \tau(x,r)$ is close to $T$.  Using the
map $i\circ \pi\colon M\to \R$, we pull the $w_{(x,r,u)}$ 
back to vectors in $M$ that
are perpendicular to $P_t$; this makes sense near $T$ since the
$w_{(x,r,u)}$ are zero at these points). For each $u$, we obtain at each
point $f_u\circ \tau(x,r)\in M$ a vector $v_{(x,r,u)}$ that points in the
$\I$-direction (i.~e.~is perpendicular to $P_t$) and maps to $w_{(x,r,u)}$
under~$(i\circ \pi)_*$.

If $k$ was a sufficiently small perturbation, the $v_{(x,r,u)}$ define a
smooth map $j_u\colon M\to M$ by $j_u(\tau(x,r))=\Exp(v_{(x,r,u)})$. Put
$g_u=j_u\circ f_u$.  Since $\mu(r)=1$ for $\epsilon/2\leq r \leq
1-\epsilon/2$, we have $i\circ \pi \circ g_u\circ \tau(x,r)=k(x,r,u)$
whenever both $\epsilon/2\leq r\leq 1-\epsilon/2$ and $\epsilon/2\leq
i\circ \pi \circ f_u\circ \tau(x,r)\leq 1-\epsilon/2$. The latter condition
says that $f_u\circ \tau(x,r)$ is in $P_s$ for some $\epsilon/2\leq s \leq
1-\epsilon/2$. Assuming that $k$ was close enough to $i\circ \pi \circ
f\circ (\tau\times 1_W)$ so that each $\pi\circ g_u\circ \tau(x,r)$ is
within $\epsilon/4$ of $\pi\circ f_u\circ \tau(x,r)$, the equality $i\circ
\pi \circ g_u\circ \tau(x,r)=k(x,r,u)$ holds whenever $\tau(x,r)$ is in a
$P_s$ and $g_u\circ \tau(x,r)$ is in a $P_t$ with $\epsilon\leq s,t\leq
1-\epsilon$.

Carrying out this construction for a sequence of $k$ that converge to
$i\circ \pi \circ f\circ (\tau\times 1_W)$, we obtain vector fields
$v_{(x,r,u)}$ that converge to the zero vector field. For those
sufficiently close to zero, $g$ will be a deformation of $f$. Choosing $g$
sufficiently close to $f$, we may ensure that $\epsilon$ still gives border
label control for~$g$.

We will now analyze the graphic of $g_u$ on $I^2_\epsilon$. For $s,t\in
[\epsilon,1-\epsilon]$, $\pi\circ g_u(x)$ equals $k_{(s,u)}(x)$ whenever
$x\in P_s$ and $g_u(x)\in P_t$. Therefore the tangencies of $g_u(P_s)$ with
$P_t$ are locally just the graphs of critical points of $k_{(s,u)}\colon
P\to \R$, so $g$ has property~(GP1).

Let $s_1,\ldots\,$, $s_n$, be the values of $s$ in $[\epsilon,1-\epsilon]$
for which $k_{(s_i,u)}\colon P\to \R$ is not a Morse function.  Each
$k_{(s_i,u)}$ is still a function of finite codimension, so has finitely
many critical points. Those critical points having critical values in
$[\epsilon,1-\epsilon]$ produce the points of the graphic of $g_u$ that lie
in the vertical line $s=s_i$, as suggested in Figure~\ref{fig:graphic}.  We
declare the $(s_i,t)$ at which $k_{(s_i,u)}$ has a critical point at $t$ to
be vertices of $\Gamma_u$.

When $s$ is not one of the $s_i$, $k_{(s,u)}$ is a Morse function, so any
tangency of $g_u(P_s)$ with $P_t$ is stable, and there is at most one such
tangency. Since these tangencies are stable, all nearby tangencies are
equivalent to them and hence also stable, so in the graphic for $g_u$ in
$I^2_\epsilon$, the pairs $(s,t)$ corresponding to levels with a single
stable tangency form ascending and descending arcs as suggested in
Figure~\ref{fig:graphic}. These arcs may enter or leave $I^2_\epsilon$, or
may end at a point corresponding to one of the finitely many points of the
graphic with $s$-coordinate equal to one of the $s_i$.  We declare the
intersection points of these arcs with $\partial I_\epsilon$ to be vertices
of $\Gamma_u$.  The conditions (GP2) and (GP3) have been achieved,
completing the proof of Theorem~\ref{thm:generalposition}.
\index{figures!figure94@ascending and descending arcs in a graphic}%
\begin{figure}
\labellist
\pinlabel $s_i$ at 66 -9
\pinlabel $s_{i+1}$ at 188 -9
\pinlabel $s$ at 264 -3
\endlabellist
\includegraphics[width=6.5cm]{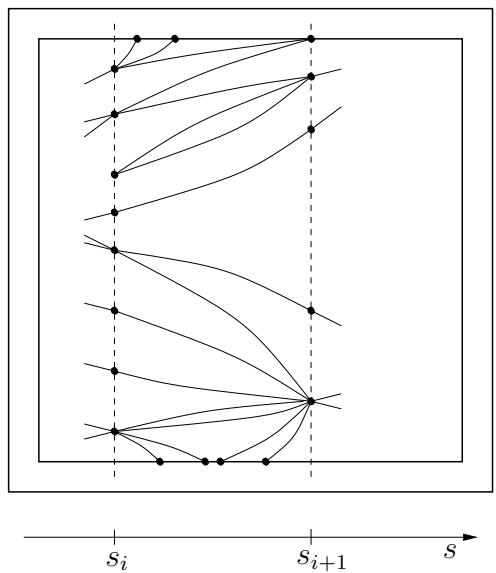}
\caption{A portion of the graphic of $g_u$.}
\label{fig:graphic}
\end{figure}

\section[Finding good regions]
{Finding good regions}
\label{sec:goodregions}

In this section, we adapt the arguments of
Section~\ref{sec:Rubinstein-Scharlemann} to general position families.  The
graphics associated to the $f_u$ of a general position family $f\colon
M\times W\to M$ satisfy property\index{RS1@(RS1), (RS2), (RS3)}
(RS\ref{item:noAB}) of
Section~\ref{sec:Rubinstein-Scharlemann} (provided that the Heegaard
splittings associated to the sweepout are strongly irreducible) and
property (RS\ref{item:edge labeling}) (since the open edges of the $\Gamma$
correspond to pairs of levels that have a single stable tangency, see the
remark after the definition of (RS\ref{item:edge labeling}) in
Section~\ref{sec:Rubinstein-Scharlemann}), but not property
(RS\ref{item:2-cell labeling}). Indeed, property (RS\ref{item:2-cell
labeling}) does not even make sense, since the vertices of $\Gamma$ can
have high valence. Property (RS\ref{item:noAB}) is what allows the map
from the $0$-cells of $K$ to the $0$-simplices of the Diagram to be
defined. Property (RS\ref{item:edge labeling}) (plus conditions on regions
near $\partial K$, which we will still have due to border label control)
allows it to be extended to a cellular map from the $1$-skeleton of $K$ to
the $1$-skeleton of the Diagram. What ensures that it still extends to the
$2$-cells is a topological fact about pairs of levels whose intersection
contains a common spine, Lemma~\ref{lem:common spine}. Because it involves
surfaces that do not meet transversely, its proof is complicated and
somewhat delicate. Since the proof does not introduce any ideas needed
elsewhere, the reader may wish to skip it on a first reading, and go
directly from the statement of Lemma~\ref{lem:common spine} to the last
four paragraphs of the section.

We specialize to the case of a parameterized family $f\colon L\times W\to
L$, where $L$ is a lens space and $W$ is a compact manifold. We retain the
notations $P_t$, $Q_s$, $V_t$, $W_t$, $X_s$, and $Y_s$ of
Section~\ref{sec:generalposition}. As was mentioned above, properties
(RS\ref{item:noAB}) and (RS\ref{item:edge labeling}) already hold for the
labels of the regions of the graphic of each~$f_u$.

\begin{theorem}
Suppose that $f\colon L\times W\to L$ is in general position with respect
to~$\tau$. Then for each $u$, there exists $(s,t)$ such that $Q_s$ meets
$P_t$ in good position.\par
\label{thm:finding good levels}
\end{theorem}
\noindent The proof of Theorem~\ref{thm:finding good levels} will
constitute the remainder of this section.

We first prove a key geometric lemma that is particular to lens spaces.
\begin{lemma}
Let $f\colon L \times W\to L$ be a parameterized family of diffeomorphisms
in general position, and let $(s,t)\in I^2_\epsilon$. If $Q_s\cap P_t$
contains a spine of $P_t$, then either $V_t$ or $W_t$ contains a core
circle which is disjoint from $Q_s$.\par
\label{lem:common spine}
\end{lemma}

\begin{proof}
We will move $Q_s$ by a sequence of isotopies of $L$. All isotopies will
have the property that if $V_t-Q_s$ (or $W_t-Q_s$) did not contain a core
circle of $V_t$ (or $W_t$) before the isotopy, then the same is true after
the isotopy.  We say this succinctly with the phrase that the isotopy
\indexdef{core circles!does not create}\textit{does not create core circles.}
Typically some of the isotopies will not be smooth, so we work in the PL
category. At the end of an initial ``flattening'' isotopy, $Q_s$ will
intersect $P_t$ nontransversely in a $2$-dimensional simplicial complex $X$
in $P_t$ whose frontier consists of points where $Q_s$ is PL embedded but
not smoothly embedded.  A sequence of simplifications called tunnel moves
and bigon moves, plus isotopies that push disks across balls, will make
$Q_s\cap P_t$ a single component $X_0$, which will then undergo a few
additional improvements.  After this has been completed, an Euler
characteristic calculation will show that a core circle disjoint from the
repositioned $Q_s$ exists in either $V_t$ or $W_t$, and consequently one
existed for the original~$Q_s$.

\index{flattening}The first step is to perform a so-called ``flattening''
isotopy. Such isotopies were already described in detail in
Lemma~\ref{flatten}, but we will give a self-contained construction here.

Since $f$ is in general position, $Q_s\cap P_t$ is a $1$-complex satisfying
the property\index{GP1@(GP1), (GP2), (GP3)}~(GP1) of
Section~\ref{sec:generalposition}. Each isolated vertex of $Q_s\cap P_t$ is
an isolated tangency of $Q_s\cap P_t$, so we can move $Q_s$ by a small
isotopy near the vertex to eliminate it from the intersection. After this
step, $Q_s\cap P_t$ is a graph $\Gamma$ which contains a spine of $Q_s\cap
P_t$, such that each vertex of $\Gamma$ has positive valence.

By property (GP1), each vertex $x$ of $\Gamma$ is a point where $Q_s$ is
tangent to $P_t$, and the edges of $\Gamma$ that emanate from $x$ are arcs
where $Q_s$ intersects $P_t$ transversely. Along each arc, $Q_s$ crosses
from $V_t$ into $W_t$ or vice versa, so there is an even number of these
arcs. Near $x$, the tangent planes of $Q_s$ are nearly parallel to those of
$P_t$, and there is an isotopy that moves a small disk neighborhood of $x$
in $Q_s$ until it coincides with a small disk neighborhood of $x$ in
$P_t$. Perform such isotopies near each vertex of~$\Gamma$. This enlarges
$\Gamma$ in $Q_s\cap P_t$ to the union of $\Gamma$ with a union $E$ of
disks, each disk containing one of the original vertices.

The closure of the portion of $\Gamma$ that is not in $E$ now consists of a
collection of arcs and circles where $Q_s$ intersects $P_t$ transversely,
except at the endpoints of the arcs, which lie in $E$. Consider one of
these arcs, $\alpha$. At points of $\alpha$ near $E$, the tangent planes to
$Q_s$ are nearly parallel to those of $P_t$, and starting from each end
there is an isotopy that moves a small regular neighborhood of a portion of
$\alpha$ in $Q_s$ onto a small regular neighborhood of the same portion of
$\alpha$ in $P_t$. This flattening process can be continued along
$\alpha$. If it is started from both ends of $\alpha$, it may be possible
to flatten all of a regular neighborhood of $\alpha$ in $Q_s$ onto one in
$P_t$. This occurs when the vectors in a field of normal vectors to
$\alpha$ in $Q_s$ are being moved to normal vectors on the same side of
$\alpha$ in $P_t$. If they are being moved to opposite sides, then we
introduce a point where the configuration is as in
Figure~\ref{fig:flattened surfaces}, in which $P_t$ appears as the
$xy$-plane, $\alpha$ appears as the points in $P_t$ with $x=-y$, and $Q_s$
appears as the four shaded half- or quarter-planes. These points will be
called 
\indexdef{crossover point}\textit{crossover} points. Perform such isotopies in disjoint
neighborhoods of all the arcs of $\Gamma-E$. For the components of $\Gamma$
that are circles of transverse intersection points, we flatten $Q_s$ near
each circle to enlarge the intersection component to an annulus.

At the end of this initial process, $\Gamma$ has been been enlarged to a
$2$-complex $X$ in $Q_s\cap P_t$ that is a regular neighborhood of
$\Gamma$, except at the crossover points where $\Gamma$ and $X$ look
locally like the antidiagonal $x=-y$ of the $xy$-plane and the set of
points with $xy\leq 0$. We will refer to $X$ as a 
\indexdef{pinched regular neighborhood}\textit{pinched regular
neighborhood} of~$\Gamma$.

Since $\Gamma$ originally contained a spine of $P_t$, $X$ contains two
circles that meet transversely in one point that lies in the interior (in
$P_t$) of $X$. Therefore $X$ contains a common spine of $P_t$ and~$Q_s$.
Let $X_0$ be the component of $X$ that contains a common spine of $Q_s$ and
$P_t$. All components of $P_t-X_0$ and $Q_s-X_0$ are open disks. Let
$X_1=X-X_0$, and for each $i$, denote $\Gamma\cap X_i$ by~$\Gamma_i$.

The next step will be to move $Q_s$ by isotopy to remove $X_1$ from
$Q_s\cap P_t$. These isotopies will be fixed near $X_0$.  Some of them will
have the effect of joining two components of $V_t-Q_s$ (or of $W_t-Q_s$)
into a single component of $V_t-Q_s$ (or of $W_t-Q_s$) for the repositioned
$Q_s$, so we must be very careful not to create core circles.

The frontier of $X_1$ in $P_t$ is a graph $\Fr(X_1)$ for which each vertex
is a crossover point, and has valence $4$ (as usual, our ``graphs'' can
have open edges that are circles). Its edges are of two types: \textit{up}
edges, for which the component of $\overline{Q_s-X}$ that contains the edge
lies in $W_t$, and \textit{down} edges, for which it lies in $V_t$. At each
disk of $E$, the up and down edges alternate as one moves around $\partial
E$ (see Figure~\ref{fig:updown}). For each of the arcs of $\Gamma_1-E$, the
flattening process creates an up edge on one side and a down edge on the
other, but there is a fundamental difference in the way that the up and
down edges appear in $Q_s$ and $P_t$. As shown in Figure~\ref{fig:updown},
up edges (the solid ones) and down edges (the dotted ones) alternate as one
moves around a crossover point, while in $Q_s$ they occur in adjacent
pairs. This is immediate upon examination of Figure~\ref{fig:flattened
surfaces}.
\index{figures!figure95@up and down edges}%
\begin{figure}
\labellist
\pinlabel $P_t$ at 240 30
\pinlabel $Q_s$ at 740 30
\endlabellist
\includegraphics[width=\textwidth]{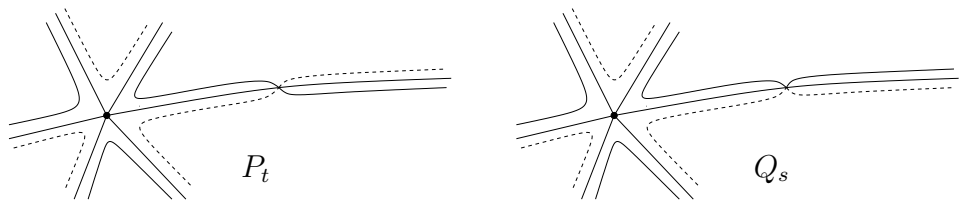}
\caption{Up and down edges of $X$ as they appear in $P_t$ and $Q_s$.}
\label{fig:updown}
\end{figure}

For our inductive procedure, we start with a pinched regular neighborhood
$X_1\subset Q_s\cap P_t$ of a graph $\Gamma_1$ in $Q_s\cap P_t$, all of
whose vertices have positive even valence.  Moreover, the edges of the
frontier of $X_1$ are up or down according to whether the portion of
$\overline{Q_s-X}$ that contains them lies in $W_t$ or $V_t$. We call this
an \textit{inductive configuration.}

To ensure that our isotopy process will terminate, we use the complexity
$-\chi(\Gamma_1)-\chi(\Fr(X_1))+N$, where $N$ is the number of components
of $\Gamma_1$.  Since all vertices of $\Gamma_1$ and $\Fr(X_1)$ have
valence at least~$2$, each of their components has nonpositive Euler
characteristic, so the complexity is a non-negative integer. The remaining
isotopies will reduce this complexity, so our procedure must terminate.

We may assume that the complexity is nonzero, since if $N=0$ then $X_1$ is
empty. Consider $X_1$ as a subset of the union of open disks
$Q_s-X_0$. Since $X_1$ is a pinched regular neighborhood of a graph with
vertices of valence at least~$2$, it separates these disks, and we can find
a closed disk $D$ in $Q_s$ with $\partial D\subset X_1$ and $D\cap
X=\partial D$.  It lies either in $V_t$ or $W_t$. Assume it is in $W_t$
(the case of $V_t$ is similar), in which case all of its edges are up
edges. Since $\partial D\subset P_t-X_0$, $\partial D$ bounds a disk $D'$
in $P_t-X_0$.  Since the interior of $D$ is disjoint from $P_t$, $D\cup D'$
bounds a $3$-ball $\Sigma$ in~$L$. Of course, $D'$ may contain portions of
the component of $X_1$ that contains $\partial D'$, or other components
of~$X_1$. Let $X_1'$ be the component of $X_1$ that contains $\partial D'$;
it is a pinched regular neighborhood of a component $\Gamma_1'$
of~$\Gamma$.

Suppose that $X_1'$ contains some vertices of $\Gamma_1$ of valence more
than $2$. We will perform an isotopy of $Q_s$ that we call a 
\indexdef{tunnel move}\textit{tunnel
move,} illustrated in Figure~\ref{fig:tunnel}, that reduces the complexity
of the inductive configuration. Near the vertex, select an arc in $X_1'$
that connects the edge of $\Fr(X_1')$ in $D'$ with another up edge of
$\Fr(X_1')$ that lies near the vertex (this arc may lie in $D'$, in a
portion of $X_1$ contained in $D'$). An isotopy of $Q_s$ is performed near
this arc, that pulls an open regular neighborhood of the arc in $X_1'$ into
$W_t$. This does not change the interior of $V_t-Q_s$ (it just adds the
regular neighborhood of the arc to $V_t-Q_s$), but in $W_t$ it creates a
tunnel that joins two different components of $W_t-Q_s$.  One of these
components was in $\Sigma$, so the isotopy cannot create core
circles. After the tunnel move, we have a new inductive configuration. The
Euler characteristic of $\Gamma_1$ has been increased by the addition of
one vertex, while $\chi(\Fr(X_1))$ and $N$ are unchanged, so the new
inductive configuration is of lower complexity. The procedure continues by
finding a new $D$ and $D'$ and repeating the process.
\index{figures!figure96@tunnel arc and tunnel move}%
\begin{figure}
\labellist
\pinlabel $D'$ at 167 132
\endlabellist
\includegraphics[width=11cm]{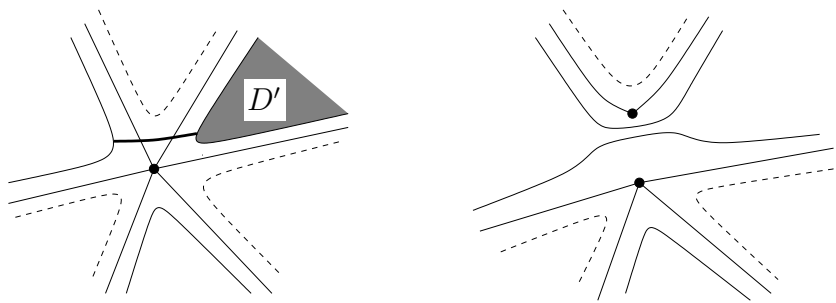}
\caption{A portion of $P_t$ showing a tunnel arc in $X_1$, and the new
$\Gamma_1$ and $X_1$ after the tunnel move.}
\label{fig:tunnel}
\end{figure}

When a $D$ has been found for which no tunnel moves are possible, all
vertices of $\Gamma_1'$ (if any) have valence $2$. Suppose that $X_1'$
contains crossover points. It must contain an even number of them, since up
and down edges at crossover points alternate in $P_t$. Some portion of
$X_1'$ is a disk $B$ whose frontier consists of two crossover points and
two edges of $\Fr(X_1)$, each connecting the two crossover points. We will
use a \indexdef{bigon move}\textit{bigon move} as in the proof of
Proposition~\ref{prop:circles}. A bigon move is an isotopy of $Q_s$,
supported in a neighborhood of $B$, that repositions $Q_s$ and replaces a
neighborhood of $B$ in $X$ with a rectangle containing no crossover
points. Figure~\ref{fig:2-gon isotopy} illustrates this isotopy. It cannot
create core circles, indeed such an isotopy changes the interiors of
$V_t-Q_s$ and $W_t-Q_s$ only by homeomorphism.

Since bigon moves increase the Euler characteristic of $\Fr(X_1)$, without
changing $\Gamma_1$ or $N$, they reduce complexity. So we eventually arrive
at the case when $X_1'$ is an annulus. Assume for now that the interior of
$D'$ is disjoint from $X_1$.  There is an isotopy of $Q_s$ that pushes $D$
across $\Sigma$, until it coincides with $D'$.  This cannot create core
circles, since its effect on the homeomorphism type of $W_t-Q_s$ is simply
to remove the component $\Sigma-Q_s$. Perform a small isotopy that pulls
$D'$ off into the interior of $V_t$, again creating no new core circles. An
annulus component of $X_1$ has been eliminated, reducing the complexity. If
$D'\cap X_1=X_1'$, then a similar isotopy eliminates~$X_1'$.

Suppose now that the interior of $D'$ contains components of $X_1$ other
than perhaps $X_1'$.  Let $X_1''$ be their union. It is a pinched regular
neighborhood of a union $\Gamma_1''$ of components of $\Gamma_1$. If
$\Gamma_1''$ has vertices of valence more than $2$, then tunnel moves can
be performed. These cannot create new core circles, since they do not
change the interior of $V_t-Q_s$, and in $W_t-Q_s$ they only connect
regions that are contained in $\Sigma$. If no tunnel move is possible, but
there are crossover points, then a bigon move may be performed. So we may
assume that every component of $X_1''$ is an annulus.

Let $S$ be a boundary circle of $X_1''$ innermost on $D'$, bounding a disk
$D''$ in $D'$ whose interior is disjoint from $X$. Let $E''$ be the disk in
$Q_s$ bounded by $S$, so that $D''\cup E''$ bounds a $3$-ball $\Sigma''$
in~$L$. Note that $E''$ does not contain $X_0$, since then a spine of $P_t$
would be contained in the $2$-sphere $E''\cup D''$.

We claim that if $(V_t -Q_s)\cup (E''\cap V_t)$ contains a core circle of
$V_t$, then $V_t-Q_s$ contained a core circle of $V_t$ (and analogously for
$W_t$).  The closures of the components of $E''- P_t$ are planar surfaces,
each lying either in $V_t$ or $W_t$. Let $F$ be one of these, lying (say)
in $V_t$. Its boundary circles lie in $P_t-X_0$, so bound disks in $P_t$.
The union of $F$ with these disks is homotopy equivalent to $S^2\vee(\vee
S^1)$ for some possibly empty collection of circles, so a regular
neighborhood in $V_t$ of the union of $F$ with these disks is a punctured
handlebody $Z(F)$ meeting $P_t$ in a union of disks.  Suppose that $C$ is a
core circle in $V_t$ that is disjoint from $Q_s-F$.  We may assume that $C$
meets $\partial Z(F)$ transversely, so cuts through $Z(F)$ is a collection
of arcs. Since $Z(F)$ is handlebody meeting $P_t$ only in disks, there is
an isotopy of $C$ that pushes the arcs to the frontier of $Z(F)$ in $V_t$
and across it, removing the intersections of $C$ with $F$ without creating
new intersections (since the arcs need only be pushed slightly outside of
$Z(F)$). Performing such isotopies for all components of $E''-X_1$ in $V_t$
produces a core circle disjoint from $E''$, proving the claim.

By virtue of the claim, an isotopy that pushes $E''$ across $\Sigma''$
until it coincides with $D''$ does not create core circles. Then, a slight
additional isotopy pulls $D''$ and the component of $X_1$ that contained
$\partial D''$ off of $P_t$, reducing the complexity.

Since we can always reduce a nonzero complexity by one of these isotopies,
we may assume that $Q_s\cap P_t=X_0$. The frontier $\Fr(X_0)$ in $P_t$ is
the union of a graph $\Gamma_2$, each of whose components has vertices of
valence~$4$ corresponding to crossover points, and a graph $\Gamma_3$ whose
components are circles.

A component of $\Gamma_3$ must bound both a disk $D_Q$ in
$\overline{Q_s-X_0}$ and a disk $D_P$ in $\overline{P_t-X_0}$. Since
$Q_s\cap P_t=X_0$, the interiors of $D_P$ and $D_Q$ are disjoint, and $D_Q$
lies either in $V_t$ or in $W_t$. So we may push $D_Q$ across the $3$-ball
bounded by $D_Q\cup D_P$ and onto~$D_P$, without creating core circles.
Repeating this procedure to eliminate the other components of $\Gamma_3$,
we achieve that the frontier of $Q_s\cap P_t$ equals the graph $\Gamma_2$.

Figure~\ref{fig:meridian disks} shows a possible intersection of $Q_s$ with
$P_t$ at this stage. The shaded region is $Q_s\cap P_t$; it is a union of a
(solid) octagon, two bigons, and a square. The closure of $Q_s-(Q_s\cap
P_t)$ consists of two meridian disks in $V_t$, bounded by the circles $C_1$
and $C_2$, and two boundary-parallel disks in $W_t$, bounded by the circles
$C_3$ and $C_4$.
\index{figures!figure97@flattened torus with two meridian disks}%
\begin{figure}
\labellist
\pinlabel $C_1$ at 78 -5
\pinlabel $C_2$ at 175 -5
\pinlabel $C_4$ at 59 59
\pinlabel $C_3$ at 90 125
\endlabellist
\includegraphics[width=8cm]{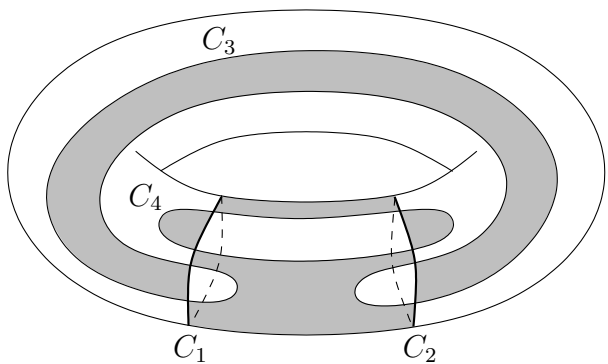}
\caption{A flattened torus containing two meridian disks.}
\label{fig:meridian disks}
\end{figure}\index{examples!example5@flattened torus with two meridian disks}%

Suppose that $Q_s$ now contains $2k_1$ meridian disks of $Q_s$ in $V_t$ and
$2k_2$ meridian disks in $W_t$ (their numbers must be even since $Q_s$ is
zero in $H_2(L)$), and a total of $k_0$ boundary-parallel disks in $V_t$
and $W_t$. Since $\chi(Q_s)=0$, we have $\chi(Q_s\cap P_t)=-k_0-2k_1-2k_2$.
To prove the lemma, we must show that either $k_1$ or $k_2$ is $0$.

Let $V$ be the number of vertices of $\Gamma_2$. Since all of its vertices
have valence $4$, $\Gamma_2$ has $2V$ edges. The remainder of $Q_s\cap P_t$
consists of $2$-dimensional faces. Each of these faces has boundary
consisting of an even number of edges, since up and down edges alternate
around a face.  If some of the faces are bigons, such as two of the faces
in Figure~\ref{fig:meridian disks}, they may be eliminated by bigon moves
(which will also change $V$). These may create additional components of the
frontier of $X_0$ that are circles, indeed this happens in the example of
Figure~\ref{fig:meridian disks}. These are eliminated as before by moving
disks of $Q_s$ onto disks in $P_t$. After all bigons have been eliminated,
each face has at least four edges, so there are at most $V/2$ faces. So we
have $\chi(Q_s\cap P_t)\leq V-2V+V/2=-V/2$.

Each boundary-parallel disk in $Q_s\cap V_t$ or $Q_s\cap W_t$ contributes
at least two vertices to the graph, since at each crossover point, $X_0$
crosses over to the other side in $P_t$ of the boundary of the disk. This
gives at least $2k_0$ vertices. The meridian disks on the two sides
contribute at least $2k_1\cdot 2k_2\cdot m$ additional vertices, where
$L=L(m,q)$, since the meridians of $V_t$ and $W_t$ have algebraic
intersection $\pm m$ in $P_t$. Thus $V\geq 2k_0+4k_1k_2m$. We calculate:
\[-k_0-2k_1-2k_2=\chi(Q_s\cap P_t)\leq -V/2\leq -k_0 -2k_1k_2m\ .\]
Since $m>2$, this can hold only when either $k_1$ or $k_2$ is $0$.
\end{proof}

Lemma~\ref{lem:common spine} fails (at the last sentence of the proof) for
the case of $L(2,1)$. Indeed, there is 
a\index{examples!example6@flattened torus in $L(2,1)$ with two meridian disks}%
flattened Heegaard torus in
$L(2,1)$ which meets $P_{1/2}$ in four squares and has two meridian disks
on each side. In a sketch somewhat like that of Figure~\ref{fig:meridian
disks}, the boundaries of these disks are two meridian circles and two
$(2,1)$-loops intersecting in a total of $8$ points, and cutting the torus
into $8$ squares. There are two choices of four of these squares to
form~$Q_s\cap P_t$.

Now, we will complete the proof of Theorem~\ref{thm:finding good
levels}. As in Section~\ref{sec:Rubinstein-Scharlemann}, assume for
contradiction that all regions are labeled, and triangulate $I^2_\epsilon$.
The map on the $1$-skeleton is defined exactly as in
Section~\ref{sec:Rubinstein-Scharlemann}, using
Lemma~\ref{lem:borderlabelconstraints} and the fact that the labels satisfy
property~(RS\ref{item:edge labeling})\index{RS1@(RS1), (RS2), (RS3)}. Using
Lemma~\ref{lem:borderlabelconstraints}, each $1$-cell maps either to a
$0$-simplex or a $1$-simplex of the Diagram, and exactly as before the
boundary circle of $K$ maps to the Diagram in an essential way. The
contradiction will be achieved once we show that the map extends over the
$2$-cells.

There is no change from before when the $2$-cell meets $\partial K$ or lies
in the interior of $K$ but does not contain a vertex of $\Gamma$, so we fix
a $2$-cell in the interior of $K$ that is dual to a vertex $v_0$ of
$\Gamma$, located at a point~$(s_0,t_0)$.

Suppose first that $Q_{s_0}\cap P_{t_0}$ contains a spine of $P_{t_0}$. By
Lemma~\ref{lem:common spine}, either $V_{t_0}$ or $W_{t_0}$ has a core
circle $C$ which is disjoint from $Q_{s_0}$; we assume it lies in
$V_{t_0}$, with the case when it lies in $W_{t_0}$ being similar.  The
letter $A$ cannot appear in the label of any region whose closure contains
$v_0$, since $C$ is a core circle for all $P_t$ with $t$ near $t_0$, and
$Q_s$ is disjoint from $C$ for all $s$ near $s_0$.  By Lemma~\ref{lem:no
lone small letters}, any letter $a$ that appears in the label of one of the
regions whose closure contains $v_0$ must appear in a combination of either
$ax$ or $ay$, so none of these regions has label $\sA$. Since each $1$-cell
maps to a $0$- or $1$-simplex of the Diagram, the map defined on the
$1$-cells of $K$ maps the boundary of the $2$-cell dual to $v_0$ into the
complement of the vertex $\sA$ of the Diagram. Since this complement is
contractible, the map can be extended over the $2$-cell.

Suppose now that $Q_{s_0}\cap P_{t_0}$ does not contain a spine of
$P_{t_0}$. Then there is a loop $C_{(s_0,t_0)}$ essential in $P_{t_0}$ and
disjoint from $Q_{s_0}$. For every $(s,t)$ near $(s_0,t_0)$, there is a
loop $C_{(s,t)}$ essential in $P_t$ and disjoint from $Q_s$, with the
property that $C_{(s,t)}$ is a meridian of $V_t$ (respectively $W_t$) if and
only if $C_{(s_0,t_0)}$ is a meridian of $V_{t_0}$ (respectively $W_{t_0}$).
In particular, any intersection circle of $Q_s$ and $P_t$ which bounds a
disk in $Q_s$ which is precompressing for $P_t$ in $V_t$ or in $W_t$ must
be disjoint from $C_{(s,t)}$. Since the meridian disks of $V_t$ and $W_t$
have nonzero algebraic intersection, the meridians for $V_t$ and $W_t$
cannot both be disjoint from $C_{(s,t)}$. So for all $(s,t)$ in this
neighborhood of $(s_0,t_0)$, either all disks in $Q_s$ that are
precompressions for $P_t$ are precompressions in $V_t$, or all are
precompressions in $W_t$. In the first case, the letter $B$ does not appear
in the label of any of the regions whose closure contain $v_0$, while in
the second case, the letter $A$ does not. In either case, the extension to
the $2$-cell can now be obtained just as in the previous paragraph. This
completes the proof of Theorem~\ref{thm:finding good levels}.

\section[From good to very good]
{From good to very good}
\label{sec:from good to very good}

\index{good position}\index{very good!position|(}By virtue of
Theorem~\ref{thm:finding good levels}, we may perturb a parameterized
family of diffeomorphisms of $M$ so that at each parameter $u$, some level
$P_t$ and some image level $f_u(P_s)$ meet in good position. In this
section, we use the methodology \index{Hatcher}of A. Hatcher \cite{H, Hold} (see
\cite{Hnew} for a more detailed version of \cite{Hold}, see 
\index{Ivanov}also N. Ivanov
\cite{I4}) to change the family so that we may assume that $P_t$ and
$f_u(P_s)$ meet in very good position. In fact, we will achieve a rather
stronger condition on discal intersections.

Following our usual notation, we fix a sweepout $\tau\colon P\times
[0,1]\to M$ of a closed orientable $3$-manifold $M$, and give $P_t$, $V_t$,
and $W_t$ their usual meanings. Given a parameterized family of
diffeomorphisms $f\colon M\times W\to M$, we give $f_u$, $Q_s$, $X_s$, and
$Y_s$ their usual parameter-dependent meanings. From now on, we refer to
the $P_t$ as \index{levels}\textit{levels} and the $Q_s$ as 
\indexdef{image labels}\indexdef{labels!image}\textit{image levels.}

Throughout this section, we assume that for each $u\in W$, there is a pair
$(s,t)$ such that $Q_s$ and $P_t$ are in good position. Before stating the
main result, we will need to make some preliminary selections.

By transversality, being in good position is an open condition, so there
exist a finite covering of $W$ by open sets $U_i$, $1\leq i\leq n$, and
pairs $(s_i,t_i)$, so that for each $u\in U_i$, $Q_{s_i}$ and $P_{t_i}$
meet in good position. By shrinking of the open cover, we can and always
will assume that all transversality and good-position conditions that hold
at parameters in $U_i$ actually hold on~$\overline{U_i}$.

We want to select the sets and parameters so that at parameters in $U_i$,
$Q_{s_i}$ is transverse to $P_{t_j}$ for all $t_j$. First note that for any
$s$ sufficiently close to $s_i$, $Q_s$ is transverse to $P_{t_i}$ at all
parameters of $U_i$ (here we are already using our condition that the
transversality for the $Q_{s_i}$ holds for all parameters in
$\overline{U_i}$). On $U_1$, $Q_{s_1}$ is already transverse to
$P_{t_1}$. \index{Sard's Theorem}Sard's Theorem ensures that at each $u\in
U_2$, there is a value $s$ arbitrarily close to $s_2$ such that $Q_s$ is
transverse to $P_{t_1}$ at all parameters in a neighborhood of $u$. Replace
$U_2$ by finitely many open sets (with associated $s$-values), for which on
each of these sets the associated $Q_s$ are transverse to $P_{t_1}$. The
new $s$ are selected close enough to $s_2$ so that these $Q_s$ still meet
$P_{t_2}$ in good position. Repeat this process for $U_3$, that is, replace
$U_3$ by a collection of sets and associated values of $s$ for which the
associated $Q_s$ are transverse to $P_{t_1}$ and still meet $P_{t_3}$ in
good position.  Proceeding through the remaining original $U_i$, we have a
new collection, with many more sets $U_i$, but only the same $t_i$ values
that we started with, and at each parameter in one of the new $U_i$,
$Q_{s_i}$ is transverse to $P_{t_1}$ as well as to $P_{t_i}$. Now proceed
to $P_{t_2}$. For the $U_i$ whose associated $t$-value is not $t_2$, we
perform a similar process, and we also select the new $s$-values so close
to $s_i$ that the new $Q_s$ are still transverse to $P_{t_1}$ and still
meet their associated $P_{t_i}$ in good position. After finitely many
repetitions, all $Q_{s_i}$ are transverse to each~$P_{t_j}$.

We may also assume the $U_i$ are connected, by making each connected
component a $U_i$. Since transversality is an open condition, we are free
to replace $s_i$ by a very nearby value, while still retaining the good
position of $Q_{s_i}$ and $P_{t_i}$ and the transverse intersection of
$Q_{s_i}$ with all $P_{t_j}$, for all parameters in $U_i$, and similarly we
may reselect any $t_j$. So (with the argument in the previous paragraph now
completed) we can and always will assume that all $s_i$ are distinct, and
all $t_i$ are distinct.

\newpage
We can now state the main result of this section. With notation as above:
\begin{theorem}
Let $f\colon W\to \diff(M)$ be a parameterized family, such that for each
$u$ there exists $(s,t)$ such that $Q_s$ and $P_t$ meet in good
position. Then $f$ may be changed by homotopy so that there exists a
covering $\set{U_i}$ as above, with the property that for all $u\in
U_i$, $Q_{s_i}$ and $P_{t_i}$ meet in very good position, and $Q_{s_i}$
has no discal intersection with any $P_{t_j}$. If these conditions already
hold for all parameters in some closed subset $W_0$ of $W$, then the
deformation of $f$ may be taken to be constant on some neighborhood of
$W_0$.\par
\label{thm:from good to very good}
\end{theorem}

Before starting the proof, we introduce a simplifying convention. Although
strictly speaking, $Q_{s_i}$ is meaningful at every parameter, as is every
$Q_s$, throughout the remainder of this section we speak of $Q_{s_i}$ only
for parameters in $\overline{U_i}$. That is, unless explicitly stated
otherwise, an assertion made about $Q_{s_i}$ means that the assertion holds
at parameters in $\overline{U_i}$, but not necessarily at other
parameters. Also, to refer to $Q_{s_i}$ at a single parameter $u$, we use
the notation $Q_{s_i}(u)$. By our convention, $Q_{s_i}(u)$ is meaningful
only when $u$ is a value in~$\overline{U_i}$.

Now, to preview some of the complications that appear in the proof of
Theorem~\ref{thm:from good to very good}, consider the problem of removing,
just for a single parameter $u\in U_i$, a discal component $c$ of the
intersection of $Q_{s_i}(u)$ with some $P_{t_j}$. Suppose that the disk
$D'$ in $Q_{s_i}(u)$ bounded by $c$ is innermost among all disks in
$Q_{s_i}(u)$ bounded by discal intersections of $Q_{s_i}(u)$ with the
$P_{t_k}$. Note that $D'$ can contain a nondiscal intersection of
$Q_{s_i}(u)$ with a $P_{t_k}$; such an intersection will be a meridian of
either $V_{t_k}$ or $W_{t_k}$ (although $k$ cannot equal $i$, since
$Q_{s_i}(u)$ and $P_{t_i}$ meet in good position). Let $D$ be the disk in
$P_{t_j}$ bounded by $c$, so that $D\cup D'$ is the boundary of a $3$-ball
$E$. There is an isotopy of $f_u$ that moves $D'$ across $E$ to $D$, and on
across $D$, eliminating $c$ and possibly other intersections of the
$Q_{s_\ell}(u)$ with the $P_{t_k}$.  We will refer to this as a 
\indexdef{basic isotopy}\indexdef{isotopy!basic}\textit{basic
isotopy.}

It is possible for a basic isotopy to remove a biessential component of
some $Q_{s_k}(u)\cap P_{t_k}$. Examples are a bit complicated to describe,
but involve ideas similar to the construction in Figure~\ref{fig:bad
torus}. Fortunately, the following lemma ensures that good position is not
lost.
\begin{lemma}\indexstate{biessential!after basic isotopy}
After a basic isotopy as described above, each $Q_{s_k}(u)\cap P_{t_k}$ still
has a biessential component.
\label{lem:no biessential elimination}
\end{lemma}

\begin{proof}
Throughout the proof of the lemma, $Q_s$ is understood to mean~$Q_s(u)$.

Suppose that a biessential component of some $Q_{s_k}\cap P_{t_k}$ is
contained in the ball $E$, and hence is removed by the isotopy. Since a
spine of $Q_{s_k}$ cannot be contained in a $3$-ball, there must be a
circle of intersection of $Q_{s_k}$ with $D$ that is essential in
$Q_{s_k}$. This implies that $k\neq j$. Now $D'$ must have nonempty
intersection with $P_{t_k}$, since otherwise $P_{t_k}$ would be contained
in $E$. An intersection circle innermost on $D'$ cannot be inessential in
$P_{t_k}$, since $c$ was an innermost discal intersection on $Q_{s_i}$, so
$D'$ contains a meridian disk $D_0'$ for either $V_{t_k}$ or $W_{t_k}$.
Choose notation so that $D$ is contained in~$V_{t_k}$ (that is, $t_j<t_k$).

Suppose first that $D_0'\subset V_{t_k}$. The basic isotopy pushing $D'$
across $E$ moves $Q_{s_k}\cap E$ into a small neighborhood of $D$, so that
it is contained in $V_{t_k}$. If there is no longer any biessential
intersection of $Q_{s_k}$ with $P_{t_k}$, then the complement in $V_{t_k}$
of the original $D_0'$ contains a spine of $Q_{s_k}$ (since the original
intersection of $Q_{s_k}$ with $D$ contained a loop essential in $Q_{s_k}$,
the spine of $Q_{s_k}$ is now on the $V_{t_k}$-side of $P_{t_k}$). This is
a contradiction, since $Q_{s_k}$ is a Heegaard torus.

Suppose now that $D_0'\subset W_{t_k}$. Since the biessential circles of
$Q_{s_k}\cap P_{t_k}$ are disjoint from $D_0'$, they are meridians for
$W_{t_k}$ and hence are essential in $V_{t_k}$. Now, let $A$ be innermost
among the annuli on $Q_{s_k}$ bounded by a biessential component $C$ of
$Q_{s_k}\cap P_{t_k}$ and a circle of $Q_{s_k}\cap D$. Since $Q_{t_k}$ and
$P_{t_k}$ meet in good position, the intersection of the interior of $A$
with $P_{t_k}$ is discal. This implies that $C$ is contractible in
$V_{t_k}$, a contradiction.
\end{proof}\longpage

\begin{proof}[Proof of Theorem~\ref{thm:from good to very good}]
We will adapt the approach used by \index{Hatcher}Hatcher \cite{H}. The
principal difference for us is that in \cite{H}, there is only a single
domain level, whereas we have the different $Q_{s_i}$ on the sets~$U_i$.

The first step is to construct a family $h_{u,t}$, $0\leq t \leq 1$ of
isotopies of the $f_u=h_{u,0}$, which eliminates the discal intersections
of every $Q_{s_i}(u)$ with every $P_{t_j}$. Let $\mathcal{C}$ be the set of
discal intersection curves of all $Q_{s_i} \cap P_{t_j}$ for all $u$, where
as previously explained, this refers only to the $Q_{s_i}(u)$ with $u\in
U_i$. Since $Q_{s_i}$ is transverse to $P_{t_j}$ at all $u\in
\overline{U_i}$, the curves in $\mathcal{C}$ fall into finitely many
families which vary by isotopy as the parameter $u$ moves over (the
connected set) $\overline{U_i}$. Thus we may regard $\mathcal{C}$ as a
disjoint union containing finitely many copies of each $\overline{U_i}$. It
projects to $W$, with the inverse image of $u$ consisting of the discal
intersection curves $\mathcal{C}_u$ of the $Q_{s_i}(u)$ and $P_{t_j}$ for
which $u\in U_i$. By assumption, no element of $\mathcal{C}$ projects to
any parameter $u\in W_0$.

Each $c \in \mathcal{C}_u$ bounds unique disks $D_c \subset P_{t_j}$ and
$D'_c \subset Q_{s_i}(u)$ for some $i$ and $j$. The inclusion relations among
the $D_c$ define a partial ordering $<_P$ on $\mathcal{C}_u$, by the rule
that $c_1 <_P c_2$ when $D_{c_1} \subset D_{c_2}$. Similarly, $c_1 <_Q c_2$
when $D'_{c_1} \subset D'_{c_2}$.

If $c$ is minimal for $<_Q$, then $D'_c \cup D_c$ is an embedded $2$-sphere
in $M$ which bounds a $3$-ball $E_c$.  
By\index{biessential!after basic isotopy} 
Lemma~\ref{lem:no biessential elimination}, the basic isotopy
that pushes $D'_c$ across $E_c$ to $D_c$ and on to the other side of $D_c$
retains the property that every $Q_{s_k}(u)\cap P_{s_k}$ has a biessential
intersection. This ensures that when all discal intersections have been
eliminated, each $Q_{s_k}(u)\cap P_{t_k}$ will still intersect, so they
will be in very good position.

Shrink the open cover $\set{U_i}$ to an open cover $\set{U_i'}$ for
which each $\overline{U_i'}\subset U_i$. To construct the $h_{u,t}$,
Hatcher introduced an auxiliary function $\Psi\colon \mathcal{C}\to (0,2)$
that gives the order in which the elements of $\mathcal{C}$ are to be
eliminated, and allows the basic isotopies to be tapered off as one nears
the frontier of $U_i$. Denoting by $\psi_u$ the restriction of $\Psi$ to
$\mathcal{C}_u$, we will select $\Psi$ so that the following conditions are
satisfied:
\begin{enumerate}
\item $\psi_u (c) < \psi_u (c')$ whenever $c <_Q c'$
\item $\psi_u (c) < 1 $ if $c \subset Q_{s_i}(u)$ and $u \in \overline{U_i'}$
\item $\psi_u (c) > 1$ if $c \subset Q_{s_i}(u)$ and $u \in
\overline{U_i}-U_i$.
\end{enumerate}
One way to construct such a $\Psi$ is to choose a Riemannian metric on
$\tau(P\times (0,1))$ for which each $P_t$ has area $1$, and define
$\Psi_0(c)$ to be the area of $f_u^{-1}(D_c')$ in $P_{s_i}$. Then,
choose continuous functions $\alpha_i$ which are $0$ on
$\overline{U_i'}$ and $1$ on $W-U_i$, and define
$\Psi(c)=\Psi_0(c)+\alpha_i(u)$ for $c\subset Q_{s_i}(u)$.

Roughly speaking, the idea of Hatcher's construction is to have $h_{u,t}$
perform the basic isotopy that eliminates $c$ during a small time interval
$I_u(c)$ which starts at the number $\psi_u(c)$. In order to retain control
of this process, preliminary steps must be taken to ensure that basic
isotopies that move points in intersecting $3$-balls $E_c$ do not occur at
the same time.

If $c$ is a discal intersection of $U_{s_i}$ and $P_{t_j}$, we denote $U_i$
and $U_i'$ by $U(c)$ and $U'(c)$. For a fixed isotopic family of $c\in
\mathcal{C}$ with $c \subset Q_{s_i}$, the points $(u,\psi_u(c))$ form a
$d$-dimensional sheet $i(c)$ lying over $\overline{U(c)}$, where $d$ is the
dimension of $W$. If $i(c_1)$ meets $i(c_2)$, then by the first property of
$\Psi$, $c_1$ and $c_2$ cannot be $<_Q$-related.

Thicken each $i(c)$ to a plate $I(c)$ intersecting each $\set{u} \times
[0,2]$ in an interval $I_u(c)=[\psi_u (c), \psi_u (c) + \epsilon]$, for
some small positive $\epsilon$. This interval will contain the $t$-support of
the portion of $h_{u,t}$ that eliminates $c$, assuming that all other loops
in $\mathcal{C}_u$ with smaller $\psi_u$-value have already been
eliminated. By condition~(1), $c$ will be $<_Q$-minimal at the times $t\in
I_u(c)$. Since $\mathcal{C}_u$ is empty for $u\in W_0$, the $h_{u,t}$ will
be constant for all $u\in W_0$.

Choose the $\epsilon$ small enough so that $I(c_1) \cap I(c_2)$ is nonempty
only near the intersections of $i(c_1)$ and $i(c_2)$. This ensures that if
basic isotopies eliminating $c_1$ and $c_2$ occur on overlapping time
intervals, then $c_1$ and $c_2$ are $<_Q$-unrelated. Also, choose
$\epsilon$ small enough so that $I_u(c)\subset [0,1]$ whenever $u\in
U_i'$.

Write $G_0$ for the union of the $i(c)$, and $G$ for the union of the
$I(c)$.

It may happen that for some $c_1,c_2\in \mathcal {C}_u$ with $\psi_u (c_1)
< \psi_u (c_2)$, we have $c_2 <_P c_1$. In this case the isotopy which
eliminates $c_1$ will also eliminate $c_2$. So reduce $G$ by deleting all
points $(u,\psi_u (c_2))$ such that $\psi_u (c_1) < \psi_u (c_2)$ for some
$c_1$ with $c_2 <_P c_1$. Make a corresponding reduction of $I(c_2)$ by
deleting points $t \in I_u(c_2)$ such that $t > \psi_u (c_1)$ for some
$c_1$ with $c_2 <_P c_1$.

There is a subtle danger here. Suppose that in the previous paragraph,
$u\in U(c_1)-\overline{(U'(c_1))}$. If part of $I_u(c_1)$ extends into
$(1,2]$, then the isotopy eliminating $c_1$ may not be completed, and
therefore $c_2$ would not be eliminated. If $u\in
\overline{U(c_2)}-\overline{U'(c_2)}$ this does not matter, since we only
need to complete the elimination of $c_2$ at parameters in $U'(c_2)$. But
the plate thickness $\epsilon$ must be selected small enough so that
$I_u(c_2)$ lies in $[0,1]$ at all $u$ for which there is a $c_2$ with $u\in
\overline{U'(c_2)}$ and $\psi_u(c_2)>\psi_u(c_1)$. This is possible because
the set of such $u$ is a compact subset of $U_i$.

At values of $t$ where the interiors of $I(c_1)$ and $I(c_2)$ still
overlap, $c_1$ and $c_2$ are $<_Q$-unrelated, and the reduction just made
ensures that they are $<_P$-unrelated. In Hatcher's context, all
intersections are discal, so the combined effect of these is to eliminate
the possibility of simultaneous isotopies on intersecting $3$-balls
$E_{c_1}$ and $E_{c_2}$. In our context, however, $E_{c_1}$ and $E_{c_2}$
can intersect on overlaps of $I(c_1)$ and $I(c_2)$ even when $c_1$ and
$c_2$ are neither $<_P$-related nor $<_Q$-related. Figure~\ref{fig:nested}
shows a simple example. The intersections of $P_{t_1}$ with $Q_{s_2}$, are
not discal, nor are the intersections of $P_{t_2}$ with $Q_{s_1}$, but
$Q_{s_2}$ has a discal intersection with $P_{t_2}$ inside $E(c_1)$.
When this happens, however, $E_{c_1}$ and $E_{c_2}$
must be either disjoint or nested:\par
\index{figures!figure98@nested ball regions}%
\begin{figure}
\labellist
\pinlabel $Q_{s_2}$ at 227 90
\pinlabel $Q_{s_1}$ at 274 218
\pinlabel $P_{t_1}$ at 377 69
\pinlabel $P_{t_2}$ at 384 169
\endlabellist
\includegraphics[width=0.75\textwidth]{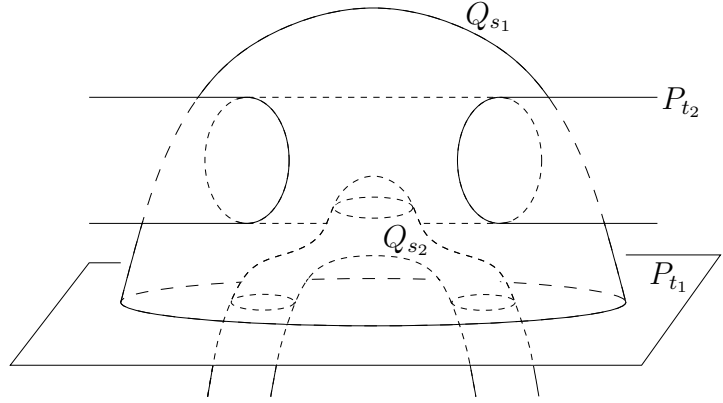}
\caption{Nested ball regions for basic isotopies.}
\label{fig:nested}
\end{figure}\index{examples!example7@nested ball regions}

\begin{lemma}
Suppose that $c_1$ and $c_2$ are $<_Q$-minimal discal intersections, and
are neither $<_P$-related nor $<_Q$-related. Then $\partial E_{c_1}$ and
$\partial E_{c_2}$ are disjoint.
\label{lem:nested E}
\end{lemma}

\begin{proof}
Since $c_1$ and $c_2$ are not $<_Q$-related, $D'(c_1)$ and $D'(c_2)$ are
disjoint, and since they are not $<_P$-related, $D(c_1)$ and $D(c_2)$ are
disjoint. An intersection circle of $D(c_1)$ and $D'(c_2)$ would be smaller
than $c_2$ in the $<_Q$-ordering, and similarly an intersection circle of
$D'(c_1)$ and $D(c_2)$ would be smaller than $c_1$ in the $<_Q$-ordering.
\end{proof}

When $E_{c_1}$ and $E_{c_2}$ are nested, say, $E_{c_2}$ lies in $E_{c_1}$,
a basic isotopy that removes $c_1$ will also remove $c_2$. So we make the
further reduction in $G_0$ of deleting all $(u,\psi_u (c_2))$ for which
there is a $c_1$ such that $i(c_1)$ meets $i(c_2)$, $\psi_u (c_1) < \psi_u
(c_2)$, and $E_{c_2}\subset E_{c_1}$. Also, reduce $I(c_2)$ by removing
any $t$ in $I_u(c_2)$ with $t>\psi_u(c_1)$. Again, this may require the
plate thickness to be decreased to ensure that $I_u(c_1)$ lies in $[0,1]$
at parameters in $U(c_1)$ where $u\in \overline{U'(c_2)}$

For fixed $u \in W$, the basic isotopies are combined by proceeding upward
in $W\times [0,2]$ from $t = 0 $ to $t = 1$, performing each basic isotopy
involving $c$ on the interval $I_u(c)$. Condition~(3) on the $\psi_u$
ensures that the basic isotopies involving $c\subset Q_{s_i}(u)$ taper off
at parameters near the frontier of $U_i$.  On a reduced interval
$I_u(c)$, which is an initial segment of $[\psi_u(c), \psi_u(c) +
\epsilon]$, perform only the corresponding initial portion of the basic
isotopy.  On the overlaps of the $I(c)$, perform the corresponding basic
isotopies concurrently; the reductions of the $I(c)$ have ensured that
these basic isotopies will have disjoint supports.  Since $\epsilon$ was
chosen small enough so that $I_u(c)\subset [0,1]$ whenever $u\in U_i'$,
the basic isotopies involving $Q_{s_i}$ will be completed at all $u$ in
$U_i'$. Since $\mathcal{C}_u$ is empty for $u\in W_0$, no isotopies take
place at parameters in $W_0$.

The remaining concern is that the basic isotopies eliminating $c\subset
Q_{s_i}(u)$ must be selected so that they fit together continuously in the
parameter $u$ on $U_i$. This can be achieved using the method in the
last paragraph on p.~345 of \cite{H} (which applies in the smooth category
by virtue of \cite{HSmale}, see also the more detailed version in~\cite{Hnew}).
\end{proof}\index{very good!position|)}

\section[Setting up the last step]
{Setting up the last step}
\label{sec:lemmas}

In this section, we present some technical lemmas that will be needed for
the final stage of the proof.

The first two lemmas give certain uniqueness properties for the fiber of
the Hopf fibration on $L$. Both are false for $\RP^3$, so require our
convention that $L=L(m,q)$ with $m>2$, and as usual we select $q$ so that
$1\leq q< m/2$. From now on, we endow $L$ with the Hopf fibering and
assume that our sweepout of $L$ is selected so that each $P_t$ is a union
of fibers. Consequently the exceptional fibers, if any, will be components
of the singular set~$S$.

\begin{lemma}\index{Heegaard torus}
Let $P$ be a Heegaard torus in $L$ which is a union of fibers, bounding
solid tori $V$ and $W$. Suppose that a loop in $P$ is a longitude for $V$
and for $W$. Then $q=1$ and the loop is isotopic in $P$ to a fiber.
\label{lem:bilongitude}
\end{lemma}

\begin{proof}
Let $a$ and $b$ be loops in $P$ which are respectively a longitude and a
meridian of $V$, and with $a$ determined by the condition that $ma+qb$ is a
meridian of $W$. Let $c$ be a loop in $P$ which is a longitude for both $V$
and $W$. Since $c$ is a longitude of $V$, it has (for one of its two
orientations) the form $a + kb$ in $H_1(P)$ for some $k$. The intersection
number of $c$ with $ma+qb$ is $q-km$, which must be $\pm1$ since $c$ is a
longitude of $W$. Since $1\leq q < m/2$ and $m>2$, this implies that
$k=0$ and $q=1$. Since $k=0$, $c$ is uniquely determined and $c=a$.  Since
$q=1$, the Hopf fibering is nonsingular, so the fiber is a longitude of
both $V$ and $W$ and hence is isotopic in $P$ to~$c$.
\end{proof}

\begin{lemma}
Let $h\colon L\to L$ be a diffeomorphism isotopic to the identity, with
$h(P_s)=P_t$. Then the image of a fiber of $P_s$ is isotopic in $P_t$ to a
fiber.\par
\label{lem:coincident levels}
\end{lemma}

\begin{proof}
Composing $f$ with a fiber-preserving diffeomorphism of $L$ that moves
$P_s$ to $P_t$, we may assume that $s=t$. Write $P$, $V$, and $W$ for
$P_t$, $V_t$, and $W_t$. Let $a$ and $b$ be loops in $P$ selected as in the
proof of Lemma~\ref{lem:bilongitude}, and write $h_*\colon H_1(P)\to
H_1(P)$ for the induced isomorphism.

Suppose first that $h(V)=V$. Since the meridian disk of $V$ is unique up to
isotopy, we have $h_*(b)=\pm b$. Since $h$ is isotopic to the identity on
$L$ and $m>2$, $h$ is orientation-preserving and induces the identity on
$\pi_1(V)$. This implies that $h_*(b)=b$. Similar considerations for $W$
show that $h_*(ma+qb)=ma+qb$, so $h_*(a)=a$. Thus $h_*$ is the identity on
$H_1(P)$ and the lemma follows for this case.

Suppose now that $h(V)=W$. Then $h$ is orientation-reversing on $P$. Since
$h$ must take a meridian of $V$ to one of $W$, we have
$h_*(b)=\epsilon(ma+qb)$ where $\epsilon=\pm1$. Writing $h_*(a)=ua+vb$,
we find that $1=a\cdot b=-h_*(a)\cdot h_*(b)=-\epsilon(qu-mv)$. The facts
that $h$ is isotopic to the identity on $L$, $a$ generates $\pi_1(L)$, and
$b$ is $0$ in $\pi_1(V)$ imply that $u\equiv 1\pmod m$, so modulo $m$ we
have $1\equiv -\epsilon q$. Since $1\leq q< m/2$, this forces $q=1$,
$\epsilon=-1$, and $h_*(b)=-ma-b$.  Since $a$ has intersection number
$-1$ with the meridian $-ma-b$ of $W$, it is also a longitude of $W$.
Since $h$ is a homeomorphism interchanging $V$ and $W$, $h(a)$ is also a
longitude of $V$ and of $W$, and an application of
Lemma~\ref{lem:bilongitude} completes the proof.
\end{proof}

We now give several lemmas which allow the deformation of diffeomorphisms
and embeddings to make them fiber-preserving or level-preserving. The first
is just a special case of Theorem~\ref{space of fp homeos}:

\begin{lemma}
Let $X$ be either a solid torus or $S^1\times S^1\times \I$, with a
fixed Seifert fibering.  Then the inclusion $\diff_f(X)\to \diff(X)$ is a
homotopy equivalence.\par
\label{lem:fiber-preserving on X}
\end{lemma}

Lemma~\ref{lem:fiber-preserving on X} guarantees that if $g\colon\Delta\to
\diff(X)$ is a continuous map from an $n$-simplex, $n\geq 1$, with
$g(\partial \Delta)\subset \diff_f(X)$, then $g$ is homotopic relative to
$\partial \Delta$ to a map with image in $\diff_f(X)$.

The next lemma is a $2$-dimensional version of Theorem~\ref{space of fp
homeos}, and can be proven using surface theory. In fact, it can be proven
by applying Theorem~\ref{space of fp homeos} to $T\times \I$, although that
would be a strange way to approach it.
\begin{lemma} Let $T$ be a torus with a fixed
$S^1$-fibering. Let $\Diff_h(T)$ be the subgroup of $\Diff(T)$ consisting
of the diffeomorphisms that take some fiber to a loop isotopic to a fiber.
Then the inclusion $\Diff_f(T)\to \Diff_h(T)$ is a homotopy equivalence.
\label{lem:fiberpreserving0}
\end{lemma}

For $e\in (0,1)$ we let $eD^2$ denote the concentric disk of radius $e$ in
the standard disk $D^2\subset \R^2$. Let $X$ be either a solid torus
$D^2\times S^1$, or $T\times \I$ where $T$ is a torus. Let $F=\cup F_i$ be a
disjoint union of finitely many tori. Fix an inclusion of $F$ into $X$ such
that each $F_i$ is of the form $\partial (e_iD^2\times S^1)$, in the solid
torus case, or of the form $T\times \set{e_i}$, in the $T^2\times \I$ case,
for distinct numbers $e_i$ in $(0,1)$. Let $\imb_{\textit{int}}(F,X)$ be
the connected component of the inclusion in the space of all embeddings of
$F$ into the interior of $X$, and let $\imb_{\textit{conc}}(F,X)$ be the
connected component of the inclusion in the set of embeddings for which
each $F_i$ is of the form $\partial (eD^2)\times S^1$ or $T\times \set{e}$
for some $e\in (0,1)$. We omit the proof of the next lemma, which is
analogous to Lemma~\ref{lem:fiberpreserving1}.
\begin{lemma}
Let $X$ be a Seifert-fibered solid torus or $S^1\times S^1\times \I$.  Then
the inclusion $\imb_{\text{conc}}(F,X)\to
\imb_{\text{int}}(F,X)$ is a homotopy equivalence.\par
\label{lem:level-preserving on F}
\end{lemma}

\section[Deforming to fiber-preserving families]
{Deforming to fiber-preserving families}
\label{sec:fiber-preserving families}

\begin{theorem}
Let $L=L(m,q)$ with $m>2$ and let $f\colon S^d\to \diff(L)$.  Then
$f$ is homotopic to a map into $\diff_f(L)$.
\label{thm:make fiber-preserving}
\end{theorem}

\begin{proof}
Applying Theorems~\ref{thm:generalposition}, \ref{thm:finding good levels},
and~\ref{thm:from good to very good}, we may assume that $f$ satisfies the
conclusion of Theorem~\ref{thm:from good to very good}. That is, there are
pairs $(s_i,t_i)$ and an open cover $\set{U_i}$ of $S^d$ with the property
that for every $u\in U_i$, $Q_{s_i}(u)$ and $P_{t_i}$ meet in very good
position, and $Q_{s_i}(u)$ meets every $P_{t_j}$ transversely, with no
discal intersections. The $U_i$ are selected to be connected, so the
intersection $Q_{s_i}(u)\cap P_{t_j}$ is independent, up to isotopy in
$P_{t_j}$, of the parameter~$u$. We remind the reader of our convention
that assertions about $Q_{s_i}$ implicitly mean ``for every $u\in U_i$.''
We can and always will assume that conditions stated for parameters in
$U_i$ actually hold for all parameters in~$\overline{U_i}$.

Since the $t_j$ are distinct, we may select notation so that
$t_1<t_2<\cdots <t_m$. The corresponding $s_i$ typically are not in
ascending order. Figure~\ref{fig:si_chaos} shows a schematic picture of a
block of three levels for which the image levels $Q_{s_1}$, $Q_{s_2}$, and
$Q_{s_3}$ have $s_1<s_3<s_2$.
\index{figures!figure99@out of order level tori}%
\begin{figure}
\labellist
\pinlabel $Q_{s_3}$ at -13 40
\pinlabel $P_{t_1}$ at -12 73
\pinlabel $P_{t_3}$ at -13 101
\pinlabel $Q_{s_1}$ at 161 26
\pinlabel $Q_{s_2}$ at 161 55
\pinlabel $P_{t_2}$ at 159 88
\endlabellist
\includegraphics[width=4cm]{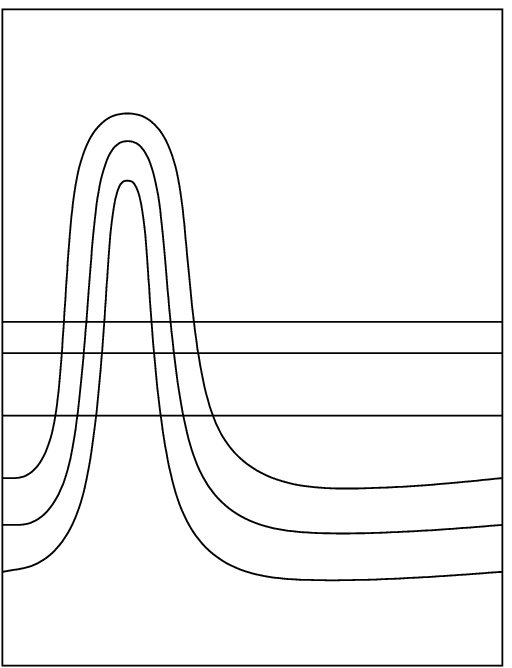}
\caption{A block of level tori with the $Q_{s_i}$ out of order.}
\label{fig:si_chaos}
\end{figure}\index{examples!example8@level tori out of order}

The basic idea of the proof is to make the $f_u$ fiber-preserving on the
$P_{s_i}$, then use Lemma~\ref{lem:fiber-preserving on X} to make the $f_u$
fiber-preserving on the complementary $S^1\times S^1\times \I$ or solid tori
of the $P_{s_i}$-levels. We must be very careful that none of the isotopic
adjustments to a $Q_{s_i}$ destroys any condition that must be preserved on
the other~$Q_{s_j}$.
\longpage

Before listing the steps in the proof of Theorem~\ref{thm:make
fiber-preserving}, a definition is needed. For each $i$, the intersection
circles of $Q_{s_i}\cap P_{t_i}$ cannot be meridians in both $V_{t_i}$ and
$W_{t_i}$, so $Q_{s_i}$ must satisfy exactly one of the following:
\begin{enumerate}
\item
The circles of $Q_{s_i}\cap P_{t_i}$ are not longitudes or meridians for
$V_{t_i}$, so the annuli of $Q_{s_i}\cap V_{t_i}$ are uniquely boundary
parallel in $V_{t_i}$.
\item
The circles of $Q_{s_i}\cap P_{t_i}$ are longitudes or meridians for
$V_{t_i}$, but are not longitudes or meridians for $W_{t_i}$, so the annuli
of $Q_{s_i}\cap W_{t_i}$ are uniquely boundary parallel in $W_{t_i}$.
\item
The circles of $Q_{s_i}\cap P_{t_i}$ are longitudes both for $V_{t_i}$ and
for $W_{t_i}$.
\end{enumerate}
In the first case, we say that $Q_{s_i}$ and $P_{t_i}$ are
\indexdef{Vcored@$V$-cored}\textit{$V$-cored,} in the second that 
they are \indexdef{Wcored@$W$-cored}\textit{$W$-cored,} and in the
third that they are 
\indexdef{bilongitudinal!levels}\textit{bilongitudinal.} If they are 
either $V$-cored or
$W$-cored, we say they are 
\indexdef{cored}\textit{cored.} Lemma~\ref{lem:bilongitude} shows
that the bilongitudinal case can occur only when $q=1$, and then only when
the intersection circles are isotopic in $P_{t_i}$ to fibers of the Hopf
fibering.

We can now list the steps in the procedure. In this list, and in the
ensuing details, ``push $Q_{s_i}$'' means perform a deformation of $f$ that
moves $Q_{s_i}$ as stated, and preserves all other conditions needed.
Making $Q_{s_i}$ ``vertical'' (at a parameter $u$) means making the
restriction of $f_u$ to $P_{s_i}$ fiber-preserving.  When we say that
something is done ``at all parameters of $U_i$,'' we mean that a
deformation of $f$ will be performed, and that $U_i$ is replaced by a
smaller set, so that the result is achieved for all parameters in the new
$\overline{U_i}$, while retaining all other needed properties (such as
that $\set{U_i}$ is an open covering of $S^d$).
\begin{enumerate}
\item[1.]  Push the $Q_{s_i}$ that meet $P_{t_j}$ out of $V_{t_j}$, for all
the $V$-cored $P_{t_j}$, at all parameters in $U(t_j)$. At the end of this
step, each $Q_{s_i}$ that was $V$-cored is parallel to $P_{t_i}$.
\vspace*{0.5 ex}
\item[2.]  Push the $Q_{s_i}$ that meet $P_{t_j}$ out of $W_{t_j}$, for all
the $W$-cored $P_{t_j}$, at all parameters in $U(t_j)$. At the end of this
step, each $Q_{s_i}$ that was $W$-cored is parallel to $P_{t_i}$.
\end{enumerate}
\noindent 
These first two steps are performed using a method of 
\index{Hatcher}Hatcher like that of
the proof of Section~\ref{sec:from good to very good}, although simpler.
After they are completed, a triangulation of $S^d$ is fixed with mesh
smaller than a Lebesgue number for the open cover by the $U_i$. Each of
the remaining steps is performed by inductive procedures that move up the
skeleta of the triangulation, achieving the objective for $Q_{s_i}$ at all
parameters that lie in a simplex completely contained in $U_i$.
\begin{enumerate}
\item[3.]  Push the $Q_{s_i}$ that originally were cored so that each one
equals some level torus. These level tori may vary from parameter to
parameter.\par
\vspace*{0.5 ex}
\item[4.]  
Push the $Q_{s_i}$ that originally were cored to be vertical.
\vspace*{0.5 ex}
\item[5.]
Push the bilongitudinal $Q_{s_i}$ to be vertical.
\vspace*{0.5 ex}
\item[6.]  
Use Lemma~\ref{lem:fiber-preserving on X} to make $f_u$ fiber-preserving
on the complementary $S^1\times S^1\times \I$ or solid tori of the
$P_{s_i}$-levels.
\end{enumerate}

The underlying fact that allows all of this pushing to be carried out
without undoing the results of the previous work is Lemma~\ref{lem:hitting
levels}. Its use involves the concepts of 
\textit{compatibility} and \textit{blocks,} which we will now define.
\newpage

Recall that $R(t_i,t_j)$ means the closure of the region between $P_{t_i}$
and $P_{t_j}$. For a connected subset $Z$ of $S^d$, which in practice will
be either a single parameter or a simplex of a triangulation, denote by
$B_Z$ the set of $t_i$ such that $Z\subset U_i$. Elements $t_i$ and
$t_j$ of $B_Z$ are called 
\indexdef{compatible!Zcompatible@$Z$-compatible parameters}%
\indexsymdef{Zcompatible}{$Z$-compatible}\textit{$Z$-compatible} when $Q_{s_i}(u)\cap
P_{t_i}$ and $Q_{s_k}(u)\cap P_{t_k}$ are homotopic in $R(t_i,t_k)$ for
every $t_k\in B_Z$ with $t_i<t_k\leq t_j$.

Because our family $f$ satisfies the conclusion of Theorem~\ref{thm:from
good to very good}, Lemma~\ref{lem:hitting levels} has the following
consequence: if $t_i$ and $t_j$ are $u$-compatible for any $u$,
then $P_{t_i}$ and $P_{t_j}$ are both $V$-cored, or both $W$-cored, or both
bilongitudinal. The next proposition is also immediate from
Lemma~\ref{lem:hitting levels}.
\begin{proposition} Suppose that $t_i,t_j,t_k\in B_Z$. Then at parameters
in $Z$, $Q_{s_k}$ can meet both $P_{t_i}$ and $P_{t_j}$ only if $t_i$
and $t_j$ are $Z$-compatible.
\label{prop:hitting levels}
\end{proposition}\longpage\longpage

For a simplex $\Delta$, write $B_{\Delta}=\set{b_1,\ldots,b_m}$ with each
$b_i<b_{i+1}$, and for each $i\leq m$ define $a_i$ to be the $s_j$ for
which $b_i=t_j$. Decompose $B_{\Delta}$ into maximal $\Delta$-compatible
blocks $C_1=\set{b_1,b_2,\ldots,b_{\ell_1}}$,
$C_2=\set{b_{\ell_1+1},\allowbreak\ldots,\allowbreak b_{\ell_2}},\ldots\,$,
$C_r=\set{b_{\ell_{r-1}+1},\ldots,b_{\ell_r}}$, with $\ell_r=m$.  Since the
blocks are maximal, Proposition~\ref{prop:hitting levels} shows that
$Q_{a_i}$ is disjoint from $P_{b_j}$ if $b_i$ and $b_j$ are not in the same
block. In steps~3 through~6, this disjointness will ensure that isotopies
of these $Q_{a_i}$ do not disturb the results of previous work.

Note that if $b_i$ and $b_j$ lie in the same block, then either both
$P_{b_i}$ and $P_{b_j}$ are $V$-cored, or both are $W$-cored, or both are
bilongitudinal. Thus we can speak of $V$-cored blocks, and so on.

When $\delta$ is a face of $\Delta$, $B_\Delta\subseteq
B_\delta$. Therefore if $b_i$ and $b_j$ in $B_\Delta$ are
$\delta$-compatible, then they are $\Delta$-compatible. So for each block
$C$ of $B_\delta$, $C\cap B_\Delta$ is contained in a block of $B_\Delta$.
However, levels that are not compatible in $B_\delta$ may become compatible
in $B_\Delta$, since the $t_i$ for intervening levels in $B_\delta$ may
fail to be in $B_\Delta$. Typically, the intersections of blocks of
$B_\delta$ with $B_\Delta$ will combine into larger blocks in~$B_\Delta$.

We should emphasize that during steps~1 through~6, the blocks of $B_Z$, and
whether a level $P_{t_i}$ is $V$-cored, $W$-cored, or bilongitudinal, are
defined with respect to the original configuration, not the new positioning
after the procedure begins. Indeed, after steps~1 and~2, many of the
$Q_{s_i}$ will be disjoint from their~$P_{t_i}$.

We now fill in the details of these procedures.

\vspace*{0.5 ex}\noindent \textsl{Step 1: Push the $Q_{s_i}$ that meet
$P_{t_j}$ out of $V_{t_j}$, for all the $V$-cored $P_{t_j}$, at all
parameters in $U(t_j)$.}\vspace*{0.5 ex}

We perform this in order of increasing $t_j$ for the $V$-cored image
levels. Begin with $t_1$. If $Q_{s_1}$ is $W$-cored or bilongitudinal, do
nothing. Suppose it is $V$-cored. Then for each $u$ in $U(t_1)$, the
$Q_{s_j}(u)$ that meet $P_{t_1}$ intersect $V_{t_1}$ in a union of
incompressible uniquely boundary-parallel annuli.  Since any such $Q_{s_j}$
are transverse to $P_{t_1}$ at each point of $U(t_j)$, the set of
intersection annuli $Q_{s_j}\cap V_{t_1}$ falls into finitely many isotopic
families, with each family a copy of the connected set $U(t_j)$.  For each
$j$ with $U(t_1)\cap U(t_j)$ nonempty, let $\mathcal{A}_j$ be the
collection of the annuli $Q_{s_j}\cap V_{t_1}$, over all parameters in
$U(t_j)$, and let $\mathcal{A}$ be the union of these~$\mathcal{A}_j$. The
nonempty intersection of $U(t_1)$ and $U(t_j)$ ensures that the loops of
$Q_{s_j}\cap P_{t_1}$ and $Q_{s_1}\cap P_{t_1}$ are all in the same isotopy
class in~$P_{t_1}$.

One might hope to push these families of annuli out of $V_{t_1}$ one at a
time, beginning with an outermost one, but an outermost family might not
exist. There could be a sequence $U(t_{j_1}),\ldots\,$, $U(t_{j_k})$ such
that $U(t_{j_i})\cap U(t_{j_{i+1}})$ is nonempty for each $i$,
$U(t_{j_k})\cap U(t_{j_1})$ is nonempty, and for some parameters $u_{j_i}$
in $U(t_{j_i})$, an annulus $Q_{s_{j_{i+1}}}(u_{j_i})\cap V_{t_1}$ lies
outside one of $Q_{s_{j_i}}(u_{j_i})\cap V_{t_1}$ for each $i$, and an
annulus of $Q_{s_{j_1}}(u_{j_k})\cap V_{t_1}$ lies outside one of
$Q_{s_{j_k}}(u_{j_k})\cap V_{t_1}$. Since an outermost family might not
exist, we will need to utilize the method of Hatcher as in the proof of
Theorem~\ref{thm:from good to very good}, but only a simple version of it.

Shrink the $U_i$ slightly, obtaining a new open cover by sets $U_i'$
with $\overline{U_i'}\subset U_i$.
We will use a function $\Psi\colon \mathcal{A}\to (0,2)$, so that at each
parameter $u$, the restriction $\psi_u$ of $\Psi$ to the annuli at that
parameter has the property that $\psi_u(A_1)<\psi_u(A_2)$ whenever $A_1,
A_2\in \mathcal{A}_i$ and $A_1$ lies in the region of parallelism between
$A_2$ and $\partial V_{t_1}$. Moreover, we will have $\psi_u(A)<1$ whenever
$A\in \mathcal{A}_i$ and $u\in \overline{U_i'}$, while $\psi_u(A)>1$ for
$u$ near the boundary of $U_i$. We construct $\Psi$ by letting
$\Psi_0(A)$ be the volume of the region of parallelism between $A$ and an
annulus in $\partial V_{t_1}$ (assuming that the volume of $L$ has been
normalized to $1$ to ensure that $\Psi_0(A)<1$), then adding on auxiliary
values $\alpha_i(u)$ as in the proof of Theorem~\ref{thm:from good to
very good}.

Form the union $G_0\subset S^d\times (0,2)$ of the $(u,\psi_u(A))$ as in
the proof of Theorem~\ref{thm:from good to very good}, and thicken each of
its sheets as was done there, obtaining an interval for each parameter.
These intervals tell the supports of the isotopies that push the annuli of
$Q_{s_j}\cap V_{t_1}$ out of $V_{t_1}$. If two sheets of $\mathcal{A}$
cross in $S^d\times (0,2)$, then the corresponding regions of parallelism
have the same volume, so must be disjoint and the isotopies can be
performed simultaneously without interference. At each individual
parameter $u$, each annulus is outermost during the time it is being
pushed out of $V_{t_1}$, but the times need to be different since there may
be no outermost family.

\newpage

After the process is completed, $Q_{s_j}$ will lie outside of $V_{t_1}$ at
all parameters in $\overline{U(t_j)'}$, whenever $U(t_j)$ had nonempty
intersection with $U(t_1)$. Replacing each $U(t_j)$ by $U(t_j)'$, we have
$Q_{s_j}$ pushed out of $V_{t_1}$ at all parameters in these $U(t_j)$.
Moreover, Lemma~\ref{lem:pushout}(2) shows that $V_{t_1}$ is concentric in
either $X_{s_1}$ or $Y_{s_1}$ at all parameters in~$U(t_1)$.
\longpage\longpage

Some of the $Q_{s_k}$ for which $U(t_k)$ did not meet $U(t_1)$ may be moved
by the isotopies of the $Q_{s_j}$ at parameters in $U(t_j)\cap U(t_k)$. The
condition that these $Q_{s_k}$ meet $P_{t_1}$ transversely may be lost, but
this will not matter, because these intersections never matter when
$U(t_k)$ does not meet~$U(t_1)$.

Now consider $t_2$. Again, we do nothing if $Q_{s_2}$ is $W$-cored or
bilongitudinal, so suppose that it is $V$-cored. Use the Hatcher process as
before, to push annuli in the $Q_{s_j}$ out of $V_{t_2}$, when $Q_{s_j}$
meets $P_{t_2}$ and $U(t_j)$ meets $U(t_2)$.  Notice that these $Q_{s_j}$
cannot meet $V_{t_1}$ at parameters in $U(t_1)$.  For if $t_2$ is not
$u$-compatible with $t_1$ at some parameters in $U(t_1)$, then (by
Lemma~\ref{lem:hitting levels}) $Q_{s_j}$ cannot meet both $P_{t_2}$ and
$P_{t_1}$, while if it is $u$-compatible at some parameter in $U(t_1)$,
then it has already been pushed out of $V_{t_1}$.  And $V_{t_1}$ cannot lie
in any of the regions of parallelism for the pushouts from $V_{t_2}$, since
the intersection circles of the $Q_{s_j}$ with $P_{t_2}$ are not longitudes
in $V_{t_2}$.

After these pushouts are completed, if $i=1$ or $i=2$ and $Q_{s_i}$ was
$V$-cored, then $V_{t_i}$ is concentric in either $X_{s_i}$ or $Y_{s_i}$ at
all parameters in~$U_i$.

We continue working up the increasing $t_i$ in this way.  At the end of
this process, $V_{t_i}$ is concentric in either $X_{s_i}$ or $Y_{s_i}$ for
all $i$ such that $Q_{s_i}$ was $V$-cored, and at all parameters in
$U_i$. For $Q_{s_i}$ that were $W$-cored or bilongitudinal, the
intersections $Q_{s_i}\cap P_{t_i}$ have not been disturbed at parameters
in $U_i$. We have not introduced any new intersections of $Q_{s_i}$ with
$P_{t_j}$, so we still have the property that at any parameter $u$ in
$U_i\cap U(t_j)$, $Q_{s_j}$ can meet $P_{t_i}$ only if $t_i$ and $t_j$
were originally $u$-compatible.

\vspace*{0.5 ex}\noindent \textsl{Step 2: Push the $Q_{s_i}$ that meet
$P_{t_j}$ out of $W_{t_j}$ for all the $Q_{s_j}$ that are $W$-cored, at all
parameters in $U(t_j)$.}\vspace*{0.5 ex}

The entire process is repeated with $W$-cored levels, except that we start
with $t_m$ and proceed in order of decreasing $t_i$. Each $W$-cored
$Q_{s_i}$ is pushed out of $W_{t_i}$, and at the end of the process
$W_{t_i}$ is concentric in either $X_{s_i}$ or $Y_{s_i}$ at all parameters
in $U_i$, whenever $Q_{s_i}$ was $W$-cored. No intersection of a
$Q_{s_j}$ with a $V$-cored or bilongitudinal level $P_{t_i}$ is changed at
any parameter in~$U_i$.

For the remaining steps, we fix a triangulation of $S^d$ with mesh smaller
than a Lebesgue number for $\set{U_i}$, which will ensure that
$B_\Delta$ is nonempty for every simplex~$\Delta$. We will no longer
proceed up or down all $t_i$-levels, working on the sets $U_i$, but
instead will work inductively up the skeleta of the triangulation. Recall
that each $B_\Delta$ is decomposed into blocks, according to the original
intersections of the $Q_{s_i}$ and $P_{t_i}$ before steps~1 and~2 were
performed.\longpage

\vspace*{0.5 ex}\noindent \textsl{Step 3: Push the $Q_{s_i}$ that were
originally cored so that each one equals some level torus.}\vspace*{0.5 ex}

We will proceed inductively up the skeleta of the triangulation, moving
cored $Q_{s_i}$ to level tori, without changing $Q_{s_k}\cap P_{s_k}$ for
the bilongitudinal $Q_{s_k}$. We want to use the fact that $V_{t_i}$ (or
$W_{t_i}$) is concentric with $X_{s_i}$ or $Y_{s_i}$ to push $Q_{s_i}$ onto
a level torus, but when moving multiple levels at a given parameter, there
is a consistency condition needed. As shown in Figure~\ref{fig:inconsistent
levels}, it might happen that $V_{t_i}$ is concentric in $X_{s_i}$ while
$V_{t_j}$ is concentric in $Y_{s_j}$. Then, we might not be able to push
$Q_{s_i}$ and $Q_{s_j}$ onto level tori without disrupting other
levels. The following lemma rules out this bad configuration.
\index{figures!figure991@inconsistent nesting}%
\begin{figure}
\labellist
\pinlabel $V_{t_i}$ at 128 40
\pinlabel $X_{s_i}$ at 235 120
\pinlabel $Y_{s_i}$ at 288 128
\pinlabel $V_{t_j}$ at 40 170
\pinlabel $Y_{s_j}$ at 75 275
\pinlabel $X_{s_j}$ at 125 320
\endlabellist
\includegraphics[width=7cm]{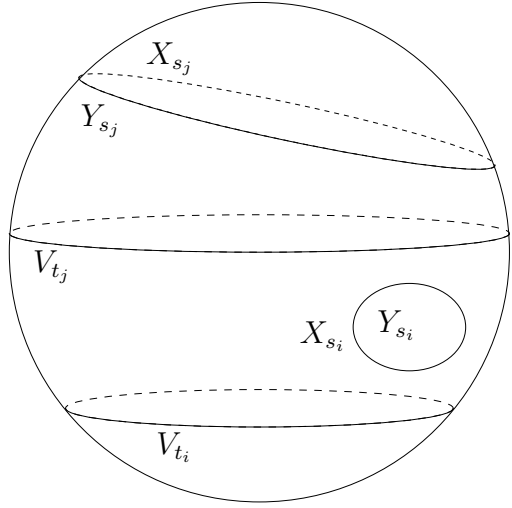}
\caption{Hypothetical inconsistent nesting: $V_{t_i}\subset X_{s_i}$ and
$V_{t_j}\subset Y_{s_j}$.}
\label{fig:inconsistent levels}
\end{figure}\index{examples!example9@hypothetical inconsistent nesting}

\begin{lemma} Suppose, after steps~1 and~2 have been completed,
that $u\in U_i\cap U(t_j)$, $t_i<t_j$, and that
$Q_{s_i}$ is $V$-cored.
\begin{enumerate}
\item The region between $Q_{s_i}$ and $Q_{s_j}$ does not contain a core
circle of $V_{t_i}$.
\item Suppose that $t_i$ and $t_j$ are $u$-compatible, and $V_{t_i}$ is
concentric in $Z_{s_i}$ where $Z$ is $X$ or $Z$ is $Y$. Then $V_{t_j}$ is
concentric in $Z_{s_j}$.
\item If $t_i$ and $t_j$ are not $u$-compatible, then $Q_{s_i}$ is parallel
to $P_{t_i}$ in $R(t_i,t_j)$.\par
\end{enumerate}\par
\noindent The analogous statement holds when $Q_{s_j}$ is $W$-cored
and $W_{t_j}$ is concentric in $Z_{s_j}$.
\label{lem:nesting}
\end{lemma}
\noindent 

\begin{proof}
It suffices to consider the case when $Q_{s_i}$ is $V$-cored. In the
situation at the start of Step~1 above, when annuli in the $Q_{s_k}$
were being pushed out of $V_{t_i}$, the intersection of $Q_{s_i}\cup
Q_{s_j}$ with $V_{t_i}$ was a union $F$ of incompressible nonlongitudinal
annuli. Since $Q_{s_i}$ met $P_{t_i}$, $F$ was nonempty. By
Proposition~\ref{prop:core circles}, exactly one complementary region of
$F$ in $V_{t_i}$ contained a core circle $C$ of $V_{t_i}$. For at least one
of $s_i$ and $s_j$, say for $s_k$, $Q_{s_k}$ met this complementary region.

Since the annuli of $F$ are nonlongitudinal, there is an embedded circle
$C'$ in $Q_{s_k}$ that is homotopic in the core region to a proper multiple
of $C$. If $C$ were in the region $R=f_u(R(s_i,s_j))$ between $Q_{s_i}$ and
$Q_{s_j}$, then the embedded circle $C'$ in $\partial R$ would be a proper
multiple in $\pi_1(R)$, which is impossible since $R$ is homeomorphic to
$S^1\times S^1\times \I$. This proves~(1).

Assume that $t_i$ and $t_j$ are $u$-compatible and suppose that
$V_{t_i}\subset X_{s_i}$ and $V_{t_j}\subset Y_{s_j}$.  Then $C$ is
contained in $X_{s_i}\cap Y_{s_j}$, forcing $s_i>s_j$ and $C$ in the region
between $Q_{s_i}$ and $Q_{s_j}$, contradicting (1). The case of
$V_{t_i}\subset Y_{s_i}$ and $V_{t_j}\subset X_{s_j}$ is similar, so (2)
holds.

For (3), if $t_i$ and $t_j$ are not $u$-compatible, then $Q_{s_i}$ was
initially disjoint from $P_{t_j}$, and hence is disjoint after steps~1
and~2. By Lemma~\ref{lem:pushout}(2), $V_{t_i}$ is concentric in $X_{s_i}$
or $Y_{s_i}$ after steps~1 and~2, and part~(3) follows.
\end{proof}

It will be convenient to extend our previous notation $R(s,t)$ for the
closure of the region between $P_s$ and $P_t$, by putting $R(0,t)=V_t$,
$R(t,1)=W_t$, and $R(0,1)=L$.

We will now define target regions. The isotopies that we will use in the
rest of our process will only change values within a single target region,
ensuring that the necessary positioning of the $Q_{s_i}$ is retained.  Let
$\Delta$ be a simplex of the triangulation, and recall the decomposition of
$B_\Delta=\set{b_1,\ldots,b_m}$ into maximal $\Delta$-compatible blocks
$C_1=\set{b_1,b_2,\ldots,b_{\ell_1}}$,
$C_2=\set{b_{\ell_1+1},\allowbreak\ldots,\allowbreak b_{\ell_2}},\ldots\,$,
$C_r=\set{b_{\ell_{r-1}+1},\ldots,b_{\ell_r}}$. Define the 
\indexdef{target region}\indexdef{region!target}\textit{target
region} of a block $C_n$ to be the submanifold $T_\Delta(C_n)$ of $L$
defined as follows. Put $\ell_0=0$, $b_0=0$, and $b_{\ell_r+1}=1$.
\begin{enumerate}
\item If $C_n$ is $V$-cored, then 
$T_\Delta(C_n)=R(b_{\ell_{n-1}+1},b_{\ell_n+1})$.
\item If $C_n$ is $W$-cored, then 
$T_\Delta(C_n)=R(b_{\ell_{n-1}},b_{\ell_n})$.
\item If $C_n$ is bilongitudinal, then
$T_\Delta(C_n)=R(b_{\ell_{n-1}},b_{\ell_n+1})$.
\end{enumerate}
We remark that $T_\Delta(C_n)$ is all of $L$ when $B_\Delta$ consists of a
single bilongitudinal block, otherwise is of the form $V_t$ when $n=1$ and
$C_1$ is $W$-cored or bilongitudinal and of the form $W_t$ when $n=r$ and
$C_n$ is $V$-cored or bilongitudinal, and in all other cases it is a region
$R(s,t)$ diffeomorphic to $S^1\times S^1\times \I$.

\newpage
As noted in the next lemma, the interior of the target region of a block
contains the $Q_{a_i}$ for the $b_i$ in the block, at this point of our
argument.
\begin{lemma} Target regions satisfy the following.
\begin{enumerate}
\item If $b_i\in C_n$ and $u\in \Delta$, then $Q_{a_i}(u)$ is in the
interior of $T_\Delta(C_n)$.
\item If $\delta$ is a face of $\Delta$, and $C'_1,\ldots\,$, $C'_{r'}$ are
the blocks of $B_\delta$, then for each $i$, there exists a $j$ such that
$T_\delta(C'_i)\subseteq T_\Delta(C_j)$.
\end{enumerate}
\label{lem:target}
\end{lemma}

\begin{proof}
Property (1) is a consequence of Proposition~\ref{prop:hitting levels} and
the fact that Steps~1 and~2 do not create new intersections of the
$Q_{s_i}(u)$ with the $P_{t_j}$. For part~(2), the proof is direct
from the definitions, dividing into various subcases.
\end{proof}

Target regions can overlap in the following ways: the target region for a
$V$-cored block $C_n$ will overlap the target region of a succeeding
$W$-cored block $C_{n+1}$, and the target region of a bilongitudinal block
will overlap the target region of a preceding $V$-cored block or of a
succeeding $W$-cored block (note that by Lemma~\ref{lem:bilongitude},
successive blocks cannot both be bilongitudinal). The latter cause no
difficulties, but the conjunctions of a $V$-cored block and a succeeding
$W$-cored block will necessitate some care during the ensuing argument.

We can now begin the process that will complete Step~3. We will start at
the parameters that are vertices of the triangulation and move the
$Q_{a_i}$ for each $V$-cored or $W$-cored block to be level, that is, so
that each $Q_{a_i}(u)$ equals some $P_t$. The isotopies will be fixed on
each $P_{b_i}$ for which $Q_{a_i}$ is bilongitudinal, and these unchanged
$Q_{a_i}\cap P_{b_i}$ will be used to work with the bilongitudinal levels
in a later step.  For each cored block, the isotopy that levels the
$Q_{a_i}$ will move points only in the interior of the target region of the
block. As we move to higher-dimensional simplices, the $Q_{a_i}$ will
already be level at parameters on the boundary, and the deformation will be
fixed at those parameters. Each deformation for the parameters in a simplex
$\delta_0$ of dimension less than $d$ must be extended to a deformation of
$f$. The extension will change an $f_u$ only when $u$ is in the open star
of $\delta_0$, by a deformation that performs some initial portion of the
deformation of $f_{u_0}$ at a parameter $u_0$ of $\delta_0$--- the
parameter that is the $\delta_0$-coordinate of $u$ when the simplex that
contains it is written as a join $\delta_0*\delta_1$ (details will be given
below). We will see that because the target regions can overlap, the
deformation of an $f_u$ might not preserve all target regions, but enough
positioning of the image levels~$Q_{a_i}$ will be retained to continue the
inductive process.

Fix a vertex $\delta_0$ of the triangulation, and consider the first block
$C_1$ of $B_{\delta_0}$. If it is bilongitudinal, we do nothing. Suppose
that it is $V$-cored.  All of the $Q_{a_1},\ldots\,$, $Q_{a_{\ell_1}}$ lie
in the interior of the target region
$T_{\delta_0}(C_1)$. Lemma~\ref{lem:nesting}(2) shows that for either $Z=X$
or $Z=Y$, $V_{b_i}$ is concentric in $Z_{a_i}$ for $b_i\in C_1$. We claim
that there is an isotopy, supported on $T_{\delta_0}(C_1)$, that moves each
$Q_{a_i}$ to be level. If $C_1$ is the only block, then
$T_{\delta_0}(C_1)=L$ and the isotopy exists by the definition of
concentric. If there is a second block, then Lemma~\ref{lem:nesting}(3)
shows that the $Q_{a_i}$ for $b_i\in C_1$ are parallel to $P_{b_1}$ in
$T_{\delta_0}(C_1)=R(b_1,b_{\ell_1+1})$, and again the isotopy exists.
After performing the isotopy, we may assume that the $Q_{a_i}(\delta_0)$
are level.

To extend this deformation of $f_{\delta_0}$ to a deformation of the
parameterized family $f$, we regard each simplex $\Delta$ of the closed
star of $\delta_0$ in the triangulation as the join $\delta_0*\delta_1$,
where $\delta_1$ is the face of $\Delta$ spanned by the vertices of
$\Delta$ other than $\delta_0$. Each point of $\Delta$ is uniquely of the
form $u=s\delta_0+(1-s)u_1$ with $u_1\in \delta_1$. Write the isotopy of
$f_{\delta_0}$ as $j_t\circ f_{\delta_0}$, with $j_0$ the identity map of
$L$. Then, at $u$ the isotopy at time $t$ is $j_t\circ f_u$ for $0\leq
t\leq s$ and $j_s\circ f_u$ for $s\leq t\leq 1$.  For any two simplices
containing $\delta_0$, this deformation agrees on their intersection, so it
defines a deformation of~$f$.

The target region $T_{\delta_0}(C_1)$ will overlap $T_{\delta_0}(C_2)$ if
$C_2$ is bilongitudinal or $W$-cored. When $C_2$ is bilongitudinal, this
does not affect any of our necessary positioning. If it is $W$-cored, then
$Q_{a_i}$ with $b_i\in C_2$ may be moved into $T_{\delta_0}(C_1)$. At
$\delta_0$, such $Q_{a_i}$ can end up somewhere between the now-level
$Q_{a_{\ell_1}}$ and $P_{b_{\ell_2}}$, and at other parameters in the star
of $\delta_0$ they will lie somewhere in $R(b_1,b_{\ell_2})$. This will
require only a bit of attention in the later argument.

In case $C_1$ was $W$-cored, we use Lemmas~\ref{lem:nesting}(2)
and~\ref{lem:level-preserving on F}, producing a deformation of
$f_{\delta_0}$ supported on the interior of the solid torus
$T_\Delta(C_1)=V_{b_{\ell_1}}$, which does not meet any other target
region. This is extended to a deformation of $f$ just as before.

We move on to consider $C_2$ in analogous fashion, doing nothing if $C_2$
is bilongitudinal, and moving the $Q_{a_i}$ to be level at the parameter
$\delta_0$. If $C_1$ was $V$-cored and $C_2$ is $W$-cored, then instead of
the initial target region $T_{\delta_0}(C_2)$ we must use the region
between the now-level $Q_{a_{\ell_1}}(u)$ and $P_{b_{\ell_2}}$, but
otherwise the argument is the same. Proceed in the same way through the
remaining blocks $C_n$ of~$B_{\delta_0}$, ending with all the cored
$Q_{a_i}(u_0)$ moved to be level. This process for $u_0$ is repeated for
each $0$-simplex of the triangulation.

Now, consider a simplex $\delta$ of positive dimension. Inductively, we may
assume that at each $u$ in $\partial \delta$, each cored $Q_{a_i}$ has been
moved to a level torus, and $Q_{a_i}\cap P{b_i}$ is unchanged for each
bilongitudinal $Q_{a_i}$. Moreover, if $a_i$ is contained in a cored block
$C_j$, then $Q_{a_i}$ lies in the corresponding target region
$T_\delta(C_j)$, or else lies in the union of the target regions for a
$V$-cored block and a succeeding $W$-cored block.

We apply Lemma~\ref{lem:level-preserving on F} to each cored block of
$B_\delta$, sequentially up the cored blocks. We obtain a sequence of
deformations of $f$ on $\delta$, constant at parameters in $\partial
\delta$. There is no interference between different blocks, except when a
$W$-cored block $C_{n+1}$ succeeds a $V$-cored block $C_n$. First, the
$Q_{a_i}$ for the $V$-cored block are moved to be level. Then, at each
parameter in $\delta$, the $Q_{a_i}(u)$ for the $W$-cored block lie between
the now-level $Q_{a_{\ell_n}}(u)$ and $P_{b_{\ell_{n+1}}}$. We regard the
union of these regions over the parameters of $\delta$ as a product
$\delta\times S^1\times S^1\times \I$, and apply
Lemma~\ref{lem:level-preserving on F}. Thus the isotopy that levels the
$Q_{a_i}$ from the $W$-cored block need not move any of the $Q_{a_i}$ from
the $V$-cored block. In other cases, the successive isotopies take place in
disjoint regions. To extend this to a deformation of $f$, we adapt the join
method from above (of course when $\delta$ is $d$-dimensional, no extension
is necessary). Regard each simplex $\Delta$ of the closed star of $\delta$
in the triangulation as the join $\delta*\delta_1$, where $\delta_1$ is the
face of $\Delta$ spanned by the vertices of $\Delta$ not in $\delta$. Each
point of $\Delta$ is uniquely of the form $u=s u_0+(1-s)u_1$ with $u_0\in
\delta$ and $u_1\in \delta_1$. Write the isotopy of $f_{u_0}$ as $j_t\circ
f_{u_0}$, with $j_0$ the identity map of $L$. Then, at $u$ the isotopy at
time $t$ is $j_t\circ f_u$ for $0\leq t\leq s$ and $j_s\circ f_u$ for
$s\leq t\leq 1$. For any two simplices containing $\delta$, this
deformation agrees on their intersection, so it defines a deformation
of~$f$.

At the completion of this process, each cored $Q_{s_i}$ is level at all
parameters in $\Delta$, whenever $\Delta\subset U_i$. The bilongitudinal
$Q_{s_i}$ may have been moved around some, but their intersections
$Q_{s_i}\cap P_{t_i}$ will not be altered at parameters for which $t_i\in
B_\Delta$ since these intersections will not lie in the interior of any
target region for a cored level.
\smallskip

\noindent \textsl{Step 4: Push all cored $Q_{s_i}$ to be
vertical, that is, make each image of a fiber of $P_{s_i}$ a fiber in $L$.
}\vspace*{0.5 ex}

Again we work our way up the simplices of the triangulation. Start at a
$0$-simplex $\delta_0$. Each cored $Q_{a_i}(\delta_0)$ for $b_i\in
B_{\delta_0}$ is now level.  By Lemma~\ref{lem:coincident levels}, the
image fibers in $Q_{a_i}(\delta_0)$ are isotopic in that level torus to
fibers of $L$. Using Lemma~\ref{lem:fiberpreserving0}, there is an isotopy
of $f_{\delta_0}$ that preserves the level tori and makes
$Q_{a_i}(\delta_0)$ vertical. This isotopy can be chosen to fix all points
in other $Q_{a_j}(\delta_0)$, and is extended to a deformation of $f$ by
using the method of Step~3. We work our way up the skeleta; if
$\delta\subset U(b_i)$, then for every $u$ in $\delta$, each $Q_{a_i}(u)$
is level torus, and at parameters $u\in\partial \delta$, $Q_{a_i}(u)$ is
vertical.  Using Lemma~\ref{lem:fiberpreserving0}, we make the $Q_{a_i}(u)$
vertical at all $u\in \delta$, and extend to a deformation of $f$ as
before. We repeat this for all levels of cored blocks.

\noindent \textsl{Step 5: Push all bilongitudinal $Q_{s_i}$
to be vertical.}\vspace*{0.5 ex}

Now, we examine the bilongitudinal levels. For a bilongitudinal level
$Q_{a_i}$ at a vertex $\delta_0$, Corollary~\ref{coro:comeridian} shows
that the intersection circles are longitudes for $X_{a_i}$ and
$Y_{a_i}$. Lemma~\ref{lem:bilongitude} then shows that the circles of
$Q_{a_i}\cap P_{b_j}$ are isotopic in $Q_{a_i}$ and in $P_{b_j}$ to fibers.
First, use Lemma~\ref{lem:fiberpreserving1}
to find an isotopy
preserving levels, such that postcomposing $f_{\delta_0}$ by the isotopy
makes the intersection circles fibers of the $P_{b_j}$. Then, use
Lemma~\ref{lem:fiberpreserving1}
applied to $f_{\delta_0}^{-1}$ to
find an isotopy preserving levels of the domain, such that precomposing
$f_{\delta_0}$ by the isotopy makes the intersection circles the images of
fibers of $P_{s_i}$. After this process has been completed for the
bilongitudinal $Q_{a_i}$, the inverse image (in their union $\cup Q_{a_i}$) of
each region $R(b_j,b_{j+1})$ with $b_j$ or $b_{j+1}$ in a bilongitudinal
block is a collection of fibered annuli which map into $R(b_j,b_{j+1})$ by
embeddings that are fiber-preserving on their boundaries. We use
Lemma~\ref{lem:fiberpreserving2} to find an isotopy that makes the
$Q_{a_i}$ vertical.  Again, we extend to a deformation of $f$ and work our
way up the skeleta, to assume that $Q_{s_i}(u)$ is vertical whenever $u\in
\Delta$ and $\Delta\subset U_i$.

\vspace*{0.5 ex}\noindent \textsl{Step 6: Make $f$ fiber-preserving on
the complementary $S^1\times S^1\times \I$ or solid tori of the
$P_{s_i}$-levels}\vspace*{0.5 ex}

We work our way up the skeleta one last time, using
Lemma~\ref{lem:fiber-preserving on X} to make $f$ fiber-preserving on the
complementary $S^1\times S^1\times \I$ or solid tori of the~$P_{a_i}$.
\longpage

There is an annoying technical problem that arises in this step. At each
parameter, the deformations that make $f_u$ fiber-preserving on the
$S^1\times S^1\times \I$ are fixed on the boundaries of these submanifolds,
but the extended diffeomorphisms may have to move points on the other side
of the frontier. Thus, a region where $f_u$ was already fiber-preserving
may be changed to make $f_u$ no longer fiber-preserving there. One fix for
this is as follows. We can arrange that the final $f$ has all $f_u$
fiber-preserving except on small product neighborhoods of a finite set of
levels at each parameter. Then by removing a neighborhood of the singular
circles and their images, we can regard $f$ as a parameterized family of
diffeomorphisms of $S^1\times S^1\times \I$ that is fiber-preserving on a
neighborhood of the boundary at each parameter. Then we apply the following
version of Theorem~\ref{space of fp homeos}:
\begin{theorem} Suppose that $\Sigma$ is a Seifert-fibered $3$-manifold
with boundary and $g\colon \Sigma \times W \to \Sigma$ is a parameterized
family of diffeomorphisms, with $W$ compact, such that each $g_u$ is
fiber-preserving on a neighborhood of $\partial \Sigma$. Then there is a
deformation of $g$, relative to $U\times W$ for some open neighborhood $U$
of $\partial \Sigma$ in $\Sigma$, to a family of fiber-preserving
diffeomorphisms.
\end{theorem}

\newpage\noindent To prove this, we know from Theorem~\ref{space of fp
homeos} that there is some deformation from $g$ to a family $h$ of
fiber-preserving diffeomorphisms. Since the inclusion $\diff_f(\partial
\Sigma)\to \diff(\partial\Sigma)$ is a homotopy equivalence, the
restriction of this deformation to $\partial \Sigma$ can be assumed to be
fiber-preserving at all times. Choosing a collar $\partial \Sigma\times \I$
in which each $\Sigma\times \{t\}$ is a union of fibers, we may use
uniqueness of collars to change the deformation to be fiber-preserving on
the collar. Now, by performing less and less of the deformation as one
moves toward $\partial \Sigma$, obtain a new deformation from $g$ to a
fiber-preserving family $h'$ such that $h'=h$ outside $\partial
\Sigma\times \I$, but $h=g$ on $\partial \Sigma\times [0,1/2]$.
\end{proof}

\section[Parameters in $D^d$]
{Parameters in $D^d$}
\label{sec:Dk}

Regard $D^d$ as the unit ball in $d$-dimensional Euclidean space, with
boundary the unit sphere $S^{d-1}$. As mentioned in
Section~\ref{sec:reduction}, to prove that $\diff_f(L)\to \diff(L)$ is a
homotopy equivalence, we actually need to work with a family of
diffeomorphisms $f$ of $L$ parameterized by $D^d$, $d\geq 1$, for which
$f(u)$ is fiber-preserving whenever $u$ lies in the boundary
$S^{d-1}$. We must deform $f$ so that each $f(u)$ is
fiber-preserving, by a deformation that keeps $f(u)$ fiber-preserving at
all times when $u\in S^{d-1}$.

We now present a trick that allows us to gain good control of what happens
on $S^{d-1}$. The Hopf fibering we are using on $L$ can be described as a
Seifert fibering of $L$ over the round $2$-sphere $S$, in such a way that
each isometry of $L$ projects to an isometry of $S$. For the cases when
$q=1$, the round sphere is the actual quotient orbifold, and when $1<q$,
the quotient orbifold has two cone points but the only induced isometries
are rotations fixing those cone points. (Section~\ref{fiberings} above
details this description for the manifolds considered in
Chapter~\ref{ch:one-sided}, full details of all cases are in~\cite{M}.) By
conjugating $\pi_1(L)$ in $\SO(4)$, we may assume that the singular fibers,
when $q>1$, are the inverse images of the poles. We choose our sweepout so that
the level tori are the inverse images of latitude circles. Denote by $p_t$ the
latitude circle that is the image of the level torus~$P_t$.

There is an isotopy $J_t$ with $J_0$ the identity map of $L$ and each $J_t$
fiber-preserving, so that the images of the level tori $P_s$ under $J_1$
project to circles in the $2$-sphere as indicated in Figure~\ref{fig:Dk
trick}. Denote the image of $J_1(P_s)$ in $S$ by~$q_s$. Their key property
is that when moved by any orthogonal rotation of $S$, each $p_t$ meets the
image of some $q_s$ transversely in two or four points.
\index{figures!figure992@projection of a perturbed level}%
\begin{figure}
\includegraphics[width=4.5cm]{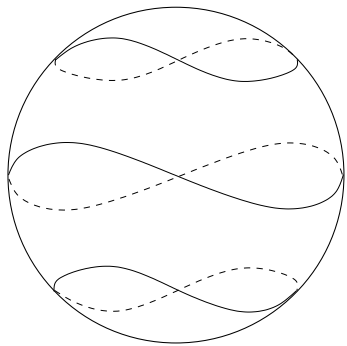}
\caption{Projections of the $J_1(P_t)$ into the $2$-sphere.}
\label{fig:Dk trick}
\end{figure}

Using Theorem~\ref{thm:reduce to fiber-preserving}, we may assume that
$f_u$ is actually an isometry of $L$ for each $u\in S^{d-1}$.  Denote the
isometry that $f_u$ induces on $S$ by $\overline{f_u}$.  Now, deform the
entire family $f$ by precomposing each $f_u$ with $J_t$. At points in
$S^{d-1}$, each $f_u\circ J_t$ is fiber-preserving, so this is an allowable
deformation of $f$.  At the end of the deformation, for each $u\in
S^{d-1}$, $f_u\circ J_1(P_s)$ is a fibered torus $Q_s$ that projects to
$\overline{f_u}(q_s)$. Since $\overline{f_u}$ is an isometry of $S$, it
follows that for any latitude circle $p_t$, some $\overline{f_u}(q_s)$
meets $p_t$ transversely, in either two or four points. So $P_t$ and this
$Q_s$ meet transversely in either two or four circles which are fibers of
$L$. In particular, they are in very good position. We call such a pair
$P_t$ and $Q_s$ at $u$ an \indexdef{instant pair}\textit{instant pair.}

Cover $S^{d-1}$ by finitely many open sets $Z_i'$ such that for each $i$,
there is an $(x_i,y_i)$ such that $Q_{x_i}$ and $P_{y_i}$ are an instant
pair at every point of $\overline{Z_i'}$. We may assume that there are open
sets $Z_i$ in $D^d$ such that $\overline{Z_i}\cap S^{d-1}=\overline{Z_i'}$
and $Q_{x_i}$ and $P_{y_i}$ meet in very good position at each point
of~$\overline{Z_i}$. For any sufficiently small deformation of $f$,
$Q_{x_i}$ and $P_{y_i}$ will still meet in very good position at all points
of~$\overline{Z_i}$. Let $V$ be a neighborhood of $S^{d-1}$ in $D^d$ such
that $\overline{V}$ is contained in the union of the~$Z_i$.

Now, we apply to $D^d$ the entire process used for the case when the
parameters lie in $S^d$, using appropriate fiber-preserving deformations at
parameters in $S^{d-1}$. Here are the steps:
\begin{enumerate}
\item By Theorem~\ref{thm:generalposition}, there are arbitrarily small
deformations of $f$ that put it in general position with respect to the
sweepout. Select the deformation sufficiently small so that the $Q_{x_i}$
and $P_{y_i}$ still meet in very good position at every point of
$\overline{Z_i}$. Within $V$, we taper the deformation off to the identity,
so that no change has taken place at parameters in $S^{d-1}$. At every
parameter, either there is already a pair in very good position, or $f_u$
satisfies the conditions 
\index{GP1@(GP1), (GP2), (GP3)}(GP1), (GP2), and~(GP3) of a general position
family.
\item
Theorem~\ref{thm:finding good levels} guarantees that at each of the
parameters in $D^d-V$, there is a pair $Q_s$ and $P_t$ meeting in good
position.
\item
Applying Theorem~\ref{thm:from good to very good} to $D^d$, with $S^{d-1}$
in the role of $W_0$, we find a deformation of $f$, fixed on $S^{d-1}$, and
a covering $U_i$ of $D^d$ and associated values $s_i$ so that for every
$u\in U_i$, $Q_{s_i}$ and $P_{t_i}$ meet in very good position, and
$Q_{s_i}$ has no discal intersection with any~$P_{t_j}$.
\item
In the pushout step of the proof of Theorem~\ref{thm:make
fiber-preserving}, we may assume that all the $U_i$ that meet $S^{d-1}$ are
the open sets $Z_i$. At parameters $u$ in $S^{d-1}$, the annuli to be
pushed out of each $V_{t_i}$ will be vertical annuli. The pushouts may be
performed using fiber-preserving isotopies at these parameters, because the
necessary deformations can be taken as lifts of deformations of circles in
the quotient sphere $S$, the lifting being possible by
Theorem~\ref{sfsquare}.
\item
After the triangulation of $D^d$ is chosen, the deformation that move the
$Q_{s_i}$ onto level tori can be performed using fiber-preserving isotopies
at parameters in $S^{d-1}$, again because the necessary deformations cover
deformations of circles in the quotient surface $S$. No further deformation
will be needed on simplices in $S^{d-1}$, since the $f_u$ are already
fiber-preserving there.
\end{enumerate}

This completes the discussion of the case of parameters in $D^d$, and the
proof of the Smale Conjecture for lens spaces.

\backmatter

\bibliography{ref}{}

\begin{thebibliography}{10}

\bibitem{ABBJRS}
C.~Aneziris, A.~P. Balachandran, M.~Bourdeau, S.~Jo, R.~D. Sorkin, and T.~R.
  Ramadas.
\newblock Aspects of spin and statistics in generally covariant theories.
\newblock {\em Internat. J. Modern Phys. A}, 4(20):5459--5510, 1989.

\bibitem{A}
Kouhei Asano.
\newblock Homeomorphisms of prism manifolds.
\newblock {\em Yokohama Math. J.}, 26, no. 1:19--25, 1978.

\bibitem{Banyaga}
Augustin Banyaga.
\newblock {\em The structure of classical diffeomorphism groups, {Mathematics
  and its Applications, 400}}.
\newblock Kluwer Academic Publishers Group, Dordrecht, 1997.

\bibitem{BP}
C.~Bessaga and Aleksander Pelczynski.
\newblock {\em Selected topics in infinite-dimensional topology, {Monografie
  Matematyczne, Tom 58. [Mathematical Monographs, Vol. 58]}}.
\newblock PWN---Polish Scientific Publishers, Warsaw, 1975.

\bibitem{Boileau-Otal}
Michel Boileau and Jean-Pierre Otal.
\newblock Scindements de {Heegaard} et groupe des homeotopies des petites
  varietes de {Seifert}.
\newblock {\em Invent. Math.}, 106, no. 1:85--107, 1991.

\bibitem{Bonahon}
Francis Bonahon.
\newblock Diff\'eotopies des espaces lenticulaires.
\newblock {\em Topology}, 22, no. 3:305--314, 1983.

\bibitem{B-W}
Glen~E. Bredon and John~W. Wood.
\newblock Non-orientable surfaces in orientable $3$-manifolds.
\newblock {\em Invent. Math}, 7:83--110, 1969.

\bibitem{Bruce}
J.~W. Bruce.
\newblock On transversality.
\newblock {\em Proc. Edinburgh Math. Soc. (2)}, 29, no. 1:115--123, 1986.

\bibitem{CG}
A.~J. Casson and C.~McA.Gordon.
\newblock Reducing {Heegaard} splittings.
\newblock {\em Topology Appl.}, 27, no. 3:275--283, 1987.

\bibitem{Cerf}
Jean Cerf.
\newblock Topologie de certains espaces de plongements.
\newblock {\em Bull. Soc. Math. France}, 89:227--380, 1961.

\bibitem{CerfS3}
Jean Cerf.
\newblock {\em Sur les diff\'eomorphismes de la sph\'ere de dimension trois
  $(\Gamma _{4}=0)$, {Lecture Notes in Mathematics, No. 53}}.
\newblock Springer-Verlag, Berlin-New York, 1968.

\bibitem{Charlap-Vasquez}
L.~S. Charlap and A.~T. Vasquez.
\newblock Compact flat riemannian manifolds. {III}. {T}he group of affinities.
\newblock {\em Amer. J. Math.}, 95:471--494, 1973.

\bibitem{Filip}
R.~P. Filipkiewicz.
\newblock Isomorphisms between diffeomorphism groups.
\newblock {\em Ergodic Theory Dynamical Systems 2}, no. 2:159--171, 1982.

\bibitem{F-S}
J.~Friedman and R.~Sorkin.
\newblock Spin 1/2 from gravity.
\newblock {\em Phys. Rev. Lett.}, 44:1100--1103, 1980.

\bibitem{Gabai}
David Gabai.
\newblock The {Smale Conjecture} for hyperbolic 3-manifolds: {${\rm
  Isom}(M^3)\simeq{\rm Diff}(M^3)$}.
\newblock {\em J. Differential Geom.}, 58, no. 1,:113--149, 2001.

\bibitem{GWLP}
Christopher~G. Gibson, Klaus Wirthm{\"u}ller, Andrew~A. du~Plessis, and Eduard
  J.~N. Looijenga.
\newblock {\em Topological stability of smooth mappings}.
\newblock Lecture Notes in Mathematics, Vol. 552. Springer-Verlag, Berlin,
  1976.

\bibitem{Giulini}
Domenico Giulini.
\newblock On the configuration space topology in general relativity.
\newblock {\em Helv. Phys. Acta}, 68, no. 1:86--111, 1995.

\bibitem{G-L}
Domenico Giulini and Jorma Louko.
\newblock No-boundary $\theta$ sectors in spatially flat quantum cosmology.
\newblock {\em Phys. Rev. D (3)}, 46, no. 10:4355--4364, 1992.

\bibitem{G}
Andre Gramain.
\newblock Le type d'homotopie du groupe des diff\'eomorphismes d'une surface
  compacte.
\newblock {\em Ann. Sci. \'Ecole Norm. Sup. (4)}, 6:53--66, 1973.

\bibitem{Hamilton}
Richard~S. Hamilton.
\newblock The inverse function theorem of {Nash and Moser}.
\newblock {\em Bull. Amer. Math. Soc. (N.S.)}, 7, no. 1:65--222, 1982.

\bibitem{Hantsche-Wendt}
W.~Hantsche and W.~Wendt.
\newblock Drei dimensionali {Euklidische Raumformen}.
\newblock {\em Math. Ann.}, 110:593--611, 1934.

\bibitem{H}
Allen~E. Hatcher.
\newblock Homeomorphisms of sufficiently large {$\mathbb{P}^{2}$-irreducible
  $3$-manifolds}.
\newblock {\em Topology}, 15, no. 4:343--347, 1976.

\bibitem{Hold}
Allen~E. Hatcher.
\newblock On the diffeomorphism group of {$S^{1}\times S^{2}$}.
\newblock {\em Proc. Amer. Math. Soc.}, 83, no. 2:427--430, 1981.

\bibitem{HSmale}
Allen~E. Hatcher.
\newblock A proof of the {Smale} conjecture, {${\rm Diff}(S^{3})\simeq {\rm
  O}(4)$}.
\newblock {\em Ann. of Math. (2)}, 117, no. 3:553--607, 1983.

\bibitem{Hnew}
Allen~E. Hatcher.
\newblock On the diffeomorphism group of {$S^{1}\times S^{2}$}.
\newblock revised version posted at
  http://www.math.cornell.edu/$_{\widetilde{\phantom{i}}}\,$hatcher/, 2007.

\bibitem{Hempel}
John Hempel.
\newblock {\em $3$-Manifolds}.
\newblock Ann. of Math. Studies, No. 86. Princeton University Press, Princeton,
  N. J, 1976.

\bibitem{Henderson}
David~W. Henderson.
\newblock Corrections and extensions of two papers about infinite-dimensional
  manifolds.
\newblock {\em General Topology and Appl.}, 1:321--327, 1971.

\bibitem{HS}
David~W. Henderson and R.~Schori.
\newblock Topological classification of infinite dimensional manifolds by
  homotopy type.
\newblock {\em Bull. Amer. Math. Soc.}, 76:121--124, 1970.

\bibitem{Hendriks}
Harrie Hendriks.
\newblock La stratification naturelle de l'espace des fonctions
  diff\'eren\-tiables re\'elles n'est pas la bonne.
\newblock {\em C. R. Acad. Sci. Paris Ser. A-B}, 274:A618--A620, 1972.

\bibitem{I}
C.~J. Isham.
\newblock Topological $\theta $-sectors in canonically quantized gravity.
\newblock {\em Phys. Lett. B}, 106, no. 3:188--192, 1981.

\bibitem{I3}
N.~V. Ivanov.
\newblock Groups of diffeomorphisms of {Waldhausen} manifolds, {Studies in
  topology, II}.
\newblock {\em Zap. Nau?n. Sem. Leningrad. Otdel. Mat. Inst. Steklov. (LOMI)},
  66:172--176, 209, 1976.

\bibitem{I1c}
N.~V. Ivanov.
\newblock Corrections: {``Homotopies} of automorphism spaces of some
  three-dimensional manifolds''.
\newblock {\em Dokl. Akad. Nauk SSSR}, 249, no. 6:1288, 1979.

\bibitem{I4}
N.~V. Ivanov.
\newblock Diffeomorphism groups of {Waldhausen} manifolds.
\newblock {\em J. Soviet Math.}, 12:115--118, 1979.

\bibitem{I1}
N.~V. Ivanov.
\newblock Homotopies of automorphism spaces of some three-di\-men\-sion\-al
  manifolds.
\newblock {\em Dokl. Akad. Nauk SSSR}, 244, no. 2:274--277, 1979.

\bibitem{I5}
N.~V. Ivanov.
\newblock Homotopy of spaces of diffeomorphisms of some three-dimen\-sional
  manifolds.
\newblock {\em Zap. Nauchn. Sem. Leningrad. Otdel. Mat. Inst. Steklov. (LOMI)},
  122:72--103, 164--165, 1982.

\bibitem{I2}
N.~V. Ivanov.
\newblock Homotopy of spaces of diffeomorphisms of some three-dimen\-sional
  manifolds.
\newblock {\em J. Soviet Math.}, 26:1646--1664, 1984.

\bibitem{Jaco}
William Jaco.
\newblock {\em Lectures on three-manifold topology {CBMS Regional Conference
  Series in Mathematics, 43}}.
\newblock American Mathematical Society, Providence, R.I., 1980.

\bibitem{JS}
William Jaco and Peter~B. Shalen.
\newblock {Seifert} fibered spaces in $3$-manifolds.
\newblock {\em Mem. Amer. Math. Soc.}, 21, no. 220:viii+192 pp., 1979.

\bibitem{Karcher}
H.~Karcher.
\newblock Riemannian center of mass and mollifier smoothing.
\newblock {\em Comm. Pure Appl. Math.}, 30(5):509--541, 1977.

\bibitem{KS}
Tsuyoshi Kobayashi and Osamu Saeki.
\newblock The {Rubinstein-Scharlemann} graphic of a 3-manifold as the
  discriminant set of a stable map.
\newblock {\em Pacific J. Math.}, 195, no. 1:101--156, 2000.

\bibitem{Kriegl-Michor}
Andreas Kriegl and Peter~W. Michor.
\newblock {\em The convenient setting of global analysis}.
\newblock Mathematical Surveys and Monographs, 53. American Mathematical
  Society, Providence, RI, 1997.

\bibitem{L}
Francois Laudenbach.
\newblock Topologie de la dimension trois: homotopie et isotopie.
\newblock {\em Ast\'erisque, Soci\'et\'e Math\'ematique de France, Paris}, No.
  12:i+152 pp, 1974.

\bibitem{F-W}
John L.Friedman and Donald~M. Witt.
\newblock Homotopy is not isotopy for homeomorphisms of $3$-manifolds.
\newblock {\em Topology}, 25, no. 1:35--44, 1986.

\bibitem{LW}
A.~Lundell and S.~Weingram.
\newblock {\em The Topology of CW Complexes}.
\newblock Van Nostrand Reinhold, Princeton, New Jersey, 1969.

\bibitem{Mather}
John~N. Mather.
\newblock Stability of $\mathrm{C}^{\infty }$ mappings. iii. {Finitely}
  determined mapgerms.
\newblock {\em Inst. Hautes Etudes Sci. Publ. Math.}, No. 35:279--308, 1968.

\bibitem{M}
Darryl McCullough.
\newblock Isometries of elliptic 3-manifolds.
\newblock {\em J. London Math. Soc. (2)}, 65, no. 1:167--182, 2002.

\bibitem{McC-Soma}
Darryl McCullough and Teruhiko Soma.
\newblock The {Smale} conjecture for {{Seifert}} fibered spaces with hyperbolic
  base orbifold.
\newblock arXiv:1005.5061, 2010.

\bibitem{N-R}
W.~Neumann and F.~Raymond.
\newblock Automorphisms of {Seifert} manifolds.
\newblock preprint, 1979.

\bibitem{Orlik}
Peter Orlik.
\newblock {\em Seifert Manifolds}.
\newblock Lecture Notes in Mathematics, Vol. 291. Springer-Verlag, Berlin-New
  York, 1972.

\bibitem{OVZ}
E.~Vogt P.~Orlik and H.~Zieschang.
\newblock Zur {Topologie} gefaserter dreidimensionaler {Mannigfaltigkeiten}.
\newblock {\em Topology}, 6:49--64, 1967.

\bibitem{P}
Richard~S. Palais.
\newblock Local triviality of the restriction map for embeddings.
\newblock {\em Comment. Math. Helv.}, 34:305--312, 1960.

\bibitem{Palais}
Richard~S. Palais.
\newblock Homotopy theory of infinite dimensional manifolds.
\newblock {\em Topology}, 5:1--16, 1966.

\bibitem{Park}
Chan-Young Park.
\newblock Homotopy groups of automorphism groups of some {Seifert} fiber
  spaces.
\newblock dissertation at the University of Michigan, 1989.

\bibitem{Park1}
Chan-Young Park.
\newblock On the weak automorphism group of a principal bundle, product case.
\newblock {\em Kyungpook Math. J.}, 31, no. 1:25--34, 1991.

\bibitem{P-R}
Jon~T. Pitts and J.~H. Rubinstein.
\newblock Applications of minimax to minimal surfaces and the topology of
  $3$-manifolds.
\newblock {\em Miniconference on geometry and partial differential equations,
  Canberra}, 2:137--170, 1987.

\bibitem{R2}
J.~H. Rubinstein.
\newblock On $3$-manifolds that have finite fundamental group and contain
  {Klein} bottles.
\newblock {\em Trans. Amer. Math. Soc.}, 251:129--137, 1979.

\bibitem{R-B}
J.~H. Rubinstein and J.~S. Birman.
\newblock One-sided {Heegaard} splittings and homeotopy groups of some
  $3$-manifolds.
\newblock {\em Proc. London Math. Soc. (3)}, 49, no. 3:517--536, 1984.

\bibitem{RS}
J.~H. Rubinstein and Martin Scharlemann.
\newblock Comparing {Heegaard} splittings of {non-Haken} $3$-manifolds.
\newblock {\em Topology}, 35, no. 4:1005--1026, 1996.

\bibitem{Sakuma}
Makoto Sakuma.
\newblock The geometries of spherical {Montesinos} links.
\newblock {\em Kobe J. Math.}, 7, no. 2:167--190, 1990.

\bibitem{Scott}
Peter Scott.
\newblock The geometries of $3$-manifolds.
\newblock {\em Bull. London Math. Soc.}, 15, no. 5:401--487, 1983.

\bibitem{Seeley}
R.~T. Seeley.
\newblock Extension of $\mathrm{C}^{\infty }$ functions defined in a half
  space.
\newblock {\em Proc. Amer. Math. Soc.}, 15:625--626, 1964.

\bibitem{Seifert}
H.~Seifert.
\newblock Topologie {Dreidimensionaler Gefaserter Raume}.
\newblock {\em Acta Math.}, 60, no. 1:147--238, 1933.

\bibitem{Sergeraert}
Francis Sergeraert.
\newblock Un th\'eor\`eme de fonctions implicites sur certains espaces de
  {Fr\'echet} et quelques applications.
\newblock {\em Ann. Sci. \'Ecole Norm. Sup. (4)}, 5:599--660, 1972.

\bibitem{Smale}
Stephen Smale.
\newblock Diffeomorphisms of the $2$-sphere.
\newblock {\em Proc. Amer. Math. Soc.}, 10:621--626, 1959.

\bibitem{S3}
Rafael~D. Sorkin.
\newblock Classical topology and quantum phases: quantum geons.
\newblock {\em Geometrical and algebraic aspects of nonlinear field theory},
  North-Holland Delta Ser., North-Holland, Amsterdam:201--218, 1989.

\bibitem{Takens}
Floris Takens.
\newblock Characterization of a differentiable structure by its group of
  diffeomorphisms.
\newblock {\em Bol. Soc. Brasil. Mat.}, 10, no. 1:17--25, 1979.

\bibitem{Tougeron1}
Jean-Claude Tougeron.
\newblock Une g\'en\`eralisation du th\'eor\`eme des fonctions implicites.
\newblock {\em C. R. Acad. Sci. Paris Ser. A-B}, 262:A487--A489, 1966.

\bibitem{Tougeron2}
Jean-Claude Tougeron.
\newblock Id\'eaux de fonctions diff\'erentiables. i.
\newblock {\em Ann. Inst. Fourier (Grenoble)}, 18, fasc. 1:177--240, 1968.

\bibitem{Waldhausen1}
Friedhelm Waldhausen.
\newblock Eine {Klasse} von $3$-dimensionalen {Mannigfaltigkeiten.} i, ii.
\newblock {\em Invent. Math.}, 3 ibid. 4:308--333; 87--117, 1967.

\bibitem{Waldhausen2}
Friedhelm Waldhausen.
\newblock Gruppen mit {Zentrum} und $3$-dimensionale {Mannigfaltigkeiten}.
\newblock {\em Topology}, 6:505--517, 1967.

\bibitem{Waldhausen}
Friedhelm Waldhausen.
\newblock On irreducible $3$-manifolds which are sufficiently large.
\newblock {\em Ann. of Math. (2)}, 87:56--88, 1968.

\bibitem{Wall}
C.~T.~C. Wall.
\newblock Finite determinacy of smooth map-germs.
\newblock {\em Bull. London Math. Soc.}, 13, no. 6:481--539, 1981.

\bibitem{W2}
Donald~M. Witt.
\newblock Symmetry groups of state vectors in canonical quantum gravity.
\newblock {\em J. Math. Phys.}, 27, no. 2:573--592, 1986.

\bibitem{Wolf}
Joseph~A. Wolf.
\newblock {\em Spaces of constant curvature, Third edition}.
\newblock Publish or Perish, Inc, Boston, Mass., 1974.

\end{thebibliography}

\bibliographystyle{plain}

\printindex

\end{document}